%% file: step-1.tex
\documentclass[12pt,twoside,reqno,openany]{amsart}
\usepackage{fullpage,times,url,amssymb,amsbsy,amsmath,amsthm,
amsfonts,amssymb,amscd,stmaryrd,graphics,color,xypic,footmisc,xy,
fancyhdr,multicol,fancybox,graphicx,mathrsfs,rotating,ifthen,
wasysym,shuffle,txfonts}
\usepackage{pifont}
\usepackage{centernot}
\usepackage{amsthm}
\usepackage{mathrsfs}
\usepackage{color}
\usepackage{extarrows}
\usepackage[colorlinks,linkcolor=red,anchorcolor=green,
citecolor=blue]{hyperref}
\usepackage{eurosym}
\usepackage{mathrsfs}
\usepackage{multirow}
\usepackage{arydshln}
\usepackage[T1]{fontenc}

\renewcommand{\dim}{\text{\footnotesize\sf dim}}
\renewcommand{\ker}{\text{\footnotesize\sf ker}}
\renewcommand{\min}{\text{\footnotesize\sf min}}
\renewcommand{\max}{\text{\footnotesize\sf max}}
\newcommand{\cdim}{\text{\footnotesize\sf codim}}
\newcommand{\rank}{\text{\footnotesize\sf rank}}

\newcommand{\vf}{\vfill\end{document}}
\newcommand{\explain}[1]{\text{\scriptsize\sf [#1]}}
\newcommand{\smallbullet}{{\scriptscriptstyle{\bullet}}}

\newcommand{\kPN}{
{}_k\mathbb{P}_{\mathbb K}^N
\rule[-0pt]{0pt}{12.5pt}^{\!\!\!\!\!\!\!\!\!\!\circ}
\,\,\,\,
}

\newcommand{\oP}{
{\mathbb{P}}
\rule[-0pt]{0pt}{12.5pt}^{\!\!\!\!\circ}
\,
}

\newcommand{\obfP}{
{\mathbf{P}^\prime}
\rule[-0pt]{0pt}{12.5pt}^{\!\!\!\!\!\!\! \circ}
\,\,
}

\newcommand{\hba}{$\overbrace{\rule[-0cm]{2cm}{0cm}}^{p}$}
\newcommand{\vpb}{$\left(\rule[-0cm]{0cm}{2cm}\right.$}
\newcommand{\vpc}{$\left.\rule[-0cm]{0cm}{2cm}\right)$}
\newcommand{\vbd}{$\left.\rule[-0cm]{0cm}{1cm}\right\}
\!{\scriptstyle{p}}$}
\newcommand{\vbe}{$\left.\rule[-0cm]{0cm}{1cm}\right\}
\!{\scriptstyle{p}}$}

\newcommand{\hbf}{$\overbrace{\rule[-0cm]{4cm}{0cm}}^{p}$}
\newcommand{\vpg}{$\left(\rule[-0cm]{0cm}{4cm}\right.$}
\newcommand{\vph}{$\left.\rule[-0cm]{0cm}{4cm}\right)$}
\newcommand{\vbi}{$\left.\rule[-0cm]{0cm}{2cm}\right\}
\!{\scriptstyle{p}}$}

\newcommand{\vbk}{$\left.\rule[-0cm]{0cm}{1cm}\right\}
\!{\scriptstyle{\tau}}$}
\newcommand{\vbl}{$\left.\rule[-0cm]{0cm}{1cm}\right\}
\!{\scriptstyle{p-\tau}}$}

\let\mathcal\mathscr

\newtheorem{The}{Theorem}[section]

\newtheorem{Theorem}[The]{Theorem}
\newtheorem{Proposition}[The]{Proposition}
\newtheorem{Lemma}[The]{Lemma}
\newtheorem{Corollary}[The]{Corollary}
\newtheorem{Observation}[The]{Observation}

\theoremstyle{definition}

\newtheorem{Definition}[The]{Definition}

\newtheorem{Remark}[The]{Remark}

\newtheorem{Conjecture}[The]{Conjecture}

\DeclareMathOperator{\Sym}{\mathrm Sym}

\makeatletter
\def\noqed{\let\QED@stack\@empty}
\makeatother

\definecolor{black}{cmyk}{1.,1.,1.,1.0}
\definecolor{blue}{cmyk}{1.,1.,0.,0.63}
\definecolor{red}{cmyk}{0.,1.,1.,0.63}
\definecolor{green}{cmyk}{1.,0.,1.,0.63}

\newcommand{\blue}{\textcolor{blue}}
\newcommand{\green}{\textcolor{green}}

\begin{document}

\title{
On the ampleness 
\\
of the cotangent bundles
\\
of complete intersections
}

\author{Song-Yan Xie}

\thanks{This work was supported by the {\sl Fondation Math\'ematique Jacques Hadamard} through the grant  N\textsuperscript{o} ANR-10-CAMP-0151-02 within the ``Programme des Investissements d'Avenir''.}

\address{Laboratoire de Math\'ematiques d'Orsay, Universit\'e Paris-Sud (France)}

\email{songyan.xie@math.u-psud.fr}



\maketitle

\begin{abstract}
Based on a geometric interpretation of Brotbek's symmetric differential forms, for the intersection family $\mathcal{X}$ of generalized
Fermat-type hypersurfaces in $\mathbb{P}_{\mathbb{K}}^N$ defined over any field $\mathbb K$, we construct\big/reconstruct
explicit symmetric differential forms by applying
Cramer's rule, skipping cohomology arguments, and we further exhibit
unveiled families of lower degree symmetric differential forms on all possible intersections
of $\mathcal{X}$ with coordinate hyperplanes.

Thereafter, we develop what we call the `{\sl moving coefficients
method}' to prove a conjecture made by Olivier Debarre: {\sl
for generic $c\geqslant N/2$ hypersurfaces
$H_1,\dots,H_c\subset \mathbb{P}_{\mathbb C}^N$ of degrees $d_1,\dots,d_c$
sufficiently large, the intersection $X:=H_1 \cap \cdots \cap H_c $
has ample cotangent bundle} $\Omega_X$, and concerning effectiveness, the lower bound
$
d_1,\dots,d_c\geqslant N^{N^2}
$  
works. 

Lastly, thanks to known results about the Fujita Conjecture, we establish the very-ampleness of 
$\mathsf{Sym}^{\kappa}\,\Omega_X$ for all
$\kappa\geqslant 
64\,
\Big(
\sum_{i=1}^c\,
d_i
\Big)^2
$.
\end{abstract}

\section{\bf Introduction}
In 2005,
Debarre
established that, in a complex abelian variety
of dimension $N$, for $c\geqslant N/2$ sufficiently ample
generic hypersurfaces $H_1, \dots, H_c$, their intersection
$
X 
:= 
H_1\cap\cdots\cap H_c
$
has ample cotangent bundle
$\Omega_X$, thereby answering a question of Lazarsfeld (cf.~\cite{Debarre-2005}). Then
naturally, by thoughtful analogies between geometry of Abelian varieties
and geometry of projective spaces, Debarre proposed the following conjecture
in Section~3 of~\cite{Debarre-2005}, extending in
fact an older question raised by Schneider~\cite{Schneider-1992} in
the surface case:

\begin{Conjecture}{\bf [Debarre Ampleness Conjecture]}
\label{Debarre Conjecture}
For all integers
$N\geqslant 2$, 
for every integer
$N/2\leqslant c < N$,
there exists a positive lower bound:
\[
d
\gg
1
\] 
such that, for all positive integers: 
\[
d_1,\dots,d_c 
\geqslant 
d,
\]
for generic choices of $c$ hypersurfaces: 
\[
H_i\,
\subset\,
\mathbb{P}_{\mathbb C}^N
\qquad
{\scriptstyle{(i\,=1\,\cdots\,c)}}
\] 
with degrees:
\[
\deg\,
H_i
=
d_i,
\] 
the intersection:
\[
X
:=
H_1
\cap
\cdots
\cap 
H_c
\]
has ample cotangent bundle $\Omega_X$.
\end{Conjecture}

Precisely, 
according to a ground conceptualization
due to Hartshorne~\cite{Hartshorne-1966},
the expected ampleness is that, for all large degrees $k
\geqslant \texttt{k}_0 \gg 1$, the global symmetric
$k$-differentials on $X$:
\[
\Gamma\,\big(X,\,{\sf Sym}^k\,\Omega_X\big)
\]
are so abundant and diverse, that firstly, at every point
$x \in X$, the first-order jet evaluation map:
\[
\Gamma\,
\big(
X,\,
{\sf Sym}^k\,\Omega_X
\big)\,
\twoheadrightarrow\,
{\sf Jet}_1\,
{\sf Sym}^k\,
\Omega_X\big\vert_x
\]
is surjective, where for every vector bundle 
$E \rightarrow X$ the first-order jet of $E$ at $x$ is defined by: 
\[
{\sf Jet}_1\,
E\big\vert_x
\,:=\,
\mathcal{O}_x(E)\,
\Big/\,
(\mathfrak{m}_x)^2\,\mathcal{O}_x(E),
\]
and that secondly, at every pair of distinct
points $x_1 \neq x_2$ in $X$, the simultaneous evaluation map:
\[
\Gamma\,
\big(
X,
{\sf Sym}^k\,
\Omega_X
\big)\,
\twoheadrightarrow\,
{\sf Sym}^k\,
\Omega_X\big\vert_{x_1}\,
\oplus\,
{\sf Sym}^k\,
\Omega_X\big\vert_{x_2}
\]
is also surjective.

The hypothesis:
\[
c\,
\geqslant\,
n
\]
appears {\em optimal}, for otherwise when $c<n$, there are no nonzero global sections for all degrees $k\geqslant 1$:
\[
\Gamma\,
\big(
X,
{\sf Sym}^k\,
\Omega_X
\big)\,
=\,
0,
\]
according to
Br\"uckmann-Rackwitz~\cite{ Bruckmann-Rackwitz-1990} and
Schneider~\cite{ Schneider-1992}, whereas, in the threshold  case $c=n$, nonzero global sections are known to exist.

As highlighted in~\cite{Debarre-2005}, projective  varieties $X$ having ample cotangent bundles enjoy several fascinating properties, for instance the following ones.

\begin{itemize}

\smallskip\item[{$\bullet$}]\,
All subvarieties of $X$ are all of general type.

\smallskip\item[{$\bullet$}]\,
There are finitely many nonconstant rational maps from any fixed 
projective variety to $X$ (\cite{Noguchi-Sunada-1982}).

\smallskip\item[{$\bullet$}]\,
If $X$ is defined over $\mathbb{C}$, then $X$ is Kobayashi-hyperbolic, i.e. every holomorphic map $\mathbb{C}\rightarrow X$ 
must be constant 
(\cite[p.~16, Proposition 3.1]{Demailly-1997},
\cite[p.~52, Proposition 4.2.1]{Diverio-Rousseau-2010}).

\smallskip\item[{$\bullet$}]\,
If $X$ is defined over a number field $K$, the set of
$K$-rational points of $X$ is expected to be finite
(Lang's conjecture, cf.~\cite{Lang-1986}, \cite{Moriwaki-1995}).
\end{itemize}

Since ampleness of cotangent bundles
potentially bridges Analytic Geometry and Arithmetic Geometry in a deep way, it is interesting to 
ask examples of such projective varieties.
In one-dimensional case, they are in fact our familar Riemann surfaces\big/algebraic curves with genus $\geqslant 2$. However, in higher dimensional case, not many examples were known, even though they were expected to be reasonable abundant. 

In this aspect, 
we would like to mention the following nice construction of Bogomolov, which is written down in the last section of~\cite{Debarre-2005}. 
 If $X_1,
\dots, X_\ell$ are smooth complex projective varieties
having positive dimensions:
\[
\dim\,X_i
\,\geqslant\,
d
\,\geqslant\,
1
\ \ \ \ \ \ \ \ \ 
{\scriptstyle{(i\,=1\,\cdots\,\ell)}},
\]
all of whose Serre line bundles $\mathcal{ O}_{\mathbb{P}(\mathrm{T}_{X_i})}(1)
\rightarrow \mathbb{P}(\mathrm{T}_{X_i})$ enjoy bigness:
\[
\dim\,
\Gamma\,
\big(\mathbb{P}(\mathrm{T}_{X_i}),\,
\mathcal{O}_{\mathbb{P}(\mathrm{T}_{X_i})}(k)
\big)
\,=\,
\dim\,
\Gamma\,\big(X_i,\,{\sf Sym}^k\Omega_{X_i}\big)
\underset{k\,\to\,\infty}{\,\geqslant\,}
\underbrace{\,{\sf constant}\,}_{>\,0}
\cdot\,
k^{2\,{\rm dim}\,X_i-1}, 
\]
then a generic complete intersection:
\[
Y 
\,\subset\,
X_1 
\times\cdots\times 
X_\ell
\] 
having dimension:
\[
\dim\,
Y
\,\leqslant\,
\frac{d\,(\ell+1)+1}{2\,(d+1)}
\]
has ample cotangent bundle $\Omega_Y$.

In his Ph.D. thesis under the direction of Mourougane,
Brotbek~\cite{Brotbek-2011-these} reached an elegant proof of the
Debarre Ampleness Conjecture in dimension $n = 2$, in  all codimensions $c
\geqslant 2$, for generic complete intersections $X^2 \subset \mathbb{
 P}^{ 2 + c} (\mathbb{C})$ having degrees:
\[
d_1,\dots,d_c\,
\geqslant\,
\frac{8\,(n+c)+2}{n+c-1},
\]
by extending the techniques of Siu~\cite{Siu-2002, Siu-2004, Siu-2012-arxiv}, 
Demailly~\cite{Demailly-1997, Demailly-El Goul-2000, Demailly-2011}, 
Rousseau~\cite{Rousseau-2007},
P{\u a}un~\cite{Paun-2008, Paun-2012}, 
 Merker~\cite{Merker-2009},
Diverio-Merker-Rousseau~\cite{DMR-2010},
Mourougane~\cite{Mourougane-2012}, 
and by employing the concept of {\sl ampleness modulo a
subvariety} introduced by Miyaoka in~\cite{Miyaoka-1983}.
Also, for smooth complete intersections $X^n \subset
\mathbb{P}^{n+c}(\mathbb{C})$ with $c \geqslant n \geqslant 2$, Brotbek showed
using holomorphic Morse inequalities that when:
\[
d_1,\dots,d_c\,\,
\geqslant\,\,
\bigg[
2^{n-1}\,
\big(2n-2\big)\,
\frac{n^2}{n+c+1}\,
\binom{2n-1}{n}
+1
\bigg]
\binom{n}{\lfloor\frac{n}{2}\rfloor}
\frac{(2n+c)!}{(n+c)!}\,
\frac{(c-n)!}{c!},
\]
bigness of the Serre line bundle $\mathcal{ O}_{\mathbb{P}(\mathrm{T}_X)}(1)
\rightarrow \mathbb{P}(\mathrm{T}_X)$ holds:
\[
\aligned
\dim\,
\Gamma\,\big(\mathbb{P}(\mathrm{T}_X),\,
\mathcal{O}_{\mathbb{P}(\mathrm{T}_X)}(k)\big)
\,=\,
\dim\,\Gamma\,\big(X,\,{\sf Sym}^k\,\Omega_X\big)
&
\underset{k\,\to\,\infty}{\geqslant\,}
{\textstyle{\frac{1}{2}}}\,
\chi_{\sf Euler}\big(X,\,{\sf Sym}^k\,\Omega_X\big)
\\
&
\underset{k\,\to\,\infty}{\geqslant\,}
\underbrace{\,{\sf constant}\,}_{>\,0}
\cdot
k^{2\,n-1},
\endaligned
\]
whereas a desirable control of the base locus of the
 inexplicitly given nonzero holomorphic sections seems impossible by means of currently available 
techniques.

\smallskip
To find an alternative approach,
a key breakthrough happened in 2014, 
when Brotbek~\cite{Brotbek-2014-arxiv} obtained  explicit
global symmetric differential forms in coordinates by an intensive 
 cohomological approach. 
More specifically, 
under the assumption that the ambient field $\mathbb{K}$ has characteristic zero,
using exact sequences and the snake lemma, Brotbek firstly provided a key series of long injective cohomology sequences, whose left initial ends consist of the most general global twisted symmetric differential forms, and whose right target ends consist of huge dimensional linear spaces well understood. Secondly, Brotbek proved that the image of each left end, going through 
the full injections sequence,
is exactly the kernel of a certain linear system at the right end. Thirdly, by focusing on pure Fermat-type hypersurface equations (\cite[p.~26]{Brotbek-2014-arxiv}):
\begin{equation}
\label{pure Fermat type hypersurfaces used by Brotbek}
F_j\,
=\,
\sum_{i=0}^N\,
s_i^j\,Z_i^e
\qquad
{\scriptstyle(j\,=\,1\,\cdots\,c)},
\end{equation}
with integers $c\geqslant N/2$, $e\geqslant 1$, where $s_i^j$ are 
some homogeneous polynomials of the same 
degree $\epsilon\geqslant 0$, 
Brotbek 
step-by-step traced back some kernel elements from each right end all the way to the left end, every middle step being an 
application of Cramer's rule, and hence he constructed global  twisted symmetric differential forms with  neat determinantal structures (\cite[p.~27--31]{Brotbek-2014-arxiv}). 

Thereafter, by employing the standard method of  counting base-locus-dimension in two ways in algebraic geometry (see e.g. Lemma~\ref{full-rank-c*(N+1)} below),
Brotbek 
established 
that the Debarre Ampleness Conjecture holds when:
\[
4c
\,\geqslant\,
3\,N-2,
\]
for equal degrees:
\begin{equation}
\label{Brotbek degree 2N+3}
d_1
=\cdots=
d_c
\,\geqslant\,
2N+3,
\end{equation}
the constructions being flexible enough to embrace `approximately
equal degrees', in the same sense as Theorem~\ref{Main Nefness Theorem 1/2} below.

\medskip

Inspired much by Brotbek's works, we propose the following answer to the Debarre Ampleness Conjecture.

\begin{Theorem}
 The cotangent bundle of the intersection in $\mathbb{ 
P}_{\mathbb{C}}^N$
 of at least $N/2$ generic hypersurfaces with degrees $\geqslant N^{N^2}$ is ample.
\end{Theorem}

In fact, we will prove
the following main theorem, which coincides with the above theorem for 
$r=0$ and $\mathbb{K}=\mathbb{C}$, and whose effective bound $\texttt{d}_0=N^{N^2}$ will be obtained in Theorem~\ref{Main Nefness Theorem with d_0=?}.

\begin{Theorem}
[\bf  Ampleness]
\label{Main Theorem}
Over any field $\mathbb{K}$ which is not finite,
for all positive integers
$N\geqslant 1$, 
for any nonnegative integers
$c,r\geqslant 0$ with:
\[
2c
+
r
\geqslant 
N,
\]
there exists a lower bound $\texttt{d}_0\gg 1$ such that,
for all positive integers: 
\[
d_1,
\dots,
d_c,
d_{c+1},
\dots,
d_{c+r}\, 
\geqslant\, 
\texttt{d}_0,
\]
for generic $c+r$ hypersurfaces:
\[
H_i\,
\subset\,
\mathbb{P}_{\mathbb K}^N
\qquad
{\scriptstyle{(i\,=1\,\cdots\,c+r)}}
\] 
with degrees:
\[
\deg\,
H_i
=
d_i,
\] 
the cotangent bundle 
$\Omega_V$ of the intersection of the first $c$ hypersurfaces:
\[
V
:=
H_1
\cap
\cdots
\cap 
H_c
\]
restricted 
to the intersection of all the $c+r$ hypersurfaces:
\[
X
:=
H_1
\cap
\cdots
\cap 
H_c
\cap
H_{c+1}
\cap
\cdots
\cap 
H_{c+r}
\]
is ample.
\end{Theorem}

First of all, remembering that 
{\em ampleness (or not) is preserved under any base change 
obtained by ambient field extension}, one only needs to prove the Ampleness Theorem~\ref{Main Theorem} for
algebraically closed fields $\mathbb{K}$.

Of course, we would like to have $\texttt{d}_0=\texttt{d}_0\,(N,c,r)$ as small as possible, yet the optimal one is at present far beyond our reach,
and we can only get exponential ones like: 
\[
\texttt{d}_0
=
N^{N^2}
\qquad
(\text{Theorem}~\ref{Main Nefness Theorem with d_0=?}),
\]
which confirms the large degree phenomena in 
Kobayashi hyperbolicity related problems (\cite{DMR-2010, Berczi-2010-arxiv, Demailly-2011, Lionel-2015, Siu-2012-arxiv}).  
When $2\,(2c+r)\geqslant 3N-2$, we obtain linear bounds for equal degrees:
\[
d_1
=
\cdots
=
d_{c+r}\,
\geqslant\,
2N
+
3,
\]
hence we recover the lower bounds~\thetag{\ref{Brotbek degree 2N+3}} in the case $r=0$, and we also obtain quadratic bounds for all large degrees:
\[
d_1,
\dots,
d_{c+r}\,
\geqslant\,
(3N+2)(3N+3).
\] 
Better estimates of the lower bound $\texttt{d}_0$ will be explained in Section~\ref{The construction of hypersurfaces revisit and better lower bounds}.
\medskip

Concerning the proof, primarily, as anticipated\big/emphasized by Brotbek and Merker (\cite{Brotbek-2014-arxiv, Merker-2013-arxiv}), it is essentially
based on constructing sufficiently many global negatively twisted symmetric differential forms, and then inevitably,
one has to struggle with the overwhelming difficulty of clearing out their base loci, which seems, at the best of our knowledge, to be an incredible mission.

In order to bypass the complexity in these two aspects, the following seven ingredients are indispensable in our approach:
\begin{itemize}
\smallskip
\item[\text{\ding{192}}]\
generalized Brotbek's symmetric differential forms
(Subsection~\ref{general-holomorphic-symmetric-forms}); 

\smallskip
\item[\text{\ding{193}}]\
global {\sl moving coefficients method} (MCM)
(Subsection~\ref{The-global-moving-coefficients-method});

\smallskip
\item[\text{\ding{194}}]\
`hidden' symmetric forms
on intersections with coordinate hyperplanes
(Subsection~\ref{Regular twisted symmetric differential forms with some vanishing coordinates});

\smallskip
\item[\text{\ding{195}}]\
MCM on intersections with coordinate hyperplanes
(Subsection~\ref{The-moving-coefficients-method-for-intersections-with-coordinate-hyperplanes});

\smallskip
\item[\text{\ding{196}}]\
Algorithm of MCM
(Subsection~\ref{subsection:constructing-algorithm});

\smallskip
\item[\text{\ding{197}}]\
Core Lemma of MCM
(Section~\ref{section: The engine of MCM});

\smallskip
\item[\text{\ding{198}}]\
product coup (Subsection~\ref{subsection: product coup}).
\end{itemize} 

\smallskip
In fact, \ding{192} is based on a geometric interpretation of Brotbek's 
symmetric differential forms (\cite[Lemma 4.5]{Brotbek-2014-arxiv}),
and has the advantage of producing symmetric differential forms by
directly copying hypersurface equations and their
differentials. Facilitated by \ding{193}, which is of certain
combinatorial interest, \ding{192} amazingly cooks a series of global
negatively twisted symmetric differential forms, which are of
nice uniform structures.  However, unfortunately, 
one still has the 
difficulty that all these obtained global symmetric forms happen to
coincide  with each other on
the intersections with any two coordinate hyperplanes, so that their base
locus stably keeps positive (large)
dimension, which is an annoying obstacle to ampleness.

Then, to overcome this difficulty
enters \ding{194}, which is arguably
the most critical ingredient in harmony with MCM, and whose
importance is much greater than its appearance as
somehow a corollary of \ding{192}. Thus, to compensate the defect of
\ding{192}-\ding{193}, it is natural to design \ding{195} which
completes the framework of MCM. And then, \ding{196} is smooth to be
devised, and it provides suitable hypersurface equations for MCM.
Now, the last obstacle to amplness is about
narrowing the base loci, an ultimate difficulty
solved by \ding{197}. Thus, the Debarre Conjecture is settled in the
central cases of almost equal degrees. Finally, the magical coup \ding{198}
thereby embraces all large degrees for the Debarre Conjecture, and
naturally shapes the formulation of the Ampleness Theorem.

\medskip
Lastly, taking account of known results
about the Fujita Conjecture in Complex Geometry (cf. survey~\cite{Demailly-1998}), we will prove in Section~\ref{Proof of the very ampleness theorem}
the following

\begin{Theorem}
[\bf Effective Very Ampleness]
\label{Very-Ampleness Theorem}
Under the same assumption and notation as in the Ampleness Theorem~\ref{Main Theorem}, if in addition the ambient field $\mathbb{K}$ has
characteristic zero, then for generic choices of $H_1, \dots, H_{c+r}$, 
the restricted cotangent bundle $\mathsf{Sym}^{\kappa}\,\Omega_V\big{\vert}_X$ is very ample on $X$, for every $\kappa\geqslant \kappa_0$, with the uniform lower bound:
\[
\kappa_0\,
=\,
16\,
\Big(
\sum_{i=1}^{c}\,d_i
+
\sum_{i=1}^{c+r}\,d_i
\Big)^2.
\]
\end{Theorem}

\smallskip
In the end, we would like to propose the following

\begin{Conjecture}
\label{Xie Conjecture 2015}
{\bf (i)}
Over an algebraically closed field $\mathbb{K}$,
for any smooth projective $\mathbb{K}$-variety ${\bf P}$ with dimension $N$,
for any integers
$c,r\geqslant 0$ with
$
2c
+
r
\geqslant 
N
$,
for any very ample line bundles $\mathcal{L}_1,\dots,\mathcal{L}_{c+r}$ on ${\bf P}$,
there exists a lower bound:
\[
\texttt{d}_0\,
=\,
\texttt{d}_0\,({\bf P}, \mathcal{L}_{\bullet})\
\gg\
1
\] 
such that,
for all positive integers: 
\[
d_1,
\dots,
d_c,
d_{c+1},
\dots,
d_{c+r}\, 
\geqslant\, 
\texttt{d}_0,
\]
for generic choices of $c+r$ hypersurfaces:
\[
H_i\,
\subset\,
{\bf P}
\qquad
{\scriptstyle{(i\,=1\,\cdots\,c+r)}}
\] 
defined by global sections:
\[
F_i\
\in\
\mathsf{H}^0\,
\big(
{\bf P},\,
\mathcal{L}_{i}^{\otimes\,d_i}
\big),
\] 
the cotangent bundle 
$\Omega_V$ of the intersection of the first $c$ hypersurfaces:
\[
V
:=
H_1
\cap
\cdots
\cap 
H_c
\]
restricted 
to the intersection of all the $c+r$ hypersurfaces:
\[
X
:=
H_1
\cap
\cdots
\cap 
H_c
\cap
H_{c+1}
\cap
\cdots
\cap 
H_{c+r}
\]
is ample.

\smallskip
\noindent
{\bf (ii)}
There exists a uniform lower bound:
\[
\texttt{d}_0\,
=\,
\texttt{d}_0\,({\bf P})\
\gg\
1
\] 
independent of the chosen very ample
line bundles $\mathcal{L}_\bullet$.

\smallskip
\noindent
{\bf (iii)}
There exists a uniform lower bound:
\[
\kappa_{0}\,
=\,
\kappa_0\,({\bf P})\
\gg\
1
\] 
independent of $d_1,\dots,d_{c+r}$,
such that for generic choices of 
$H_1,\dots,H_{c+r}$,
the restricted cotangent bundle $\mathsf{Sym}^{\kappa}\,\Omega_V\big{\vert}_X$ is very ample on $X$, for every $\kappa\geqslant \kappa_0$.
\end{Conjecture}

\medskip
\noindent {\bf Acknowledgements.}  First of all, I would like to
express my sincere gratitude to my thesis advisor Jo\"el Merker, for
his suggestion of the Debarre Ampleness Conjecture, for his exemplary
guidance, patience, encouragement, and notably for his invaluable
instructions in writing and in \LaTeX, as well as for his prompt help
in the introduction.  

Mainly, I deeply express my debt to the recent
breakthrough~\cite{Brotbek-2014-arxiv} by Damian Brotbek.  Especially,
I thank him for explaining and sharing his ideas during a private
seminar in Merker's office on 16 October 2014; even beyond the Debarre
conjecture, it became clear that Damian Brotbek's approach will give
higher order jet differentials on certain classes of hypersurfaces.
Next, I would like to thank Junjiro Noguchi for the interest he showed
on the same occasion, when I presented for the first time the 
moving coefficients method which led to this article.

Moreover, I thank Zhi Jiang, Yang Cao, Yihang Zhu and Zhizhong Huang for helpful
discussions. Also, I thank Lionel Darondeau, Junyan Cao,
Huynh Dinh Tuan and Tongnuo Wei for their encouragement.

Lastly, I heartily thank Olivier Debarre for suggesting the title, 
and Nessim Sibony, Junjiro Noguchi, Damian Brotbek, Lie Fu, Yihang Zhu, Zhizhong Huang for useful remarks. 

\section{\bf Preliminaries and Restatements of the Ampleness Theorem~\ref{Main Theorem}}

\subsection{Two families of hypersurface intersections in $\mathbb{P}_{\mathbb{K}}^N$}
Fix an arbitrary algebraically closed field  $\mathbb{K}$.
Now, we introduce the fundamental object of this paper:
the intersection family of $c+r$ hypersurfaces with degrees
$d_1,\dots,d_{c+r}\geqslant 1$ in the $\mathbb{K}$-projective space $\mathbb{P}_{\mathbb{K}}^N$ of dimension $N$, equipped with homogeneous coordinates
$[z_0:z_1:\cdots:z_N]$. 

Recalling that the projective parameter space of degree $d\geqslant 1$ hypersurfaces in $\mathbb{P}_{\mathbb{K}}^N$ is:
\[
\mathbb{P}\,
\Big(
\underbrace{
H^0
\big(
\mathbb{P}_{\mathbb{K}}^N, \mathcal{O}_{\mathbb{P}_{\mathbb{K}}^N}(d)
\big)
}_{
\dim_\mathbb{K}
=
\binom{N+d}{N}
}
\Big)\,
=\,
\mathbb{P}\,
\Big{\{}
\sum_{|\alpha|=d}\,
A_\alpha\,z^{\alpha}
\colon
A_\alpha
\in
\mathbb{K}
\Big{\}},
\]
we may denote by:
\[
\mathbb{P}\,
\Big(
\underbrace{
\oplus_{i=1}^{c+r}\,
H^0
\big(
\mathbb{P}_{\mathbb{K}}^N, \mathcal{O}_{\mathbb{P}_{\mathbb{K}}^N}(d_i)
\big)
}_{
\dim_\mathbb{K}
\,
=
\,
\sum_{i=1}^{c+r}\,\binom{N+d_i}{N}
}
\Big)
=
\mathbb{P}\,
\Big\{
\oplus_{i=1}^{c+r}\,
\sum_{|\alpha|=d_i}
A_\alpha^i\,z^{\alpha}\,
\colon\,
A_\alpha^i
\in
\mathbb{K}
\Big{\}}
\]
the projective parameter space of $c+r$ hypersurfaces with 
degrees $d_1,\dots,d_{c+r}$. 
This $\mathbb{K}$-projective space has dimension: 
\begin{equation}
\label{diamond=?}
\diamondsuit
:=
\sum_{i=1}^{c+r}\,
\binom{N+d_i}{N}
-1,
\end{equation}
hence we write it as:
\begin{equation}
\label{diamond-projective space}
\mathbb{P}_{\mathbb{K}}^{\diamondsuit}
=
\mathrm{Proj}\,
\mathbb{K}
\big[
\{
A_{\alpha}^i
\}
_{
\substack{
1\leqslant i \leqslant c+r
\\
|\alpha|=d_i
}}
\big],
\end{equation}
where, as shown above, $A_{\alpha}^i$ are the homogeneous coordinates indexed by the serial number $i$ of each hypersurface and by all multi-indices $\alpha$ with the weight $|\alpha|=d_i$ associated to the degree $d_i$ monomials 
$z^\alpha \in \mathbb{K}[z_0,\dots,z_N]$.

Now, we introduce the two subschemes:
\[
\mathcal{X}\,
\subset\,
\mathcal{V}\,
\subset\,
\mathbb{P}_{\mathbb{K}}^{\diamondsuit}
\times_{\mathbb{K}}
\mathbb{P}_{\mathbb{K}}^N,
\]
where $\mathcal{X}$ is defined by `all' the $c+r$ 
bihomogeneous polynomials: 
\begin{equation}
\label{universal intersecion X}
\mathcal{X}
:=
\mathrm{V}
\Big(
\sum_{|\alpha|=d_1}\,
A_\alpha^1\,z^{\alpha},
\dots,
\sum_{|\alpha|=d_{c}}\,
A_\alpha^{c}\,z^{\alpha},
\sum_{|\alpha|=d_{c+1}}\,
A_\alpha^{c+1}\,z^{\alpha},
\dots,
\sum_{|\alpha|=d_{c+r}}\,
A_\alpha^{c+r}\,z^{\alpha}
\Big),
\end{equation}
and where $\mathcal{V}$ is defined by
the `first' $c$ bihomogeneous polynomials:
\begin{equation}
\label{universal intersections V}
\mathcal{V}
:=
\mathrm{V}
\Big(
\sum_{|\alpha|=d_1}\,
A_\alpha^1\,z^{\alpha},
\dots,
\sum_{|\alpha|=d_c}\,
A_\alpha^c\,z^{\alpha}
\Big).
\end{equation}

Then we view
$
\mathcal{X},
\mathcal{V}
\subset
\mathbb{P}_{\mathbb{K}}^{\diamondsuit}
\times_{\mathbb{K}}
\mathbb{P}_{\mathbb{K}}^N
$
as two families of closed subschemes of $\mathbb{P}_{\mathbb{K}}^N$ parametrized by the projective parameter space $\mathbb{P}_{\mathbb{K}}^{\diamondsuit}$.

\subsection{The relative cotangent sheaves family of $\mathcal{V}$}
A comprehensive reference on sheaves of relative differentials is \cite[Section $6.1.2$]{Liu-2002}.

Let $\mathrm{pr}_1, \mathrm{pr}_2$ be the two canonical projections: 
\begin{equation}
\label{pr1-pr2}
\xymatrix{
&
\mathbb{P}_{\mathbb{K}}^{\diamondsuit}
\times_{\mathbb{K}}
\mathbb{P}_{\mathbb{K}}^N
\ar[ld]_-{\mathrm{pr}_1} \ar[rd]^-{\mathrm{pr}_2}
\\
\mathbb{P}_{\mathbb{K}}^{\diamondsuit}
& 
&  
\mathbb{P}_{\mathbb{K}}^N.
}
\end{equation}
Then, by composing with the subscheme inclusion: 
\[
i
\colon\ \ \
\mathcal{V}\,
\hookrightarrow\,
\mathbb{P}_{\mathbb{K}}^{\diamondsuit}
\times_{\mathbb{K}}
\mathbb{P}_{\mathbb{K}}^N,
\] 
we receive a morphism:
\[
\mathrm{pr}_1
\circ
i
\colon
\ \ \
\mathcal{V}
\longrightarrow
\mathbb{P}_{\mathbb{K}}^{\diamondsuit},
\]
together with a sheaf $\Omega_{\mathcal{V}/\mathbb{P}_{\mathbb{K}}^{\diamondsuit}}^1$   of relative differentials of degree $1$
 of $\mathcal{V}$ over $\mathbb{P}_{\mathbb{K}}^{\diamondsuit}$.
 
Since $\mathrm{pr}_1$ is of finite type and $\mathbb{P}_{\mathbb{K}}^{\diamondsuit}$ is noetherian, a standard theorem (\cite[p.~216, Proposition~$1.20$]{Liu-2002})
shows that the sheaf $\Omega_{\mathcal{V}/\mathbb{P}_{\mathbb{K}}^{\diamondsuit}}^1$
is coherent. 

We may view
$\Omega_{\mathcal{V}/\mathbb{P}_{\mathbb{K}}^{\diamondsuit}}^1$ as the family of the cotangent bundles for the intersection family $\mathcal{V}$, 
since the coherent sheaf  
$\Omega_{\mathcal{V}/\mathbb{P}_{\mathbb{K}}^{\diamondsuit}}^1$
is indeed locally free on the Zariski open set that consists of smooth complete intersections. 

\subsection{The projectivizations and the Serre line bundles}
We refer the reader to \cite[pp.~160-162]{Hartshorne-1977}
for the considerations in this subsection.

Starting with the noetherian scheme $\mathcal{V}$ and the coherent  degree $1$ relative differential sheaf  $\Omega_{\mathcal{V}/\mathbb{P}
_{\mathbb{K}}^{\diamondsuit}}^1$
on it, we consider the sheaf of relative 
$\mathcal{O}_\mathcal{V}$-symmetric differential algebras:
\[
\mathbf{Sym}^{\bullet}\,\Omega_{\mathcal{V}/\mathbb{P}_{\mathbb{K}}^{\diamondsuit}}^1
:=
\bigoplus_{i\geqslant 0}\,
\mathsf{Sym}^{i}\,
\Omega_{\mathcal{V}/\mathbb{P}_{\mathbb{K}}^{\diamondsuit}}^1.
\]

According to the construction of \cite[p.~160]{Hartshorne-1977}, 
noting that this sheaf has a natural structure of graded $\mathcal{O}_{\mathcal{V}}$-algebras, and moreover that it satisfies the condition $(\dag)$ there, we receive 
the projectivization of 
$\Omega_{\mathcal{V}/\mathbb{P}_{\mathbb{K}}^{\diamondsuit}}^1$:
\begin{equation}\label{Proj( Omega_cal(X) )}
\mathbf{P}
\big(
\Omega_{\mathcal{V}/\mathbb{P}_{\mathbb{K}}^{\diamondsuit}}^1
\big)
:=
\mathbf{Proj}\,
\big(
\mathbf{Sym}^{\bullet}\,\Omega_{\mathcal{V}/\mathbb{P}_{\mathbb{K}}^{\diamondsuit}}^1
\big).
\end{equation}
As described in \cite[p.~160]{Hartshorne-1977}, 
$\mathbf{P}(\Omega_{\mathcal{V}/\mathbb{P}_{\mathbb{K}}^{\diamondsuit}}^1)$
is naturally equipped with the so-called 
{\sl Serre line bundle} $\mathcal{O}_{\mathbf{P}(\Omega_{\mathcal{V}/\mathbb{P}_{\mathbb{K}}^{\diamondsuit}}^1)}(1)$ on it. 

Similarly, replacing $\mathcal{V}$ by 
$\mathbb{P}_{\mathbb{K}}^{\diamondsuit}
\times_{\mathbb{K}}
\mathbb{P}_{\mathbb{K}}^N$, we obtain the relative differentials sheaf of 
$\mathbb{P}_{\mathbb{K}}^{\diamondsuit}
\times_{\mathbb{K}}
\mathbb{P}_{\mathbb{K}}^N$
with respect to $\mathrm{pr_1}$ in \thetag{\ref{pr1-pr2}}:
\[
\Omega_{
\mathbb{P}_{\mathbb{K}}^{\diamondsuit}
\times_{\mathbb{K}}
\mathbb{P}_{\mathbb{K}}^N
/\mathbb{P}_{\mathbb{K}}^{\diamondsuit}}^1\,
\cong\,
\mathrm{pr}_2^{*}\,
\Omega_{\mathbb{P}_{\mathbb{K}}^N}^1,
\]
and we thus obtain its projectivization:
\begin{equation}
\label{Proj(...)=P^diamond*Proj(...)}
\mathbf{P}
\big(
\Omega_{
\mathbb{P}_{\mathbb{K}}^{\diamondsuit}
\times_{\mathbb{K}}
\mathbb{P}_{\mathbb{K}}^N
/\mathbb{P}_{\mathbb{K}}^{\diamondsuit}}^1
\big)\,
:=\,
\mathbf{Proj}\,
\big(
\mathbf{Sym}^{\bullet}\,\Omega_{
\mathbb{P}_{\mathbb{K}}^{\diamondsuit}
\times_{\mathbb{K}}
\mathbb{P}_{\mathbb{K}}^N
/\mathbb{P}_{\mathbb{K}}^{\diamondsuit}}^1
\big)\,
\cong\,
\mathbb{P}_{\mathbb{K}}^{\diamondsuit}
\times_{\mathbb{K}}
\mathbf{Proj}\,
\big(
\mathbf{Sym}^{\bullet}\,\Omega_{\mathbb{P}_{\mathbb{K}}^N}^1
\big).
\end{equation}
We will abbreviate 
$
\mathbf{Proj}\,
\big(
\mathbf{Sym}^{\bullet}\,\Omega_{\mathbb{P}_{\mathbb{K}}^N}^1
\big)
$
as 
$\mathbf{P}(\Omega_{\mathbb{P}_{\mathbb{K}}^{N}}^1)$, and denote its Serre line bundle by $\mathcal{O}_{\mathbf{P}(\Omega_{\mathbb{P}_{\mathbb{K}}^{N}}^1)}(1)$.
Then, the Serre line bundle $\mathcal{O}_{\mathbf{P}(\Omega_{
\mathbb{P}_{\mathbb{K}}^{\diamondsuit}
\times_{\mathbb{K}}
\mathbb{P}_{\mathbb{K}}^N
/\mathbb{P}_{\mathbb{K}}^{\diamondsuit}}^1)}(1)$ on the left hand side of~\thetag{\ref{Proj(...)=P^diamond*Proj(...)}} is nothing but the line bundle
$
\widetilde{\pi}_2^{*}\,
\mathcal{O}_{\mathbf{P}(\Omega_{\mathbb{P}_{\mathbb{K}}^{N}}^1)}(1)
$
on the right hand side, where $\widetilde{\pi}_2$ is the canonical
 projection:
\begin{equation}
\label{tilde pi_2}
\widetilde{\pi}_2
\colon\ \ \
\mathbb{P}_{\mathbb{K}}^{\diamondsuit}
\times_{\mathbb{K}}
\mathbf{P}(\Omega_{\mathbb{P}_{\mathbb{K}}^{N}}^1)
\rightarrow
\mathbf{P}(\Omega_{\mathbb{P}_{\mathbb{K}}^{N}}^1).
\end{equation}

Now, note that the commutative diagram:
\[
\xymatrix{
\mathcal{V}\, 
\ar@{^{(}->}[r]^{i}
\ar[rd]_{\mathrm{pr_1\circ i}} 
&\,
\mathbb{P}_{\mathbb{K}}^{\diamondsuit}
\times_{\mathbb{K}}
\mathbb{P}_{\mathbb{K}}^N
\ar[d]^{\mathrm{pr_1}}
\\
&
\mathbb{P}_{\mathbb{K}}^{\diamondsuit}}
\]
induces the surjection (cf. \cite[p.~176, Proposition 8.12]{Hartshorne-1977}):
\[
i^*\,
\Omega_{
\mathbb{P}_{\mathbb{K}}^{\diamondsuit}
\times_{\mathbb{K}}
\mathbb{P}_{\mathbb{K}}^N
/\mathbb{P}_{\mathbb{K}}^{\diamondsuit}}^1\,
\twoheadrightarrow\,
\Omega_{\mathcal{V}/\mathbb{P}_{\mathbb{K}}^{\diamondsuit}}^1,
\]
and hence yields the surjection:
\[
i^*\,
\mathbf{Sym}^{\bullet}\,
\Omega_{
\mathbb{P}_{\mathbb{K}}^{\diamondsuit}
\times_{\mathbb{K}}
\mathbb{P}_{\mathbb{K}}^N
/\mathbb{P}_{\mathbb{K}}^{\diamondsuit}}^1\,
\twoheadrightarrow\,
\mathbf{Sym}^{\bullet}\,
\Omega_{\mathcal{V}/\mathbb{P}_{\mathbb{K}}^{\diamondsuit}}^1.
\]
Taking `$\mathbf{Proj}$',
thanks to~\thetag{\ref{Proj(...)=P^diamond*Proj(...)}},
we obtain the commutative diagram:
\begin{equation}
\label{key commutative diagram 1}
\xymatrix{
\mathbf{P}(\Omega_{\mathcal{V}/\mathbb{P}_{\mathbb{K}}^{\diamondsuit}}^1)\,
\ar[d] 
\ar@{^{(}->}[r]^{\widetilde{i}} 
&\,
\mathbb{P}_{\mathbb{K}}^{\diamondsuit}
\times_{\mathbb{K}}
\mathbf{P}(\Omega_{\mathbb{P}_{\mathbb{K}}^{N}}^1)
\ar[d]\\
\mathcal{V}\, 
\ar@{^{(}->}[r]^i
&\,
\mathbb{P}_{\mathbb{K}}^{\diamondsuit}
\times_{\mathbb{K}}
\mathbb{P}_{\mathbb{K}}^N.
}
\end{equation}
Thus, the Serre line bundle
$\mathcal{O}_{\mathbf{P}(\Omega_{\mathcal{V}/\mathbb{P}_{\mathbb{K}}^{\diamondsuit}}^1)}(1)$
becomes exactly the pull back of `the Serre line bundle'  
$
\widetilde{\pi}_2^{*}\,
\mathcal{O}_{\mathbf{P}(\Omega_{\mathbb{P}_{\mathbb{K}}^{N}}^1)}(1)
$
under the inclusion $\widetilde{i}$:
\begin{equation}
\label{Serre line bundle = another Serre line bundle}
\mathcal{O}_{\mathbf{P}(\Omega_{\mathcal{V}/\mathbb{P}_{\mathbb{K}}^{\diamondsuit}}^1)}(1)\,
=\,
\widetilde{i}^*
\big(
\widetilde{\pi}_2^{*}\,
\mathcal{O}_{\mathbf{P}(\Omega_{\mathbb{P}_{\mathbb{K}}^{N}}^1)}(1)
\big)\,
=\,
(\widetilde{\pi}_2\circ \widetilde{i})^*\,
\mathcal{O}_{\mathbf{P}(\Omega_{\mathbb{P}_{\mathbb{K}}^{N}}^1)}(1).
\end{equation}

\subsection{Restatement of Theorem~\ref{Main Theorem}}
\label{subsection: Restatement of Main Theorem}
Let $\widetilde\pi$ be the canonical projection:
\[
\widetilde\pi
\colon\ \ \
\mathbb{P}_{\mathbb{K}}^{\diamondsuit}
\times_{\mathbb{K}}
\mathbf{P}(\Omega_{\mathbb{P}_{\mathbb{K}}^{N}}^1)
\rightarrow
\mathbb{P}_{\mathbb{K}}^{\diamondsuit}
\times_{\mathbb{K}}
\mathbb{P}_{\mathbb{K}}^N,
\]
and let $\pi_1,\pi_2$ be the compositions of $\widetilde\pi$ with $\mathrm{pr_1},\mathrm{pr_2}$:

\begin{equation}
\label{pi_1-proper}
\xymatrix{
& 
\mathbb{P}_{\mathbb{K}}^{\diamondsuit}
\times_{\mathbb{K}}
\mathbf{P}(\Omega_{\mathbb{P}_{\mathbb{K}}^{N}}^1)
\ar[ldd]_{\pi_1:=\mathrm{pr}_1\circ\widetilde\pi}
\ar[rdd]^{\pi_2:=\mathrm{pr}_2\circ\widetilde\pi}
\ar[d]^{\!\widetilde\pi} 
& 
\\
& 
\mathbb{P}_{\mathbb{K}}^{\diamondsuit}
\times_{\mathbb{K}}
\mathbb{P}_{\mathbb{K}}^N
\ar[ld]^{\!\!\mathrm{pr_1}} 
\ar[rd]_{\mathrm{pr_2}\!\!} 
& 
\\
\mathbb{P}_{\mathbb{K}}^{\diamondsuit} 
&&
\mathbb{P}_{\mathbb{K}}^N.
}
\end{equation}
Let:
\begin{equation}
\label{key background inclusions}
\mathbf{P}\,
:=\,
\widetilde\pi^{-1}
(\mathcal{X})\,
\cap
\mathbf{P}(\Omega_{\mathcal{V}/\mathbb{P}_{\mathbb{K}}^{\diamondsuit}}^1)\
\subset\
\mathbf{P}(\Omega_{\mathcal{V}/\mathbb{P}_{\mathbb{K}}^{\diamondsuit}}^1)\
\subset\
\mathbb{P}_{\mathbb{K}}^{\diamondsuit}
\times_{\mathbb{K}}
\mathbf{P}(\Omega_{\mathbb{P}_{\mathbb{K}}^{N}}^1)
\end{equation}
be `the pullback' of:
\[
\mathcal{X}\,
\subset\,
\mathcal{V}\,
\subset\,
\mathbb{P}_{\mathbb{K}}^{\diamondsuit}
\times_{\mathbb{K}}
\mathbb{P}_{\mathbb{K}}^N
\]
under the map $\widetilde\pi$, and let:
\begin{equation}
\label{key Serre line bundles equality}
\mathcal{O}_{\mathbf{P}}(1)\,
:=\,
\mathcal{O}_{\mathbf{P}(\Omega_{\mathcal{V}/\mathbb{P}_{\mathbb{K}}^{\diamondsuit}}^1)}(1)
\big\vert
_{\mathbf{P}}\,
=\,
\widetilde{\pi}_2^*\,
\mathcal{O}_{\mathbf{P}(\Omega_{\mathbb{P}_{\mathbb{K}}^{N}}^1)}(1)
\big\vert
_{\mathbf{P}}
\qquad
\explain{see~\thetag{\ref{Serre line bundle = another Serre line bundle}}}
\end{equation}
be the restricted Serre line bundle.

Now, we may view $\mathbf{P}$
as a family of subschemes of $\mathbf{P}(\Omega_{\mathbb{P}_{\mathbb{K}}^{N}}^1)$ parametrized by the projective parameter space
$\mathbb{P}_{\mathbb{K}}^{\diamondsuit}$
under the restricted map:
\begin{equation}
\label{second pi_1}
\pi_1
\colon
\ \ \
\mathbf{P}
\longrightarrow
\mathbb{P}_{\mathbb{K}}^{\diamondsuit}.
\end{equation}
Thus
Theorem~\ref{Main Theorem} can be reformulated as below,
with the assumption that the hypersurface degrees $d_1,\dots,d_{c+r}$ are sufficiently large:
\[
d_1,
\dots,
d_{c+r}
\gg
1.
\]

\smallskip
\noindent
\textbf{Theorem~\ref{Main Theorem} (Version A).}
{\em
For a generic point $t\in \mathbb{P}_{\mathbb{K}}^{\diamondsuit}$,
over the fibre: 
\[
\mathbf{P}_t
:=
\pi_1^{-1}(t)\,
\cap
\mathbf{P},
\]
the restricted Serre line bundle:
\begin{equation}
\label{O_(P_i)(1)}
\mathcal{O}_{\mathbf{P}_t}(1)
:=
\mathcal{O}_{\mathbf{P}}(1)
\big\vert
_{\mathbf{P}_t}
\end{equation}
is ample.
}

\smallskip

From now on, every closed point:
\[
t
=
\Big[
\{
A_{\alpha}^i
\}
_{\substack
{
1
\leqslant
i
\leqslant
c+r
\\
|\alpha|
=
d
}}
\Big]\
\in \
\mathbb{P}_{\mathbb{K}}^{\diamondsuit}
\]
will be abbreviated as:
\[
t\,
=\,
[F_1\colon\cdots\colon F_{c+r}],
\]
where:
\[
F_i
:=
\sum_{|\alpha|=d_i}\,
A_\alpha^i\,z^{\alpha}
\qquad
{\scriptstyle (i\,=\,1\,\cdots\,c+r)}.
\]
Then
we 
have:
\[
\mathbf{P}_t\,
=\,
\{t\}\,
\times_{\mathbb{K}}\,
{}_{F_{c+1},\dots,F_{c+r}}\mathbf{P}_{F_1,\dots,F_c},
\]
for a uniquely defined subscheme:
\begin{equation}
\label{_...P_...}
{}_{F_{c+1},\dots,F_{c+r}}\mathbf{P}_{F_1,\dots,F_c}\
\subset\
\mathbf{P}
\big(
\Omega_{\mathbb{P}_{\mathbb{K}}^{N}}^1
\big).
\end{equation}

\smallskip
\noindent
\textbf{Theorem~\ref{Main Theorem} (Version B).}
{\em
For a generic closed point: 
\[
[F_1\colon\cdots\colon F_{c+r}]\
\in \
\mathbb{P}_{\mathbb{K}}^{\diamondsuit},
\]
the Serre line bundle
$\mathcal{O}_{\mathbf{P}(\Omega_{\mathbb{P}_{\mathbb{K}}^{N}}^1)}(1)$
is ample on
${}_{F_{c+1},\dots,F_{c+r}}\mathbf{P}_{F_1,\dots,F_c}$.
}

To have a better understanding of the above statements,
we now investigate the geometry behind.

\section{\bf The background geometry}
\label{subsection: The background geometry}
Since $\mathbb{K}$ is an algebraically closed field,
throughout this section, we view each scheme in the classical sense (cf. \cite[Chapter 1]{Hartshorne-1977}), i.e. its underlying topological space ($\mathbb{K}$-variety) consists of all the closed points. 

\subsection{The geometry of $\mathbb{P}_{\mathbb{K}}^N$ and 
$\mathcal{O}_{\mathbb{P}_{\mathbb{K}}^N}(1)$}
Recall that, the projective $N$-space $\mathbb{P}_{\mathbb{K}}^N$ is
obtained by projectivizing the Euclidian $(N+1)$-space $\mathbb{K}^{N+1}$, i.e. is defined as the set of lines passing through the origin:
\begin{equation}
\label{N-projective space = (N+1)-Euclidian space / relation}
\mathbb{P}_{\mathbb{K}}^N
:=
\mathbb{P}
\big(
\mathbb{K}^{N+1}
\big)
:=
\mathbb{K}^{N+1}
\big\backslash 
\{0\}
\ \
\big/\,
\sim,
\end{equation}
where the quotient relation $\sim$ for 
$z\,\in\mathbb{K}^{N+1}\backslash\{0\}$ is:
\[
z
\sim
\lambda z
\qquad
{\scriptstyle
(
\forall\,\lambda\,\in\,\mathbb{K}^{\times}
)
}.
\]

On $\mathbb{P}_{\mathbb{K}}^N$, there is the so-called 
{\sl tautological line bundle} $\mathcal{O}_{\mathbb{P}_{\mathbb{K}}^N}(-1)$, which at every point $[z]\in \mathbb{P}_{\mathbb{K}}^N$ has fibre:
\[
\mathcal{O}_{\mathbb{P}_{\mathbb{K}}^N}(-1)
\big\vert_{[z]}\,
:=\,
\mathbb{K}
\cdot
z
\
\subset\
\mathbb{K}^{N+1}.
\]
Its dual line bundle is the well known:
\[
\mathcal{O}_{\mathbb{P}_{\mathbb{K}}^N}(1)
:=
\mathcal{O}_{\mathbb{P}_{\mathbb{K}}^N}(-1)^{\vee}.
\]

\subsection{The geometry of $\mathbf{P}(\Omega_{\mathbb{P}_{\mathbb{K}}^{N}}^1)$
and $\mathcal{O}_{\mathbf{P}(\Omega_{\mathbb{P}_{\mathbb{K}}^{N}}^1)}(1)$}
For every point $[z]\in
\mathbb{P}_{\mathbb{K}}^N$, the tangent space of $\mathbb{P}_{\mathbb{K}}^N$ at $[z]$ is:
\[
\mathrm{T}_{\mathbb{P}_{\mathbb{K}}^N}
\big{\vert}
_{[z]}
=
\mathbb{K}^{N+1}
\big\slash
\mathbb{K}\cdot z,
\]
and the total tangent space of $\mathbb{P}_{\mathbb{K}}^N$:
\begin{equation}
\label{total tangent space of the N-projective space}
\mathrm{T}_{\mathbb{P}_{\mathbb{K}}^N}
:=
\mathrm{T}_{\mathrm{hor}}\mathbb{K}^{N+1}
\big/
\sim,
\end{equation}
is the quotient space of the horizontal tangent space of
$\mathbb{K}^{N+1}\setminus \{0\}$:
\begin{equation}
\label{horizontal tangent space of (N+1)-Euclidian space}
\mathrm{T}_{\mathrm{hor}}\mathbb{K}^{N+1}
:=
\Big\{
(z,[\xi])\,\colon \,z\in \mathbb{K}^{N+1}
\setminus \{0\} \text{ and }
[\xi]\in \mathbb{K}^{N+1}
\big\slash
\mathbb{K}\cdot z 
\Big\},
\end{equation}
by the quotient relation $\sim:$
\[
(z,[\xi])
\sim
(\lambda z,[\lambda \xi])
\qquad
{
\scriptstyle
(
\forall\,\lambda\,\in\,\mathbb{K}^{\times}
)
}.
\]

\begin{center}
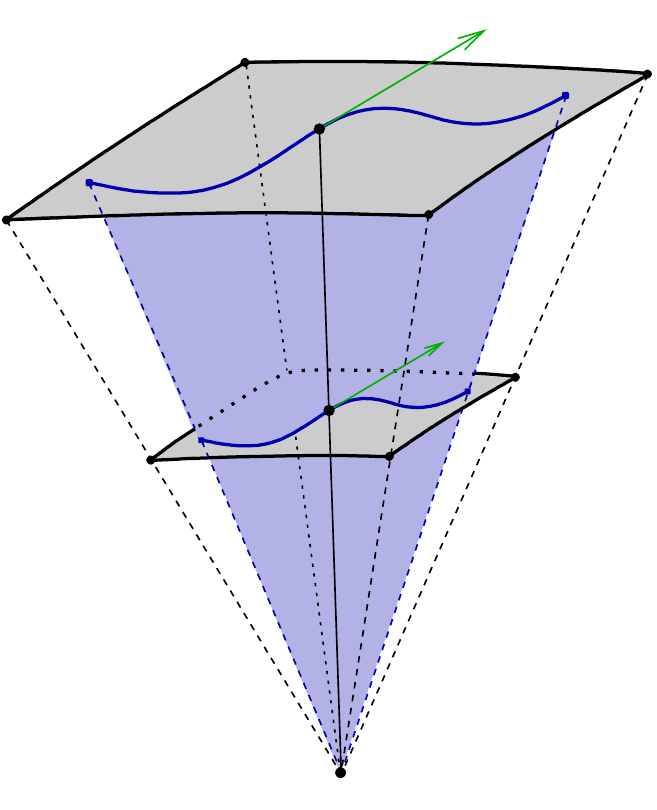
\end{center}

Now, the $\mathbb K$-variety associated to 
$\mathbf{P}(\Omega_{\mathbb{P}_{\mathbb{K}}^{N}}^1)$ is just the projectivized tangent space 
$\mathbb{P}(\mathrm{T}_{\mathbb{P}_{\mathbb{K}}^N})$, which is obtained by projectivizing 
each tangent space $\mathrm{T}_{\mathbb{P}_{\mathbb{K}}^N}\big\vert_{[z]}$
at every point $[z]\in \mathbb{P}_{\mathbb{K}}^N$:
\[
\mathbb{P}(\mathrm{T}_{\mathbb{P}_{\mathbb{K}}^N})
\big{\vert}
_{[z]}
:=
\mathbb{P}
\big(
\mathrm{T}_{\mathbb{P}_{\mathbb{K}}^N}
\big{\vert}
_{[z]}
\big).
\]
And the Serre line bundle 
$\mathcal{O}_{\mathbf{P}(\Omega_{\mathbb{P}_{\mathbb{K}}^{N}}^1)}(1)$ on $\mathbf{P}(\Omega_{\mathbb{P}_{\mathbb{K}}^{N}}^1)$ corresponds to the `Serre line bundle' 
$\mathcal{O}_{\mathbb{P}(\mathrm{T}_{\mathbb{P}_{\mathbb{K}}^N})}(1)$ on $\mathbb{P}(\mathrm{T}_{\mathbb{P}_{\mathbb{K}}^N})$,
which after restricting on $\mathbb{P}(\mathrm{T}_{\mathbb{P}_{\mathbb{K}}^N})
\big{\vert}
_{[z]}$ becomes 
$\mathcal{O}_{\mathbb{P}(\mathrm{T}_{\mathbb{P}_{\mathbb{K}}^N}
\mid_{[z]})}(1)$.
In other words, the Serre line bundle 
$\mathcal{O}_{\mathbb{P}(\mathrm{T}_{\mathbb{P}_{\mathbb{K}}^N})}(1)$ is the dual of the tautological line bundle 
$\mathcal{O}_{\mathbb{P}(\mathrm{T}_{\mathbb{P}_{\mathbb{K}}^N})}(-1)$, where the latter one, at every point
$([z],[\xi])\in \mathbb{P}(\mathrm{T}_{\mathbb{P}_{\mathbb{K}}^N})$, has fibre:
\[
\mathcal{O}_{\mathbb{P}(\mathrm{T}_{\mathbb{P}_{\mathbb{K}}^N})}(-1)
\big{\vert}
_{
([z],[\xi])
}\,
:=\,
\mathbb{K}
\cdot
[\xi]\ 
\subset\
\mathrm{T}_{\mathbb{P}_{\mathbb{K}}^N}
\big{\vert}
_{[z]}
=
\mathbb{K}^{N+1}\,
\big/\,
\mathbb{K}
\cdot
z.
\]

\subsection{The geometry of $\mathbf{P}(\Omega_{\mathcal{V}/\mathbb{P}_{\mathbb{K}}^{\diamondsuit}}^1)$, $\mathbf{P}$ and $\mathbf{P}_t$}
Recalling~\thetag{\ref{universal intersections V}}, the $\mathbb{K}$-variety $\mathbb{V}$ associated to 
$\mathcal{V}
\subset 
\mathbb{P}_{\mathbb{K}}^{\diamondsuit}
\times_{\mathbb{K}}
\mathbb{P}_{\mathbb{K}}^N$
is:
\[
\mathbb{V}\,
:=\,
\Big\{
\big(
[F_1,\dots,F_{c+r}],[z]
\big)
\in
\mathbb{P}_{\mathbb{K}}^{\diamondsuit}
\times
\mathbb{P}_{\mathbb{K}}^N\
\colon\
F_i(z)=0,\,
\forall\,
i=1\cdots c\,
\Big\}.
\]
Moreover, recalling~\thetag{\ref{key commutative diagram 1}}, 
the $\mathbb{K}$-variety: 
\[
\mathbb{P}(\mathrm{T}_{\mathbb{V}/\mathbb{P}_{\mathbb{K}}^{\diamondsuit}})
\subset
\mathbb{P}_{\mathbb{K}}^{\diamondsuit}
\times
\mathbb{P}(\mathrm{T}_{\mathbb{P}_{\mathbb{K}}^N})
\]
associated to
$\mathbf{P}(\Omega_{\mathcal{V}/\mathbb{P}_{\mathbb{K}}^{\diamondsuit}}^1)
\subset
\mathbb{P}_{\mathbb{K}}^{\diamondsuit}
\times_{\mathbb{K}}
\mathbf{P}(\Omega_{\mathbb{P}_{\mathbb{K}}^{N}}^1)$
is:
\[
\mathbb{P}(\mathrm{T}_{\mathbb{V}/\mathbb{P}_{\mathbb{K}}^{\diamondsuit}})\,
:=\,
\Big\{
\big(
[F_1,\dots,F_{c+r}],([z],[\xi])
\big)\
\colon\
F_i(z)
=
0,
dF_i
\big{\vert}_z(\xi)
=
0,
\forall\,
i=1\cdots c\,
\Big\}.
\]
Similarly, the $\mathbb{K}$-variety: 
\[
\mathbb{P}\
\subset\
\mathbb{P}(\mathrm{T}_{\mathbb{V}/\mathbb{P}_{\mathbb{K}}^{\diamondsuit}})
\]
associated to $\mathbf{P}\subset\mathbf{P}(\Omega_{\mathcal{V}/\mathbb{P}_{\mathbb{K}}^{\diamondsuit}}^1)$ is:
\[
\mathbb{P}\,
:=\,
\Big\{
\big(
[F_1,\dots,F_{c+r}],([z],[\xi])
\big)\
\colon\
F_i(z)
=
0,
dF_j\big{\vert}_z(\xi)
=
0,
\forall\,
i=1\cdots c+r,
\forall\,
j=1\cdots c\,
\Big\},
\]
and the $\mathbb{K}$-variety:
\[
{}_{F_{c+1},\dots,F_{c+r}}\mathbb{P}_{F_1,\dots,F_c}\
\subset\
\mathbb{P}(\mathrm{T}_{\mathbb{P}_{\mathbb{K}}^N})
\]
associated to~\thetag{\ref{_...P_...}} is:
\begin{equation}
\label{closed points of _(...)P_(...)}
{}_{F_{c+1},\dots,F_{c+r}}\mathbb{P}_{F_1,\dots,F_c}\,
:=\,
\Big\{
([z],[\xi])\
\colon\
F_i(z)
=
0,
dF_j\big{|}_z(\xi)
=
0,
\forall\,
i=1\cdots c+r,
\forall\,
j=1\cdots c\,
\Big\}.
\end{equation}
Now, the $\mathbb{K}$-variety $\mathbb{P}_t$ of $\mathbf{P}_t$ is:
\[
\mathbb{P}_t\,
:=\,
\{t\}
\times
{}_{F_{c+1},\dots,F_{c+r}}\mathbb{P}_{F_1,\dots,F_c}.
\] 

\section{\bf Some hints on the Ampleness Theorem~\ref{Main Theorem}}
\label{Some hints of the Debarre Conjecture}
The first three Subsections~\ref{Ampleness is Zariski open}--\ref{Nefness of negative twisted cotangent sheaf suffices}
consist of some standard knowledge in algebraic geometry,
and the last Subsection~\ref{subsection: A practical nefness criterion} presents a helpful nefness criterion which suits our  moving coefficients method.

\subsection{Ampleness is Zariski open}
\label{Ampleness is Zariski open}
The foundation of our approach is the following classical theorem of Grothendieck (cf. \cite[III.4.7.1]{EGA-III} or \cite[p.~29, Theorem 1.2.17]{Lazarsfeld-2004}).

\begin{Theorem}
{\bf[Amplitude in families]}\, 
Let $f: X \rightarrow T$ be a proper morphism of schemes, and let $\mathcal{L}$ be a line bundle on $X$. For every point $t \in T$, denote by:
\[
X_t
:=
f^{-1}(t),
\ \ \ \
\mathcal{L}_t:=\mathcal{L}\big{\vert}_{X_t}.
\]
Assume that, for some point $0$ $\in T$, $\mathcal{L}_0$ is ample on $X_0$ . Then in $T$, there is a Zariski open set $U$ containing $0$ such that $\mathcal{L}_t$ is ample on $X_t$, for all $t\in U$.
\end{Theorem}

Note that in \thetag{\ref{pi_1-proper}}, 
$\pi_1={\rm pr}_1\circ \pi$ is a composition of two proper morphisms, hence is proper, and so is \thetag{\ref{second pi_1}}. Therefore, by virtue of the above theorem, we only need to find one (closed) point $t\in \mathbb{P}_{\mathbb{K}}^{\diamondsuit}$ such that:
\begin{equation}
\label{ample-example}
\mathcal{O}_{\mathbf{P}_t}(1)
\text{ is ample on } 
\mathbf{P}_t.
\end{equation}

\subsection{Largely twisted Serre line bundle is ample}
Let: 
\[
\pi_0
\colon\ \ \
\mathbf{P}(\Omega_{\mathbb{P}_{\mathbb{K}}^{N}}^1)
\rightarrow
\mathbb{P}_{\mathbb{K}}^{N}
\]
be the canonical projection.
\cite[p.~161, Proposition 7.10]{Hartshorne-1977} yields that,
for all sufficiently large integer $\ell$
\footnote{ In fact, $\ell\geqslant 3$ is enough, see~\thetag{\ref{l=3 is OK}} below.}, the twisted line bundle below is ample on $\mathbf{P}(\Omega_{\mathbb{P}_{\mathbb{K}}^{N}}^1)$:
\begin{equation}
\label{what is the minimum l}
\mathcal{O}_{\mathbf{P}(\Omega_{\mathbb{P}_{\mathbb{K}}^{N}}^1)}
(1)\,
\otimes\,
\pi_0^*\,
\mathcal{O}_{\mathbb{P}_{\mathbb{K}}^{N}}(\ell).
\end{equation}

Recalling~\thetag{\ref{tilde pi_2}} and \thetag{\ref{pi_1-proper}},
and noting that:
\begin{equation}
\label{pi_2=pi_0 * widetilde pi_2}
\pi_2\,
=\,
\pi_0
\circ
\widetilde{\pi}_2,
\end{equation}
for the following ample line bundle $\mathcal{H}$ on the scheme $\mathbb{P}_{\mathbb{K}}^{\diamondsuit}
\times_{\mathbb{K}}
\mathbb{P}_{\mathbb{K}}^N$:
\[
\mathcal{H}
:=
\mathrm{pr}_1^{*}\,
\mathcal{O}_{\mathbb{P}_{\mathbb{K}}^{\diamondsuit}}(1)
\,
\otimes
\,
\mathrm{pr}_2^{*}\,
\mathcal{O}_{\mathbb{P}_{\mathbb{K}}^{N}}(1),
\] 
the twisted line bundle below is ample on 
$\mathbb{P}_{\mathbb{K}}^{\diamondsuit}
\times_{\mathbb{K}}
\mathbf{P}(\Omega_{\mathbb{P}_{\mathbb{K}}^{N}}^1)$: 
\begin{align}
\widetilde{\pi}_2^*\,
\mathcal{O}_{\mathbf{P}(\Omega_{\mathbb{P}_{\mathbb{K}}^{N}}^1)}(1)\,
\otimes\,
\widetilde\pi^{*}
\mathcal{H}^\ell
&
=
\widetilde{\pi}_2^*\,
\mathcal{O}_{\mathbf{P}(\Omega_{\mathbb{P}_{\mathbb{K}}^{N}}^1)}(1)\,
\otimes\,
\widetilde\pi^{*}
\Big(
\mathrm{pr_1}^{*}
\mathcal{O}_{\mathbb{P}_{\mathbb{K}}^{\diamondsuit}}(\ell)
\,
\otimes
\,
\mathrm{pr_2}^{*}
\mathcal{O}_{\mathbb{P}_{\mathbb{K}}^{N}}(\ell)
\Big)
\notag
\\
&
=
\widetilde{\pi}_2^*\,
\mathcal{O}_{\mathbf{P}(\Omega_{\mathbb{P}_{\mathbb{K}}^{N}}^1)}(1)\,
\otimes\,
(\mathrm{pr_1}\circ \widetilde\pi)^{*}
\mathcal{O}_{\mathbb{P}_{\mathbb{K}}^{\diamondsuit}}(\ell)
\,
\otimes
\,
(\mathrm{pr_2}\circ \widetilde\pi)^{*}
\mathcal{O}_{\mathbb{P}_{\mathbb{K}}^{N}}(\ell)
\notag
\\
\explain{use~\thetag{\ref{pi_1-proper}}}
\ \ \ \ \ \ \ \ \ \ \
&
=
\widetilde{\pi}_2^*\,
\mathcal{O}_{\mathbf{P}(\Omega_{\mathbb{P}_{\mathbb{K}}^{N}}^1)}(1)\,
\otimes\,
\pi_1^{*}
\mathcal{O}_{\mathbb{P}_{\mathbb{K}}^{\diamondsuit}}(\ell)
\,
\otimes
\,
\pi_2^{*}
\mathcal{O}_{\mathbb{P}_{\mathbb{K}}^{N}}(\ell)
\label{ample twisted Serre line bundle}
\\
\explain{use~\thetag{\ref{pi_2=pi_0 * widetilde pi_2}}}
\ \ \ \ \ \ \ \ \ \ \
&
=
\widetilde{\pi}_2^*\,
\mathcal{O}_{\mathbf{P}(\Omega_{\mathbb{P}_{\mathbb{K}}^{N}}^1)}(1)\,
\otimes\,
\pi_1^{*}
\mathcal{O}_{\mathbb{P}_{\mathbb{K}}^{\diamondsuit}}(\ell)
\,
\otimes
\,
\widetilde{\pi}_2^{*}\,
\big(
\pi_0^*\,
\mathcal{O}_{\mathbb{P}_{\mathbb{K}}^{N}}(\ell)
\big)
\notag
\\
&
=
\pi_1^{*}
\underbrace{
\mathcal{O}_{\mathbb{P}_{\mathbb{K}}^{\diamondsuit}}(\ell)
}_{\text{ample}}
\,
\otimes
\,
\widetilde{\pi}_2^*\,
\big(
\underbrace{
\mathcal{O}_{\mathbf{P}(\Omega_{\mathbb{P}_{\mathbb{K}}^{N}}^1)}(1)\,
\otimes\,
\pi_0^*\,
\mathcal{O}_{\mathbb{P}_{\mathbb{K}}^{N}}(\ell)
}_{\text{ample on}~\mathbf{P}(\Omega_{\mathbb{P}_{\mathbb{K}}^{N}}^1)}
\big).
\notag
\end{align}
In particular, for every point $t\in \mathbb{P}_{\mathbb{K}}^{\diamondsuit}$, 
recalling~\thetag{\ref{key Serre line bundles equality}}, \thetag{\ref{O_(P_i)(1)}}, restricting~\thetag{\ref{ample twisted Serre line bundle}}
to the subscheme: 
\[
\mathbf{P}_t\,
=\,
\pi_1^{-1}(t)
\cap
\mathbf{P}, 
\]
we receive an ample line bundle:
\begin{equation}
\label{ample-twisted}
\mathcal{O}_{\mathbf{P_t}}(1)
\otimes
\!\!
\underbrace{
\pi_1^{*}
\mathcal{O}_{\mathbb{P}_{\mathbb{K}}^{\diamondsuit}}(\ell)
}_{
\text{trivial line bundle}
}
\!\!\otimes\,
\pi_2^{*}
\mathcal{O}_{\mathbb{P}_{\mathbb{K}}^{N}}(\ell)
=
\mathcal{O}_{\mathbf{P_t}}(1)
\otimes
\pi_2^{*}
\mathcal{O}_{\mathbb{P}_{\mathbb{K}}^{N}}(\ell).
\end{equation}

\subsection{Nefness of negatively twisted cotangent sheaf suffices}
\label{Nefness of negative twisted cotangent sheaf suffices}
As we mentioned at the end of Subsection~\ref{Ampleness is Zariski open}, our goal is to show the existence of one such
(closed) point $t\in \mathbb{P}_{\mathbb{K}}^{\diamondsuit}$
satisfying~\thetag{\ref{ample-example}}.
In fact, we can relax this requirement thanks to the following theorem.

\begin{Theorem}
\label{nefness proves Debarre}
For every point
$t\in \mathbb{P}_{\mathbb{K}}^{\diamondsuit}$,
the following properties are equivalent. 
\begin{itemize}
\smallskip
\item[\bf (i)]
$\mathcal{O}_{\mathbf{P}_t}(1)$ 
is ample on 
$\mathbf{P}_t$.

\smallskip
\item[\bf (ii)]
There exist two positive integers $a, b\geqslant 1$ such that
$\mathcal{O}_{\mathbf{P_t}}(a)
\otimes
\pi_2^{*}
\mathcal{O}_{\mathbb{P}_{\mathbb{K}}^{N}}(-b)$
is ample on  
$\mathbf{P}_t$.

\smallskip
\item[\bf (iii)]
There exist two positive integers $a, b\geqslant 1$ such that
$\mathcal{O}_{\mathbf{P_t}}(a)
\otimes
\pi_2^{*}
\mathcal{O}_{\mathbb{P}_{\mathbb{K}}^{N}}(-b)$
is nef on  
$\mathbf{P}_t$.
\end{itemize}
\end{Theorem}

\begin{proof}
It is clear that {\bf (i)} $\implies$ {\bf (ii)} $\implies$ {\bf (iii)}, and we now show that 
{\bf (iii)} $\implies$ {\bf (i)}.

In fact, the nefness of the negatively twisted Serre line bundle:
\begin{equation}
\label{nef-twisted}
\mathcal{S}_t^a(-b)
:=
\mathcal{O}_{\mathbf{P_t}}(a)
\otimes
\pi_2^{*}
\mathcal{O}_{\mathbb{P}_{\mathbb{K}}^{N}}(-b)
\end{equation}
implies that:
\[
\underbrace{
\thetag{\ref{ample-twisted}}^{\otimes\,b}
}_{
\text{ample}
}\,\,
\otimes\,\,
\underbrace{
\thetag{\ref{nef-twisted}}^{\otimes\,\ell}
}_{
\text{nef}
}\,\,
=\,\,
\underbrace{
\mathcal{O}_{\mathbf{P_t}}(b+a\,\ell)
}_{
\text{ample !}
}
\,\,
=\,\,
\underbrace{
\mathcal{O}_{\mathbf{P_t}}(1)
}_{
\text{ample}
}
{}^{\otimes\,(b\,+\,a\,\ell)}
\]
is also ample, because of
the well known fact that
``ample $\otimes$ nef $=$ ample'' (cf. \cite[p.~53, Corollary 1.4.10]{Lazarsfeld-2004}).
\end{proof}

By definition, the nefness of~\thetag{\ref{nef-twisted}}
means that for every irreducible curve
$C\subset \mathbf{P}_t$, the intersection number
$C\cdot \mathcal{S}_t^a(-b)$ is $\geqslant 0$.
Recalling now the classical result~\cite[p.~295, Lemma 1.2]{Hartshorne-1977}, we only need to show that the line bundle $\mathcal{S}_t^a(-b)$ has a nonzero section on the curve $C$:
\begin{equation}
\label{section-not-empty}
H^0
\big(
C,
\mathcal{S}_t^a(-b)
\big)
\neq
\{0\}.
\end{equation}
To this end, of course we like to construct sufficiently many global sections:
\[
s_1,\dots,s_{m}\,
\in\,
H^0
\big(
\mathbf{P}_t,
\mathcal{S}_t^a(-b)
\big)
\]
such that their base locus is empty or discrete, whence 
one of
$
s_1\big{\vert}_C,
\dots,
s_m\big{\vert}_C
$
suffices to conclude~\thetag{\ref{section-not-empty}}.

More flexibly, we have:
\begin{Theorem}
\label{a=max(...), b=min(...)}
Suppose that there exist $m\geqslant 1$ nonzero sections
of certain negatively twisted Serre line bundles:
\[
s_i\ 
\in\ 
H^0
\big(
\mathbf{P}_t,
\mathcal{S}_t^{a_i}(-b_i)
\big)
\qquad
{\scriptstyle(i\,=\,1\,\cdots\,m;\,a_i,\,b_i\,\geqslant\,1)}
\]
such that their base locus is discrete or empty:
\[
\dim\,
\cap_{i=1}^m\,
\mathsf{BS}\,
(s_i)\,
\leqslant\,
0,
\]
then for all positive integers $a, b$ with: 
\[
{\textstyle{
\frac
{a}
{b}}}
\geqslant 
\max\,
\big\{
{\textstyle{
\frac
{a_1}
{b_1}}},
\dots,
{\textstyle{
\frac
{a_m}
{b_m}}}
\big\},
\] 
the twisted Serre
line bundle $\mathcal{S}_t^{a}(-b)$ is nef.
\end{Theorem}

\begin{proof}
For every irreducible curve 
$C\subset \mathbf{P}_t$, 
noting that:
\[
\underbrace{
C
}_{
\dim\,
=\,
1
}\,
\not\subset\,
\underbrace{
\cap_{i=1}^m\,
\mathsf{BS}\,
(s_i)
}_{
\dim\,
\leqslant\,
0
}\ \ ,
\]
there exists some integer $1\leqslant i\leqslant m$ such that:
\[
C\,
\not\subset\,
\mathsf{BS}\,
(s_i).
\]
Therefore $s_i\big{\vert}_C$ is a nonzero section of $\mathcal{S}_t^{a_i}(-b_i)$ on the curve $C$: 
\[
s_i\ 
\in\ 
H^0
\big(
C,
\mathcal{S}_t^{a_i}(-b_i)
\big)\,
\setminus\,
\{
0
\},
\]
and hence:
\[
C
\cdot
\mathcal{S}_t^{a_i}(-b_i)\,
\geqslant\,
0.
\]
Thus we have the estimate:
\[
\aligned
0\,
&
\leqslant\,
C
\cdot
\big(
\mathcal{S}_t^{a_i}(-b_i)
\big)
^{\otimes\,a}
\qquad\qquad\qquad
\explain{$\,= a\,\,C\cdot\mathcal{S}_t^{a_i}(-b_i)$ }
\\
&
=\,
C
\cdot
\mathcal{O}_{\mathbf{P_t}}(a_i\,a)
\otimes
\pi_2^{*}
\mathcal{O}_{\mathbb{P}_{\mathbb{K}}^{N}}(-b_i\,a)
\qquad
\explain{see~\thetag{\ref{nef-twisted}}}
\\
&
=\,
a_i\,
C
\cdot
\big(
\mathcal{O}_{\mathbf{P_t}}(a)
\otimes
\pi_2^{*}
\mathcal{O}_{\mathbb{P}_{\mathbb{K}}^{N}}(-b)
\big)\,
-\,
(b_i\,a-a_i\,b)\,\,
C
\cdot
\pi_2^{*}
\mathcal{O}_{\mathbb{P}_{\mathbb{K}}^{N}}(1)
\\
&
=\,
a_i\,\,
C
\cdot
\big(
\mathcal{S}_t^{a}(-b)
\big)\,
-\,
b\,
b_i\,
(
\underbrace{
a/b-a_i/b_i
}_{
\geqslant\, 
0
}
)\,\,
C
\cdot
\pi_2^{*}
\mathcal{O}_{\mathbb{P}_{\mathbb{K}}^{N}}(1).
\endaligned
\]
Noting that 
$\mathcal{O}_{\mathbb{P}_{\mathbb{K}}^{N}}(1)$ 
is nef and hence is 
$\pi_2^{*}\,\mathcal{O}_{\mathbb{P}_{\mathbb{K}}^{N}}(1)$
(cf. \cite[p.~51, Example 1.4.4]{Lazarsfeld-2004}), the above
estimate immediately yields:
\[
C
\cdot
\big(
\mathcal{S}_t^{a}(-b)
\big)\,
\geqslant\,
\frac
{b\,b_i}
{a_i}\,
(
a/b-a_i/b_i
)\,\,
\underbrace{
C
\cdot
\pi_2^{*}
\mathcal{O}_{\mathbb{P}_{\mathbb{K}}^{N}}(1)
}_{
\geqslant\, 
0
}\,
\geqslant\,
0.
\eqno
\qed
\]
\noqed
\end{proof}

Repeating the same reasoning as in the above two theorems, we obtain:

\begin{Proposition}
\label{nef + epsilon implies ample}
For every point
$t\in \mathbb{P}_{\mathbb{K}}^{\diamondsuit}$,
if 
$\mathcal{O}_{\mathbf{P_t}}(\ell_1)
\otimes
\pi_2^{*}
\mathcal{O}_{\mathbb{P}_{\mathbb{K}}^{N}}(-\ell_2)$
is nef on
$\mathbf{P}_t$ for some
positive integers $\ell_1, \ell_2\geqslant 1$, then for any positive integers
$\ell_1', \ell_2'\geqslant 1$ with
$\ell_2'/\ell_1'< \ell_2/\ell_1$, the twisted line bundle
$\mathcal{O}_{\mathbf{P_t}}(\ell_1')
\otimes
\pi_2^{*}
\mathcal{O}_{\mathbb{P}_{\mathbb{K}}^{N}}(-\ell_2')$ is ample on
$\mathbf{P}_t$.
\qed
\end{Proposition}

\subsection{A practical nefness criterion}
\label{subsection: A practical nefness criterion}
However, in practice, it is often difficult to gather enough global sections (with discrete base locus) to guarantee nefness of a line bundle. We need to be more clever to improve such a coarse nefness criterion with the help of nonzero sections of the same bundle restricted to proper subvarieties. First, let us introduce the theoretical reason behind. 

\begin{Definition}\,
Let $X$ be a variety, and let  $Y\subset X$ be a subvariety. 
A line bundle $\mathcal{L}$ on $X$ is said to be 
{\sl nef outside $Y$} if, for every irreducible curve $C\subset X$ with $C\not\subset Y$, the intersection number $C\cdot \mathcal{L}\geqslant 0$.
\end{Definition}

Of course, $\mathcal{L}$ is nef on $X$ if and only if $\mathcal{L}$ is nef outside the empty set $\emptyset\subset X$.

\begin{Theorem}[\bf Nefness Criterion]
\label{refined nefness criterion}
Let $X$ be a noetherian variety, and let $\mathcal{L}$ be a line bundle
on $X$. Assume that
there exists a set $\mathcal{V}$ of closed subvarieties of $X$ satisfying:

\begin{itemize}

\smallskip\item[\bf (i)]
$\emptyset \in \mathcal{V}$ and $X \in \mathcal{V}$;

\smallskip\item[\bf (ii)]
for every element $Y\in\mathcal{V}$ with $Y\neq \emptyset$, there exist finitely many elements $Z_1, \dots, Z_\flat \in \mathcal{V}$ with 
$Z_1,\dots,Z_\flat\subsetneqq Y$ such that 
the restricted line bundle $\mathcal{L}\big{\vert}_Y$ is nef outside the union 
$Z_1\cup\cdots\cup Z_\flat$.
\end{itemize}
Then $\mathcal{L}$ is nef on $X$.
\end{Theorem}

\begin{proof}
For every irreducible curve $C\subset X$, we have to show that
$C\cdot \mathcal{L}\geqslant 0$.

Assume on the contrary that 
$C\cdot \mathcal{L} < 0$. Then introduce the subset $\mathcal{N}\subset \mathcal{V}$ consisting of all subvarieties $Y\in \mathcal{V}$ which contain the curve $C$.
Clearly, $\mathcal{N}\ni X$, so $\mathcal{N}$ is nonempty.
Note that there is a natural partial order `$<$' on $\mathcal{N}$ given by the strict inclusion relation `$\subsetneqq$'.
Since $X$ is noetherian, 
$\mathcal{N}$ has a minimum element $M\supset C$. We now show a contradiction.

In fact, according to {\bf (ii)}, there exist some elements $\mathcal{V} \ni Z_1,\dots,Z_\flat \subsetneqq M$ such that $\mathcal{L}\big{\vert}_M$ is nef outside $Z_1\cup\cdots\cup Z_\flat$. Rembering that: 
\[
0 
> 
C
\cdot 
\mathcal{L}
=
C
\cdot
\mathcal{L}\big{\vert}_M,
\]
the curve $C$ is forced to lie in the union
$Z_1\cup\cdots\cup Z_\flat$, and 
thanks to irreducibility, it is furthermore contained in one certain:
\[
\underbrace{Z_i}_{\subsetneqq\,M}\in \mathcal{N},
\]
which contradicts the minimality of $M$!
\end{proof}

Now, using the same idea as around~\thetag{\ref{section-not-empty}}, 
we may realize {\bf (ii)} above with the help of sections over proper subvarieties.

\begin{Corollary}
\label{blueprint-1}
Let $X$ be a noetherian variety, and let $\mathcal{L}$ be a line bundle
on $X$. Assume that
there exists a set $\mathcal{V}$ of closed subvarieties of $X$ satisfying:

\begin{itemize}

\smallskip\item[\bf (i)]
$\emptyset \in \mathcal{V}$ and $X \in \mathcal{V}$;

\smallskip\item[\bf (ii')]
every element $\emptyset\neq Y \in \mathcal{V}$ is a union of some 
elements $Y_1,\dots,Y_{\text{\eighthnote}}\in \mathcal{V}$ such that the union of base loci:
\[
\cup_{\smallbullet\,=1}^{\eighthnote}\,
\Big(
\cap_{s\in H^0(Y_\smallbullet,\,\mathcal{L}|_{Y_\smallbullet})}\,
\big\{
s=0
\big\}
\Big)
\] 
is contained in a union of some elements $\mathcal{V} \ni Z_1,\dots,
Z_\flat
\subsetneqq Y$,
except discrete points.
\end{itemize}
Then $\mathcal{L}$ is nef on $X$.
\qed
\end{Corollary}

\section{\bf A proof blueprint of the Ampleness Theorem~\ref{Main Theorem}}
\subsection{Main Nefness Theorem}
Recalling Theorem~\ref{nefness proves Debarre}, the Ampleness Theorem~\ref{Main Theorem} is a consequence of the theorem below, whose 
effective bound $\texttt{d}_0(\varheartsuit)$ for 
$\varheartsuit = 1$
will be given in Theorem~\ref{Main Nefness Theorem with d_0=?}.

\begin{Theorem}
\label{Main Nefness Theorem}
Given any positive integer $\varheartsuit\geqslant 1$, there exists a lower degree bound
$\texttt{d}_0(\varheartsuit)\gg 1$ such that,
for all degrees $d_1,\dots,d_{c+r}\geqslant \texttt{d}_0(\varheartsuit)$, for a very generic\footnote{\
$t\in \mathbb{P}_{\mathbb{K}}^{\diamondsuit}\setminus \cup_{i=1}^{\infty}\,Z_i$
for 
some countable proper subvarieties
 $Z_i\subsetneqq \mathbb{P}_{\mathbb{K}}^{\diamondsuit}$.
}
$t\in \mathbb{P}_{\mathbb{K}}^{\diamondsuit}$,
the negatively twisted Serre line bundle
$\mathcal{O}_{\mathbf{P}_t}(1)
\otimes
\pi_2^{*}
\mathcal{O}_{\mathbb{P}_{\mathbb{K}}^{N}}
(-\varheartsuit)$
is nef on  
$\mathbf{P}_t$.
\end{Theorem}

It suffices to find one such 
$t\in \mathbb{P}_{\mathbb{K}}^{\diamondsuit}$ to
guarantee`very generic' (cf.~\cite[p.~56, Proposition 1.4.14]{Lazarsfeld-2004}).

\medskip
We will prove Theorem~\ref{Main Nefness Theorem} in two steps. 
At first, in Subsection~\ref{The central cases of relatively the same large degrees}, we sketch the proof in the central cases when all $c+r$ hypersurfaces are
approximately of the same large degrees. Then, in Subsection~\ref{subsection: product coup}, we play a 
{\sl product coup} to embrace all large degrees. 

\subsection{The central cases of relatively the same large degrees}
\label{The central cases of relatively the same large degrees}

\begin{Theorem}
\label{Main Nefness Theorem 1/2}
For any fixed $c+r$ positive integers
$
\epsilon_1,
\dots,
\epsilon_{c+r}
\geqslant 
1$,
for every sufficiently large integer
$
d\,
\gg\,
1$,  
Theorem~\ref{Main Nefness Theorem} holds with 
$
d_i\,
=\,
d
+
\epsilon_i
$, $i=1\cdots c+r$.
\end{Theorem}

When $c+r\geqslant N$, generically $X$ is discrete or empty, so there is nothing to prove. Assuming 
$c+r\leqslant N-1$, we now outline the proof.

\medskip
\noindent
{\sf Step 1.} 
In the entire family of $c+r$ hypersurfaces with degrees $d+\lambda_1,\dots,d+\lambda_{c+r}$, whose projective parameter space is $\mathbb{P}_{\mathbb{K}}^{\diamondsuit}$ (see~\thetag{\ref{diamond=?}}),
we select a specific subfamily which best suits our moving coefficients method, whose projective parameter space is a subvariety:
\[
\mathbb{P}_{\mathbb{K}}^{\text{\ding{169}}}\ \
\subset\ \
\mathbb{P}_{\mathbb{K}}^{\diamondsuit}
\qquad
\explain{see~\thetag{\ref{euro = ?}}}.
\]
For the details of this subfamily, see Subsection~\ref{subsection:constructing-algorithm}.

Recalling~\thetag{\ref{universal intersecion X}} and~\thetag{\ref{pr1-pr2}}, we then consider the subfamily of intersections $\mathcal{Y}\subset\mathcal{X}$:
\[
\mathrm{pr}_1^{-1}\,
\Big(
\mathbb{P}_{\mathbb{K}}^{\text{\ding{169}}}
\Big)\,
\cap\,
\mathcal{X}\,
=:\,
\underline{
\mathcal{Y}\ 
\subset\ 
\mathbb{P}_{\mathbb{K}}^{\text{\ding{169}}}
\times_{\mathbb{K}}
\mathbb{P}_{\mathbb{K}}^N
}\,
=\,
\mathrm{pr}_1^{-1}\,
\Big(
\mathbb{P}_{\mathbb{K}}^{\text{\ding{169}}}
\Big).
\] 
Recalling~\thetag{\ref{pi_1-proper}}, \thetag{\ref{key background inclusions}}, we introduce the subscheme of $\mathbf{P}$:
\[
\mathbf{P}^\prime\,
:=\,
\widetilde\pi^{-1}
(
\mathcal{Y}
)
\cap
\mathbf{P}\
\subset\
\mathbb{P}_{\mathbb{K}}^{\diamondsuit}
\times_{\mathbb{K}}
\mathbf{P}(\Omega_{\mathbb{P}_{\mathbb{K}}^{N}}^1),
\]
which is parametrized by $\mathcal{Y}$. 
By restriction, \thetag{\ref{pi_1-proper}} yields the commutative diagram: 
\begin{equation}
\label{Y diagram}
\xymatrix{
& 
\mathbf{P}^\prime 
\ar[ldd]_{\pi_1\,=\,\mathrm{pr_1}\circ\widetilde\pi}
\ar[rdd]^{\pi_2\,=\,\mathrm{pr_2}\circ\widetilde\pi}
\ar[d]^{\widetilde\pi} 
& 
\\
& 
\mathcal{Y} 
\ar[ld]^{\!\!\mathrm{pr_1}} 
\ar[rd]_{\mathrm{pr_2}\!\!} 
& 
\\
\mathbb{P}_{\mathbb{K}}^{\text{\ding{169}}}
&&
\mathbb{P}_{\mathbb{K}}^N.
}
\end{equation}

Introducing the restricted Serre line bundle 
$
\mathcal{O}_{\mathbf{P}^\prime}(1):=
\mathcal{O}_{\mathbf{P}}(1)\big{\vert}_{\mathbf{P}^\prime}
$
over
$\mathbf{P}^\prime$,
in order to establish Theorem~\ref{Main Nefness Theorem 1/2},
 it suffices 
to provide one such example. In fact, we will prove

\begin{Theorem}
\label{generic nefness in black diamond}
For a generic closed point 
$t\in \mathbb{P}_{\mathbb{K}}^{\text{\ding{169}}}$,
the bundle
$\mathcal{O}_{\mathbf{P_t'}}(1)
\otimes
\pi_2^{*}
\mathcal{O}_{\mathbb{P}_{\mathbb{K}}^{N}}
(-\varheartsuit)$
is nef on  
$\mathbf{P}_t':=\mathbf{P}_t$.
\end{Theorem}

\medskip
\noindent
{\sf Step 2.} 
The central objects now are the {\sl universal negatively twisted Serre line bundles}:
\[
\mathcal{O}_{\mathbf{P}^\prime}(a,b,-c)\,
:=\,
\mathcal{O}_{\mathbf{P}^\prime}(a)\,
\otimes
\pi_1^{*}\,
\mathcal{O}_
{\mathbb{P}_{\mathbb{K}}^{\text{\ding{169}}}}(b)\,
\otimes\,
\pi_2^{*}\,
\mathcal{O}_{\mathbb{P}_{\mathbb{K}}^{N}}(-c),
\]
where $a$, $c$ are positive integers such that
$c/a\geqslant \varheartsuit$,
and where $b$ are any integers.

Taking advantage of the moving coefficients method, 
firstly, we construct a series of global universal negatively twisted symmetric differential $n$-forms:
\begin{equation}
\label{S_ell}
S_{\ell}\ \
\in\ \
\Gamma
\big(
\mathbf{P}^\prime,\,
\mathcal{O}_{\mathbf{P}^\prime}(n, N, -\,\heartsuit_\ell)
\big)
\qquad
{\scriptstyle(\ell\,=\,1\,\cdots\,\text{\ding{100}})},
\end{equation}
where $n:=N-(c+r)\geqslant 1$ and all $\heartsuit_\ell/n\geqslant \varheartsuit$,
and
where {\em we always use} the symbol `$\text{\ding{100}}$' to denote auxiliary positive integers, which vary according to the context. 

Secondly, for every integer $1\leqslant \eta \leqslant n-1$, for every sequence of ascending indices :
\[
0\leqslant v_1<\dots<v_{\eta}\leqslant N,
\]  
considering the vanishing part of the corresponding $\eta$ coordinates:
\[
{}_{v_1,\dots,v_{\eta}}\mathbf{P}^\prime\,
:=\,
\mathbf{P}^\prime\
\cap\
\pi_2^{-1}
\underbrace{
\{
z_{v_1}
=
\cdots
=
z_{v_\eta}
=
0
\}
}_{
=:\,
{}_{v_1,\dots,v_{\eta}}\mathbb{P}^N
},
\]
we construct a series of universal negatively twisted symmetric differential $(n-\eta)$-forms on it:
\begin{equation}
\label{vanishing coordinates S_i}
{}_{v_1,\dots,v_{\eta}}S_{\ell}\ \
\in\ \
\Gamma
\big(
{}_{v_1,\dots,v_{\eta}}\mathbf{P}^\prime,\,
\mathcal{O}_{\mathbf{P}^\prime}(n-\eta, N-\eta, -\,{}_{v_1,\dots,v_{\eta}}\heartsuit_\ell)
\big)
\qquad
{\scriptstyle(\ell\,=\,1\,\cdots\,\text{\ding{100}})}.
\end{equation}
where all
${}_{v_1,\dots,v_{\eta}}\heartsuit_\ell/(n-\eta)\geqslant \varheartsuit$.

This step will be accomplished in Sections~\ref{Cramer-type symmetric differentials forms} and \ref{section:moving-coefficient-method}.

\medskip
\noindent
{\sf Step 3.}
From now on, we view every scheme as its $\mathbb{K}$-variety.
 
Firstly, we control the base locus of all the global sections obtained in~\thetag{\ref{S_ell}}:
\[
\mathsf{BS}
:=
\text{Base Locus of }
\{
S_\ell
\}_{
1
\leqslant 
\ell
\leqslant
\text{\ding{100}}
}
\ \
\subset
\ \
\mathbf{P}^\prime.
\]
In fact, on the coordinates nonvanishing part of $\mathbf{P}^\prime$:
\[
\obfP\,
:=\,
\mathbf{P}^\prime\,
\cap\,
\pi_2^{-1}
\{
z_0
\cdots
z_N
\neq
0
\}\ , 
\]
we prove that:
\begin{equation}
\label{dimension estimate of base loucs I}
\dim\,\,
\mathsf{BS}\,
\cap\,
\obfP\,
\leqslant\,
\dim\,\,
\mathbb{P}_{\mathbb{K}}^{\text{\ding{169}}}.
\end{equation}

Secondly, we control the base locus of all the sections obtained in~\thetag{\ref{vanishing coordinates S_i}}:
\[
{}_{v_1,\dots,v_{\eta}}
\mathsf{BS}
:=
\text{Base Locus of }
\{
{}_{v_1,\dots,v_{\eta}}
S_\ell
\}_{
1
\leqslant 
\ell
\leqslant
\text{\ding{100}}
}
\ \
\subset
\ \
{}_{v_1,\dots,v_{\eta}}\mathbf{P}^\prime.
\]
In fact, on the corresponding `coordinates nonvanishing part' of ${}_{v_1,\dots,v_{\eta}}\mathbf{P}^\prime$:
\[
{}_{v_1,\dots,v_{\eta}}\obfP\,
:=\,
{}_{v_1,\dots,v_{\eta}}\mathbf{P}^\prime\,
\cap\,
\pi_2^{-1}
\{
z_{r_0}
\cdots
z_{r_{N-\eta}}
\neq
0
\}\ , 
\]
where:
\begin{equation}
\label{r_0, ..., r_(N-eta)}
\{
r_0,
\dots,
r_{N-\eta}
\}\,
:=\,
\{
0,\dots,N
\}\,
\setminus\,
\{
v_1,
\dots,
v_\eta
\}\ ,
\end{equation}
we prove that:
\begin{equation}
\label{dimension estimate of base loucs II}
\dim\,\,
{}_{v_1,\dots,v_{\eta}}\mathsf{BS}\,
\cap\,
{}_{v_1,\dots,v_{\eta}}\obfP\,
\leqslant\,
\dim\,\,
\mathbb{P}_{\mathbb{K}}^{\text{\ding{169}}}.
\end{equation}

This crucial step will be accomplished in Sections~\ref{section: Controlling the base locus} and \ref{section: The engine of MCM}.
Anticipating, we would like to emphasize that, in order to lower down
dimensions of base loci for global symmetric differential forms (or for
higher order jet differential forms in Kobayashi hyperbolicity conjecture), a substantial amount of
algebraic geometry work is required, mainly because some already
known\big/constructed sections have the annoying tendency to
proliferate by multiplying each other without shrinking their
base loci ($0 \times {\sf anything} = 0$). Hence the first main
difficulty is to devise a wealth of independent symmetric
differential forms, which the
{\sl Moving Coefficients Method is designed for}, and the second main
difficulty is to establish the emptiness\big/discreteness of their
base loci, an ultimate difficulty that will be settled in the {\sl
Core Lemma}~\ref{Core-lemma-of-MCM}.

\medskip
\noindent
{\sf Step 4.} 
Firstly, for the regular map:
\[
\pi_1
\colon\ \ \
\mathbf{P}^\prime\,
\longrightarrow\,
\mathbb{P}_{\mathbb{K}}^{\text{\ding{169}}},
\]
noting the dimension estimates~
\thetag{\ref{dimension estimate of base loucs I}}, \thetag{\ref{dimension estimate of base loucs II}} of the base loci, 
applying now a classical theorem~\cite[p.~132, Theorem 11.12]{Harris-1992}, we know that there exists a proper closed algebraic subvariety:
\[
\Sigma\ \
\subsetneqq\ \
\mathbb{P}_{\mathbb{K}}^{\text{\ding{169}}}
\]
such that, for every closed point $t$ outside $\Sigma$:
\[
t\ \
\in\ \
\mathbb{P}_{\mathbb{K}}^{\text{\ding{169}}}\,
\setminus\,
\Sigma,
\]
\begin{itemize}
\smallskip
\item[\bf (i)]

the base locus of the restricted symmetric differential $n$-forms:
\[
\mathsf{BS}_t
:=
\text{Base Locus of }
\big\{
S_\ell(t)\, 
:=\,
S_\ell\big{\vert}_{{\bf P}^\prime_t}
\big\}_{
1
\leqslant 
\ell
\leqslant
\text{\ding{100}}
}
\ \
\subset
\ \
\mathbf{P}^\prime_t
\] 
is discrete or empty over the coordinates nonvanishing part:
\begin{equation}
\label{discrete base locus I}
\dim\,\,
\mathsf{BS}_t\,
\cap\,
\obfP_t\,
\leqslant\,
0,
\end{equation}
where:
\[
\obfP_t\,
:=\,
\obfP\,
\cap\,
\pi_1^{-1}
(t);
\]

\smallskip
\item[\bf (ii)]

the base locus of the restricted symmetric differential $(n-\eta)$-forms:
\[
{}_{v_1,\dots,v_{\eta}}\mathsf{BS}_t
:=
\text{Base Locus of }
\Big\{
{}_{v_1,\dots,v_{\eta}}S_\ell(t)\, 
:=\,
{}_{v_1,\dots,v_{\eta}}S_\ell\big{\vert}_{{}_{v_1,\dots,v_{\eta}}{\bf P}^\prime_t}
\Big\}_{
1
\leqslant 
\ell
\leqslant
\text{\ding{100}}
}
\ \
\subset
\ \
{}_{v_1,\dots,v_{\eta}}\mathbf{P}^\prime_t
\] 
is discrete or empty over the corresponding `coordinates nonvanishing part':
\begin{equation}
\label{discrete base locus II}
\dim\,\,
{}_{v_1,\dots,v_{\eta}}\mathsf{BS}_t\,
\cap\,
{}_{v_1,\dots,v_{\eta}}\obfP_t\,
\leqslant\,
0\ ,
\end{equation}
where:
\[
{}_{v_1,\dots,v_{\eta}}\obfP_t\,
:=\,
{}_{v_1,\dots,v_{\eta}}\obfP\,
\cap\,
\pi_1^{-1}
(t).
\]
\end{itemize}

\smallskip

Secondly, there exists a proper closed algebraic subvariety:
\[
\Sigma^\prime\ \
\subsetneqq\ \
\mathbb{P}_{\mathbb{K}}^{\text{\ding{169}}}
\]
such that, for every closed point $t$ outside $\Sigma^\prime$:
\[
t\ \
\in\ \
\mathbb{P}_{\mathbb{K}}^{\text{\ding{169}}}\,
\setminus\,
\Sigma^\prime,
\]
the fibre:
\[
\mathcal{Y}_t\,
:=\,
\mathcal{Y}\,
\cap\,
\mathrm{pr}_1^{-1}
(t)
\]
is smooth and of dimension $n=N-(c+r)$, and it satisfies:
\begin{equation}
\label{intersects n hyperplanes = finite points}
\dim\
\mathcal{Y}_t\,
\cap\,
\mathrm{pr}_2^{-1}
\big(
{}_{v_1,\dots,v_{n}}\mathbb{P}^N
\big)\,
=\,
0
\qquad
{\scriptstyle(0\,\leqslant\,v_1\,<\,\cdots\,<\,v_n\,\leqslant\,N)},
\end{equation}
i.e. the intersection of $\mathcal{Y}_t$ --- (under the regular map ${\rm pr}_2$) viewed as a dimension $n$ subvariety in
$\mathbb{P}^N$ --- with every $n$ coordinate hyperplanes:
\[
{}_{v_1,\dots,v_{n}}\mathbb{P}^N\,
:=\,
\{
z_{v_1}
=
\cdots
=
z_{v_n}
=0
\}
\]
is just finitely many points, which we denote by:
\begin{equation}
\label{finite intersection points with n cooridinate hyperplanes}
\underbrace{
{}_{v_1,\dots,v_{n}}\mathcal{Y}_t
}_{
\#\,<\,\infty
}\ \
\subset\ \
{}_{v_1,\dots,v_{n}}\mathbb{P}^N.
\end{equation}

\smallskip
Now, we shall conclude Theorem~\ref{generic nefness in black diamond} 
for every closed point 
$
t
\in 
\mathbb{P}_\mathbb{K}^{\text{\ding{169}}}
\setminus
(
\Sigma\,
\cup\,
\Sigma^\prime
)
$.

\begin{proof}[Proof of Theorem~\ref{generic nefness in black diamond}]
For the line bundle
$\mathcal{L}
=
\mathcal{O}_{\mathbf{P}_t'}(1)
\otimes
\pi_2^{*}
\mathcal{O}_{\mathbb{P}_{\mathbb{K}}^{N}}
(-\varheartsuit)$ over the variety $\mathbf{P}^\prime_t$,
we claim that
the set of subvarieties: 
\[
\mathcal{V}\,
:=\,
\Big\{
\emptyset,\,
\mathbf{P}^\prime_t,\,
{}_{v_1,\dots,v_{\eta}}\mathbf{P}^\prime_t
\Big\}_{
\substack
{
1
\leqslant 
\eta
\leqslant
n
\\
0
\leqslant
v_1
<
\cdots
<
v_\eta
\leqslant
N
}
}
\]
satisfies the conditions of Theorem~\ref{refined nefness criterion}.

Indeed, firstly, recalling~\thetag{\ref{discrete base locus I}}, the sections
$
\{S_{\ell}(t)\}_
{\ell=1\cdots\text{\ding{100}}}
$
have empty\big/discrete base locus over the coordinates nonvanishing part,
i.e. outside
$
\cup_{j=0}^N\,
{}_j\mathbf{P}^\prime_t
$.
Hence, using an adaptation of Theorem~\ref{a=max(...), b=min(...)},
remembering
$\varheartsuit/1
\leqslant
\min\,
\{
\heartsuit_\ell/n
\}_{1\leqslant\ell\leqslant \text{\ding{100}}}
$
,
the line bundle 
$
\mathcal{O}_{\mathbf{P}_t^\prime\,}(1)\,
\otimes\,
\pi_2^{*}\,
\mathcal{O}_{\mathbb{P}^{N}}(-\varheartsuit)
$
is nef outside 
$
\cup_{j=0}^N\
{}_j\mathbf{P}^\prime_t
$.

Secondly, for every integer $\eta=1\cdots n-1$, recalling the dimension estimate~\thetag{\ref{discrete base locus II}},
again by Theorem~\ref{a=max(...), b=min(...)},
remembering
$\varheartsuit/1
\leqslant
\min\,
\{
{}_{v_1,\dots,v_{\eta}}\heartsuit_\ell/(n-\eta)
\}_{1\leqslant\ell\leqslant \text{\ding{100}}}
$,
the line bundle
$
\mathcal{O}_{\mathbf{P}_t^\prime\,}(1)\,
\otimes\,
\pi_2^{*}\,
\mathcal{O}_{\mathbb{P}^{N}}(-\varheartsuit)
$
is nef on ${}_{v_1,\dots,v_{\eta}}\mathbf{P}^\prime_t$ outside
$
\cup_{j=0}^{N-\eta}\,
{}_{v_1,\dots,v_{\eta},r_j}\mathbf{P}^\prime_t
$
(see~\thetag{\ref{r_0, ..., r_(N-eta)}}).

Lastly, for $\eta=n$, noting that under the projection
$
\pi
\colon
\mathbf{P}_t^\prime
\rightarrow
\mathcal{Y}_t
$,
thanks to~\thetag{\ref{intersects n hyperplanes = finite points}},
every subvariety
$
{}_{v_1,\dots,v_{n}}\mathbf{P}^\prime_t
$
contracts to discrete points ${}_{v_1,\dots,v_{n}}\mathcal{Y}_t$, 
we see that on 
$
{}_{v_1,\dots,v_{n}}\mathbf{P}^\prime_t
$, the line bundle
$
\mathcal{O}_{\mathbf{P}_t'}(1)
\otimes
\pi_2^{*}
\mathcal{O}_{\mathbb{P}_{\mathbb{K}}^{N}}
(-\varheartsuit)
\cong
\mathcal{O}_{\mathbf{P}_t^\prime\,}(1)$
is not only nef, but also ample!

Summarizing the above three parts, 
by Theorem~\ref{refined nefness criterion},
we conclude the proof.
\end{proof}

\subsection{Product Coup}
\label{subsection: product coup}
We will use in an essential way Theorem~\ref{Main Nefness Theorem 1/2} with
all $\epsilon_i$ equal to either $1$ or $2$.
To begin with, we need an elementary

\begin{Observation}
\label{Bezout Theorem}
For all positive integers $d \geqslant 1$, every integer 
$d_0\geqslant d^2+d$  is a sum of nonnegative multiples of $d+1$ and 
$d+2$.
\end{Observation}

\begin{proof}
According to the Euclidian division, we can write $d_0$ as:
\[
d_0
=
p\,
(d+1)
+
q
\]
for some positive integer $p\geqslant 1$ and residue number $0\leqslant q \leqslant d$.
We claim that $p\geqslant q$.

Otherwise, we would have:
\[
p
\leqslant
q-1
\leqslant
d-1,
\]
which would imply the estimate:
\[
d
=
p\,
(d+1)
+
q
\leqslant
(d-1)\,(d+1)
+
d
=
d^2
+
d
-1,
\]
contradicting our assumption.

Therefore, we can write $d_0$ as:
\[
d_0
=
\underbrace{
(p-q)
}_{\geqslant 0}
\,(d+1)
+
q\,
(d+2),
\]
which concludes the proof.
\end{proof}

\begin{proof}[Proof of Theorem~\ref{Main Nefness Theorem}]
Take one sufficiently large integer $d$ such that Theorem~\ref{Main Nefness Theorem 1/2} holds for any integers
$
\epsilon_i
\in
\{
1,2
\}
$,
$i=1\cdots c+r
$.
Now, the above observation says that all large degrees
$
d_1,\dots,d_{c+r}
\geqslant
d^2
+
d
$
can be written as
$
d_i
=
p_i\,(d+1)
+
q_i\,(d+2)
$,
with some nonnegative integers $p_i,q_i\geqslant 0$, $i=1\cdots c+r$.
Let
$
F_i
:=
f^i_{1}
\cdots
f^i_{p_i}
f^i_{p_i+1}
\cdots
f^{i}_{p_i+q_i}
$
be a product of some $p_i$ homogeneous polynomials
$f^i_{1},\dots,f^i_{p_i}$ each 
of degree $d+1$ and of some $q_i$ homogeneous polynomials 
$f^i_{p_i+1},\dots,f^{i}_{p_i+q_i}$
each of degree $d+2$, so that $F_i$ has degree $d_i$.

Recalling~\thetag{\ref{closed points of _(...)P_(...)}}, 
a point $([z],[\xi])\in 
\mathbb{P}(\mathrm{T}_{\mathbb{P}_{\mathbb{K}}^N})$ lies in
${}_{F_{c+1},\dots,F_{c+r}}\mathbb{P}_{F_1,\dots,F_c}$ if and only if:
\[
F_i(z)
=
0,\,
dF_j\big{|}_z(\xi)
=
0
\qquad
{\scriptstyle
(
\forall\,
i\,=\,1\,\cdots\,c+r,\,
\forall\,
j\,=\,1\,\cdots\,c
)
}.
\]
Note that, for every $j=1\cdots c$, the pair of equations:
\begin{equation}
\label{equations F_j=0, dF_j=0}
F_j(z)=0,\,
dF_j
\big\vert
_z
(\xi)
=0
\end{equation}
is equivalent to either:
\begin{equation}
\label{case 1}
\exists\
1\,
\leqslant\,
v_j\,
\leqslant\,
p_j+q_j\ \ \ \
s.t.\ \ \ \
f^j_{v_j}(z)=0,\,
df^j_{v_j}
\big\vert
_z
(\xi)
=0,
\end{equation}
or to:
\begin{equation}
\label{case 2}
\exists\
1\,
\leqslant\,
w_j^1\,
<\,
w_j^2\,
\leqslant\,
p_j
+
q_j\ \ \ \
s.t.\ \ \ \
f_{w_j^1}^j(z)=0,\,
f_{w_j^2}^j(z)=0.
\end{equation}
Therefore, 
$([z],[\xi])\in 
{}_{F_{c+1},\dots,F_{c+r}}\mathbb{P}_{F_1,\dots,F_c}$ is equivalent to say that
there exists a subset $
\{
i_1,\dots,i_k
\}
\subset
\{
1,\dots,c
\}
$ of cardinality $k$ ($k=0$ for $\emptyset$)
such that, firstly,
for every index $j\in\{i_1\cdots i_k\}$, $(z,\xi)$ is a solution of~\thetag{\ref{equations F_j=0, dF_j=0}} of type~\thetag{\ref{case 1}},
secondly, for every index $j\in\,\{1,\dots,c\}\setminus\{i_1\cdots i_k\}$, 
$(z,\xi)$ is a solution of~\thetag{\ref{equations F_j=0, dF_j=0}} of type~\thetag{\ref{case 2}}, and lastly, for every 
$j=c+1\cdots c+r$, one of $f_1^j,\dots,f_{p_j+q_j}^j$ vanishes at $z$.
Thus, we see that the variety ${}_{F_{c+1},\dots,F_{c+r}}\mathbb{P}_{F_1,\dots,F_c}$ actually decomposes into a union of subvarieties:
\[
\aligned
{}_{F_{c+1},\dots,F_{c+r}}\mathbb{P}_{F_1,\dots,F_c}\,
=\,
&
\cup_{k=0\cdots c}
\cup_{1\leqslant i_1<\cdots < i_k \leqslant c}
\cup_{
\substack{
1\leqslant v_{i_j}\leqslant p_{v_j}+q_{v_j}
\\
j=1\cdots k
}
}
\cup_{
\substack{
\{
r_1,\dots,r_{c-k}
\}
=
\{1,\dots,c\}
\setminus
\{
i_1,\dots,i_k
\}
\\
1
\leqslant 
w_{r_l}^1
<
w_{r_l}^2
\leqslant
p_{r_l}
+
q_{r_l}
\\
l=1\cdots c-k
}
}
\cup_{
\substack{
1
\leqslant
u_j
\leqslant
p_j+q_j
\\
j=c+1\cdots c+r
}
}
\\
&
\ \ \ \ \ \ \ \ \ \ \ \ \ \
{}_{f_{w_{r_1}^1}^{r_1},f_{w_{r_1}^2}^{r_1},\dots,
f_{w_{r_{c-k}}^1}^{r_{c-k}},f_{w_{r_{c-k}}^2}^{r_{c-k}},
f_{u_{c+1}}^{c+1},\dots,f_{u_{c+r}}^{c+r}
}
\mathbb{P}_{f_{v_{i_1}}^{i_1},\dots,f_{v_{i_k}}^{i_k}}.
\endaligned
\]

Similarly, we can show that the scheme ${}_{F_{c+1},\dots,F_{c+r}}\mathbf{P}_{F_1,\dots,F_c}$ also decomposes into a union of subschemes:
\begin{equation}
\aligned
{}_{F_{c+1},\dots,F_{c+r}}\mathbf{P}_{F_1,\dots,F_c}\,
=\,
&
\cup_{k=0\cdots c}
\cup_{1\leqslant i_1<\cdots < i_k \leqslant c}
\cup_{
\substack{
1\leqslant v_{i_j}\leqslant p_{v_j}+q_{v_j}
\\
j=1\cdots k
}
}
\cup_{
\substack{
\{
r_1,\dots,r_{c-k}
\}
=
\{1,\dots,c\}
\setminus
\{
i_1,\dots,i_k
\}
\\
1
\leqslant 
w_{r_l}^1
<
w_{r_l}^2
\leqslant
p_{r_l}
+
q_{r_l}
\\
l=1\cdots c-k
}
}
\cup_{
\substack{
1
\leqslant
u_j
\leqslant
p_j+q_j
\\
j=c+1\cdots c+r
}
}
\\
&
\ \ \ \ \ \ \ \ \ \ \ \ \ \
{}_{f_{w_{r_1}^1}^{r_1},f_{w_{r_1}^2}^{r_1},\dots,
f_{w_{r_{c-k}}^1}^{r_{c-k}},f_{w_{r_{c-k}}^2}^{r_{c-k}},
f_{u_{c+1}}^{c+1},\dots,f_{u_{c+r}}^{c+r}
}
\mathbf{P}_{f_{v_{i_1}}^{i_1},\dots,f_{v_{i_k}}^{i_k}}.
\endaligned
\end{equation}
Note that, for each subscheme
on the right hand side,
the number of polynomials on the lower-left of `$\mathbf{P}$' is
$
\#_L
=
2(c-k)
+
r
$,
and the number of polynomials on the lower-right is
$
\#_R
=
k
$,
whence
$
2
\#_R
+
\#_L
=
2c+r
\geqslant N
$.
Now, applying Theorem~\ref{Main Nefness Theorem 1/2}, we can choose one
$\{f_{\bullet}^{\bullet}\}_{\bullet,\bullet}$ such that the twisted Serre line bundle 
$\mathcal{O}_{\mathbf{P}(\Omega_{\mathbb{P}_{\mathbb{K}}^{N}}^1)}(1)\otimes \pi_0^*\,\mathcal{O}_{\mathbb{P}_{\mathbb{K}}^N}(-\,\varheartsuit)$ is nef on each subscheme
$
{}_{f_{w_{r_1}^1}^{r_1},f_{w_{r_1}^2}^{r_1},\dots,
f_{w_{r_{c-k}}^1}^{r_{c-k}},f_{w_{r_{c-k}}^2}^{r_{c-k}},
f_{u_{c+1}}^{c+1},\dots,f_{u_{c+r}}^{c+r}
}
\mathbf{P}_{f_{v_{i_1}}^{i_1},\dots,f_{v_{i_k}}^{i_k}}
$, and therefore
is also nef on their union 
${}_{F_{c+1},\dots,F_{c+r}}\mathbf{P}_{F_1,\dots,F_c}$.
Since nefness is a very generic property in family, we 
conclude the proof.
\end{proof}

\section{\bf Generalization of Brotbek's symmetric  differentials forms}
\label{Cramer-type symmetric differentials forms}

\subsection{Preliminaries on symmetric differential
forms in projective space}
\label{subsection: Preliminaries on symmetric differential
forms in projective space}
For a fixed algebraically closed field $\mathbb{K}$,
for three fixed integers $N,c,r\geqslant 0$ such that $N\geqslant 2$,
$2c+r\geqslant N$ and $c+r\leqslant N-1$, for $c+r$ positive integers $d_1,
\dots, d_{c+r}$, let: 
\[
H_i\
\subset\
\mathbb{P}_{\mathbb{K}}^N 
\qquad
{\scriptstyle
(i\,=\,1\,\cdots\,c+r)
} 
\]
be $c+r$
hypersurfaces defined by some degree $d_i$ homogeneous polynomials:
\[
F_i\,
\in\,
\mathbb{K}[z_0,\dots,z_N],
\] 
let $V$ be the intersection of the first $c$ hypersurfaces:
\begin{equation}
\label{V:=?}
\aligned
V\,
&
:=\,
H_1
\cap 
\cdots 
\cap 
H_c
\\
&\,\,
=
\big\{
[z]
\in
\mathbb{P}_{\mathbb{K}}^N\
\colon\
F_i(z)=0,
\forall\,
i=1\cdots c
\big\}
,
\endaligned
\end{equation}
and let $X$ be the intersection of all the $c+r$ hypersurfaces:
\begin{equation}
\label{X:=?}
\aligned
X\,
&
:=\,
\underbrace{
H_1
\cap 
\cdots 
\cap
H_{c}
}_{=\,V}
\cap
\underbrace{
H_{c+1}
\cap 
\cdots 
\cap
H_{c+r}
}_{
r~\text{more hypersurfaces}
}
\\
&\,\,
=
\big\{
[z]
\in
\mathbb{P}_{\mathbb{K}}^N\
\colon\
F_i(z)=0,
\forall\,
i=1\cdots c+r
\big\}.
\endaligned
\end{equation}
It is well known that, for generic choices of $\{F_i\}_{i=1}^{c+r}$,
the intersection $V=\cap_{i=1}^c\,H_i$ 
and $X=\cap_{i=1}^{c+r}\,H_i$ are both smooth complete, and we shall
assume this henceforth. In Subsections~\ref{subsection: Preliminaries on symmetric differential
forms in projective space}--\ref{Regular twisted symmetric differential forms with some vanishing coordinates}, we focus on smooth $\mathbb{K}$-varieties to provide a geometric approach to generalize Brotbek's symmetric differential forms, where the ambient field $\mathbb{K}$ is assumed to
be algebraically closed. In addition, in Subsection~\ref{A scheme theoretic point of view}, we will give another quick algebraic approach, without any assumption on the ambient field $\mathbb{K}$.

Recalling~\thetag{\ref{N-projective space = (N+1)-Euclidian space / relation}},
let us denote by:
\[
\pi\colon\ \ \
\mathbb{K}^{N+1}\setminus \{0\}
\longrightarrow
\mathbb{P}_{\mathbb K}^N
\]
the canonical projection.

For every integer $k$, the standard twisted regular function sheaf
$\mathcal{O}_{\mathbb{P}_{\mathbb K}^N}(k)$, geometrically, can be defined as, for all Zariski open subset $U$ in
$\mathbb{P}_{\mathbb K}^N$,  the corresponding section set
$\Gamma\big(U,\,\mathcal{O}_{\mathbb{P}_{\mathbb K}^N}(k)\big)$
consists of all the regular functions $\widehat f$ on $\pi^{-1}(U)$ satisfying:
\begin{equation}\label{characterize-of-O-k}
\widehat{f}(\lambda z)
=
\lambda^k\,
\widehat{f}(z)
\qquad
{\scriptstyle(\forall\,z\,\in\,\pi^{-1}(U), 
\,\lambda\,\in\,\mathbb{K}^{\times})}.
\end{equation}

For the cone $\widehat V:=\pi^{-1}(V)$ of $V$:
\[
\widehat V\,
=\,
\big\{
z
\in
\mathbb{K}^{N+1}
\setminus
\{0\}\
\colon\
F_i(z)=0,\,
\forall\,i=1\cdots c
\big\},
\]
recalling~\thetag{\ref{total tangent space of the N-projective space}}, \thetag{\ref{horizontal tangent space of (N+1)-Euclidian space}},
we can similarly define
its horizontal tangent bundle $\mathrm{T}_{\mathrm{hor}}\widehat V$
which has fibre at any point $z\in\widehat V$:
\[
\mathrm{T}_{\mathrm{hor}}\widehat V\big{\vert}_z
=
\big\{[\xi]\in\mathbb{K}^{N+1}\Big/\,\mathbb{K}\cdot z\
\colon\
dF_i\big\vert_z(\xi)
=
0,\,\forall\,i=1\cdots c\big\}.
\]
Its total space is:
\begin{equation}\label{def-T-hor-X-hat}
\mathrm{T}_{\mathrm{hor}}\widehat V
:=
\big\{
(z,[\xi])\
\colon\
z\in \widehat{V},\,
[\xi]\in\mathbb{K}^{N+1}\Big/\,\mathbb{K}\cdot z,
dF_i\big\vert_z(\xi)=0,\,\forall\, i=1\cdots c\big\}.
\end{equation}
Then similarly we receive the total tangent bundle $\mathrm{T}_V$ of
$V$ as:
\[
\mathrm{T}_V
=
\mathrm{T}_{\mathrm{hor}}\widehat V/\sim,
\ \ \ \ \
\text{where}\,\,(z,[\xi])
\sim
(\lambda z,[\lambda\xi]),\,\forall\,\lambda\in\mathbb{K}^\times.
\]

Let $\Omega_V$ be the dual bundle of $\mathrm{T}_V$, i.e. the
cotangent bundle of $V$, and let $\Omega_{ \mathrm{ hor}} \widehat V$
be the dual bundle of $\mathrm{T}_{ \mathrm{hor}} \widehat{V}$.  For
all positive integers $l\geqslant 1$ and all integers $\heartsuit \in \mathbb{
Z}$, we use the standard notation $\mathsf{Sym}^l\, \Omega_V$ to
denote the symmetric $l$-tensor-power of the vector bundle $\Omega_V$,
and we use $\mathsf{Sym}^l\,\Omega_V(\heartsuit)$ to denote the twisted vector
bundle $\mathrm{Sym}^l\,\Omega_V\otimes \mathcal{O}_V(\heartsuit)$.

\begin{Proposition}\label{determine-the-degree-of-symmetric-differential-forms}
For two fixed integers $l\geqslant 1$, $\heartsuit\in\mathbb{Z}$, and for every
Zariski open set $U\subset V$ together with its cone
$\widehat{U}:=\pi^{-1}(U)$,
there is a canonical injection:
\[
\Gamma\big(U,\mathsf{Sym}^l\,\Omega_V(\heartsuit)\big)
\hookrightarrow
\Gamma\big(\widehat U,
\mathsf{Sym}^l\,\Omega_{\mathrm{hor}}\widehat V\big),
\]
whose image is the set of sections $\Phi$ enjoying:
\begin{equation}\label{Phi-1}
\Phi\big(\lambda z,[\lambda\xi]\big)
=
\lambda^\heartsuit\,\Phi\big(z,[\xi]\big),
\end{equation}
for all $z\in\widehat U$, for all $[\xi]\in\mathrm{T}_{\mathrm{hor}}\widehat V\big{\vert}_z$ and for all $\lambda\in\mathbb{K}^\times$.
\end{Proposition}

\begin{proof}
Note that we have two canonical injections of
vector bundles:
\begin{align*}
\pi^*\mathsf{Sym}^l\,\Omega_V
&\hookrightarrow
\mathsf{Sym}^l\Omega_{\mathrm{hor}}\widehat{V},\\
\pi^*\mathcal{O}_V(\heartsuit)
&\hookrightarrow
\mathcal{O}_{\widehat V},
\end{align*}
since the tensor functor is left exact (torsion free) in the category
of $\mathbb{K}$-vector bundles, the tensoring of the above two injections
remains an injection:
\[
\pi^*\mathsf{Sym}^l\,\Omega_V
\otimes
\pi^*\mathcal{O}_V(\heartsuit)
\hookrightarrow
\mathsf{Sym}^l\Omega_{\mathrm{hor}}\widehat{V}
\otimes
\mathcal{O}_{\widehat V}.
\]
Recalling that:
\[
\mathsf{Sym}^l\,\Omega_V(\heartsuit)= \mathsf{Sym}^l\,\Omega_V
\otimes
\mathcal{O}_V(\heartsuit),
\]
we can rewrite the above injection as:
\[
\pi^*\mathsf{Sym}^l\,\Omega_V(\heartsuit)
\hookrightarrow
\mathsf{Sym}^l\Omega_{\mathrm{hor}}\widehat{V}.
\]
With $U \subset X$ Zariski open,
applying the global section functor $\Gamma(\widehat U,\,\cdot\,)$, which is left exact, we receive:
\[
\Gamma\big(\widehat U, \,\pi^*\mathsf{Sym}^l\,\Omega_V(\heartsuit)\big)
\hookrightarrow
\Gamma\big(\widehat U, \,\mathsf{Sym}^l\Omega_{\mathrm{hor}}\widehat{V}\big).
\]
Lastly, 
we have an injection:
\[
\Gamma\big(U,\, \mathsf{Sym}^l\,\Omega_V(\heartsuit)\big)
\hookrightarrow
\Gamma\big(\widehat U,\, \pi^*\mathsf{Sym}^l\,\Omega_V(\heartsuit)\big),
\]
whence, by composing the previous two injections, we conclude:
\[
\Gamma\big(U, \,\mathsf{Sym}^l\,\Omega_V(\heartsuit)\big)
\hookrightarrow
\Gamma\big(\widehat U, \,\mathsf{Sym}^l\Omega_{\mathrm{hor}}\widehat{V}\big).
\]

To view explicitly the image of this injection, notice
that in the case $l=0$, it is the standard injection:
\begin{align*}
\Gamma\big(U,\mathcal{O}_V(\heartsuit)\big)
&
\hookrightarrow
\Gamma\big(\widehat U,\mathcal{O}_{\widehat V}) \\
f
&
\mapsto \pi^*f,
\end{align*}
whose image consists of, as a consequence of the definition \thetag{\ref{characterize-of-O-k}} above, all functions $\widehat f$ on
$\widehat U$ satisfying $\widehat f(\lambda z)=\lambda^\heartsuit\widehat f(z)$,
for all $z\in\widehat U$ and for all $\lambda\in\mathbb{K}^\times$.

Furthermore, in the case $\heartsuit=0$, the image of the injection:
\begin{align*}
\Gamma\big(U,\mathsf{Sym}^l\,\Omega_V\big)
&\hookrightarrow
\Gamma\big(\widehat U,\mathsf{Sym}^l\Omega_{\mathrm{hor}}\widehat{V}\big)
\\
\omega&\mapsto \pi^*\omega,
\end{align*}
consists of sections
$\widehat{ \omega}$ on $\widehat{ U}$
satisfying:
\[
\widehat\omega\,
(z,[\xi])
=
\widehat\omega\,
(\lambda z,[\lambda\xi]),
\]
for all $z\in\widehat U$, all $[\xi]
\in\mathrm{T}_{\mathrm{hor}}\widehat V\big{\vert}_z$ and
all $\lambda\in\mathbb{K}^\times$.

As $\mathsf{Sym}^l\,\Omega_V(\heartsuit)=\mathsf{Sym}^l\,\Omega_V\otimes
\mathcal{O}_V(\heartsuit)$,
composing the above two observations by tensoring the corresponding
two injections, we see that any element $\Phi$ in the
image of the injection:
\begin{equation}\label{injection-U-Uhat}
\Gamma\big(U,\mathsf{Sym}^l\,\Omega_V(\heartsuit)\big)
\hookrightarrow
\Gamma\big(\widehat U,
\mathsf{Sym}^l\,\Omega_{\mathrm{hor}}\widehat V\big),
\end{equation}
automatically satisfies~\thetag{\ref{Phi-1}}.
On the other hand, for every element $\Phi$ in
$\Gamma\big(\widehat U,
\mathsf{Sym}^l\,\Omega_{\mathrm{hor}}\widehat V\big)$ satisfying
\thetag{\ref{Phi-1}}, we can construct the corresponding element $\phi$
in $\Gamma\big(U,\mathsf{Sym}^l\,\Omega_V(\heartsuit)\big)$, which maps to
$\Phi$ under the injection \thetag{\ref{injection-U-Uhat}}.
\end{proof}

Let $Y\subset V$ be a regular subvariety.
Replacing the underground variety $V$ by $Y$,
in much the same way we can show:

\begin{Proposition}\label{determine-the-degree-of-symmetric-differential-forms}
For two fixed integers $l\geqslant 1$, $\heartsuit\in\mathbb{Z}$, and for every
Zariski open set $U\subset Y$ together with its cone
$\widehat{U}:=\pi^{-1}(U)$,
there is a canonical injection:
\[
\Gamma\big(U,\mathsf{Sym}^l\,\Omega_V(\heartsuit)\big)
\hookrightarrow
\Gamma\big(\widehat U,
\mathsf{Sym}^l\,\Omega_{\mathrm{hor}}\widehat V\big),
\]
whose image is the set of sections $\Phi$ enjoying:
\begin{equation}\label{Phi}
\Phi\big(\lambda z,[\lambda\xi]\big)
=
\lambda^\heartsuit\,\Phi\big(z,[\xi]\big),
\end{equation}
for all $z\in\widehat U$, for all $[\xi]\in\mathrm{T}_{\mathrm{hor}}\widehat V\big{\vert}_z$ and for all $\lambda\in\mathbb{K}^\times$.
\qed
\end{Proposition}

In future applications, we will mainly interest in the sections:
\[
\Gamma\big(Y,\mathsf{Sym}^l\,\Omega_V(\heartsuit)\big),
\]
where $Y=X$ or 
$Y=X\cap \{z_{v_1}=0\}\cap \cdots \cap \{z_{v_\eta}=0\}$ for some vanishing coordinate indices $0\leqslant v_1<\cdots<v_{\eta}\leqslant N$. 

\subsection{Global regular symmetric horizontal differential forms.}
In our coming applications, we will be mainly 
concerned with {\sl Fermat-type hypersursurfaces}
$H_i$ defined by some homogeneous polynomials $F_i$ of the form:
\begin{equation}
\label{General c+r Fermat type hypersurfaces}
F_i
=
\sum_{j=0}^N\,
A_i^j\,z_j^{\lambda_j}
\qquad
{\scriptstyle{(i\,=\,1\cdots\,c+r)}},
\end{equation}
where $\lambda_0,\dots,\lambda_N$ are some positive integers
and where $A_i^j\in \mathbb{K}[z_0,z_1,\dots,z_N]$
are some homogeneous polynomials,
with all terms of $F_i$ having the same degree:
\begin{equation}\label{deg-F_i}
\deg A_i^j+\lambda_j
=
\deg F_i
=:
d_i
\ \ \ \ \ \ \ \ \ \ \ \ \
{\scriptstyle{(i\,=\,1\cdots\,c+r;\,j\,=\,0\,\cdots\,N)}}.
\end{equation}
Differentiating $F_i$, we receive:
\begin{equation}
\label{dF_i first time}
dF_i
=
\sum_{j=0}^N\,{\sf B}_i^j\,z_j^{\lambda_j-1},
\end{equation}
where:
\begin{equation}
\label{B_i^j}
{\sf B}_i^j
:=
z_j\,dA_i^j+\lambda_j\,A_i^j\,dz_j
\ \ \ \ \ \ \ \ \ \ \ \ \
{\scriptstyle{(i\,=\,1\cdots\,c+r;\,j\,=\,0\,\cdots\,N)}}.
\end{equation}
To make the terms of $F_i$ have the same structure as that of $dF_i$, let us denote:
\begin{equation}
\label{tilde-A_i^j}
\widetilde {\sf A}_i^j
:=
A_i^j\,z_j,
\end{equation}
so that:
\[
F_i
=
\sum_{j=0}^N\,\widetilde{\sf A}_i^j\,z_j^{\lambda_j-1}.
\]

Recalling~\thetag{\ref{X:=?}}, we denote the cone of $X$ by:
\[
\widehat X\,
=\,
\big\{
z
\in
\mathbb{K}^{N+1}
\setminus
\{0\}\
\colon\
F_i(z)=0,\,
\forall\,i=1\cdots c+r
\big\}.
\]
For all
$z\in \widehat{X}$ and
$[\xi]\in \mathrm{T}_{\mathrm{hor}}\widehat V\big{\vert}_{z}$, 
by the very
definition \thetag{\ref{def-T-hor-X-hat}} of $\mathrm{T}_{\mathrm{hor}}\widehat V$, we have:
\begin{equation}\label{F-dF}
\begin{cases}
\sum_{j=0}^N\,\widetilde{\sf A}_i^j\,z_j^{\lambda_j-1}\,(z)
=
0
\ \ \ \ \ \ \ \ \ \ \ \ \ \ \ \ \ \ \ \ \
{\scriptstyle{(i\,=\,1\cdots\,c+r)}},
\\
\vspace{-0.2cm}
\\
\sum_{j=0}^N\,{\sf B}_i^j(z,\xi)\,z_j^{\lambda_j-1}(z)
=
0
\ \ \ \ \ \ \ \ \ \ \ \ \
{\scriptstyle{(i\,=\,1\cdots\,c)}}.
\end{cases}
\end{equation}
For convenience, dropping
$z,\xi$,
we rewrite the above equations as:
\[
\begin{cases}
\sum_{j=0}^N\,\widetilde{\sf A}_i^j\,z_j^{\lambda_j-1}
=
0
\ \ \ \ \ \ \ \ \ \ \ \ \
{\scriptstyle{(i\,=\,1\cdots\,c+r)}},
\\
\vspace{-0.2cm}
\\
\sum_{j=0}^N\,{\sf B}_i^j\,z_j^{\lambda_j-1}
=
0
\ \ \ \ \ \ \ \ \ \ \ \ \
{\scriptstyle{(i\,=\,1\cdots\,c)}},
\end{cases}
\]
and formally, we view them as a system of linear equations with respect to the unknown variables $z_0^{\lambda_0-1},\dots,z_N^{\lambda_N-1}$, of which the associated coefficient matrix, of size
$(c+r+c)\times (N+1)$, is:
\begin{equation}
\label{C=...}
{\sf C}
:=
\begin{pmatrix}
\widetilde{\sf A}_1^0 & \cdots & \widetilde{\sf A}_1^N \\
\vdots & & \vdots \\
\widetilde{\sf A}_{c+r}^0 & \cdots & \widetilde{\sf A}_{c+r}^N \\
{\sf B}_1^0 & \cdots & {\sf B}_1^N \\
\vdots & & \vdots \\
{\sf B}_c^0 & \cdots & {\sf B}_c^N \\
\end{pmatrix},
\end{equation}
so that the system reads as:
\begin{equation}
\label{C(z_0^...,...,z_N^...)=0}
{\sf C}
\,
\begin{pmatrix}
z_0^{\lambda_0-1}\\
\vdots\\
z_N^{\lambda_N-1}
\end{pmatrix}=0.
\end{equation}

Recalling our assumption:
\[
n
=
\underbrace{
N-(c+r)
}_{
=\,
\dim\,X
}
\geqslant 1,
\]
since $2c+r\geqslant N$, we have
$1\leqslant n \leqslant c$.

Let now $\sf D$ be the upper 
$(\underbrace{c+r+n}_{= N})\times (N+1)$ submatrix of $\sf C$, i.e. consisting of the first $(c+r+n)$ rows of $\sf C$:
\begin{equation}
\label{matrix D = ...}
\sf D
:=
\begin{pmatrix}
\widetilde{\sf A}_1^0 & \cdots & \widetilde{\sf A}_1^N \\
\vdots & & \vdots \\
\widetilde{\sf A}_{c+r}^0 & \cdots & \widetilde{\sf A}_{c+r}^N \\
{\sf B}_1^0 & \cdots & {\sf B}_1^N \\
\vdots & & \vdots \\
{\sf B}_n^0 & \cdots & {\sf B}_n^N \\
\end{pmatrix}.
\end{equation}
For $j=0\cdots N$, let $\widehat{\sf D}_j$ denote the submatrix of $\sf D$ obtained by omitting the
$(j+1)$-th column:
\begin{equation}
\label{obtained-by-deleting-column}
\widehat{\sf D}_j
:=
\begin{pmatrix}
\widetilde{\sf A}_1^0 & \cdots & \widehat{\widetilde{{\sf A}}_1^{j}} & \dots & \widetilde{\sf A}_1^N \\
\vdots & & \vdots \\
\widetilde{\sf A}_{c+r}^0 & \cdots & \widehat{\widetilde{{\sf A}}_{c+r}^{j}} & \dots & \widetilde{\sf A}_{c+r}^N \\
{\sf B}_1^0 & \cdots & \widehat{{\sf B}_1^{j}} & \dots & {\sf B}_1^N \\
\vdots & & \vdots \\
{\sf B}_n^0 & \cdots & \widehat{{\sf B}_n^{j}} & \dots & {\sf B}_n^N \\
\end{pmatrix},
\end{equation}
and let $\sf D_j$ denote the $(j+1)$-th column of $\sf D$.

Denote:
\begin{equation}
\label{V_j}
W_j
:=
\{z_j\neq 0\}\
\subset\
\mathbb{P}^{N}
\ \ \ \ \ \ \ \ \ \ \ \ \
{\scriptstyle{(j\,=\,0\,\cdots\,N)}}
\end{equation}
the canonical affine open subsets, whose cones are:
\begin{equation}
\label{widehat-V_j}
\widehat{W}_j:=\pi^{-1}(W_j)\,\subset\,
\mathbb{K}^{N+1}
\setminus
\{0\}.
\end{equation}
Denote also:
\begin{equation}
\label{U_j}
U_j
:=
W_j
\cap
X
\end{equation}
the open subsets of $X$, whose cones are:
\begin{equation}
\label{widehat-U_j}
\widehat{U}_j
:=
\pi^{-1}(U_j)\
\subset\
\widehat{X}.
\end{equation}

Recalling the horizontal tangent bundle of $\mathbb{K}^{N+1}$:
\[
\mathrm{T}_{\mathrm{hor}}\mathbb{K}^{N+1}
=
\big\{(z,[\xi])\
\colon\
z\in \mathbb{K}^{N+1}
\setminus \{0\} \text{ and }
[\xi]\in \mathbb{K}^{N+1}/\mathbb{K}\cdot z \big\},
\]
now let $\Omega_{\mathrm{hor}} \mathbb{K}^{N+1}$ be its dual bundle.

\begin{Proposition}
\label{omega-is-well-defined}
For every $j=0\cdots N$, on the affine set:
\[
\widehat{W}_j
=
\{
z_j
\neq
0
\}\,
\subset \,
\mathbb{K}^{N+1}
\setminus
\{0\},
\]
the following affine symmetric horizontal
 differential $n$-form is well defined:
\begin{equation}
\label{widehat omega_j = ...}
\widehat\omega_j
:=
\frac{(-1)^{j}}{z_j^{\lambda_j-1}}
\det\big(\widehat{\sf D}_j\big)\ \
\in \ \
\Gamma\Big(\widehat{W}_j,\,\mathsf{Sym}^{n}\,
\Omega_{\mathrm{hor}}\mathbb{K}^{N+1}\Big).
\end{equation}
\end{Proposition}

The essence of this proposition lies in the famous
{\em Euler's Identity}.

\begin{Lemma}
{\bf [Euler's Identity]}
For every homogeneous polynomial
$P \in \mathbb{K}[z_0,\dots,z_N]$,
one has:
\[
\sum_{j=0}^N \,
\frac{\partial F}
{\partial z_j}
\cdot
z_j
=
\mathrm{deg}\,F \cdot F,
\]
where using differential writes as:
\begin{equation}\label{Euler-identity}
dF\,\big{|}_z(z)
=:
dF(z,z)
=
\mathrm{deg}\,F
\cdot
F(z),
\end{equation}
at all points $z=(z_0,\dots,z_N)\in \mathbb{K}^{N+1}$.
\qed
\end{Lemma}

\smallskip
\noindent
{\em Proof of Proposition
 \ref{omega-is-well-defined}.}
Without loss of generality, we only prove the case
$j=0$.

Recalling the notation \thetag{\ref{tilde-A_i^j}} and
\thetag{\ref{B_i^j}} where all $\widetilde{\sf A}_i^j$ are regular functions and all
${\sf B}_i^j$ are regular $1$-forms on $\mathbb{K}^{N+1}$, we can see
without difficulty that:
\begin{align*}
\widehat\omega_0
&
=
\frac{1}{z_0^{\lambda_0-1}}
\det\big(\widehat{\sf D}_0\big)\\
&
=
\frac{1}{z_0^{\lambda_0-1}}\ \
\det
\begin{pmatrix}
\widetilde{\sf A}_1^1 & \cdots & \widetilde{\sf A}_1^N \\
\vdots & & \vdots \\
\widetilde{\sf A}_{c+r}^1 & \cdots & \widetilde{\sf A}_{c+r}^N \\
{\sf B}_1^1 & \cdots & {\sf B}_1^N \\
\vdots & & \vdots \\
{\sf B}_n^1 & \cdots & {\sf B}_n^N \\
\end{pmatrix}\ \
\in \ \
\Gamma\Big(\widehat{V}_0,\,\mathsf{Sym}^{n}\,
\Omega\,\mathbb{K}^{N+1}\Big)
\end{align*}
is a well defined regular symmetric differential $n$-form. Now we need an:

\begin{Observation}
Let $N\geqslant 1$ be a positive integer,
let $L$ be a field with $\mathsf{Card}\,L=\infty$, and let $F$ be a polynomial:
\[
F\,
\in\,
L[z_0,\dots,z_N].
\]
Then $F$ is a polynomial without the variable $z_0$:
\[
F\,
\in\, 
L[z_1,\dots,z_N]\,
\subset\,
L[z_0,\dots,z_N]
\]
if and only if the evaluation map:
\[
\aligned
{\sf ev}_{F}
\colon\ \ \ \ \ \
L^{N+1}
&\,
\longrightarrow\,
L
\\
(x_0,\dots,x_N)
&\,
\longmapsto\,
F(x_0,\dots,x_N)
\endaligned
\]
is independent of the first variable $x_0\in L$.
\qed
\end{Observation}

For the same reason as the above Observation, in order to show that $\widehat{\omega}_0$ descends to a regular
symmetric horizontal differential $n$-form in
$\Gamma\big(\widehat{V}_0,\,\mathsf{Sym}^{n}\,
\Omega_{\mathrm{hor}}\mathbb{K}^{N+1}\big)$, we only have to show, at every point
 $z\in \widehat{V}_0$, for all $\xi\in
\mathrm{T}_{z}\mathbb{K}^{N+1}\cong \mathbb{K}^{N+1}$, $\lambda\in \mathbb{K}^\times$, that:
\begin{equation}\label{omega-vanish-along-the-line-z}
\widehat\omega_0(z, \xi+\lambda\,z)
=
\widehat\omega_0(z, \xi).
\end{equation}

In fact, applying Euler's Identity~\thetag{\ref{Euler-identity}} to
the above formula \thetag{\ref{B_i^j}}, we receive:
\begin{align*}
{\sf B}_i^j(z,z)
&=
\lambda_j\,A_i^j(z)\,dz_j(z,z)
+
dA_i^j(z,z)\,z_j(z)
\\
&=
\lambda_j\,A_i^j(z)\,z_j(z)
+
\mathrm{deg}\,A_i^j\cdot A_i^j(z)\,z_j(z)
\\
&=
(\lambda_j+\mathrm{deg}\,A_i^j)\,
\widetilde{{\sf A}}_i^j(z).
\end{align*}
Since $B_i^j$ are $1$-forms, we obtain:
\[
\aligned
{\sf B}_i^j(z,\xi+\lambda\,z)
&
=
{\sf B}_i^j(z,\xi)
+
\lambda\,{\sf B}_i^j(z,z)
\\
&
=
{\sf B}_i^j(z,\xi)
+
\underbrace{
\lambda\,
(\lambda_j+\mathrm{deg}
\,A_i^j)
}_{
\text{`constant'}
}\,
\widetilde{{\sf A}}_i^j(z).
\endaligned
\]

Therefore, the matrix:
\[
\begin{pmatrix}
\widetilde{\sf A}_1^1 & \cdots & \widetilde{\sf A}_1^N \\
\vdots & & \vdots \\
\widetilde{\sf A}_{c+r}^1 & \cdots & \widetilde{\sf A}_{c+r}^N \\
{\sf B}_1^1 & \cdots & {\sf B}_1^N \\
\vdots & & \vdots \\
{\sf B}_n^1 & \cdots & {\sf B}_n^N \\
\end{pmatrix}
(z,\xi+\lambda\,z)
\]
not only has the same first $c+r$ rows as the matrix:
\[
\begin{pmatrix}
\widetilde{\sf A}_1^1 & \cdots & \widetilde{\sf A}_1^N \\
\vdots & & \vdots \\
\widetilde{\sf A}_{c+r}^1 & \cdots & \widetilde{\sf A}_{c+r}^N \\
{\sf B}_1^1 & \cdots & {\sf B}_1^N \\
\vdots & & \vdots \\
{\sf B}_n^1 & \cdots & {\sf B}_n^N \\
\end{pmatrix}
(z,\xi),
\]
but also for $\ell=1\cdots n$, the $(c+r+\ell)$-th row of the former one equals to the 
$(c+r+\ell)$-th row of the latter one plus a multiple of the $\ell$-th row. Therefore both matrices have the same determinant, which verifies \thetag{\ref{omega-vanish-along-the-line-z}}.
\qed
\smallskip

Inspired by the explicit global symmetric
differential forms in Lemma~4.5 of Brotbek's paper~\cite{
Brotbek-2014-arxiv}, we carry out a simple proposition employing the above
notation. First, let us recall the well known Cramer's rule in a less familiar formulation (cf. \cite[p.~513, Theorem 4.4]{Lang-2002}).

\begin{Theorem}
{\bf [Cramer's rule]}
\label{Theorem: Cramer's rule}
In a commutative ring $R$, for all positive integers $N\geqslant 1$, let:
\[
A^0,
A^1,
\dots,
A^N
\in 
R^N
\] 
be $(N+1)$ column vectors, and suppose that $z_0,z_1,\dots,z_N\in R$ satisfy:
\begin{equation}
\label{(N+1)-sum=0}
A^0\,z_0
+
A^1\,z_1
+
\cdots
+
A^N\,z_N
=
\mathbf{0}.
\end{equation}
Then for all index pairs $0\leqslant i<j\leqslant N$, there holds the identity:
\begin{equation}
\label{Cramer's-rule}
(-1)^j
\det\,
\big(
A^0,\dots,\widehat{A^j},\dots,A^N
\big)\,
z_i
=
(-1)^i
\det\,
\big(A^0,\dots,\widehat{A^i},\dots,A^N
\big)\,
z_j.
\end{equation}
\end{Theorem}

\begin{proof}
By permuting the indices, without loss of generality, we may assume $i=0$. 

First, note that~\thetag{\ref{(N+1)-sum=0}} yields: 
\begin{equation}
\label{A_0z_0=-(...)}
A^0\,z_0
=
-
\sum_{\ell=1}^N\,
A^\ell\,z_\ell.
\end{equation}
Hence we may compute the left hand side of~\thetag{\ref{Cramer's-rule}} as:
\[
\aligned
(-1)^j
\det\,
\big(
A^0,A^1,\dots,\widehat{A^j},\dots,A^N
\big)\,
z_0\,
&
=\,
(-1)^j
\det\,
\big(
A^0\,z_0,A^1,\dots,\widehat{A^j},\dots,A^N
\big)
\\
\explain{substitute~\thetag{\ref{A_0z_0=-(...)}}}
\qquad
&
=\,
(-1)^j
\det\,
\bigg(
-\sum_{\ell=1}^N\,A^\ell\,z_\ell,
A^1,
\dots,
\widehat{A^j},
\dots,
A^N
\bigg)
\\
&
=\,
(-1)^{j+1}
\sum_{\ell=1}^N\,
\det\,
\big(
A^\ell,
A^1,
\dots,
\widehat{A^j},
\dots,
A^N
\big)\,
z^\ell
\\
\explain{only $\ell=j$ is nonzero}
\qquad
&
=\,
(-1)^{j+1}
\det\,
\big(
A^j,
A^1,
\dots,
\widehat{A^j},
\dots,
A^N
\big)\,
z^j
\\
&
=\,
(-1)^0
\det\,
\big(
\widehat{A^0},
A^1,
\dots,
A^N
\big)\,
z^j,
\endaligned
\]
which is exactly the right hand side.
\end{proof}

\begin{Proposition}
\label{first-holomorphic-symmetric-horizontal-n-forms}
The following $(N+1)$ affine regular symmetric horizontal differential $n$-forms:
\[
\widehat\omega_j
:=
\frac{(-1)^{j}}{z_j^{\lambda_j-1}}
\det\big(\widehat{\sf D}_j\big)
\ \
\in
\Gamma\Big(\widehat{U}_j,\,\mathsf{Sym}^{n}\,
\Omega_{\mathrm{hor}}\widehat{V}\Big)
\ \ \ \ \ \ \ \ \ \ \ \ \
{\scriptstyle{(j\,=\,0\,\cdots\,N)}}
\]
glue together to make a regular symmetric horizontal differential $n$-form on $\widehat{X}$:
\[
\omega\,
\in\,
\Gamma\Big(\widehat{X},\,\mathsf{Sym}^{n}\,
\Omega_{\mathrm{hor}}\widehat{V}\Big).
\]
\end{Proposition}

\begin{proof}
Our proof divides into two parts.

\smallskip

{\em Part 1:}
To show that these affine regular symmetric horizontal differential $n$-forms $\widehat{\omega}_0,\dots,\widehat{\omega}_N$ are well defined.

\smallskip

{\em Part 2:}
To show that any two different affine regular symmetric horizontal differential $n$-forms
$\widehat{\omega}_{j_1}$ and $\widehat{\omega}_{j_2}$ glue together along the intersection set
$\widehat{U}_{j_1} \cap\widehat{U}_{j_2}$.

\smallskip\noindent{\em Proof of Part 1.}
The Proposition~\ref{omega-is-well-defined} above shows that the:
\[
\widehat\omega_j
:=
\frac{(-1)^{j}}{z_j^{\lambda_j-1}}
\det\big(\widehat{\sf D}_j\big)
\ \
\in
\Gamma\Big(\widehat{W}_j,\,\mathsf{Sym}^{n}\,
\Omega_{\mathrm{hor}}\mathbb{K}^{N+1}\Big)
\ \ \ \ \ \ \ \ \ \ \ \ \
{\scriptstyle{(j\,=\,0\,\cdots\,N)}},
\]
are well defined, where:
\[
\widehat{W}_j
=
\{
z_j
\neq
0
\}\,
\subset \,
\mathbb{K}^{N+1}\setminus \{0\}.
\]
Thanks to the canonical inclusion embedding of vector bundles:
\[
\Big(\widehat{U}_j,\,
\mathrm{T}_{\mathrm{hor}}\widehat{V}\Big)\,
\hookrightarrow
\Big(\widehat{W}_j,\,
\mathrm{T}_{\mathrm{hor}}\mathbb{K}^{N+1}\Big),
\]
a pullback of $\widehat\omega_j$ concludes the first part.

\smallskip\noindent{\em Proof of Part 2.}
Recalling the equations~\thetag{\ref{C(z_0^...,...,z_N^...)=0}}, 
in particular, granted that $\sf D$ consists of the first $(c+r+n)$ rows of $\sf C$, we have:
\[
\sf D
\,
\begin{pmatrix}
z_0^{\lambda_0-1}\\
\vdots\\
z_N^{\lambda_N-1}
\end{pmatrix}
=
0.
\]
Now applying the above Cramer's rule to all the $(N+1)$ columns of $\sf D$ and the $(N+1)$ values $z_0^{\lambda_0-1},\dots,z_N^{\lambda_N-1}$, for every index pair
$0\leqslant j_1 < j_2 \leqslant N$, we receive:
\[
(-1)^{j_2}
\det\big(\widehat{\sf D}_{j_2}\big)\,
z_{j_1}^{\lambda_{j_1}-1}
=
(-1)^{{j_1}}
\det\big(\widehat{\sf D}_{j_1}\big)\,
z_{j_2}^{\lambda_{j_2}-1}.
\]
When $z_{j_1}\neq 0, z_{j_2}\neq 0$, this immediately yields:
\[
\underbrace{
\frac{(-1)^{{j_1}}}
{z_{j_1}^{\lambda_{j_1}-1}}
\det\big(\widehat{\sf D}_{j_1}\big)
}_{=\,\widehat{ \omega}_{j_1}}\,
=\,
\underbrace{
\frac{(-1)^{j_2}}
{z_{j_2}^{\lambda_{j_2}-1}}
\det\big(\widehat{\sf D}_{j_2}\big)
}_{=\,\widehat{ \omega}_{j_2}}
,
\]
thus the two affine symmetric horizontal differential $n$-forms
$\widehat{ \omega}_{j_1}$ and $\widehat{ \omega}_{j_2}$ glue together
along their overlap set
$\widehat{U}_{j_1} \cap \widehat{U}_{j_2}$.
\end{proof}

By permuting the indices, the above
Proposition~\ref{first-holomorphic-symmetric-horizontal-n-forms}
can be trivially generalized to, instead of the particular upper
$(c+r+n)\times (N+1)$ submatrix ${\sf D}$, all
$(c+r+n)\times (N+1)$ submatrices of ${\sf C}$ containing the upper $c+r$ rows,
as follows.

For all $n$ ascending positive integers $1 \leqslant j_1 <
\dots < j_n \leqslant c$, denote ${\sf C}_{j_1,\dots,j_n}$ the
$(c+r+n)\times (N+1)$ submatrix of ${\sf C}$ consisting of the first upper $c+r$
rows and the rows $c+r+j_1,\dots,c+r+j_n$. Also, for $j=0\cdots N$, let
$\widehat{{\sf C}}_{j_1,\dots,j_n;\,j}$ denote the submatrix of
${\sf C}_{j_1,\dots,j_n}$ obtained by omitting the $(j+1)$-th column.

\begin{Proposition}\label{general-holomorphic-symmetric-horizontal-n-forms}
The following $(N+1)$ affine regular symmetric horizontal
differential $n$-forms:
\[
\widehat\omega_{j_1,\dots,j_n;\,j}
:=
\frac{(-1)^{j}}{z_j^{\lambda_{j}-1}}
\det\big(\widehat{{\sf C}}_{j_1,\dots,j_n;\,j}\big)
\ \
\in
\Gamma\Big(\widehat{U}_j,\,\mathsf{Sym}^{n}\,
\Omega_{\mathrm{hor}}\widehat{V}\Big)
\ \ \ \ \ \ \ \ \ \ \ \ \
{\scriptstyle{(j\,=\,0\,\cdots\,N)}}
\]
glue together to make a global regular symmetric horizontal differential $n$-form on $\widehat{X}$.
\qed
\end{Proposition}

One step further, the above Proposition~\ref{general-holomorphic-symmetric-horizontal-n-forms} can be generalized to a larger class of $(c+r+n)\times (N+1)$ submatrices of $C$, as follows.

For any two positive integers $l_1 \geqslant l_2$ with:
\[
l_1+l_2=c+r+n,
\]
for any two ascending sequences of positive indices:
\[
\aligned
&
1\leqslant i_1 < \cdots < i_{l_1} \leqslant c+r,
\\
&
1\leqslant j_1 < \cdots < j_{l_2} \leqslant c
\endaligned
\]
satisfying:
\[
\{j_1,\dots,j_{l_2}\}\ \
\subset \ \
\{i_1,\dots,i_{l_1}\},
\]
denote ${\sf C}_{j_1,\dots,j_{l_2}}^{i_1,\dots,i_{l_1}}$ the
$(c+r+n)\times (N+1)$ submatrix of ${\sf C}$ consisting of the rows
$i_1,\dots,i_{l_1}$ and the rows $c+r+j_1,\dots,c+r+j_{l_2}$.
Also, for $j=0,\dots,N$, let
$\widehat{{\sf C}}_{j_1,\dots,j_{l_2};\,j}^{i_1,\dots,i_{l_1}}$ denote the submatrix of ${\sf C}_{j_1,\dots,j_{l_2}}^{i_1,\dots,i_{l_1}}$ obtained by omitting the
$(j+1)$-th column.

By much the same proof of Proposition
\ref{first-holomorphic-symmetric-horizontal-n-forms}, we obtain:

\begin{Proposition}
\label{general-holomorphic-symmetric-horizontal-forms}
The following $N+1$ affine regular symmetric horizontal differential $l_2$-forms:
\[
\widehat\omega_{j_1,\dots,j_{l_2};\,j}^{i_1,\dots,i_{l_1}}
:=
\frac{(-1)^{j}}{z_j^{\lambda_j-1}}
\det
\Big(\widehat{{\sf C}}_{j_1,\dots,j_{l_2};\,j}^{i_1,\dots,i_{l_1}}
\Big)
\,
\in
\Gamma\Big(\widehat{U}_j,\,\mathsf{Sym}^{l_2}\,
\Omega_{\mathrm{hor}}\widehat{V}\Big)
\ \ \ \ \ \ \ \ \ \ \ \ \
{\scriptstyle{(j\,=\,0\,\cdots\,N)}}
\]
glue together to make a global regular symmetric horizontal differential $l_2$-form:
\[
\widehat\omega_{j_1,\dots,j_{l_2}}^{i_1,\dots,i_{l_1}}\
\in\
\Gamma\Big(\widehat X,\,\mathsf{Sym}^{l_2}\,
\Omega_{\mathrm{hor}}\widehat{V}\Big).
\eqno
\qed
\]
\end{Proposition}

\subsection{Global twisted regular symmetric differential forms.}
\label{subsection:Global-twisted-holomorphic-symmetric-differential-forms}
Now, using the structure of the above explicit global forms, and applying the previous Proposition~\ref{determine-the-degree-of-symmetric-differential-forms}, we receive a result which, in the case of pure Fermat-type hypersurfaces~\thetag{\ref{pure Fermat type hypersurfaces used by Brotbek}} where
all $\lambda_0=\cdots=\lambda_N=\epsilon$ are equal, with also equal
$\deg F_1=\cdots=\deg F_c=e+\epsilon$, coincides with Brotbek's
Lemma~4.5 in~\cite{Brotbek-2014-arxiv}; Brotbek also implicitly obtained such symmetric differential 
forms by his cohomological approach. 

\begin{Proposition}
\label{general-holomorphic-symmetric-forms}
Under the assumptions and notation of Proposition~\ref{general-holomorphic-symmetric-horizontal-forms}, the global regular symmetric horizontal differential $l_2$-form $\widehat\omega_{j_1,\dots,j_{l_2}}^{i_1,\dots,i_{l_1}}$ is the image of a global twisted regular symmetric differential $l_2$-form:
\[
\omega_{j_1,\dots,j_{l_2}}^{i_1,\dots,i_{l_1}}
\in
\Gamma\big(X,\mathsf{Sym}^{l_2}\,\Omega_V(\heartsuit)\big)
\]
under the canonical injection as a particular case of Proposition~\ref{determine-the-degree-of-symmetric-differential-forms}:
\[
\Gamma\big(X,\mathsf{Sym}^{l_2}\,\Omega_V(\heartsuit)\big)
\hookrightarrow
\Gamma\big(\widehat X,
\mathsf{Sym}^{l_2}\,\Omega_{\mathrm{hor}}\widehat V\big),
\]
where the degree:
\begin{equation}\label{value-of-k}
\heartsuit
:=
\sum_{p=1}^{l_1}\,\deg F_{i_p}
+
\sum_{q=1}^{l_2}\,\deg F_{j_q}
-
\sum_{j=0}^{N}\,\lambda_j
+N
+
1.
\end{equation}
For all homogeneous polynomials $P \in \Gamma\big(\mathbb{P}^N,
\mathcal{O}_{\mathbb{P}^N}(\deg P)\big)$, by multiplication, one
receives more global twisted regular symmetric differential
$l_2$-forms:
\[
P\,\omega_{j_1,\dots,j_{l_2}}^{i_1,\dots,i_{l_1}}
\in
\Gamma\big(X,\mathsf{Sym}^{l_2}\,\Omega_V(\deg P+\heartsuit)\big).
\eqno
\qed
\]
\end{Proposition}

It is worth to mention that, again by applying Cramer's rule in
linear algebra, one can construct determinantal shape sections
concerning higher-order jet bundles on Fermat type hypersurfaces, as well as on their
intersections.

\begin{proof}
According to the criterion \thetag{\ref{Phi}} of Proposition
\ref{determine-the-degree-of-symmetric-differential-forms},  it is necessary and sufficient to show,
for all $z\in\widehat X$, for all $[\xi]\in \mathrm{T}_{\mathrm{hor}}\widehat V\big\vert_z$ and for all $\lambda\in\mathbb{K}^\times$, that:
\begin{equation}\label{verify-degree-k}
\widehat\omega_{j_1,\dots,j_{l_2}}^{i_1,\dots,i_{l_1}}\big(\lambda z,[\lambda\xi]\big)
=
\lambda^\heartsuit\,\widehat\omega_{j_1,\dots,j_{l_2}}^{i_1,\dots,i_{l_1}}\big(z,[\xi]\big).
\end{equation}
We may assume $z \in \widehat{U}_0=\{z_0 \neq 0\}$ for instance. Now, applying Proposition \ref{general-holomorphic-symmetric-horizontal-forms}, we receive:
\begin{equation}\label{compute-degree-k}
\aligned
\widehat\omega_{j_1,\dots,j_{l_2}}^{i_1,\dots,i_{l_1}}\big(\lambda z,[\lambda\xi]\big)
&
=
\widehat\omega_{j_1,\dots,j_{l_2};0}^{i_1,\dots,i_{l_1}}
\big(\lambda z,[\lambda\xi]\big)
\\
&
=
\frac
{1}
{z_0^{\lambda_0-1}}
\begin{pmatrix}
\widetilde{\sf A}_{i_1}^1 & \cdots & \widetilde{\sf A}_{i_1}^N \\
\vdots & & \vdots \\
\widetilde{\sf A}_{i_{l_1}}^1 & \cdots & \widetilde{\sf A}_{i_{l_1}}^N \\
{\sf B}_{j_1}^1 & \cdots & {\sf B}_{j_1}^N \\
\vdots & & \vdots \\
{\sf B}_{j_{l_2}}^1 & \cdots & {\sf B}_{j_{l_2}}^N \\
\end{pmatrix}
\big(
\lambda z,
\lambda\xi
\big).
\endaligned
\end{equation}
For all $j=0\cdots N$, for all $p=1\cdots l_1$,
and for all $q=1\cdots l_2$, recall the degree identity~\thetag{\ref{deg-F_i}} which shows that the entry
$\widetilde{\sf A}_{i_p}^j = z_j\,A_i^j$ is a homogeneous polynomial of degree:
\[
\deg A_{i_p}^j+1
=
\deg F_{i_p}-\lambda_j+1,
\]
and therefore satisfies:
\begin{equation}\label{degree-of-tilde-A_i^j}
\widetilde{\sf A}_{i_p}^j(\lambda z)
=
\lambda^{\deg F_{i_p}-\lambda_j+1}\widetilde{\sf A}_{i_p}^j(z).
\end{equation}
Recalling also the notation \thetag{\ref{B_i^j}}, the entry ${\sf B}_{j_q}^j$ is a $1$-form satisfying:
\begin{equation}\label{degree-of-B-i^j}
\aligned
{\sf B}_{j_q}^j
\big(\lambda z,\lambda[\xi]\big)
&
=
\lambda^{\deg A_{j_q}^j+1}\,
{\sf B}_{j_q}^j
\big(z,[\xi]\big)
\\
&
=
\lambda^{\deg F_{j_q}-\lambda_j+1}\,
{\sf B}_{j_q}^j
\big(z,[\xi]\big).
\endaligned
\end{equation}
Now, let us continue to compute
\thetag{\ref{compute-degree-k}}, starting by expanding the determinant:
\begin{footnotesize}
\begin{equation}
\label{compute-k-step1}
\aligned
\begin{pmatrix}
\widetilde{\sf A}_{i_1}^1 & \cdots & \widetilde{\sf A}_{i_1}^N \\
\vdots & & \vdots \\
\widetilde{\sf A}_{i_{l_1}}^1 & \cdots & \widetilde{\sf A}_{i_{l_1}}^N \\
{\sf B}_{j_1}^1 & \cdots & {\sf B}_{j_1}^N \\
\vdots & & \vdots \\
{\sf B}_{j_{l_2}}^1 & \cdots & {\sf B}_{j_{l_2}}^N \\
\end{pmatrix}
\big(\lambda z,\lambda\xi\big)
=
\sum_{\sigma\in \mathbb{S}_N}\,
\mathrm{sign}(\sigma)\,
\widetilde{\sf A}_{i_1}^{\sigma(1)}
\cdots
\widetilde{{\sf A}}_{i_{l_1}}^{\sigma(l_1)}
{\sf B}_{j_1}^{\sigma(l_1+1)}
\cdots
{\sf B}_{j_{l_2}}^{\sigma(l_1+l_2)}
\big(\lambda z,\lambda\xi\big).
\endaligned
\end{equation}
\end{footnotesize}\!\!
With the help of the above two entry identities
\thetag{\ref{degree-of-tilde-A_i^j}}
and
\thetag{\ref{degree-of-B-i^j}}, each term in the above 
sum equals to:
\[
\mathrm{sign}(\sigma)\,
\widetilde{\sf A}_{i_1}^{\sigma(1)}
\cdots
\widetilde{\sf A}_{i_{l_1}}^{\sigma(l_1)}
{\sf B}_{j_1}^{\sigma(l_1+1)}
\cdots
{\sf B}_{j_{l_2}}^{\sigma(l_1+l_2)}
\big(z,\xi\big)
\]
multiplied by $\lambda^{?}$, where:
\[
\aligned
?\
&
=\
\sum_{p=1}^{l_1}\,
\big(\deg F_{i_{p}}-\lambda_{\sigma(p)}+1\big)
+
\sum_{q=1}^{l_2}\,
\big(\deg F_{j_{q}}-\lambda_{\sigma(l_1+q)}+1\big)
\\
&
=\
\sum_{p=1}^{l_1}\,
\deg F_{i_p}
+
\sum_{q=1}^{l_2}\,
\deg F_{j_q}
-
\underbrace{
\Big(\sum_{p=1}^{l_1}\,\lambda_{\sigma(p)}
+
\sum_{q=1}^{l_2}\,\lambda_{\sigma(l_1+q)}
}
_{=\,\sum_{j=1}^{l_1+l_2}\,\lambda_{\sigma(j)}}\Big)
+
\underbrace{
l_1
+
l_2
}
_{=N}
\\
&
=\
\sum_{p=1}^{l_1}\,
\deg F_{i_p}
+
\sum_{q=1}^{l_2}\,
\deg F_{j_q}
-
\sum_{j=1}^N\,\lambda_j
+
N
\\
\explain{Use \thetag{\ref{value-of-k}}}
\ \ \ \ \ \ \ \ \ \ \ \ \ \ \ \ \ \ \
&
=\
\heartsuit
+
\lambda_0
-
1,
\endaligned
\]
therefore \thetag{\ref{compute-k-step1}} factors as:
\begin{footnotesize}
\[
\aligned
\lambda^{\heartsuit+\lambda_0-1}\,
\sum_{\sigma\in \mathbb{S}_N}\,
\mathrm{sign}(\sigma)\,
\widetilde{\sf A}_{i_1}^{\sigma(1)}
\cdots
\widetilde{\sf A}_{i_{l_1}}^{\sigma(l_1)}
{\sf B}_{j_1}^{\sigma(l_1+1)}
\cdots
{\sf B}_{j_{l_2}}^{\sigma(l_1+l_2)}
\big(z,\xi\big)\
&
=\
\lambda^{\heartsuit+\lambda_0-1}\,
\begin{pmatrix}
\widetilde{\sf A}_{i_1}^1 & \cdots & \widetilde{\sf A}_{i_1}^N \\
\vdots & & \vdots \\
\widetilde{\sf A}_{i_{l_1}}^1 & \cdots & \widetilde{\sf A}_{i_{l_1}}^N \\
{\sf B}_{j_1}^1 & \cdots & {\sf B}_{j_1}^N \\
\vdots & & \vdots \\
{\sf B}_{j_{l_2}}^1 & \cdots & {\sf B}_{j_{l_2}}^N \\
\end{pmatrix}
\big(z,\xi\big),
\endaligned
\]
\end{footnotesize}
and thus \thetag{\ref{compute-degree-k}} becomes:
\[
\footnotesize
\aligned
\frac
{1}
{(\lambda z_0)^{\lambda_0-1}}\,
\lambda^{\heartsuit+\lambda_0-1}\,
\begin{pmatrix}
\widetilde{\sf A}_{i_1}^1 & \cdots & \widetilde{\sf A}_{i_1}^N \\
\vdots & & \vdots \\
\widetilde{\sf A}_{i_{l_1}}^1 & \cdots & \widetilde{\sf A}_{i_{l_1}}^N \\
{\sf B}_{j_1}^1 & \cdots & {\sf B}_{j_1}^N \\
\vdots & & \vdots \\
{\sf B}_{j_{l_2}}^1 & \cdots & {\sf B}_{j_{l_2}}^N \\
\end{pmatrix}
\big(z,\xi\big)\
&
=\
\lambda^{\heartsuit}\,
\frac
{1}
{z_0^{\lambda_0-1}}\,
\begin{pmatrix}
\widetilde{\sf A}_{i_1}^1 & \cdots & \widetilde{\sf A}_{i_1}^N \\
\vdots & & \vdots \\
\widetilde{\sf A}_{i_{l_1}}^1 & \cdots & \widetilde{\sf A}_{i_{l_1}}^N \\
{\sf B}_{j_1}^1 & \cdots & {\sf B}_{j_1}^N \\
\vdots & & \vdots \\
{\sf B}_{j_{l_2}}^1 & \cdots & {\sf B}_{j_{l_2}}^N \\
\end{pmatrix}
\big(z,\xi\big)
\\
&
=\
\lambda^{\heartsuit}\ \
\widehat\omega_{j_1,\dots,j_{l_2};0}^{i_1,\dots,i_{l_1}}
\big(z,[\xi]\big)
\\
&
=\
\lambda^{\heartsuit}\ \
\widehat\omega_{j_1,\dots,j_{l_2}}^{i_1,\dots,i_{l_1}}
\big(z,[\xi]\big),
\endaligned
\]
which is exactly our desired equality
\thetag{\ref{verify-degree-k}}.
\end{proof} 

Now, let ${\sf K}$ be the $(c+r+c) \times (N+1)$ matrix whose first $c+r$ rows
consist of all $(N+1)$ terms in the expressions of $F_1,\dots,F_{c+r}$ in the exact order, i.e. the $(i,j)$-th entry of ${\sf K}$ is:
\begin{equation}
\label{K-matrix}
{\sf K}_{i,j}
:=
A_i^{j-1}\,z_{j-1}^{\lambda_{j-1}}
\ \ \ \ \ \ 
{\scriptstyle{(i=1\,\cdots\,c+r;\,j\,=\,1\,\cdots\,N+1)}},
\end{equation}
and whose last $c$ rows consist of all $(N+1)$ terms in the expressions of
$dF_1,\dots,dF_c$ in the exact order, i.e. the $(c+r+i,j)$-th
entry of ${\sf K}$ is:
\[
{\sf K}_{c+r+i,j}
:=
d\big(A_i^{j-1}\,z_{j-1}^{\lambda_{j-1}}\big)
\ \ \ \ \ \ 
{\scriptstyle{(i=1\,\cdots\,c;\,j\,=\,1\,\cdots\,N+1)}}.
\] 
The $j$-th column ${\sf K}_j$ of ${\sf K}$ and the $j$-th column ${\sf C}_j$ of ${\sf C}$ are proportional:
\begin{equation}\label{K-C-proportional}
{\sf K}_j
=
{\sf C}_j\,
z_{j-1}^{\lambda_{j-1}-1}
\ \ \ \ \ \ 
{\scriptstyle{(j\,=\,1\,\cdots\,N+1)}}.
\end{equation}
In later applications, we will use  Proposition~\ref{general-holomorphic-symmetric-forms}
in the case: 
\[
l_1
=
c+r,\ \ \
l_2
=
n,
\]
and in abbreviation, dropping the upper indices, we will write these global symmetric differential forms $\omega_{j_1,\dots,j_n}^{1,\dots,c+r}$ as
$\omega_{j_1,\dots,j_n}$.
Since we will mainly consider the case where all coordinates are nonvanishing:
\[
z_0
\neq 
0,
\dots,
z_N
\neq
0,
\]
the corresponding symmetric horizontal differential $n$-forms
$\widehat\omega_{j_1,\dots,j_n;\,j}$ of
 Proposition~\ref{general-holomorphic-symmetric-horizontal-n-forms} read, in the set
$\{z_0\cdots z_N\neq 0\}$, as:
\begin{equation}
\label{omega-det(K)-first}
\aligned
\widehat\omega_{j_1,\dots,j_n;\,j}
&
=
\frac{(-1)^{j}}{z_j^{\lambda_{j}-1}}
\det\big(\widehat{\sf C}_{j_1,\dots,j_n;\,j}\big)
\\
\explain{Use \thetag{\ref{K-C-proportional}}}
\ \ \ \ \ \ \ \ \ \ \ \ \ \ \ \ \ \ \
&
=
\frac{(-1)^{j}}{z_j^{\lambda_{j}-1}}
\Big(\prod_{0\leqslant i\leqslant N,i\neq j}
\frac{1}{z_i^{\lambda_{i}-1}}\Big)\,
\det\big(\widehat{{\sf K}}_{j_1,\dots,j_n;\,j}\big)
\\
&
=
\frac{(-1)^{j}}{z_0^{\lambda_{0}-1}\cdots z_N^{\lambda_{N}-1}}
\det\big(\widehat{{\sf K}}_{j_1,\dots,j_n;\,j}\big)
\ \ \ \ \ \ \ \ \
{\scriptstyle{(j\,=\,0\,\cdots\,N)}},
\endaligned
\end{equation}
where $\widehat{{\sf K}}_{j_1,\dots,j_n;\,j}$ is defined as an analog of $\widehat{\sf C}_{j_1,\dots,j_n;\,j}$ in the obvious way.

\subsection{Regular twisted symmetric differential forms with some vanishing coordinates.}
\label{Regular twisted symmetric differential forms with some vanishing coordinates}
Investigating further the construction of symmetric differential forms via Cramer's rule, for every integer $1\leqslant \eta \leqslant n-1$, for every sequence of ascending indices:
\[
0\leqslant v_1<\dots<v_{\eta}\leqslant N,
\] 
by focusing on the intersection of $X$ with the $\eta$ coordinate
hyperplanes:
\[
{}_{v_1,\dots,v_\eta}X
:=
X
\cap
\{z_{v_1}=0\}
\cap
\dots
\cap
\{z_{v_\eta}=0\}
,
\] 
we can also construct several twisted symmetric differential $(n-\eta)$-forms:
\[
\Gamma \Big({}_{v_1,\dots,v_\eta}X, \,\Sym^{n-\eta}\,\Omega_V(?) \Big)
\ \ \ \ \ \ \ 
{\scriptstyle{(?\, \text{are twisted degrees})}}
\]
as follows, which will be essential ingredients towards the solution of the Debarre Ampleness Conjecture. 

For every two positive integers $l_1 \geqslant l_2$ with:
\[
l_1+l_2=c+r+n-\eta=N-\eta,
\]
for any two sequences of ascending  positive integers:
\[
\aligned
&
1\leqslant i_1 < \cdots < i_{l_1} \leqslant c+r
\\
&
1\leqslant j_1 < \cdots < j_{l_2} \leqslant c
\endaligned
\]
such that the second one is a subsequence of the first one:
\[
\{j_1,\dots,j_{l_2}\}\ \
\subset \ \
\{i_1,\dots,i_{l_1}\},
\]
let us denote by ${}_{v_1,\dots,v_\eta}{\sf C}_{j_1,\dots,j_{l_2}}^{i_1,\dots,i_{l_1}}$ the
$(N-\eta)\times (N-\eta+1)$ submatrix of ${\sf C}$ determined by the
$(N-\eta)$ rows 
$i_{1},\dots,i_{l_1},c+r+j_1,\dots,c+r+j_{l_2}$
and the $(N-\eta+1)$ columns which are complement to the columns $v_1+1,\dots,v_\eta+1$.
Also, for every index $j\in \{0,\dots,N\}\setminus \{v_1,\dots,v_{\eta}\}$, let
${}_{v_1,\dots,v_{\eta}}\widehat{{\sf C}}_{j_1,\dots,j_{l_2};\,j}^{i_1,\dots,i_{l_1}}$ denote the submatrix of ${}_{v_1,\dots,v_{\eta}}{\sf C}_{j_1,\dots,j_{l_2}}^{i_1,\dots,i_{l_1}}$ obtained by deleting the column which is originally contained in the
$(j+1)$-th column of ${\sf C}$.
Analogously to \thetag{\ref{V_j}}-\thetag{\ref{widehat-U_j}}, 
we denote:
\[
{}_{v_1,\dots,v_{\eta}}W_j
:=
\{z_{v_1}=0\}\cap\cdots \cap \{z_{v_{\eta}}=0\}
 \cap \{z_j\neq 0\}
\subset \,
\mathbb{P}^{N},
\]
whose cone is:
\[
{}_{v_1,\dots,v_{\eta}}\widehat{W}_j
:=
\pi^{-1}
\big(
{}_{v_1,\dots,v_{\eta}}W_j
\big)
\subset
\mathbb{K}^{N+1}
\setminus
\{0\},
\]
and we denote also:
\[
{}_{v_1,\dots,v_{\eta}}U_j
:=
{}_{v_1,\dots,v_{\eta}}W_j
\cap
X
\subset
{}_{v_1,\dots,v_{\eta}}X,
\]
whose cone is:
\[
{}_{v_1,\dots,v_{\eta}}\widehat{U}_j
:=
\pi^{-1}
\big(
{}_{v_1,\dots,v_{\eta}}U_j
\big)\,
\subset\,
{}_{v_1,\dots,v_{\eta}}\widehat{X}
:=
\pi^{-1}
\big(
{}_{v_1,\dots,v_{\eta}}X
\big).
\]

Now we have two very analogs of Propositions \ref{general-holomorphic-symmetric-horizontal-forms} and \ref{general-holomorphic-symmetric-forms}.

First, write the $(N-\eta+1)$ remaining numbers of the set-minus: 
\[
\{0,\dots,N\}
\setminus 
\{v_1,\dots,v_{\eta}\}
\] 
in the ascending order:
\begin{equation}
\label{ascending order}
r_0<r_1<\cdots<r_{N-\eta}.
\end{equation}
It is necessary to assume that
$\lambda_0,\dots,\lambda_N\geqslant 2$.

\begin{Proposition}
\label{Prop-3.8}
For all $j=0\cdots N-\eta$, the following $(N+1-\eta)$ affine regular symmetric horizontal differential $l_2$-forms:
\[
{}_{v_1,\dots,v_{\eta}}\widehat\omega_{j_1,\dots,j_{l_2};\,r_j}^{i_1,\dots,i_{l_1}}
:=
\frac{(-1)^{j}}{z_{r_j}^{\lambda_{r_j}-1}}
\det\big({}_{v_1,\dots,v_{\eta}}\widehat{{\sf C}}_{j_1,\dots,j_{l_2};\,r_j}^{i_1,\dots,i_{l_1}}\big)
\,
\in
\Gamma\Big({}_{v_1,\dots,v_{\eta}}\widehat{U}_{r_j},\,\mathsf{Sym}^{l_2}\,
\Omega_{\mathrm{hor}}\widehat{V}\Big)
\]
glue together to make a regular symmetric horizontal differential $l_2$-form on ${}_{v_1,\dots,v_{\eta}}\widehat X$:
\[
{}_{v_1,\dots,v_{\eta}}\widehat\omega_{j_1,\dots,j_{l_2}}^{i_1,\dots,i_{l_1}}\,
\in
\Gamma\Big({}_{v_1,\dots,v_{\eta}}\widehat X,\,\mathsf{Sym}^{l_2}\,
\Omega_{\mathrm{hor}}\widehat{V}\Big).
\]
\end{Proposition}

\begin{Proposition}
\label{Prop-3.9}
Under the assumptions and notation of the above proposition, the regular symmetric horizontal differential $l_2$-form 
${}_{v_1,\dots,v_{\eta}}\widehat\omega_{j_1,\dots,j_{l_2}}^{i_1,\dots,i_{l_1}}$ on ${}_{v_1,\dots,v_{\eta}}\widehat X$ is the image of a twisted regular symmetric differential $l_2$-form on ${}_{v_1,\dots,v_{\eta}}X$:
\[
{}_{v_1,\dots,v_{\eta}}\omega_{j_1,\dots,j_{l_2}}^{i_1,\dots,i_{l_1}}\
\in\
\Gamma\,\big({}_{v_1,\dots,v_{\eta}}X,\mathsf{Sym}^{l_2}\,\Omega_V
({}_{v_1,\dots,v_{\eta}}\heartsuit_{j_1,\dots,j_{l_2}}^{i_1,\dots,i_{l_1}})\big)
\]
under the canonical injection:
\[
\Gamma\,\big({}_{v_1,\dots,v_{\eta}}X,\mathsf{Sym}^{l_2}\,\Omega_V({}_{v_1,\dots,v_{\eta}}\heartsuit_{j_1,\dots,j_{l_2}}^{i_1,\dots,i_{l_1}})\big)
\hookrightarrow
\Gamma\,\big({}_{v_1,\dots,v_{\eta}}\widehat X,
\mathsf{Sym}^{l_2}\,\Omega_{\mathrm{hor}}\widehat V\big),
\]
where the twisted degree is:
\begin{equation}\label{value-of-heart}
{}_{v_1,\dots,v_{\eta}}\heartsuit_{j_1,\dots,j_{l_2}}^{i_1,\dots,i_{l_1}}
:=
\sum_{p=1}^{l_1}\,\deg F_{i_p}
+
\sum_{q=1}^{l_2}\,\deg F_{j_q}
-
\Big(
\sum_{j=0}^{N}\,\lambda_j
-
\sum_{\mu=1}^{\eta}\,\lambda_{v_{\mu}}
\Big)
+(N-\eta)
+
1.
\end{equation}
Furthermore, for all homogeneous polynomials $P \in \Gamma\big(\mathbb{P}^N,
\mathcal{O}_{\mathbb{P}^N}(\deg P)\big)$, by multiplication, one
receives more twisted regular symmetric differential
$l_2$-forms:
\[
P\,{}_{v_1,\dots,v_{\eta}}\omega_{j_1,\dots,j_{l_2}}^{i_1,\dots,i_{l_1}}
\in
\Gamma\,\big({}_{v_1,\dots,v_{\eta}}X,\mathsf{Sym}^{l_2}\,\Omega_V(\deg P+{}_{v_1,\dots,v_{\eta}}\heartsuit_{j_1,\dots,j_{l_2}}^{i_1,\dots,i_{l_1}})\big).
\eqno
\qed
\]
\end{Proposition}

In our coming applications, we will use Proposition~\ref{Prop-3.9}
in the case: 
\[
l_1
=
c+r,\ \ \
l_2
=
n-\eta,
\]
and in abbreviation we write these symmetric differential forms ${}_{v_1,\dots,v_{\eta}}\omega_{j_1,\dots,j_{n-\eta}}^{1,\dots,c+r}$ as
${}_{v_1,\dots,v_{\eta}}\omega_{j_1,\dots,j_{n-\eta}}$.
Since we will mainly consider the case when all coordinates but $z_{v_1},\dots,z_{v_{\eta}}$ are nonvanishinig:
\[
z_{r_0}
\neq 0,
\dots,
z_{r_{N-\eta}}
\neq
0,
\]
the corresponding symmetric horizontal differential $(n-\eta)$-forms ${}_{v_1,\dots,v_{\eta}}\omega_{j_1,\dots,j_{n-\eta}}$
of Proposition~\ref{Prop-3.8} read, in the set
$\{z_{r_0}\cdots z_{r_{N-\eta}}\neq 0\}$, as:
\begin{equation}\label{omega-det(K)-second}
\aligned
{}_{v_1,\dots,v_{\eta}}\widehat\omega_{j_1,\dots,j_{l_2};\,r_j}
&
:=
\frac{(-1)^{j}}{z_{r_j}^{\lambda_{r_j}-1}}
\det\big({}_{v_1,\dots,v_{\eta}}\widehat{{\sf C}}_{j_1,\dots,j_{l_2};\,r_j}^{i_1,\dots,i_{l_1}}\big)
\\
\explain{Use \thetag{\ref{K-C-proportional}}}
\ \ \ \ \ \ \ \ \ \ \ \ \ \ \ \ \ \ \
&
=
\frac{(-1)^{j}}{z_{r_j}^{\lambda_{r_j}-1}}
\Big(\prod_{0\leqslant i\leqslant N-\eta,i\neq j}
\frac{1}{z_{r_i}^{\lambda_{r_i}-1}}\Big)\,
\det\big({}_{v_1,\dots,v_{\eta}}\widehat{{\sf K}}_{j_1,\dots,j_{n-\eta};\,r_j}\big)
\\
&
=
\frac{(-1)^{j}}{z_{r_0}^{\lambda_{r_0}-1}\cdots z_{r_{N-\eta}}^{\lambda_{r_{N-\eta}}-1}}
\det\big({}_{v_1,\dots,v_{\eta}}\widehat{{\sf K}}_{j_1,\dots,j_{n-\eta};\,r_j}\big)
\ \ \ \ \ \ \ \ \
{\scriptstyle{(j\,=\,0\,\cdots\,N-\eta)}},
\endaligned
\end{equation}
where ${}_{v_1,\dots,v_{\eta}}\widehat{{\sf K}}_{j_1,\dots,j_{n-\eta};\,r_j}$ is defined as an analog of ${}_{v_1,\dots,v_{\eta}}\widehat{\sf C}_{j_1,\dots,j_{n-\eta};\,r_j}$ in the obvious way.

The two formulas~\thetag{\ref{omega-det(K)-first}}, \thetag{\ref{omega-det(K)-second}}
will enable us to efficiently narrow the base loci of the obtained 
symmetric differential forms, as the matrix ${\sf K}$ directly copies the original equations\big/differentials of the hypersurface polymonials $F_1,\dots,F_{c+r}$. We will heartily appreciate such a formalism when
a wealth of moving coefficient terms happen to tangle together.

\subsection{A scheme-theoretic point of view}
\label{A scheme theoretic point of view}
In future applications, we will only consider symmetric forms {\sl in coordinates}. Nevertheless, in this subsection, let us reconsider the obtained symmetric forms in an algebraic way,
 dropping the assumption `algebraically-closed' on the ambient
field $\mathbb{K}$.

Recalling~\thetag{\ref{diamond=?}}, \thetag{\ref{diamond-projective space}}, we may denote
the projective parameter space of the $c+r$ hypersurfaces in \thetag{\ref{General c+r Fermat type hypersurfaces}} by:
\[
\mathbb{P}_{\mathbb{K}}^{\text{\ding{169}}}
=
\mathrm{Proj}\,
\mathbb{K}\,
\bigg[
\Big\{
A_{i,\alpha}^{j}
\Big\}
_{
\substack{
i=1\cdots c+r
\\
j=0\cdots N
\\
|\alpha|=d_i-\lambda_j
}}
\bigg],
\]
so that the hypersurface coefficient polynomials $A_i^j$ are written as:
\begin{equation}
\label{A_i^j= sum ....}
A_i^j\,
:=\,
\sum_{|\alpha|=d_i-\lambda_j}\,
A_{i,\alpha}^{j}\,
z^{\alpha}
\qquad
{\scriptstyle(i\,=\,1\,\cdots\,c+r,\,j\,=\,0\,\cdots\,N)}.
\end{equation}

Now, we give a scheme-theoretic explanation of Proposition~\ref{omega-is-well-defined}, firstly by expressing $\widehat\omega_j$ in terms of affine coordinates.

For every index $j=0\cdots N$, in each affine set:
\[
\widehat{W}_j
=
\{
z_j
\neq
0
\}\,
\subset \,
\mathbb{K}^{N+1}
\setminus
\{0\},
\]
the $c+r$ homogeneous hypersurface equations~\thetag{\ref{General c+r Fermat type hypersurfaces}} in affine coordinates: 
\[
\bigg(
\frac{z_0}{z_j},
\dots,
\widehat{
\frac{z_j}{z_j}
},
\dots,
\frac{z_N}{z_j}
\bigg)
\]
become:
\begin{equation}
\label{dehomogenize F_i in W_j}
\aligned
\big(
F_i
\big)_j\,
&
=\,
\sum_{k=0}^N\,
\big(
A_i^k
\big)_{j}\,\,
\Big(
\frac{z_k}{z_j}
\Big)^{\lambda_k}
\\
\explain{see~\thetag{\ref{tilde-A_i^j}}}
\qquad\qquad
&
=\,
\sum_{k=0}^N\,
\big(
\widetilde {\sf A}_i^k
\big)_{j}\,\,
\Big(
\frac{z_k}{z_j}
\Big)^{\lambda_k-1},
\endaligned
\end{equation}
where for any homogeneous polynomial $P$, we dehomogenize:
\[
\big(
P
\big)_j\,
:=\,
\frac
{P}
{z_j^{\deg P}}.
\]
Differentiating \thetag{\ref{dehomogenize F_i in W_j}} for $i=1\cdots c$,
we receive:
\[
d\,
\big(
F_i
\big)_j\,
=\,
\sum_{k=0}^N\,
{\sf B}_{i,j}^k\,
\Big(
\frac{z_k}{z_j}
\Big)^{\lambda_k-1},
\]
where:
\begin{equation}
\label{B_i,j^k definition}
{\sf B}_{i,j}^k
:=
\frac{z_k}{z_j}\,d\,\big(A_i^k\big)_j
+
\lambda_k\,\big(A_i^k\big)_j\,d\,\bigg(\frac{z_k}{z_j}\bigg)
\ \ \ \ \ \ \ \ \ \ \ \ \
{\scriptstyle{(j,\,k\,=\,0\,\cdots\,N)}}.
\end{equation}
Computing $z_k\,d\,\big(
A_i^k
\big)_{j}$, we receive:
\[
\aligned
z_k\,
d\,
\big(
A_i^k
\big)_{j}\,
&
=\,
z_k\,
d\,
\bigg(
\frac{A_i^k}{z_j^{d_i-\lambda_k}}
\bigg)
\qquad\qquad
\explain{use~\thetag{\ref{deg-F_i}}}
\\
\explain{Leibniz's rule}\qquad\qquad
&
=\,
\frac{z_k\,d\,A_i^k}{z_j^{d_i-\lambda_k}}\,
-\,
(d_i-\lambda_k)\,
\frac{z_k\,A_i^k}{z_j^{d_i-\lambda_k+1}}\,
dz_j
\\
\explain{use~\thetag{\ref{dF_i first time}}}
\qquad\qquad
&
=\,
\big(
{\sf B}_i^k
-
\lambda_k\,A_i^k\,dz_k
\big)\,
\frac{1}{z_j^{d_i-\lambda_j}}
-\,
(d_i-\lambda_k)\,
\frac{z_k\,A_i^k}{z_j^{d_i-\lambda_j+1}}\,
dz_j,
\endaligned
\]
therefore~\thetag{\ref{B_i,j^k definition}} become:
\begin{align}
{\sf B}_{i,j}^k\,
&
=\,
\big(
{\sf B}_i^k
-
\lambda_k\,A_i^k\,dz_k
\big)\,
\frac{1}{\,z_j^{d_i-\lambda_j+1}}\,
-\,
(d_i-\lambda_k)\,
\frac{z_k\,A_i^k}{z_j^{d_i-\lambda_j+2}}\,
dz_j\,
+\,
\lambda_k\,\big(A_i^k\big)_j\,d\,\bigg(\frac{z_k}{z_j}\bigg)
\notag
\\
&
=\,
\frac
{{\sf B}_i^k}
{\,z_j^{d_i-\lambda_j+1}}
-
\lambda_k\,
\big(A_i^k\big)_j\,
\frac
{dz_k}
{z_j}
-
(d_i-\lambda_k)\,
\big(A_i^k\big)_j\,
\frac{z_k}{z_j^{2}}\,
dz_j
+
\lambda_k\,\big(A_i^k\big)_j\,
\Big(
\frac{dz_k}{z_j}
-
\frac{z_k}{z_j^2}\,dz_j
\Big)
\notag
\\
&
=\,
\frac
{1}
{\,z_j^{d_i-\lambda_j+1}}\,
{\sf B}_i^k
-
d_i\,
\frac{z_k}{z_j^{2}}\,
dz_j
\,\,
\big(A_i^k\big)_j
\notag
\\
&
=\,
\frac
{1}
{\,z_j^{d_i-\lambda_j+1}}\,
{\sf B}_i^k
-
\frac{d_i}{z_j}\,
dz_j
\,\,
\big(
\widetilde {\sf A}_i^k
\big)_{j}.
\label{dehomogenize-formula for B_i^k}
\end{align}

Recalling the matrix ${\sf C}$ in~\thetag{\ref{C=...}}, which is obtained
by copying the homogeneous hypersurface equations 
$F_1,\dots,F_{c+r}$ and  the differentials
$dF_1,\dots,dF_c$, we define the matrix:
\begin{equation}
\label{dehomogenize C in coordinates D(z_j)}
\big({\sf C}\big)_j
:=
\begin{pmatrix}
\big(\widetilde{\sf A}_1^0\big)_j & \cdots & \big(\widetilde{\sf A}_1^N\big)_j \\
\vdots & & \vdots \\
\big(\widetilde{\sf A}_{c+r}^0\big)_j & \cdots & \big(\widetilde{\sf A}_{c+r}^N\big)_j \\
{\sf B}_{1,j}^0 & \cdots & {\sf B}_{1,j}^N \\
\vdots & & \vdots \\
{\sf B}_{c,j}^0 & \cdots & {\sf B}_{c,j}^N \\
\end{pmatrix},
\end{equation}
which is obtained by copying the dehomogenized hypersurface equations
$\big(F_1\big)_j,\dots,\big(F_{c+r}\big)_j$ and the differentials
$d\big(F_1\big)_j,\dots,d\big(F_{c}\big)_j$. 
Recalling the matrices~\thetag{\ref{matrix D = ...}}, \thetag{\ref{obtained-by-deleting-column}}, in the obvious way we also define $\big({\sf D}\big)_j$, $\big(\widehat{\sf D}_k\big)_j$ as:
\begin{equation}
\label{dehomogenize D in coordinates D(z_j)}
\big(
\sf D
\big)_j
:=
\begin{pmatrix}
\big(\widetilde{\sf A}_1^0\big)_j & \cdots & \big(\widetilde{\sf A}_1^N\big)_j \\
\vdots & & \vdots \\
\big(\widetilde{\sf A}_{c+r}^0\big)_j & \cdots & \big(\widetilde{\sf A}_{c+r}^N\big)_j \\
{\sf B}_{1,j}^0 & \cdots & {\sf B}_{1,j}^N \\
\vdots & & \vdots \\
{\sf B}_{n,j}^0 & \cdots & {\sf B}_{n,j}^N \\
\end{pmatrix},
\end{equation}
and:
\begin{equation}
\label{widehat D_k dehomogenize in D(z_j)}
\big(
\widehat{\sf D}_k
\big)_j
:=
\begin{pmatrix}
\big(\widetilde{\sf A}_1^0\big)_j & \cdots & \widehat{\big(\widetilde{{\sf A}}_1^{k}\big)_j} & \dots & \big(\widetilde{\sf A}_1^N\big)_j \\
\vdots & & \vdots \\
\big(\widetilde{\sf A}_{c+r}^0\big)_j & \cdots & \widehat{\big(\widetilde{{\sf A}}_{c+r}^{k}\big)_j} & \dots & \big(\widetilde{\sf A}_{c+r}^N\big)_j \\
{\sf B}_{1,j}^0 & \cdots & \widehat{{\sf B}_{1,j}^{k}} & \dots & {\sf B}_{1,j}^N \\
\vdots & & \vdots \\
{\sf B}_{n,j}^0 & \cdots & \widehat{{\sf B}_{n,j}^{k}} & \dots & {\sf B}_{n,j}^N \\
\end{pmatrix}
\qquad
{\scriptstyle(k\,=\,0\,\cdots\,N)}.
\end{equation}
Recalling $\widehat{\omega}_j$ of Proposition~\ref{omega-is-well-defined}, now thanks to~\thetag{\ref{dehomogenize-formula for B_i^k}}, we have the following nice:
\begin{Observation}
\label{Observation: wonderful formula}
For every $j=0\cdots N$, one has the identity:
\begin{equation}
\label{wonderful formula}
\widehat{\omega}_j\,
=\,
\frac{(-1)^{j}}{z_j^{\lambda_j-1}}
\det\big(\widehat{\sf D}_j\big)\,
=\,
\frac{(-1)^{j}}{z_j^{-\heartsuit}}
\det
\Big(
\big(
\widehat{\sf D}_j
\big)_j
\Big),
\end{equation}
where for the moment $\heartsuit$ is defined in~\thetag{\ref{value-of-k}} for 
$\omega_{1,\dots,c}^{1,\dots,c+r}$:
\[
\heartsuit
:=
\sum_{p=1}^{c+r}\,d_{p}
+
\sum_{q=1}^{n}\,d_{q}
-
\sum_{j=0}^{N}\,\lambda_j
+N
+
1.
\]
\end{Observation}

The proof is much the same as that of Proposition~\ref{general-holomorphic-symmetric-forms}, hence we omit it here.
\qed

\medskip

Now, let $\mathrm{pr}_1, \mathrm{pr}_2$ be the two canonical projections: 
\[
\xymatrix{
&
\mathbb{P}_{\mathbb{K}}^{\text{\ding{169}}}
\times_{\mathbb{K}}
\mathbb{P}_{\mathbb{K}}^N
\ar[ld]_-{\mathrm{pr}_1} \ar[rd]^-{\mathrm{pr}_2}
\\
\mathbb{P}_{\mathbb{K}}^{\text{\ding{169}}}
& 
&  
\mathbb{P}_{\mathbb{K}}^N.
}
\]
Then thanks to the formula~\thetag{\ref{wonderful formula}}, we may view $\widehat{\omega}_j$ as a section of the 
twisted sheaf:
\[
\mathsf{Sym}^{n}\,
\Omega_{
\mathbb{P}_{\mathbb{K}}^{\text{\ding{169}}}
\times_{\mathbb{K}}
\mathbb{P}_{\mathbb{K}}^N
/\mathbb{P}_{\mathbb{K}}^{\text{\ding{169}}}}^1\,
\otimes\,
\mathrm{pr}_1^{*}\,
\mathcal{O}_
{\mathbb{P}_{\mathbb{K}}^{\text{\ding{169}}}}(N)\,
\otimes\,
\mathrm{pr}_2^{*}\,
\mathcal{O}_{\mathbb{P}_{\mathbb{K}}^{N}}(\heartsuit)
\]
over the pullback: 
\[
\mathrm{pr}_2^{-1}(\mathsf{W}_j)\
\subset\
\mathbb{P}_{\mathbb{K}}^{\text{\ding{169}}}
\times_{\mathbb{K}}
\mathbb{P}_{\mathbb{K}}^N
\] 
of the canonical affine scheme
\[
\mathsf{W}_j\,
:=\,
\mathsf{D}\,(z_j)\
\subset\
\mathbb{P}_{\mathbb{K}}^N.
\]

Using the same notation as~\thetag{\ref{universal intersecion X}},~\thetag{\ref{universal intersections V}}, recalling~\thetag{\ref{General c+r Fermat type hypersurfaces}}, \thetag{\ref{A_i^j= sum ....}},
we now introduce the two subschemes:
\[
\mathcal{X}\,
\subset\,
\mathcal{V}\,
\subset\,
\mathbb{P}_{\mathbb{K}}^{\text{\ding{169}}}
\times_{\mathbb{K}}
\mathbb{P}_{\mathbb{K}}^N,
\]
where $\mathcal{X}$ is defined by `all' the $c+r$ 
bihomogeneous polynomials: 
\[
\mathcal{X}
:=
\mathrm{V}
\Big(
\sum_{j=0}^N\,
A_1^j\,z_j^{\lambda_j},
\dots,
\sum_{j=0}^N\,
A_c^j\,z_j^{\lambda_j},
\sum_{j=0}^N\,
A_{c+1}^j\,z_j^{\lambda_j},
\dots,
\sum_{j=0}^N\,
A_{c+r}^j\,z_j^{\lambda_j}
\Big),
\]
and where $\mathcal{V}$ is defined by
the `first' $c$ bihomogeneous polynomials:
\[
\mathcal{V}
:=
\mathrm{V}
\Big(
\sum_{j=0}^N\,
A_1^j\,z_j^{\lambda_j},
\dots,
\sum_{j=0}^N\,
A_c^j\,z_j^{\lambda_j}
\Big).
\]

Now, we may view each entry of the matrix
\thetag{\ref{dehomogenize C in coordinates D(z_j)}} as
a section in:
\[
\Gamma\,\
\Big(
\mathcal{X}\cap
\mathrm{pr}_2^{-1}(\mathsf{W}_j),\,
\mathsf{Sym}^{\bullet}\,\Omega_{\mathcal{V}/\mathbb{P}_{\mathbb{K}}^{\text{\ding{169}}}}^1
\,\otimes\,
\mathrm{pr}_1^{*}\,
\mathcal{O}_
{\mathbb{P}_{\mathbb{K}}^{\text{\ding{169}}}}(1)
\Big),
\]
where the symmetric degrees are $0$ for the first $c+r$ rows and $1$ for the last $n$ rows.
Noting that the $N+1$ columns $C^0,\cdots,C^{N}$ of this matrix satisfy the relation:
\[
\sum_{k=0}^N\,
C^k\,
\frac{z_k^{\lambda_k-1}}{z_j^{\lambda_k-1}}\,
=\,
{\bf 0},
\] 
in particular, so do the columns of the submatrix~\thetag{\ref{dehomogenize D in coordinates D(z_j)}}.
Now, recalling the submatrices~\thetag{\ref{widehat D_k dehomogenize in D(z_j)}} of~\thetag{\ref{dehomogenize D in coordinates D(z_j)}}, an application of Cramer's rule (Theorem~\ref{Theorem: Cramer's rule}) yields: 
\begin{equation}
\label{babay Cramer's rule applicaiton}
\aligned
(-1)^{k_1}
\det
\Big(
\big(
\widehat{\sf D}_{k_1}
\big)_j
\Big)\,
\frac{z_{k_2}^{\lambda_{k_2-1}}}{z_j^{\lambda_{k_2}-1}}\,
&
=\,
(-1)^{k_2}
\det
\Big(
\big(
\widehat{\sf D}_{k_2}
\big)_j
\Big)\,
\frac{z_{k_1}^{\lambda_{k_1}-1}}{z_j^{\lambda_{k_1}-1}}\,
\\
&
\in\
\Gamma\,\
\Big(
\mathcal{X}\cap
\mathrm{pr}_2^{-1}(\mathsf{W}_j),\,
\mathsf{Sym}^{n}\,\Omega_{\mathcal{V}/\mathbb{P}_{\mathbb{K}}^{\text{\ding{169}}}}^1
\,\otimes\,
\mathrm{pr}_1^{*}\,
\mathcal{O}_
{\mathbb{P}_{\mathbb{K}}^{\text{\ding{169}}}}(N)
\Big)
\qquad
{\scriptstyle(k\,=\,0\,\cdots\,N)}.
\endaligned
\end{equation}

Now, recalling~\thetag{\ref{wonderful formula}}, we may interpret Proposition~\ref{first-holomorphic-symmetric-horizontal-n-forms} as follows. First, for $j=0\cdots N$, we view each:
\[
\widehat{\omega}_j
=
\frac{(-1)^{j}}{z_j^{-\heartsuit}}
\det
\Big(
\big(
\widehat{\sf D}_j
\big)_j
\Big)
\]
as a section in:
\[
\Gamma\,\
\Big(
\mathcal{X}\cap
\mathrm{pr}_2^{-1}(\mathsf{W}_j),\,
\mathsf{Sym}^{n}\,\Omega_{\mathcal{V}/\mathbb{P}_{\mathbb{K}}^{\text{\ding{169}}}}^1
\,\otimes\,
\mathrm{pr}_1^{*}\,
\mathcal{O}_
{\mathbb{P}_{\mathbb{K}}^{\text{\ding{169}}}}(N)\,
\otimes\,
\mathrm{pr}_2^{*}\,
\mathcal{O}_{\mathbb{P}_{\mathbb{K}}^{N}}(\heartsuit)
\Big).
\]
Then, thanks to an observation below, for every different indices $j_1<j_2$, over the open set:
\[
\mathcal{X}\cap
\mathrm{pr}_2^{-1}(\mathsf{W}_{j_1}\cap \mathsf{W}_{j_2})\
\subset\
\mathcal{X},
\] 
the twisted sheaf:
\[
\mathsf{Sym}^{n}\,\Omega_{\mathcal{V}/\mathbb{P}_{\mathbb{K}}^{\text{\ding{169}}}}^1
\,\otimes\,
\mathrm{pr}_1^{*}\,
\mathcal{O}_
{\mathbb{P}_{\mathbb{K}}^{\text{\ding{169}}}}(N)\,
\otimes\,
\mathrm{pr}_2^{*}\,
\mathcal{O}_{\mathbb{P}_{\mathbb{K}}^{N}}(\heartsuit)
\]
has the two coinciding sections:
\[
\aligned
\widehat{\omega}_{j_1}\,
&
=\,
\frac{(-1)^{{j_1}}}{z_{j_1}^{-\heartsuit}}
\det
\Big(
\big(
\widehat{\sf D}_{j_1}
\big)_{j_1}
\Big)
\qquad
\explain{Observation~\ref{Observation: wonderful formula}}
\\
\explain{
use~\thetag{\ref{babay Cramer's rule applicaiton}}
}
\qquad
&
=\,
\frac{(-1)^{{j_2}}}{z_{j_1}^{-\heartsuit}}\,
\frac{z_{j_1}^{\lambda_{j_2}-1}}{z_{j_2}^{\lambda_{j_2}-1}}\,
\det
\Big(
\big(
\widehat{\sf D}_{j_2}
\big)_{j_1}
\Big)
\\
\explain{
Observation~\ref{determinant transformations between charts} below
}
\qquad
&
=\,
\frac{(-1)^{{j_2}}}{z_{j_1}^{-\heartsuit}}\,
\frac{z_{j_1}^{\lambda_{j_2}-1}}{z_{j_2}^{\lambda_{j_2}-1}}\,
\frac{z_{j_2}^{\heartsuit+\lambda_{j_2}-1}}{z_{j_1}^{\heartsuit+\lambda_{j_2}-1}}
\det
\Big(
\big(
\widehat{\sf D}_{j_2}
\big)_{j_2}
\Big)
\\
&
=\,
\frac{(-1)^{{j_2}}}{z_{j_2}^{-\heartsuit}}
\det
\Big(
\big(
\widehat{\sf D}_{j_2}
\big)_{j_2}
\Big)
\\
&
=\,
\widehat{\omega}_{j_2}.
\endaligned
\]
Thus, the $N+1$ sections $\widehat{\omega}_{0},\dots,\widehat{\omega}_{N}$ glue together to make a global section:
\[
\widehat{\omega}\
\in\
\Gamma\,\
\Big(
\mathcal{X},\,
\mathsf{Sym}^{n}\,\Omega_{\mathcal{V}/\mathbb{P}_{\mathbb{K}}^{\text{\ding{169}}}}^1
\,\otimes\,
\mathrm{pr}_1^{*}\,
\mathcal{O}_
{\mathbb{P}_{\mathbb{K}}^{\text{\ding{169}}}}(N)\,
\otimes\,
\mathrm{pr}_2^{*}\,
\mathcal{O}_{\mathbb{P}_{\mathbb{K}}^{N}}(\heartsuit)
\Big).
\]

\begin{Observation}
\label{determinant transformations between charts}
For all distinct indices $0\leqslant j_1,j_2\leqslant N$,
one has the transition identities:
\[
\det
\Big(
\big(
\widehat{\sf D}_{j_2}
\big)_{j_1}
\Big)\,
=\,
\frac{z_{j_2}^{\heartsuit+\lambda_{j_2}-1}}{z_{j_1}^{\heartsuit+\lambda_{j_2}-1}}
\det
\Big(
\big(
\widehat{\sf D}_{j_2}
\big)_{j_2}
\Big).
\]
\end{Observation}

The proof is but elementary computations, so we omit it here.
\qed

\medskip

Next, repeating the same reasoning, using the obvious notation, we interpret Propositions~\ref{general-holomorphic-symmetric-horizontal-forms} and~\ref{general-holomorphic-symmetric-forms} as:

\begin{Proposition}
Each of the following $N+1$ symmetric forms:
\[
\widehat\omega_{j_1,\dots,j_{l_2};\,j}^{i_1,\dots,i_{l_1}}\,
=\,
\frac{(-1)^{j}}{z_j^{\lambda_j-1}}
\det
\Big(\widehat{{\sf C}}_{j_1,\dots,j_{l_2};\,j}^{i_1,\dots,i_{l_1}}
\Big)\,
=\,
\frac{(-1)^{j}}{z_j^{-\heartsuit}}
\det
\bigg(
\underbrace{
\Big(
\widehat{{\sf C}}_{j_1,\dots,j_{l_2};\,j}^{i_1,\dots,i_{l_1}}
\Big)_j
}_{\text{guess what?}}
\bigg)
\ \ \ \ \ \ \ \ \ \ \ \ \
{\scriptstyle{(j\,=\,0\,\cdots\,N)}}
\]
can be viewed as a section of:
\[
\Gamma\,
\Big(
\mathrm{pr}_2^{-1}(\mathsf{W}_j),\,
\mathsf{Sym}^{n}\,
\Omega_{
\mathbb{P}_{\mathbb{K}}^{\text{\ding{169}}}
\times_{\mathbb{K}}
\mathbb{P}_{\mathbb{K}}^N
/\mathbb{P}_{\mathbb{K}}^{\text{\ding{169}}}}^1
\,\otimes\,
\mathrm{pr}_1^{*}\,
\mathcal{O}_
{\mathbb{P}_{\mathbb{K}}^{\text{\ding{169}}}}(N)\,
\otimes\,
\mathrm{pr}_2^{*}\,
\mathcal{O}_{\mathbb{P}_{\mathbb{K}}^{N}}(\heartsuit)
\Big),
\]
with the twisted degree:
\[
\heartsuit
:=
\sum_{p=1}^{l_1}\,\deg F_{i_p}
+
\sum_{q=1}^{l_2}\,\deg F_{j_q}
-
\sum_{j=0}^{N}\,\lambda_j
+N
+
1.
\]
Moreover, restricting on $\mathcal{X}$, they
glue together to make a global section:
\[
\widehat\omega_{j_1,\dots,j_{l_2}}^{i_1,\dots,i_{l_1}}\
\in\
\Gamma\,\
\Big(
\mathcal{X},\,
\mathsf{Sym}^{l_2}\,\Omega_{\mathcal{V}/\mathbb{P}_{\mathbb{K}}^{\text{\ding{169}}}}^1
\,\otimes\,
\mathrm{pr}_1^{*}\,
\mathcal{O}_
{\mathbb{P}_{\mathbb{K}}^{\text{\ding{169}}}}(N)\,
\otimes\,
\mathrm{pr}_2^{*}\,
\mathcal{O}_{\mathbb{P}_{\mathbb{K}}^{N}}(\heartsuit)
\Big).
\eqno
\qed
\]
\end{Proposition}

We may view Propositions~\ref{Prop-3.8} and~\ref{Prop-3.9}
in a similar way.

\section{\large\bf Moving Coefficients Method}
\label{section:moving-coefficient-method}

\subsection{Algorithm}
\label{subsection:constructing-algorithm}
As explained in Subsection~\ref{Nefness of negative twisted cotangent sheaf suffices}, we wish to construct sufficiently many {\sl negatively twisted} symmetric differential forms, and for this purpose we investigate the {\em moving coefficients method} as follows. We will be concerned only with the central cases that all $c+r$ hypersurfaces are of the approximating big degrees $d+\epsilon_1,\dots,d+\epsilon_{c+r}$, where $\epsilon_1,\dots,\epsilon_{c+r}$ are some given positive integers negligible compared with the large integer 
$d \gg 1$ to be specified.   

To understand the essence of the moving coefficients method while avoiding unnecessary complexity (see Section~\ref{The construction of hypersurfaces revisit and better lower bounds}),    
we first consider the following $c+r$ cumbersome homogeneous polynomials $F_1, \dots, F_{c+r}$, each being the sum of a {\sl dominant} Fermat-type polynomial plus an `army' of {\sl moving coefficient} terms:
\begin{equation}
\label{F_i-moving-coefficient-method-full-strenghth}
F_i
\,
=
\,
\sum_{j=0}^N\,
A_i^j\,
z_j^{d}
\,
+
\,
\sum_{l=c+r+1}^{N}\,
\sum_{0\leqslant j_0<\cdots<j_l\leqslant N}\,
\sum_{k=0}^l\,
M_i^{j_0,\dots,j_l;j_k}\,
z_{j_0}^{\mu_{l,k}}
\cdots
\widehat{z_{j_k}^{\mu_{l,k}}}
\cdots
z_{j_l}^{\mu_{l,k}}
z_{j_k}^{d-l\mu_{l,k}},
\end{equation}
where all coefficients $A_i^{\bullet}, M_i^{\bullet;\bullet}
\in \mathbb{K}[z_0,\dots,z_N]$ are some degree $\epsilon_i\geqslant 1$ homogeneous polynomials, and all integers $\mu^{l,k}\geqslant 2$ , $d\gg 1$ are chosen subsequently by the following {\em Algorithm}, which is designed to make all the twisted symmetric differential forms obtained later have {\sl negative twisted degrees}.

The procedure is to first construct $\mu^{l,k}$ in a lexicographic order with respect to indices $(l,k)$, for $l=c+r+1 \cdots N, k=0\cdots l$, along with a set of positive integers $\delta_l$.

Recall the integer $\varheartsuit\geqslant 1$ in Theorem~\ref{Main Nefness Theorem}.
We start by setting:
\begin{equation}
\label{delta_(c+r+1)}
\delta_{c+r+1}\,
\geqslant\,
\max\,
\{
\epsilon_1,
\dots,
\epsilon_{c+r}
\}.
\end{equation} 
For every integer $l=c+r+1 \cdots N$, in this step, we begin with choosing $\mu_{l,0}$ that satisfies:
\begin{equation}\label{mu_0>?}
\mu_{l,0}
\geqslant
l\,
\delta_l
+
l\,(\delta_{c+r+1}+1)
+
1
+
(l-c-r)\,\varheartsuit,
\end{equation}
then  inductively we choose $\mu_{l,k}$ with:
\begin{equation}\label{mu_k>?}
\mu_{l,k}
\geqslant
\sum_{j=0}^{k-1}\,
l\,\mu_{l,j}
+
(l-k)\,\delta_l
+
l\,(\delta_{c+r+1}+1)
+
1
+
(l-c-r)\,\varheartsuit
\ \ \ \ \ \ \ \ \ \ \ \ \ \ \ \
{\scriptstyle{(k\,=\,1\,\cdots\,l)}}.
\end{equation}
If $l<N$, we end this step by setting:
\begin{equation}
\label{delta_l=?}
\delta_{l+1}
:=
l\,\mu_{l,l}
\end{equation}
as the starting point for the next step $l+1$. At the end $l=N$, we demand the integer $d\gg 1$ to be big enough:
\begin{equation}\label{d>?}
d
\geqslant
(N+1)\,\mu_{N,N}.
\end{equation}

Roughly speaking, the Algorithm above is designed for the following three properties. 

\begin{itemize}

\smallskip\item[{\bf (i)}]
For every integer $l=c+r+1\cdots N$, in this step, $\mu_{l,\bullet}\,(\bullet=0 \cdots l)$ grows so drastically that the former ones are negligible compared with the later ones:
\begin{equation}\label{purpose-1}
\mu_{l,0}
\ll
\mu_{l,1}
\ll
\dots
\ll
\mu_{l,l}.
\end{equation}

\smallskip\item[{\bf (ii)}]
For all integer pairs $(l_1,l_2)$ with $c+r+1\leqslant l_1< l_2\leqslant N$, all the integers $\mu_{l_1,\bullet_1}$ chosen in the former step $l_1$ are negligible compared with the integers $\mu_{l_2,\bullet_2}$ chosen in the later step $l_2$:
\begin{equation}\label{purpose-2}
\mu_{l_1,\bullet_1}
\ll
\mu_{l_2,\bullet_2}
\ \ \ \ \ \ \ \ \ \ \ \ \ \ \ \ 
{\scriptstyle{(\forall\,0\,\leqslant\,\bullet_1\,\leqslant\, l_1;\, 0\,\leqslant\,\bullet_2\,\leqslant\,l_2)}}.
\end{equation}

\smallskip\item[{\bf (iii)}]
All integers $\mu_{l,k}$ are negligible compared with the integer $d$:
\begin{equation}\label{purpose-3}
\mu_{l,k}
\ll
d
\ \ \ \ \ \ \ \ \ \ \ \ \ \ \ \ 
{\scriptstyle{(\forall\,c+r+1\,\leqslant\,l\,\leqslant\,N;\, 0\,\leqslant\,k\,\leqslant\,l)}}.
\end{equation}

\end{itemize}\smallskip

\subsection{Global moving coefficients method}
\label{The-global-moving-coefficients-method}
First, for all $i=1\cdots c+r$, we write the polynomial $F_i$ by extracting the terms for which $l=N$:
\begin{equation}\label{original-F}
\aligned
F_i
=
\sum_{j=0}^N\,
A_i^j\,
z_j^{d}
&
+
\sum_{l=c+r+1}^{N-1}\,
\sum_{0\leqslant j_0<\cdots<j_l\leqslant N}\,
\sum_{k=0}^l\,
M_i^{j_0,\dots,j_l;j_k}\,
z_{j_0}^{\mu_{l,k}}
\cdots
\widehat{z_{j_k}^{\mu_{l,k}}}
\cdots
z_{j_l}^{\mu_{l,k}}
z_{j_k}^{d-l\mu_{l,k}}
+
\\
& 
+
\sum_{k=0}^N\,
M_i^{0,\dots,N;k}\,
z_{0}^{\mu_{N,k}}
\cdots
\widehat{z_{k}^{\mu_{N,k}}}
\cdots
z_{N}^{\mu_{N,k}}
z_{k}^{d-N\,\mu_{N,k}},
\endaligned
\end{equation}
and now this second line consists of exactly all the moving coefficient terms which associate to all variables $z_0,\dots,z_N$, namely of the form $M_i^{\bullet;\bullet}\,z_0^{\bullet}\cdots z_N^{\bullet}$.

To simplify the structure of the first line, associating each term in the second sums:
\[
M_i^{j_0,\dots,j_l;j_k}\,
z_{j_0}^{\mu_{l,k}}
\cdots
\widehat{z_{j_k}^{\mu_{l,k}}}
\cdots
z_{j_l}^{\mu_{l,k}}
z_{j_k}^{d-l\mu_{l,k}}
\]
with the `corresponding' term in the first sum:
\[
A_i^{j_k}\,
z_{j_k}^{d},
\] 
and noting a priori the inequalities guaranteed by the Algorithm:
\[
d-l\,\mu_{l,k}
\geqslant
d
-
\underbrace
{
(N-1)\,\mu_{N-1,N-1}
}
_{=\,\,\delta_N \ \ \ \explain{By \thetag{\ref{delta_l=?}}}}
\ \ \ \ \ \ \ \ \ \
{\scriptstyle{
(\forall\,c+r+1\,\leqslant\,l\,\leqslant\,N-1;\,
0\,\leqslant\,k\,\leqslant\,l)
}}
,
\]
we rewrite the $F_i$ as:
\begin{equation}
\label{first-rewrite-F_i}
F_i
\,
=
\,
\sum_{j=0}^N\,
C_i^j\,
z_j^{d-\delta_N}
+
\sum_{k=0}^N\,
M_i^{0,\dots,N;k}\,
z_{0}^{\mu_{N,k}}
\cdots
\widehat{z_{k}^{\mu_{N,k}}}
\cdots
z_{N}^{\mu_{N,k}}
z_{k}^{d-N\,\mu_{N,k}},
\end{equation}
where the homogeneous polynomials $C_i^j$
are uniquely determined by gathering:
\begin{equation}
\label{C_i^j*z_j^(d-delta_N)}
C_i^j\,
z_j^{d-\delta_N}
=
A_i^j\,z_j^{d}
+
\sum_{l=c+r+1}^{N-1}\,
\sum_{
\substack
{0\leqslant j_0<\cdots<j_l\leqslant N
\\
j_k=j\,\text{for some}\,0\leqslant k \leqslant l
}}\,
M_i^{j_0,\dots,j_l;j_k}\,
z_{j_0}^{\mu_{l,k}}
\cdots
\widehat{z_{j_k}^{\mu_{l,k}}}
\cdots
z_{j_l}^{\mu_{l,k}}
z_{j_k}^{d-l\mu_{l,k}},
\end{equation}
namely, after dividing out the common factor $z_j^{d-\delta_N}$ of both sides above:
\begin{equation}\label{C_i^j}
C_i^j
:=
A_i^j\,z_j^{\delta_N}
+
\sum_{l=c+r+1}^{N-1}\,
\sum_{
\substack
{0\leqslant j_0<\cdots<j_l\leqslant N
\\
j_k=j\,\text{for some}\,0\leqslant k \leqslant l
}}\,
M_i^{j_0,\dots,j_l;j_k}\,
z_{j_0}^{\mu_{l,k}}
\cdots
\widehat{z_{j_k}^{\mu_{l,k}}}
\cdots
z_{j_l}^{\mu_{l,k}}
z_{j_k}^{\delta_N-l\mu_{l,k}}.
\end{equation}
 
Next, we have two ways to manipulate the $(N+1)$ remaining moving coefficient terms in~\thetag{\ref{first-rewrite-F_i}}:
\[
M_i^{0,\dots,N;k}\,
z_{0}^{\mu_{N,k}}
\cdots
\widehat{z_{k}^{\mu_{N,k}}}
\cdots
z_{N}^{\mu_{N,k}}
z_{k}^{d-N\,\mu_{N,k}}
\ \ \ \ \ \ \ \ \ \
{\scriptstyle{
(k\,=\,0\,\cdots\,N)
}},
\]
in order to ensure the negativity of the symmetric differential forms to be obtained later.

The first kind of manipulations are, for every chosen index 
$\nu=0\cdots N$, to associate all these $(N+1)$ moving coefficient terms:
\[
\sum_{k=0}^N\,
M_i^{0,\dots,N;k}\,
z_{0}^{\mu_{N,k}}
\cdots
\widehat{z_{k}^{\mu_{N,k}}}
\cdots
z_{N}^{\mu_{N,k}}
z_{k}^{d-N\,\mu_{N,k}}
\]
with the term $C_i^{\nu}\,z_{\nu}^{d-\delta_N}$ by rewriting $F_i$ in~\thetag{\ref{first-rewrite-F_i}} as:
\begin{equation}\label{first-major-manipulation}
F_i
=
\sum_
{\substack{
j=0\\
j\neq \nu}}
^N\,
C_i^j\,
z_j^{d-\delta_{N}}
+
T_i^{\nu}\,
z_{\nu}^{\mu_{N,0}},
\end{equation}
where $T_i^{\nu}$ is the homogeneous polynomial uniquely determined by solving:
\begin{equation}
\label{equation-T}
T_i^{\nu}\,
z_{\nu}^{\mu_{N,0}}
=
C_i^{\nu}\,z_{\nu}^{d-\delta_N}
+
\sum_{k=0}^N\,
M_i^{0,\dots,N;k}\,
z_{0}^{\mu_{N,k}}
\cdots
\widehat{z_{k}^{\mu_{N,k}}}
\cdots
z_{N}^{\mu_{N,k}}
z_{k}^{d-N\,\mu_{N,k}};
\end{equation}
in fact, guided by properties \thetag{\ref{purpose-1}}, \thetag{\ref{purpose-2}}, \thetag{\ref{purpose-3}}, our algorithm a priori implies:
\[
\mu_{N,0}\
\leqslant\
d
-
\delta_N,
\ \
\mu_{N,k}, 
\ \ 
d-N\,\mu_{N,k}
\ \ \ \ \ \ \ 
{\scriptstyle{(k\,=\,0\,\cdots\,N)}},
\] 
thus the right hand side of \thetag{\ref{equation-T}} has a common factor $z_{\nu}^{\mu_{N,0}}$.

The second kind of manipulations are, for every chosen integer $\tau=0 \cdots N-1$, for every chosen index $\rho=\tau+1 \cdots N$, to
associate each of the first $(\tau+1)$ moving coefficient terms:
\[
M_i^{0,\dots,N;k}\,
z_{0}^{\mu_{N,k}}
\cdots
\widehat{z_{k}^{\mu_{N,k}}}
\cdots
z_{N}^{\mu_{N,k}}
z_{k}^{d-N\,\mu_{N,k}}
\ \ \ \ \ \ \ \ \
{\scriptstyle{(k\,=\,0\,\cdots\,\tau)}} 
\]
with the corresponding terms $C_i^{k}\,z_{k}^{d-\delta_N}$ and to associate the remaining $(N-\tau)$ moving coefficient terms:
\[
\sum_{j=\tau+1}^N\,
M_i^{0,\dots,N;j}\,
z_{0}^{\mu_{N,j}}
\cdots
\widehat{z_{j}^{\mu_{N,j}}}
\cdots
z_{N}^{\mu_{N,j}}
z_{j}^{d-N\,\mu_{N,j}}
\]
with the term $C_i^{\rho}\,z_{\rho}^{d-\delta_N}$ by rewriting $F_i$ as:
\begin{equation}\label{second-major-manipulation}
F_i
=
\sum_{k=0}^{\tau}\,
E_i^k\,z_k^{d-N\,\mu_{N,k}}
+
\sum_
{\substack{
j=\tau+1\\
j\neq \rho}}
^N\,
C_i^j\,
z_j^{d-\delta_{N}}
+
P_i^{\tau,\rho}\,z_{\rho}^{\mu_{N,\tau+1}},
\end{equation}
where $E_i^k$ and $P_i^{\tau,\rho}$ are the homogeneous polynomials uniquely determined by
solving:
\begin{equation}\label{equation-E-R}
\aligned
E_i^k\,z_k^{d-N\,\mu_{N,k}}
&
=
C_i^{k}\,z_{k}^{d-\delta_N}
+
M_i^{0,\dots,N;k}\,
z_{0}^{\mu_{N,k}}
\cdots
\widehat{z_{k}^{\mu_{N,k}}}
\cdots
z_{N}^{\mu_{N,k}}
z_{k}^{d-N\,\mu_{N,k}}
\ \ \ \ \ \
{\scriptstyle{(k\,=\,0\,\cdots\,\tau)}},
\\
P_i^{\tau,\rho}\,z_{\rho}^{\mu_{N,\tau+1}}
&
=
C_i^{\rho}\,z_{\rho}^{d-\delta_N}
+
\sum_{j=\tau+1}^N\,
M_i^{0,\dots,N;j}\,
z_{0}^{\mu_{N,j}}
\cdots
\widehat{z_{j}^{\mu_{N,j}}}
\cdots
z_{N}^{\mu_{N,j}}
z_{j}^{d-N\,\mu_{N,j}},
\endaligned
\end{equation}
which is direct by the inequalities listed below granted by the Algorithm:
\begin{equation}
\label{some-inequalities-by-the-algorithm}
\aligned
d
-
N\,\mu_{N,k}
&
\leqslant
d
-
\delta_N
&
\ \ \ \ \ \ \ \ \ \ \ \ \ \ \ \ \ \
{\scriptstyle{(k\,=\,0\,\cdots\,\tau)}},
\\
\mu_{N,\tau+1}
&
\leqslant
\mu_{N,j}, 
&
\\
\mu_{N,\tau+1}
&
\leqslant
d
-
N\,\mu_{N,j}
&
\ \ \ \ \ \ \ \ \
{\scriptstyle{(j\,=\,\tau+1\,\cdots\,N)}}.
\endaligned
\end{equation}

Now thanks to the above two kinds of manipulations~\thetag{\ref{first-major-manipulation}}, \thetag{\ref{second-major-manipulation}}, applying Proposition~\ref{general-holomorphic-symmetric-horizontal-n-forms}, \ref{general-holomorphic-symmetric-forms}, we receive the corresponding 
 twisted symmetric differential forms with {\em negative degrees} as follows.

Firstly, for every index $\nu=0 \cdots N$, applying Proposition~\ref{general-holomorphic-symmetric-horizontal-n-forms}, \ref{general-holomorphic-symmetric-forms} with respect to the first kind of manipulation~\thetag{\ref{first-major-manipulation}} on the hypersurface polynomial equations $F_1,\dots,F_{c+r}$, for every $n$-tuple $1 \leqslant j_1 < \cdots < j_n \leqslant c$, we receive a  twisted symmetric  differential $n$-form:
\begin{equation}\label{first-class-symmetric-differential-n-forms}
\phi_{j_1,\dots,j_n}^{\nu}
\in
\Gamma
\big(
X,\mathsf{Sym}^{n}\,\Omega_V(\heartsuit_{j_1,\dots,j_n}^{\nu})
\big),
\end{equation}
whose twisted degree
$\heartsuit_{j_1,\dots,j_n}^{\nu}$, according to the formula~\thetag{\ref{value-of-k}}, is negative:
\[
\aligned
\explain{Use $\deg F_i=d+\epsilon_i\leqslant d+\delta_{c+r+1}$}
\ \ \ \ \ \ \ \ \ \
\heartsuit_{j_1,\dots,j_n}^{\nu}
&
\leqslant
N\,(d+\delta_{c+r+1})
-
\big[
N\,(d-\delta_N)
+
\mu_{N,0}
\big]
+
N
+
1
\\
&
=
N\,\delta_N
+
N\,(\delta_{c+r+1}+1)
+
1
-
\mu_{N,0}
\\
\explain{Use \thetag{\ref{mu_0>?}} for $l=N$}
\ \ \ \ \ \ \ \ \ \ \ \ \ \ \ \ \ \ \ \ \ \ \
&
\leqslant
-n\,\varheartsuit.
\endaligned
\]

Secondly, for every integer $\tau=0 \cdots N-1$, for every index $\rho=\tau+1 \cdots N$, applying Proposition~\ref{general-holomorphic-symmetric-horizontal-n-forms}, \ref{general-holomorphic-symmetric-forms} with respect to the second kind of manipulation~\thetag{\ref{second-major-manipulation}}  on the hypersurface polynomial equations $F_1,\dots,F_{c+r}$, for every $n$-tuple $1 \leqslant j_1 < \cdots < j_n \leqslant c$, we receive a twisted symmetric  differential $n$-form:
\begin{equation}
\label{second-class-symmetric-differential-n-forms}
\psi_{j_1,\dots,j_n}^{\tau,\rho}
\in
\Gamma
\big(
X,\mathsf{Sym}^{n}\,\Omega_V(\heartsuit_{j_1,\dots,j_n}^{\tau,\rho})
\big),
\end{equation}
whose twisted degree $\heartsuit_{j_1,\dots,j_n}^{\tau,\rho}$, according to the formula~\thetag{\ref{value-of-k}}, is negative too:
\[
\aligned
\heartsuit_{j_1,\dots,j_n}^{\tau,\rho}
&
\leqslant
N\,(d+\delta_{c+r+1})
-
\sum_{k=0}^{\tau}\,
(d-N\,\mu_{N,k})
-
(N-\tau-1)\,
(d-\delta_N)
-
\mu_{N,\tau+1}
+
N
+1
\\
&
=
\sum_{k=0}^{\tau}\,
N\,\mu_{N,k}
+
(N-\tau-1)\,
\delta_N
+
N\,
(\delta_{c+r+1}+1)
+
1
-
\mu_{N,\tau+1}
\\
&
\leqslant
-n\,\varheartsuit
\ \ \ \ \ \ \ \ \ \ \ \ \ \ \ \ \ \ \ \ 
\explain{use \thetag{\ref{mu_k>?}} \text{ for } $l=N$, $k=\tau+1$}.
\endaligned
\]

\subsection{Moving coefficients method with some vanishing coordinates}
\label{The-moving-coefficients-method-for-intersections-with-coordinate-hyperplanes}
To investigate further the moving coefficients method, for all integers $1\leqslant \eta \leqslant n-1$, for every sequence of ascending indices :
\[
0\leqslant v_1<\dots<v_{\eta}\leqslant N,
\] 
take the intersection of $X$ with the $\eta$ coordinate
hyperplanes:
\[
{}_{v_1,\dots,v_{\eta}}X
:=
X
\cap
\{z_{v_1}=0\}
\cap
\dots
\cap
\{z_{v_{\eta}}=0\}
.
\] 
Applying Proposition~\ref{Prop-3.9},
in order to obtain more symmetric differential $(n-\eta)$-forms having negative twisted degree,
we carry on manipulations as follows, which are much the same as before.

First, write the $(N-\eta+1)$ remaining numbers of the set-minus:
\[
\{0,\dots,N\}
\setminus 
\{v_1,\dots,v_{\eta}\}
\] 
in the ascending order:
\[
r_0<\cdots<r_{N-\eta}.
\]
Note that in Proposition~\ref{Prop-3.9}, the coefficient terms associated with the vanishing variables $z_{v_1},\dots,z_{v_{\eta}}$ play no role, therefore we decompose $F_i$ into two parts. The first part (the first two lines below) is a very analog of~\thetag{\ref{original-F}}
involving only the  variables $z_{r_0},\dots,z_{r_{N-\eta}}$, while the second part (the third line) collects all the residue terms involving at least one of the vanishing coordinates $z_{v_1},\dots,z_{v_{\eta}}$:
\begin{equation}\label{F-eta}
\aligned
F_i
=
\sum_{j=0}^{N-\eta}\,
A_i^{r_j}\,
z_{r_j}^{d}
&
+
\sum_{l=c+r+1}^{N-\eta-1}\,
\sum_{0\leqslant j_0<\cdots<j_l\leqslant N-\eta}\,
\sum_{k=0}^l\,
M_i^{r_{j_0},\dots,r_{j_l};r_{j_k}}\,
z_{r_{j_0}}^{\mu_{l,k}}
\cdots
\widehat{z_{r_{j_k}}^{\mu_{l,k}}}
\cdots
z_{r_{j_l}}^{\mu_{l,k}}
z_{r_{j_k}}^{d-l\mu_{l,k}}
+
\\
& 
+
\sum_{k=0}^{N-\eta}\,
M_i^{r_0,\dots,r_{N-\eta};r_k}\,
z_{r_{0}}^{\mu_{N-\eta,k}}
\cdots
\widehat{z_{r_{k}}^{\mu_{N-\eta,k}}}
\cdots
z_{r_{N-\eta}}^{\mu_{N-\eta,k}}
z_{r_{k}}^{d-(N-\eta)\,\mu_{N-\eta,k}}
+
\\
&
+ 
\underbrace
{
\mathrm{(Residue\,Terms)}_i^{v_1,\dots,v_\eta}
}
_{\text{negligible in the coming applications}}.
\endaligned
\end{equation}

Since every power associated with the vanishing variables $z_{v_1},\dots,z_{v_{v_\eta}}$ is $\geqslant 2$ thanks to the Algorithm in subsection~\ref{subsection:constructing-algorithm}, all the 
$\mathrm{(Residue\,Terms)}_i^{v_1,\dots,v_\eta}$ lie in the ideal:
\[
(z_{v_1}^2,\dots,z_{v_{\eta}}^2) 
\subset 
\mathbb{K}[z_0,\dots,z_N].
\] 
Moreover, using for instance the lexicographic order, we can write them as:
\begin{equation}\label{Residue-Terms=?}
\mathrm{(Residue\,Terms)}_i^{v_1,\dots,v_\eta}
=
\sum_{j=1}^{\eta}\,
R_i^{v_1,\dots,v_\eta;v_j}\,
z_{v_j}^2,
\end{equation}
where $R_i^{v_1,\dots,v_\eta;v_j}$ are the homogeneous polynomials uniquely
determined by solving:
\begin{footnotesize}
\[
R_i^{v_1,\dots,v_\eta;v_j}\,
z_{v_j}^2
=
A_i^{v_j}\,z_{v_j}^{d}
+
\sum_{l=c+r+1}^{N}\,
\sum_{\substack
{
0\leqslant j_0<\cdots<j_l\leqslant N\\
{\rm min}\,
\big(
\{
j_0,\dots,j_l
\}
\setminus
\{
v_1,\dots,v_{\eta}
\}
\big)
=
v_j
}}\,
\sum_{k=0}^l\,
M_i^{j_0,\dots,j_l;j_k}\,
z_{j_0}^{\mu_{l,k}}
\cdots
\widehat{z_{j_k}^{\mu_{l,k}}}
\cdots
z_{j_l}^{\mu_{l,k}}
z_{j_k}^{d-l\mu_{l,k}}.
\]
\end{footnotesize}

Observing that the first two lines of~\thetag{\ref{F-eta}} have exactly the same structure as~\thetag{\ref{original-F}}, by mimicking the manipulation of rewriting~\thetag{\ref{original-F}} as~\thetag{\ref{first-rewrite-F_i}}, we can rewrite the first two lines of~\thetag{\ref{F-eta}} as:
\begin{equation}\label{first-part-elt}
\sum_{j=0}^{N-\eta}\,
{}_{v_1,\dots,v_{\eta}}C_i^{r_j}\,
z_{r_j}^{d-\delta_{N-\eta}}
+
\sum_{k=0}^{N-\eta}\,
M_i^{r_0,\dots,r_{N-\eta};r_k}\,
z_{r_0}^{\mu_{N-\eta,k}}
\cdots
\widehat{z_{r_k}^{\mu_{N-\eta,k}}}
\cdots
z_{r_{N-\eta}}^{\mu_{N-\eta,k}}
z_{r_k}^{d-(N-\eta)\,\mu_{N-\eta,k}},
\end{equation}
where the integer $\delta_{N-\eta}$ was defined in ~\thetag{\ref{delta_l=?}} for $l=N-\eta-1$:
\[
\delta_{N-\eta}
=
(N-\eta-1)\,
\mu_{N-\eta-1,N-\eta-1},
\] 
and where the homogeneous polynomials ${}_{v_1,\dots,v_{\eta}}C_i^{r_j}$
are obtained in the same way as $C_i^j$ in~\thetag{\ref{C_i^j}}:
\begin{footnotesize}
\[
{}_{v_1,\dots,v_{\eta}}C_i^{r_j}
:=
A_i^{r_j}\,z_j^{\delta_{N-\eta}}
+
\sum_{l={c+r+1}}^{N-\eta-1}\,
\sum_{
\substack
{0\leqslant j_0<\cdots<j_l\leqslant N-\eta
\\
j_k=j\,\text{for some}\,0\leqslant k \leqslant l
}}\,
M_i^{r_{j_0},\dots,r_{j_l};r_{j_k}}\,
z_{r_{j_0}}^{\mu_{l,k}}
\cdots
\widehat{z_{r_{j_k}}^{\mu_{l,k}}}
\cdots
z_{r_{j_l}}^{\mu_{l,k}}
z_{r_{j_k}}^{(N-\eta-1)\mu_{N-\eta-1,N-\eta-1}-l\mu_{l,k}}.
\]
\end{footnotesize}

Now substituting~\thetag{\ref{first-part-elt}}, \thetag{\ref{Residue-Terms=?}} into the equation~\thetag{\ref{F-eta}}, we rewrite $F_i$ as:
\begin{equation}\label{F_i-rewrite-eta}
\aligned
F_i
=
\sum_{j=0}^{N-\eta}\,
{}_{v_1,\dots,v_{\eta}}C_i^{r_j}\,
z_{r_j}^{d-\delta_{N-\eta}}
&
+
\sum_{k=0}^{N-\eta}\,
M_i^{r_0,\dots,r_{N-\eta};r_k}\,
z_{r_0}^{\mu_{N-\eta,k}}
\cdots
\widehat{z_{r_k}^{\mu_{N-\eta,k}}}
\cdots
z_{r_{N-\eta}}^{\mu_{N-\eta,k}}
z_{r_k}^{d-(N-\eta)\,\mu_{N-\eta,k}}
+
\\
&
\ \ \ \ \ \ \ \ \ \ \ \ \ \ \ \ \ \ \ 
+
\underbrace{
\sum_{j=1}^{\eta}\,
R_i^{v_1,\dots,v_\eta;v_j}\,
z_{v_j}^2.
}_{\text{negligible in the coming applications}}
\endaligned
\end{equation}

Now, noting that the first line of $F_i$ in~\thetag{\ref{F_i-rewrite-eta}}
has exactly the same structure as~\thetag{\ref{first-rewrite-F_i}}, 
 we repeat the two kinds of manipulations, 
as briefly summarized below.

The first kind of manipulations are, for every chosen index 
$\nu=0\,\cdots\,N-\eta$, to associate all these $(N+1-\eta)$ moving coefficient terms:
\[
\sum_{k=0}^{N-\eta}\,
M_i^{r_0,\dots,r_{N-\eta};r_k}\,
z_{r_0}^{\mu_{N-\eta,k}}
\cdots
\widehat{z_{r_k}^{\mu_{N-\eta,k}}}
\cdots
z_{r_{N-\eta}}^{\mu_{N-\eta,k}}
z_{r_k}^{d-(N-\eta)\,\mu_{N-\eta,k}}
\]
with the term 
${}_{v_1,\dots,v_{\eta}}C_i^{r_\nu}\,
z_{r_\nu}^{d-\delta_{N-\eta}}$ 
by rewriting \thetag{\ref{first-part-elt}} as:
\begin{equation}\label{first-major-manipulation-eta}
\sum_
{\substack{
j=0\\
j\neq \nu}}
^{N-\eta}\,
{}_{v_1,\dots,v_{\eta}}C_i^{r_j}\,
z_{r_j}^{d-\delta_{N-\eta}}
+
{}_{v_1,\dots,v_{\eta}}
T_i^{r_\nu}\,
z_{r_\nu}^{\mu_{N-\eta,0}},
\end{equation}
where ${}_{v_1,\dots,v_{\eta}}T_i^{r_\nu}$ 
is the homogeneous polynomial uniquely determined by
solving:
\begin{footnotesize}
\begin{equation}\label{equation-T-eta}
{}_{v_1,\dots,v_{\eta}}
T_i^{r_\nu}\,
z_{r_\nu}^{\mu_{N-\eta,0}}
=
{}_{v_1,\dots,v_{\eta}}C_i^{r_\nu}\,
z_{r_\nu}^{d-\delta_{N-\eta}}
+
\sum_{k=0}^{N-\eta}\,
M_i^{r_0,\dots,r_{N-\eta};r_k}\,
z_{r_0}^{\mu_{N-\eta,k}}
\cdots
\widehat{z_{r_k}^{\mu_{N-\eta,k}}}
\cdots
z_{r_{N-\eta}}^{\mu_{N-\eta,k}}
z_{r_k}^{d-(N-\eta)\,\mu_{N-\eta,k}}.
\end{equation}
\end{footnotesize}

The second kind of manipulations are, for every integer $\tau=0 \cdots N-\eta-1$, for every index $\rho=\tau+1 \cdots N-\eta$,
to associate each of the first $(\tau+1)$ moving coefficient terms:
\[
M_i^{r_0,\dots,r_{N-\eta};r_k}\,
z_{r_0}^{\mu_{N-\eta,k}}
\cdots
\widehat{z_{r_k}^{\mu_{N-\eta,k}}}
\cdots
z_{r_{N-\eta}}^{\mu_{N-\eta,k}}
z_{r_k}^{d-(N-\eta)\,\mu_{N-\eta,k}}
\ \ \ \ \ \ \ \ \
{\scriptstyle{(k\,=\,0\,\cdots\,\tau)}} 
\]
with the corresponding term ${}_{v_1,\dots,v_{\eta}}C_i^{r_k}\,z_{r_k}^{d-\delta_{N-\eta}}$ 
and to associate the remaining $(N-\eta-\tau)$ moving coefficient terms:
\[
\sum_{j=\tau+1}^{N-\eta}\,
M_i^{r_0,\dots,r_{N-\eta};r_j}\,
z_{r_0}^{\mu_{N-\eta,j}}
\cdots
\widehat{z_{r_j}^{\mu_{N-\eta,j}}}
\cdots
z_{r_{N-\eta}}^{\mu_{N-\eta,j}}
z_{r_j}^{d-(N-\eta)\,\mu_{N-\eta,j}},
\]
with the term 
${}_{v_1,\dots,v_{\eta}}C_i^{r_\rho}\,
z_{r_\rho}^{d-\delta_{N-\eta}}$  
by rewriting \thetag{\ref{first-part-elt}} as:
\begin{equation}\label{second-major-manipulation-eta}
\sum_{k=0}^{\tau}\,
{}_{v_1,\dots,v_{\eta}}E_i^{r_k}\,
z_{r_k}^{d-(N-\eta)\,\mu_{N-\eta,k}}
+
\sum_
{\substack{
j=\tau+1\\
j\neq \rho}}
^{N-\eta}\,
{}_{v_1,\dots,v_{\eta}}C_i^{r_j}\,
z_{r_j}^{d-\delta_{N-\eta}}
+
{}_{v_1,\dots,v_{\eta}}P_i^{r_\tau,r_\rho}\,
z_{r_\rho}^{\mu_{N-\eta,\tau+1}},
\end{equation}
where ${}_{v_1,\dots,v_{\eta}}E_i^{r_k}$ and 
${}_{v_1,\dots,v_{\eta}}P_i^{r_\tau,r_\rho}$ 
are the homogeneous polynomials uniquely determined by
solving:
\begin{footnotesize}
\begin{equation}\label{equations-E-R-eta}
\aligned
{}_{v_1,\dots,v_{\eta}}E_i^{r_k}\,
z_{r_k}^{d-(N-\eta)\,\mu_{N-\eta,k}}
&
=
{}_{v_1,\dots,v_{\eta}}C_i^{r_k}\,
z_{r_k}^{d-\delta_{N-\eta}}
+
M_i^{r_0,\dots,r_{N-\eta};r_k}\,
z_{r_0}^{\mu_{N-\eta,k}}
\cdots
\widehat{z_{r_k}^{\mu_{N-\eta,k}}}
\cdots
z_{r_{N-\eta}}^{\mu_{N-\eta,k}}
z_{r_k}^{d-(N-\eta)\,\mu_{N-\eta,k}},
\\
{}_{v_1,\dots,v_{\eta}}P_i^{r_\tau,r_\rho}\,
z_{r_\rho}^{\mu_{N-\eta,\tau+1}}
&
=
{}_{v_1,\dots,v_{\eta}}C_i^{r_\rho}\,
z_{r_\rho}^{d-\delta_{N-\eta}}
+
\sum_{j=\tau+1}^{N-\eta}\,
M_i^{r_0,\dots,r_{N-\eta};r_j}\,
z_{r_0}^{\mu_{N-\eta,j}}
\cdots
\widehat{z_{r_j}^{\mu_{N-\eta,j}}}
\cdots
z_{r_{N-\eta}}^{\mu_{N-\eta,j}}
z_{r_j}^{d-(N-\eta)\,\mu_{N-\eta,j}},
\endaligned
\end{equation}
\end{footnotesize}\!\!
which is possible by the Algorithm in subsection~\ref{subsection:constructing-algorithm}.

To summarize, taking the two forms~\thetag{\ref{first-major-manipulation-eta}}, \thetag{\ref{second-major-manipulation-eta}} of the first line of \thetag{\ref{F_i-rewrite-eta}} into account,  we can rewrite $F_i$ in the following two ways. The first one is:
\begin{equation}\label{first-way-F_i}
\aligned
F_i
=
\sum_
{\substack{
j=0\\
j\neq \nu}}
^{N-\eta}\,
{}_{v_1,\dots,v_{\eta}}C_i^{r_j}\,
z_{r_j}^{d-\delta_{N-\eta}}
+
{}_{v_1,\dots,v_{\eta}}
T_i^{r_\nu}\,
z_{r_\nu}^{\mu_{N-\eta,0}}
+
\underbrace
{
\sum_{j=1}^{\eta}\,
R_i^{v_1,\dots,v_\eta;v_j}\,
z_{v_j}^2
}
_{\text{negligible in our coming applications}},
\endaligned
\end{equation}
and the second one is:
\begin{footnotesize}
\begin{equation}\label{second-way-F_i}
F_i
=
\sum_{k=0}^{\tau}\,
{}_{v_1,\dots,v_{\eta}}E_i^{r_k}\,
z_{r_k}^{d-(N-\eta)\,\mu_{N-\eta,k}}
+
\sum_
{\substack{
j=\tau+1\\
j\neq \rho}}
^{N-\eta}\,
{}_{v_1,\dots,v_{\eta}}C_i^{r_j}\,
z_{r_j}^{d-\delta_{N-\eta}}
+
{}_{v_1,\dots,v_{\eta}}P_i^{r_\tau,r_\rho}\,
z_{r_\rho}^{\mu_{N-\eta,\tau+1}}
+
\underbrace
{
\sum_{j=1}^{\eta}\,
R_i^{v_1,\dots,v_\eta;v_j}\,
z_{v_j}^2
}
_{\text{negligible in our coming applications}}.
\end{equation}
\end{footnotesize}

Firstly, applying Proposition~\ref{Prop-3.9} to~\thetag{\ref{first-way-F_i}}, for every $(n-\eta)$-tuple: 
\[
1 
\leqslant 
j_1 
< 
\cdots 
< 
j_{n-\eta}
\leqslant 
c,
\] 
we receive a  twisted symmetric  differential $(n-\eta)$-form:
\begin{equation}\label{first-class-symmetric-differential-n-eta-forms}
{}_{v_1,\dots,v_{\eta}}\phi_{j_1,\dots,j_{n-\eta}}^{\nu}
\in
\Gamma
\big(
{}_{v_1,\dots,v_{\eta}}X,\mathsf{Sym}^{n-\eta}\,\Omega_V\,({}_{v_1,\dots,v_{\eta}}\heartsuit_{j_1,\dots,j_{n-\eta}}^{\nu})
\big),
\end{equation}
whose twisted degree ${}_{v_1,\dots,v_{\eta}}\heartsuit_{j_1,\dots,j_{n-\eta}}^{\nu}$, according to the formula~\thetag{\ref{value-of-heart}}, is negative:
\begin{footnotesize}
\[
\aligned
{}_{v_1,\dots,v_{\eta}}\heartsuit_{j_1,\dots,j_{n-\eta}}^{\nu}
&
\leqslant
(N-\eta)\,(d+\delta_{c+r+1})
-
\big[
(N-\eta)\,(d-\delta_{N-\eta})
+
\mu_{N-\eta,0}
\big]
+(N-\eta)
+1
\\
&
=
(N-\eta)\,\delta_{N-\eta}
+
(N-\eta)\,
(\delta_{c+r+1}+1)
+
1
-
\mu_{N-\eta,0}
\\
\explain{use \thetag{\ref{mu_0>?}} for $l=N-\eta$}
\ \ \ \ \ \ \ \ \ \ \ \ \ \ \
&
\leqslant
-\,(n-\eta)\,\varheartsuit.
\endaligned
\]
\end{footnotesize}

Secondly, applying Proposition~\ref{Prop-3.9} to~\thetag{\ref{second-way-F_i}},
for every $(n-\eta)$-tuple: 
\[
1 
\leqslant 
j_1 
< 
\cdots 
< 
j_{n-\eta}
\leqslant 
c,
\]  
we receive a  twisted symmetric  differential $(n-\eta)$-form:
\begin{equation}\label{second-class-symmetric-differential-n-forms-tau,rho}
{}_{v_1,\dots,v_{\eta}}\psi_{j_1,\dots,j_{n-\eta}}^{\tau,\rho}
\in
\Gamma
\big(
{}_{v_1,\dots,v_{\eta}}X,\mathsf{Sym}^{n-\eta}\,\Omega_V({}_{v_1,\dots,v_{\eta}}\heartsuit_{j_1,\dots,j_{n-\eta}}^{\tau,\rho})
\big),
\end{equation}
whose twisted degree ${}_{v_1,\dots,v_{\eta}}\heartsuit_{j_1,\dots,j_{n-\eta}}^{\tau,\rho}$, according to the formula~\thetag{\ref{value-of-heart}}, is negative also:
\begin{footnotesize}
\[
\aligned
{}_{v_1,\dots,v_{\eta}}\heartsuit_{j_1,\dots,j_{n-\eta}}^{\tau,\rho}
&
\leqslant
(N-\eta)\,(d+\delta_{c+r+1})
-
\sum_{k=0}^{\tau}\,
(d-(N-\eta)\,\mu_{N-\eta,k})
-
(N-\eta-\tau-1)\,
(d-\delta_{N-\eta})
-
\\
&\ \ \ \ \ \ \ \ \ \ \ \ \ \ \
-
\mu_{N-\eta,\tau+1}
+
N-\eta
+1
\\
&
=
\sum_{k=0}^{\tau}\,
(N-\eta)\,\mu_{N-\eta,k}
+
(N-\eta-\tau-1)\,
\delta_{N-\eta}
+
(N-\eta)\,
(\delta_{c+r+1}+1)
+
1
-
\mu_{N-\eta,\tau+1}
\\
&
\leqslant
-\,(n-\eta)\,\varheartsuit
\ \ \ \ \ \ \ \ \ \ \ \ \ \ \ \ \ \ \ \ 
\explain{use \thetag{\ref{mu_k>?}} \text{ for } $l=N-\eta$, $k=\tau+1$}.
\endaligned
\]
\end{footnotesize}

\section{\bf Basic technical preparations}
\subsection{Fibre dimension estimates}
The first theorem below coincides with our geometric intuition,
and one possible proof is mainly based on Complex Implicit Function Theorem. Here,
we give a short proof by the same method as~\cite[p.~76, Theorem 7]{Shafarevich-1994}.

\begin{Theorem}
{\bf (Analytic Fibre Dimension Estimate)}
\label{analytic-fibre-Dimension-Estimate}
Let $X, Y$ be two complex spaces and let $f\colon X\rightarrow Y$ be a regular map.  Then the maximum fibre dimension is bounded from below by the dimension of the source space $X$ minus the dimension of the target space $Y$:
\[
\max_{y\,\in\,Y}\,
\dim_{\mathbb{C}}\,
f^{-1}(y)
\geqslant
\dim_{\mathbb{C}}\,
X
-
\dim_{\mathbb{C}}\,
Y.
\]
Equivalently:
\begin{equation}
\label{analytic-fundamental-dimension-inequality}
\dim_{\mathbb{C}}\,
X
\leqslant 
\dim_{\mathbb{C}}\,
Y
+
\max_{y\,\in\,Y}\,
\dim_{\mathbb{C}}\,
f^{-1}(y).
\end{equation}
\end{Theorem}

\begin{proof}
For every point $x\in X$, let $f(x)=:z\,\in\,Y$, and 
denote the germ dimension of $Y$ at this point by:
\[
d_z
:=
\dim_{\mathbb{C}}\,
(Y,z).
\]
Then we can find holomorphic function germs
$g_1,\dots,g_{d_z}\in \mathcal{O}_{Y,z}$ vanishing at $z$ such that:
\[
(Y,z)
\cap
\{
g_1
=
\cdots
=
g_{d_z}
=
0
\}\,
=\,
\{
z
\}.
\]
Pulling back by the holomorphic map $f$, we therefore realize:
\[
(X,x)
\cap
\{
g_1
\circ
f
=
\cdots
=
g_{d_z}
\circ
f
=
0
\}\,
=\,
\big(
f^{-1}(z),
x
\big).
\]
Now, counting the germ dimension, we receive the estimate:
\[
\dim_{\mathbb{C}}\,
\big(
f^{-1}(z),
x
\big)
\geqslant
\dim_{\mathbb{C}}\,
(X,x)
-
d_z
=
\dim_{\mathbb{C}}\,
(X,x)
-
\dim_{\mathbb{C}}\,
(Y,z),
\] 
hence:
\[
\aligned
\dim_{\mathbb{C}}\,
(X,x)
&
\leqslant
\dim_{\mathbb{C}}\,
(Y,z)
+
\dim_{\mathbb{C}}\,
\big(
f^{-1}(z),
x
\big)
\\
&
\leqslant
\dim_{\mathbb{C}}\,
Y
+
\max_{y\,\in\,Y}\,
\dim_{\mathbb{C}}\,
f^{-1}(y).
\endaligned
\]
Finally, let $x\in X$ vary in the above estimate, thanks to:
\[
\dim_{\mathbb{C}}\,
X
=
\max_{x\,\in\,X}\,
\dim_{\mathbb{C}}\,
(X,x),
\]
we receive the desired estimate~\thetag{\ref{analytic-fundamental-dimension-inequality}}.
\end{proof}

With the same proof (cf. \cite[p.~169, Proposition 12.30; p.~140, Corollary 10.27]{Peskine-1996}), here is an algebraic version of the analytic fibre dimension estimate above, for every algebraically closed field $\mathbb K$ and for the category of $\mathbb K$-varieties in the classical sense (\cite[\textsection 1.3, p. 15]{Hartshorne-1977}), where  dimension is defined to be the Krull dimension (\cite[\textsection 1.1, p. 6]{Hartshorne-1977}).

\begin{Theorem}
[{\bf Algebraic Fibre Dimension Estimate}]
\label{algebraic-fibre-Dimension-Estimate}
Let $X,Y$ be two $\mathbb K$-varieties, and let $f\colon X\rightarrow Y$ be a morphism. Then the dimension of the source variety $X$ is bounded from above by the sum of the dimension of the target variety $Y$ plus the maximum fibre dimension:
\begin{equation}
\label{inq:algebraic-fibre-dimension-estimate}
\dim\,
X
\leqslant 
\dim\,
Y
+
\max_{y\,\in\,Y}\,
\dim\,
f^{-1}(y).
\end{equation}
\end{Theorem}

In our future applications, $f$ will always be surjective, so one may also refer to~\cite[p.~76, Theorem 7]{Shafarevich-1994}.
The above theorem will prove fundamental in estimating every base locus involved in this paper. 

\begin{Corollary}
\label{transferred-codimension-estimate}
Let $X,Y$ be two $\mathbb K$-varieties, and let $f\colon X\rightarrow Y$ be a morphism such that every fibre satisfies the dimension estimate:
\[
\dim\,
f^{-1}(y)
\leqslant
\dim\,
X
-
\dim\,
Y
\qquad
{\scriptstyle(\forall\,y\,\in\,Y)}.
\]  
Then for every subvariety 
$Z\subset Y$, its inverse image:
\[
f^{-1}(Z)\,
\subset\,
X
\]
satisfies the transferred codimension estimate:
\[
\cdim\,
f^{-1}(Z)
\geqslant
\cdim\,
Z.
\eqno
\qed
\]
\end{Corollary}

\subsection{Matrix-rank estimates}
This subsection recalls some elementary rank estimates in linear algebra.

\begin{Lemma}
Let $\mathbb{K}$ be a field and let $W$ be a finite-dimensional $\mathbb{K}$-vector space generated by a set of vectors $\mathcal{B}$. Then every subset $\mathcal{B}_1\subset \mathcal{B}$ that consists of $\mathbb{K}$-linearly independent vectors can be extended to a bigger subset $\mathcal{B}_2\subset \mathcal{B}$ which forms a basis of $W$.
\qed
\end{Lemma}

\begin{Lemma}
\label{rank(2c*(N-k+1)<N-k}
Let $\mathbb{K}$ be a field, and let $V$ be a $\mathbb{K}$-vector space. For all
positive integers $e,k,l\geqslant 1$ with $k\geqslant l$, let $v_1,\dots,v_e,v_{e+1},\dots,v_{e+k}$ be $(e+k)$ vectors such that:

\begin{itemize}

\smallskip\item[{\bf (i)}] 
$v_1,\dots, v_e$ are $\mathbb{K}$-linearly independent;

\smallskip\item[{\bf (ii)}] 
for every sequence of $l$ ascending indices between $e+1$ and $e+k$:
\[
e+1
\leqslant
i_1
<
\cdots
<
i_l
\leqslant
e+k,
\]
there holds the rank inequality:
\[
\rank_\mathbb{K}\,
\{
v_1,\dots,v_e,v_{i_1},\dots,v_{i_l}
\}
\leqslant
e+l-1.
\]  
\end{itemize}

\noindent
Then there holds the rank estimate:
\[
\rank_\mathbb{K}\,
\{
v_1,\dots,v_e,v_{e+1},\dots,v_{e+k}
\}
\leqslant
e+l-1.
\]
\end{Lemma}

\begin{proof}
Assume on the contrary that:
\[
\rank_\mathbb{K}\,
\{
v_1,\dots,v_e,v_{e+1},\dots,v_{e+k}
\}
=:
e+l_0
\geqslant
e+l,
\]
that is, $l_0\geqslant l$.

Now applying the above lemma to:
\[
\aligned
W
&
=
\mathsf{Span}_\mathbb{K}\,
\underbrace{
\{
v_1,\dots,v_e,v_{e+1},\dots,v_{e+k}
\}}_{=\,\mathcal B}\,\,, 
\\
\mathcal{B}_1
&
=
\{
v_1,
\dots,
v_e
\},
\endaligned
\]
we receive a certain basis of $V$:
\[
\mathcal{B}_2
=
\{
v_1,\dots,v_e,v_{i_1},\dots,v_{i_{l_0}}
\}.
\]
In particular, as $l_0\geqslant l$, the first $(e+l)$ vectors in $\mathcal{B}_2$ are $\mathbb{K}$-linearly independent:
\[
\rank_\mathbb{K}\,
\{
v_1,\dots,v_e,v_{i_1},\dots,v_{i_l}
\}
=
e+l,
\]  
which contradicts condition {\bf (ii)}.
\end{proof}

Let $\mathbb{K}$ be a field, and let 
$p,q,e,l$ be positive integers with: 
\[
\min
\{
p,q
\}
\geqslant 
e+l.
\]
Let $\mathsf{M} \in \mathsf{Mat}_{p\times q}(\mathbb{K})$ be a 
$p\times q$ matrix. For all sequences of ascending indices :
\[
1
\leqslant 
i_1 
<
\cdots
<
i_k
\leqslant
p,
\]
let us denote by $\mathsf{M}_{i_1,\dots,i_k}$ the $k\times q$ submatrix of $\mathsf{M}$ that consists of the rows $i_1,\dots,i_k$, and for all sequences of ascending indices:
\[
1
\leqslant
j_1
<
\cdots
<
j_l
\leqslant
q,
\]
let us denote by 
$\mathsf{M}_{i_1,\dots,i_k}^{j_1,\dots,j_l}$ the $k\times l$ submatrix of
$\mathsf{M}_{i_1,\dots,i_k}$ that consists of the columns
$j_1,\dots,j_l$.
  
\begin{Lemma}
If the first $e$ rows of the matrix $\mathsf{M}$ are of full rank:
\[
\rank_{\mathbb{K}}\,
\mathsf{M}_{1,\dots,e}
=
e,
\]
and if all the $(e+l)\times (e+l)$ submatrices always selecting the first $e$ rows of $\mathsf{M}$  are degenerate:
\begin{equation}\label{rank-M<=m+l-1}
\rank_{\mathbb{K}}\,
\mathsf{M}_{1,\dots,e,i_1,\dots,i_l}^{j_1,\dots,j_{e+l}}
\leqslant 
e+l-1
\ \ \ \ \ \ \ \ \ \
{\scriptstyle{(\forall\,e+1\,\leqslant\,i_1\,<\,\cdots\,<\,i_l\,\leqslant\,p;\,1\,\leqslant\, j_1\,<\,\cdots\,<\,j_{e+l}\,\leqslant\,q)}},
\end{equation}
then there holds the rank estimate:
\[
\rank_{\mathbb{K}}\,
\mathsf{M}
\leqslant
e+l-1
\] 
\end{Lemma}

\begin{proof}
For every fixed sequence of ascending indices:
\[
e+1\leqslant i_1<\cdots<i_l\leqslant p,
\]
the rank inequalities~\thetag{\ref{rank-M<=m+l-1}} yields:
\[
\rank_{\mathbb{K}}\,
\mathsf{M}_{1,\dots,e,i_1,\dots,i_l}
\leqslant
e+l-1
\]
Now applying the previous lemma to the rows of the matrix $\mathsf{M}$, we conclude the desired rank estimate.
\end{proof}

\begin{Lemma}
\label{rank widehat-H_j = rank H}
Let $\mathbb{K}$ be a field and let $e,m$ be positive integers. Let ${\sf H}\in \mathsf{Mat}_{e\times m}(\mathbb{K})$ be an $e\times m$ matrix with entries
in $\mathbb{K}$ such that
the sum of all $m$ columns of ${\sf H}$ vanishes:
\begin{equation}\label{H_1+...+H_m=0}
{\sf H}_1
+
\cdots
+
{\sf H}_{m}
=
\mathbf{0},
\end{equation}
where we denote by ${\sf H}_i$ the $i$-th column of ${\sf H}$.
Then for every integer $j=1\cdots m$, the $e\times m$ submatrix $\widehat{{\sf H}}_j$ of ${\sf H}$ obtained by omitting the $j$-th column still has the same rank:
\[
\rank_{\mathbb{K}}\,
\widehat{{\sf H}}_j
=
\rank_{\mathbb{K}}\,
{\sf H}.
\]
\end{Lemma}

\begin{proof}
Note that \thetag{\ref{H_1+...+H_m=0}} yields: 
\[
{\sf H}_j
=
-({\sf H}_1+\cdots+{\sf H}_{j-1}+{\sf H}_{j+1}+\cdots+{\sf H}_m),
\]
therefore ${\sf H}_j$ lies in the $\mathbb{K}$-linear space generated
by the columns of the matrix $\widehat{{\sf H}}_j$, thus we receive:
\[
\mathsf{Span}_{\mathbb{K}}\,
\{
{\sf H}_1,\dots,\widehat{{\sf H}_j},\dots,{\sf H}_m
\}
=
\mathsf{Span}_{\mathbb{K}}\,
\{
{\sf H}_1,\dots,{\sf H}_m
\}.
\]
Taking the dimension of both sides, we receive the desired rank equality.
\end{proof}

\subsection{Classical codimension formulas}
In an algebraically closed field $\mathbb K$, for all positive integers $p,q\geqslant 1$, 
denote by:
\[
{\sf Mat}_{p\times q}(\mathbb{K})
=
\mathbb{K}^{p\times q}
\]
the space of all $p\times q$ matrices with entries in $\mathbb K$.
For every integer
$0 \leqslant \ell \leqslant \max \{p,q\}$, 
we have a classical formula 
(cf. \cite[p.~247 exercise 10.10, and the proof in p.~733]{Eisenbud-1995}) for the codimension of the subvariety:
\[
\Sigma_\ell^{p,q}
\subset
{\sf Mat}_{p\times q}(\mathbb{K})
\]
which consists of all matrices with rank $\leqslant \ell$.

\begin{Lemma}\label{classical-codimension-formula}
There holds the codimension formula:
\[
\cdim\,
\Sigma_\ell^{p,q}
=
\max\,
\big\{
(p-\ell)\,(q-\ell),\,
0
\big\}.
\eqno
\qed
\]
\end{Lemma}

In applications, we will use the following two direct consequences.

\begin{Corollary} 
\label{classical-codimension-formula-sum-0}
For every integer 
$0 \leqslant \ell \leqslant \max \{p,q-1\}$,
the codimension of the subvariety:
\[ 
{}^0\Sigma_\ell^{p,q}\,
\subset\,
\Sigma_\ell^{p,q},
\]
which consists of matrices whose sum of all the columns vanish,
is:
\[
\cdim\,
{}^0\Sigma_\ell^{p,q}
=
\max\,
\big\{
(p-\ell)\,(q-1-\ell),\,
0
\big\}
+
p.
\eqno
\qed
\]
\end{Corollary}

\begin{proof}
Since every matrix in 
${}^0\Sigma_\ell^{p,q}$ is uniquely determined by
the first $(q-1)$ columns, thanks to 
Lemma~\ref{rank widehat-H_j = rank H},
the projection morphism into the first $(q-1)$ columns: 
\[
\pi
\colon
\ \ \
{}^0\Sigma_\ell^{p,q}\,
\longrightarrow\,
\Sigma_\ell^{p,q-1}
\]
is an isomorphism. 
Remembering that:
\[
\dim\,
\Sigma_\ell^{p,q}
=
\dim\,
\Sigma_\ell^{p,q-1}
+
p,
\]
now a direct application of the preceding lemma finishes the proof.
\end{proof}

\subsection{Surjectivity of evaluation maps}
Given a field $\mathbb K$, for all positive integers $N\geqslant 1$, denote the affine coordinate ring of $\mathbb{K}^{N+1}$ by:
\[
\mathcal{A}
(
\mathbb{K}^{N+1}
)
:=
\mathbb{K}[z_0,\dots,z_N].
\]
For all positive integers $\lambda\geqslant 1$, also denote by:
\[
\mathcal{A}_{\lambda}(\mathbb{K}^{N+1})\,
\subset\,
\mathcal{A}(\mathbb{K}^{N+1})
\] 
the $\mathbb{K}$-linear space spanned by all the degree $\lambda$ homogeneous polynomials:
\[
\mathcal{A}_{\lambda}(\mathbb{K}^{N+1})\,\,
:=
\ 
\oplus_{
\substack{
\alpha_0
+
\cdots
+
\alpha_N
=
\lambda
\\
\alpha_0,
\dots,
\alpha_N
\geqslant 
0
}}\,\,
\mathbb{K}\,
\cdot
z_0^{\lambda_0}
\cdots
z_N^{\lambda_N}
\,\,
\cong\,\,
\mathbb{K}^{\binom{N+\lambda}{N}}.
\]

For every point $z\in \mathbb{K}^{N+1}$, denote by $v_z$ the $\mathbb{K}$-linear evaluation map:
\[
\aligned
v_z\,
\colon\,
\mathcal{A}
(
\mathbb{K}^{N+1}
)\,
&
\longrightarrow\,
\mathbb{K}
\\
f\,
&
\longmapsto\,
f(z),
\endaligned
\]
and for every tangent vector 
$\xi \in \mathrm{T}_z \mathbb{K}^{N+1}\cong\mathbb{K}^{N+1}$, 
denote by $d_z(\xi)$ the $\mathbb{K}$-linear differential evaluation map:
\[
\aligned
d_z(\xi)\,
\colon\,
\mathcal{A}
(
\mathbb{K}^{N+1}
)\,
&
\longrightarrow
\mathbb{K}
\\
f\,
&
\longmapsto\,
d\,f\big{\vert}_z(\xi).
\endaligned
\]
For every polynomial $g\in \mathcal{A}(\mathbb{K}^{N+1})$, 
for every point $z\in \mathbb{K}^{N+1}$, denote by $(g\cdot v)_z$ the $\mathbb{K}$-linear evaluation map:
\[
\aligned
(g\cdot v)_z\,
\colon\,
\mathcal{A}
(
\mathbb{K}^{N+1}
)\,
&
\longrightarrow\,
\mathbb{K}
\\
f\,
&
\longmapsto\,
(gf)(z),
\endaligned
\]
and for every tangent vector 
$\xi \in \mathrm{T}_z \mathbb{K}^{N+1}\cong\mathbb{K}^{N+1}$, 
denote by $d_{z}(g\cdot\,)(\xi)$ the $\mathbb{K}$-linear differential 
evaluation map:
\[
\aligned
d_{z}(g\cdot\,)(\xi)\,
\colon\,
\mathcal{A}
(
\mathbb{K}^{N+1}
)\,
&
\longrightarrow
\mathbb{K}
\\
f\,
&
\longmapsto\,
d\,(gf)\big{\vert}_z(\xi).
\endaligned
\] 

The following Lemma was obtained by Brotbek in another affine coordinates version~\cite[p.~36, Proof of Claim 3]{Brotbek-2014-arxiv}.

\begin{Lemma}
\label{valuation-maps-are-surjective}
For all positive integers $\lambda\geqslant 1$, at every nonzero point $z\in \mathbb{K}^{N+1}\setminus\{0\}$, 
for every tangent vector 
$\xi \in \mathrm{T}_z \mathbb{K}^{N+1}\cong \mathbb{K}^{N+1}$ which does not lie in the line of $z$: 
\[
\xi\,
\not\in\,
\mathbb{K}\cdot z,
\]
restricting on the subspace:
\[
\mathcal{A}_{\lambda}(\mathbb{K}^{N+1})\,
\subset\,
\mathcal{A}(\mathbb{K}^{N+1}),
\] 
the evaluation maps $v_{z}$ and 
$d_{z}(\xi)$
are $\mathbb{K}$-linearly independent. In other words,
the map:
\[
\begin{pmatrix}
v_z
\\
d_z(\xi)
\end{pmatrix}\,
\colon\,
\mathcal{A}_{\lambda}(\mathbb{K}^{N+1})\,
\longrightarrow\,
\mathbb{K}^2
\]
is surjective.
\end{Lemma}

\begin{proof}
{\em Step 1.}
For the case $\lambda=1$, this lemma is evident. In fact, now
every polynomials $\ell\in\mathcal{A}_1(\mathbb{K}^{N+1})$ can be viewed as, by evaluating $\ell(z)$ at every point 
$z\in \mathbb{K}^{N+1}$, a $\mathbb K$-linear form:
\[
\ell
\in
\big(
\mathbb{K}^{N+1}
\big)^\vee,
\]
thus there is a canonical $\mathbb{K}$-linear isomorphism:
\[
\mathcal{A}_1(\mathbb{K}^{N+1})
\cong
\big(
\mathbb{K}^{N+1}
\big)^\vee.
\]
Moreover, it is easy to see: 
\[
d\ell\big{\vert}_z(\xi)=\ell(\xi).
\]
Since $z,\xi\in \mathbb{K}^{N+1}$ are $\mathbb{K}$-linearly independent,
now recalling the Riesz Representation Theorem in linear algebra:
\begin{equation}
\label{Riesz-Theorem}
\mathbb{K}^{N+1}
\cong
\Big(
\big(
\mathbb{K}^{N+1}
\big)^{\vee}
\Big)^{\vee},
\end{equation}
we conclude the claim.

\smallskip
{\em Step 2.}
For the general case $\lambda\geqslant 2$, 
first, we choose a degree $(\lambda-1)$ homogeneous polynomial
$g\in \mathcal{A}_{\lambda-1}(\mathbb{K}^{N+1})$ 
with $g(z)\neq 0$ (for instance, one of $z_0^{\lambda-1},\dots,z_N^{\lambda-1}$ succeeds), and then we claim,  restricting on the $\mathbb{K}$-linear subspace obtained by multiplying $\mathcal{A}_{1}(\mathbb{K}^{N+1})$ with $g$:
\[
g
\cdot
\mathcal{A}_{1}(\mathbb{K}^{N+1})
\,
\subset\,
\mathcal{A}_{\lambda}(\mathbb{K}^{N+1}),
\] 
that the evaluation maps $v_{z}$ and 
$d_{z}(\xi)$
are $\mathbb{K}$-linearly independent.

In fact, for all $f\in \mathcal{A}(\mathbb{K}^{N+1})$, we have:
\[
\aligned
(g\cdot v)_{z}\,(f)
&
=
(gf)(z)
\\
&
=
g(z)\,
f(z)
\\
&
=
g(z)\,
v_z(f),
\endaligned
\]
and by Leibniz's rule:
\[
\aligned
d_{z}(g\cdot\,)(\xi)\,
(f)
&
=
d\,(gf)\big{\vert}_z(\xi)
\\
&
=
g(z)\,
d\,f\big{\vert}_z(\xi)
+
f(z)\,
d\,g\big{\vert}_z(\xi)
\\
&
=
g(z)\,
d_z(\xi)\,(f)
+
d\,g\big{\vert}_z(\xi)\,
v_z(f),
\endaligned
\]
in other words:
\begin{equation}
\label{matrix-relation}
\begin{pmatrix}
(g\cdot v)_{z}
\\
d_{z}(g\cdot\,)(\xi)
\end{pmatrix}
=
\underbrace{
\begin{pmatrix}
g(z) 
&
0
\\
d\,g\big{\vert}_z(\xi)
&
g(z)
\end{pmatrix}
}_{\text{invertible, since } g(z)\,\neq\,0}
\,
\begin{pmatrix}
v_z
\\
d_z(\xi)
\end{pmatrix}.
\end{equation}
Now, restricting~\thetag{\ref{matrix-relation}} on the $\mathbb K$-linear subspace:
\[
\mathcal{A}_{1}(\mathbb{K}^{N+1})\,
\subset\,
\mathcal{A}(\mathbb{K}^{N+1}),
\]
and recalling the result of Step 1 that the evaluation maps
$v_z, d_z(\xi)$
are $\mathbb K$-linearly independent, we immediately see that the 
evaluation maps
$(g\cdot v)_{z}, d_{z}(g\cdot\,)(\xi)$
are $\mathbb K$-linearly independent too. In other words,
restricting on the $\mathbb K$-linear subspace:
\[
g
\cdot
\mathcal{A}_{1}(\mathbb{K}^{N+1})\,
\subset\,
\mathcal{A}_{\lambda}(\mathbb{K}^{N+1}),
\]
the evaluation maps
$v_z, d_z(\xi)$
are $\mathbb K$-linearly independent.
\end{proof}

\begin{Lemma}
\label{linear-valuation-map-is-surjective}
For all positive integers $\lambda\geqslant 1$, for all polynomials $g\in \mathcal{A}(\mathbb{K}^{N+1})$, at every nonzero point $z\in \mathbb{K}^{N+1}\setminus\{0\}$ where $g$ does not vanish:
\[
g(z)\,
\neq\, 
0,
\] 
and for every tangent vector 
$\xi \in \mathrm{T}_z \mathbb{K}^{N+1}\cong \mathbb{K}^{N+1}$ which does not lie in the line of $z$: 
\[
\xi\,
\not\in\,
\mathbb{K}\cdot z,
\]
restricting on the subspace:
\[
\mathcal{A}_{\lambda}(\mathbb{K}^{N+1})\,
\subset\,
\mathcal{A}(\mathbb{K}^{N+1}),
\] 
the evaluation maps $(g\cdot v)_{z}$ and 
$d_{z}(g\cdot\,)(\xi)$
are $\mathbb{K}$-linearly independent.
In other words,
the map:
\[
\begin{pmatrix}
(g\cdot v)_{z}
\\
d_{z}(g\cdot\,)(\xi)
\end{pmatrix}\,
\colon\,
\mathcal{A}_{\lambda}(\mathbb{K}^{N+1})\,
\longrightarrow\,
\mathbb{K}^2
\]
is surjective.
\end{Lemma}

\begin{proof}
This is a direct consequence of formula~\thetag{\ref{matrix-relation}} and of the preceding lemma.
\end{proof}

\subsection{Codimensions of affine cones}
Usually, it is  more convenient to count dimension in an Euclidian space rather than in a
projective space. Therefore we carry out the following lemma
(cf. \cite[p.~12, exercise 2.10]{Hartshorne-1977}), 
which is geometrically obvious, as one point ($\dim_{\mathbb{K}}=0$) in the projective space 
$\mathbb{P}_{\mathbb{K}}^N$ corresponds to one $\mathbb{K}$-line ($\dim_{\mathbb{K}}=1$) in $\mathbb{K}^{N+1}$.   

\begin{Lemma}
\label{algebraic-codim V=codim V-hat}
In an algebraically closed field $\mathbb K$,
let $\pi\colon \mathbb{K}^{N+1}\rightarrow \mathbb{P}_{\mathbb K}^N$ be the canonical projection, and  
let:
\[Y
\subset 
\mathbb{P}_{\mathbb K}^N
\] 
be a nonempty algebraic set defined by a homogeneous ideal:
\[
I
\subset
\mathbb{K}[z_0,\dots,z_N].
\]
Denote by $C(Y)$ the affine cone over $Y$:
\[
C(Y)
:=
\pi^{-1}
(Y)
\cup
\{
{\bf 0}
\}\,
\subset
\mathbb{K}^{N+1}.
\]
Then $C(Y)$ is an algebraic set in $\mathbb{K}^{N+1}$ which is also defined by the ideal $I$ (considered as an ordinary ideal in $\mathbb{K}[z_0,\dots,z_N]$), and it has dimension one more than $Y$:
\[
\dim\,
C(Y)
=
\dim\,
Y
+
1.
\]
In other words, they have the same codimension:
\[
\cdim\,
C(Y)
=
\cdim\,
Y.
\eqno
\qed
\]
\end{Lemma}

The essence of the above geometric lemma is the following theorem in commutative algebra (cf. \cite[p.~73, Cor. 5.21]{Liu-2002}):

\begin{Theorem}
Let $B$ be a homogeneous algebra over a field $\mathbb{K}$, then:
\[
\dim\,
\mathrm{Spec}\,
B
=
\dim\,
\mathrm{Proj}\,
B
+
1.
\eqno
\qed
\]
\end{Theorem}

\subsection{Full rank of hypersurface equation matrices}
In an algebraically closed field $\mathbb K$, for all positive integers $N\geqslant 2$, for all integers $e=1\cdots N$, for all positive integers $\epsilon_1,\dots,\epsilon_e\geqslant 1$ and $d\geqslant 1$, consider the following $e$ hypersurfaces: 
\[
H_1, \dots,H_e\,
\subset\, 
\mathbb{P}_{\mathbb K}^N,
\]
each being defined as the zero set of a degree $(d+\epsilon_i)$ Fermat-type homogeneous polynomial:
\begin{equation}\label{F_i-before-technical-lemma}
F_i
\,
=
\,
\sum_{j=0}^N\,
A_i^j\,
z_j^{d}
\ \ \ \ \ \ \ \ \ \ \ \ \ \
{\scriptstyle{(i\,=\,1\,\cdots\,e)}},
\end{equation}
where all $A_i^j\in \mathcal{A}_{\epsilon_i}(\mathbb{K}^{N+1})$ are some degree $\epsilon_i$ homogeneous polynomials. 

Now, denote by $\sf H$ the $e\times (N+1)$ matrix whose $i$-th row copies the $(N+1)$ terms of $F_i$ in the exact order, i.e. the $(i,j)$-th entries of $\sf H$ are:
\[
{\sf H}_{i,j}
=
A_i^{j-1}\,
z_{j-1}^{d}
\ \ \ \ \ \ \ \ \ \ \ \ \ \
{\scriptstyle{(i\,=\,1\cdots\,e;\,j\,=\,1\,\cdots\,N+1)}},
\] 
so $\sf H$ writes as:
\begin{equation}
\label{H-matrix}
{\sf H}
:=
\begin{pmatrix}
A_1^0\,z_0^{d} & \cdots & A_1^N\,z_N^{d} \\[3pt]
\vdots                 & \ddots & \vdots  \\[3pt]
A_e^0\,z_0^{d} & \cdots & A_e^N\,z_N^{d} \\[3pt]
\end{pmatrix},
\end{equation}
which we call {\em the hypersurface equation matrix} of $F_1,\dots,F_e$.
{\em Passim}, remark that by~\thetag{\ref{F_i-before-technical-lemma}}, the sum of all columns of $\sf H$ vanishes at every point $[z]\in X:=H_1\cap\cdots\cap H_e$.

Also introduce:
\[
\mathbb P
\mathcal
(
\mathcal{M}
)
:
=
\mathbb P
\Big(
\underbrace{
\oplus_{\substack{1\leqslant i \leqslant e\\0\leqslant j \leqslant N}}\,
\mathcal{A}_{\epsilon_i}(\mathbb{K}^{N+1})
}_{=:\,\mathcal{M}}
\Big)
\]
the projectivized parameter space of $(A_i^j)_{\substack{1\leqslant i\leqslant e\\ 0\leqslant j \leqslant N}}\in\mathcal{M}$. 

First, let us recall a classical theorem
(cf. \cite[p.~57, Theorem 2]{Shafarevich-1994}) that somehow foreshadows Remmert's proper mapping theorem.

\begin{Theorem}
\label{closed-image-theorem}
The image of a projective variety under a regular map is closed.
\qed
\end{Theorem}

The following lemma was proved by Brotbek in another
version~\cite[p.~36, Proof of Claim 1]{Brotbek-2014-arxiv},
and the proof there is in affine coordinates
$
(
\frac{z_0}{z_j},
\dots,
\widehat{
\frac{z_j}{z_j}
},
\dots,
\frac{z_N}{z_j}
)
$:
\[
\mathbb{K}^N\,
\cong\,
\{
z_j
\neq
0
\}\
\subset\
\mathbb{P}_{\mathbb{K}}^N
\qquad
{\scriptstyle(j\,=\,0\,\cdots\,N)}.
\] 
Here, we may
present a proof by much the same arguments in ambient coordinates
$(z_0,\dots,z_N)$:
\[
\aligned
\mathbb{K}^{N+1}
\setminus
\{
0
\}
&
\longrightarrow
\mathbb{P}_{\mathbb{K}}^N
\\
(z_0,\dots,z_N)
&
\longmapsto
[z_0:\cdots:z_N].
\endaligned
\]

\begin{Lemma}
\label{full-rank-c*(N+1)}
In $\mathbb{P}(\mathcal{M})$,
there exists a proper algebraic subset:
\[
\Sigma
\subsetneqq
\mathbb{P}(\mathcal{M})
\] 
such that,
for every choice of parameter outside $\Sigma$: 
\[
\Big[
\big(
A_i^j
\big)
_{
\substack{
1\leqslant i\leqslant e
\\
0\leqslant j \leqslant N
}}
\Big]
\in\ \
\mathbb{P}(\mathcal{M})
\setminus
\Sigma,
\]
on the corresponding intersection:
\[
X
=
H_1
\cap
\cdots
\cap
H_e
\,
\subset\,
\mathbb{P}_{\mathbb K}^N,
\]
the matrix $\sf H$ has full rank $e$ everywhere:
\[
\rank_{\mathbb K}\,{\sf H}(z)
=
e
\qquad
{\scriptstyle{(\forall\,[z]\,\in\,X)}}.
\]
\end{Lemma}

Sharing the same spirit as the famous {\em Fubini principle} in combinatorics, the essence of the proof below is to count dimension in two ways, which is a standard method in algebraic geometry having various forms (e.g. the proof of Bertini's Theorem in~\cite[p.~179]{Hartshorne-1977}, main arguments in~\cite{Debarre-2005, Brotbek-2014-arxiv}, etc).

\smallskip
\noindent
{\em Proof.}
Now, introduce the universal family 
$\mathcal{X}
\hookrightarrow
\mathbb{P}(\mathcal{M})
\times
\mathbb{P}_{\mathbb K}^N$
of the intersections of such $e$  Fermat-type hypersurfaces:
\[
\mathcal{X}
:=
\Big\{
\big(
[A_i^j],
[z]
\big)\,
\in\,
\mathbb{P}(\mathcal{M})
\times
\mathbb{P}_{\mathbb K}^N\,
\colon\,
{\textstyle{\sum_{j=0}^N}}\,
A_i^j\,z_j^{d}
=
0,\,
\text{for}\,\,
 i=1\cdots e
\Big\},
\]
and then consider the subvariety $\mathcal{B}\subset \mathcal{X}$ that consists of all `bad points' defined by:
\begin{equation}
\label{definition-of-B}
\rank_{\mathbb K}\,
{\sf H}
\leqslant 
e-1
.
\end{equation}

Let $\pi_1$, $\pi_2$ below be the two canonical projections:
\[
\xymatrix{
&
\mathbb{P}(\mathcal{M})
\times
\mathbb{P}_{\mathbb K}^N
\ar[ld]_-{\pi_1} \ar[rd]^-{\pi_2}
\\
\mathbb{P}(\mathcal{M}) 
& 
&  
\mathbb{P}_{\mathbb K}^N.
}
\]
Since 
$\mathbb{P}(\mathcal{M})
\times
\mathbb{P}_{\mathbb K}^N
\supset
\mathcal{B}$ 
is a projective variety and $\pi_1$ is a regular map, now applying Theorem~\ref{closed-image-theorem}, 
we see that: 
\[
\pi_1(\mathcal{B})
\subset
\mathbb{P}(\mathcal{M})
\] 
is an algebraic subvariety. Hence it is necessary and sufficient to show that: 
\begin{equation}\label{pi_1(B)!=whole-moduli-space}
\pi_1(\mathcal{B})
\neq
\mathbb{P}(\mathcal{M}).
\end{equation}
Our strategy is as follows. 

\smallskip
{\em Step $1$.} 
To decompose $\mathbb{P}_{\mathbb K}^{N}$ into a union of quasi-subvarieties:
\begin{equation}\label{notation-circle-CP^N}
\mathbb{P}_{\mathbb K}^{N}
=
\cup_{k=0}^N\,
\kPN,
\end{equation}
where
$\kPN$ consists of points 
$[z]=[z_0\colon z_1\colon\cdots\colon z_N] \in \mathbb{P}_{\mathbb K}^N$ with exactly $k$ vanishing homogeneous coordinates, the other ones being nonzero. 

\smallskip
{\em Step $2$.} 
For every integer $k=0\cdots N$, for every point 
$[z]\in \kPN$, to establish the fibre dimension identity:
\begin{equation}\label{fibre-dimension}
\dim\,
\pi_2^{-1}([z])
\cap
\mathcal{B}
=
\dim\,
\mathbb{P}(\mathcal{M})
-
\big(
\max\,
\{
N-k-e+1,0
\}
+
e
\big).
\end{equation}

\begin{proof}[Proof of Step $2$.]
Without loss of generality, we may assume that the last $k$ homogeneous coordinates of $[z]$ vanish:
\begin{equation}\label{k-vanishing-coordinates}
z_{N-k+1}
=
\cdots
=
z_N
=
0,
\end{equation}
and then by the definition of 
$\kPN$, none of the first $(N-k+1)$ coordinates $z_0,\dots,z_{N-k}$
vanish.

Noting that:
\[
\pi_2^{-1}([z])
\cap
\mathcal{B}
\,\,=
\!\!\!\!
\underbrace{
\pi_1
\Big(
\pi_2^{-1}([z])
\cap
\mathcal{B}
\Big)
}_{\text{by Theorem~\ref{closed-image-theorem} is an algebraic set}}
\!\!\!\!
\times
\underbrace{
\big\{
[z]
\big\}
}_{\text{one point set}},
\] 
and considering the canonical projection:
\[
\widehat{\pi}
\colon
\ \ \
\mathcal{M}
\setminus 
\{0\}\,
\longrightarrow\,
\mathbb{P}(\mathcal{M}),
\]
we receive:
\begin{equation}\label{compute-dimension}
\aligned
\dim\,
\pi_2^{-1}([z])
\cap
\mathcal{B}
&
=
\dim\,
\pi_1
\Big(
\pi_2^{-1}([z])
\cap
\mathcal{B}
\Big)
\\
\explain{use Lemma ~\ref{algebraic-codim V=codim V-hat}}
\ \ \ \ \ \ \ \ \ \ \ \ \ \ \ \
&
=
\dim\,
\widehat\pi^{-1}
\Big(
\pi_1
\big(
\pi_2^{-1}([z])
\cap
\mathcal{B}
\big)
\Big)
\cup
\{
0
\}
-
1.
\endaligned
\end{equation}

Now, observe that whatever choice of parameters:
\[
\big(
A_i^j
\big)_{\substack{1\leqslant i\leqslant e\\ 0\leqslant j \leqslant N}}
\in
\mathcal{M},
\]
the vanishing of the last $k$ coordinates of $[z]$ in~\thetag{\ref{k-vanishing-coordinates}} makes
the last $k$ columns of ${\sf H}(z)$ in~\thetag{\ref{H-matrix}} vanish. It is therefore natural to introduce the submatrix ${}^{N+1-k}{\sf H}$ of $\sf H$ that
consists of the remaining nonvanishing columns, i.e. the first $(N+1-k)$ ones. Since the sum of all columns of ${\sf H}(z)$ vanishes by
\thetag{\ref{F_i-before-technical-lemma}}, the sum of all columns of ${}^{N+1-k}{\sf H}(z)$ also vanishes.  

Observe that the set:
\[
\aligned
\mathcal{M}\,
\supset\,
&
\widehat\pi^{-1}
\Big(
\pi_1
\big(
\pi_2^{-1}([z])
\cap
\mathcal{B}
\big)
\Big)
\cup
\{0\}\,
\\
&
=
\Big\{
\big(
A_i^j
\big)_{\substack{1\leqslant i\leqslant e\\ 0\leqslant j \leqslant N}}
\in
\mathcal{M}\,
\colon\,
\text{sum of all the columns of}\,
{}^{N+1-k}{\sf H}(z)\,
\text{vanishes,}
\\
& 
\ \ \ \ \ \ \ \ \ \ \ \ \ 
\ \ \ \ \ \ \ \ \ \ \ \ \ 
\ \ \ \ \ \ \ \ \ \ \ \ \ \ \ \ \
\text{and}\,
\underbrace{
\rank_{\mathbb K}\,
{}^{N+1-k}{\sf H}(z)
\leqslant
e-1
}_{
\text{from}\,\thetag{\ref{definition-of-B}}
}
\Big\}
\endaligned
\]
is nothing but the inverse image of:
\[
{}^{0}\Sigma_{e-1}^{e,N+1-k}
\subset
\mathsf{Mat}_{e\times (N+1-k)}(\mathbb{K})
\ \ \ \ \ \ \ \ \ \ \ \ \ \ \ \
\explain{use notation of Lemma~\ref{classical-codimension-formula-sum-0}}
\]
under the $\mathbb K$-linear map:
\[
\aligned
{}^{N-k+1}{\sf H}_z\,
\colon\,
\mathcal{M}
&
\longrightarrow
\mathsf{Mat}_{e\times (N+1-k)}(\mathbb{K})
\\
(A_i^j)_{i,j}
&
\longmapsto
 \ \ \
{}^{N-k+1}{\sf H}(z),
\endaligned
\]
which is surjective by Lemma~\ref{linear-valuation-map-is-surjective}.

Therefore we have the codimension identity:
\begin{equation}\label{compute-codimension}
\aligned
\cdim\,
\widehat{\pi}^{-1}
\Big(
\pi_1
\big(
\pi_2^{-1}([z])
\cap
\mathcal{B}
\big)
\Big)
\cup
\{0\}
&
=
\cdim\,
{}^{0}\Sigma_{e-1}^{e,N+1-k}
\ \ \ \ \ \ \
\explain{${}^{N-k+1}{\sf H}_z$ is linear and surjective}
\\
\explain{use Lemma~\thetag{\ref{classical-codimension-formula-sum-0}}}
\ \ \ \ \ \ \ \ \ \ \ \ \ \
&
=
\max\,
\{
N-k-e+1,0
\}
+
e,
\endaligned
\end{equation}
and thereby we receive:
\[
\aligned
\dim\,
\pi_2^{-1}([z])
\cap
\mathcal{B}
&
=
\dim\,
\widehat{\pi}^{-1}
\Big(
\pi_1
\big(
\pi_2^{-1}([z])
\cap
\mathcal{B}
\big)
\Big)
\cup
\{
0
\}
-
1
\ \ \ \ \ \ \ \ \ \ \ \ \ \ \ \
\explain{use \thetag{\ref{compute-dimension}}}
\\
\explain{by definition of codimension}
\ \ \ \ \ \ \ \ \ \ \ 
&
=
\dim\,
\mathcal{M}
-
\cdim\,
\widehat{\pi}^{-1}
\Big(
\pi_1
\big(
\pi_2^{-1}([z])
\cap
\mathcal{B}
\big)
\Big)
\cup
\{0\}
-1
\\
\explain{exercise}
\ \ \ \ \ \ \ \ \ \ \ 
&
=
\dim\,
\mathbb{P}(\mathcal{M})
-
\cdim\,
\widehat{\pi}^{-1}
\Big(
\pi_1
\big(
\pi_2^{-1}([z])
\cap
\mathcal{B}
\big)
\Big)
\cup
\{0\}
\\
\explain{use \thetag{\ref{compute-codimension}}}
\ \ \ \ \ \ \ \ \ \ \ 
&
=
\dim\,
\mathbb{P}(\mathcal{M})
-
\big(
\max\,
\{
N-k-e+1,0
\}
+
e
\big),
\endaligned
\]
which is exactly our claimed fibre dimension identity~\thetag{\ref{fibre-dimension}}.
\end{proof}

{\em Step $3$.}
Applying Lemma~\ref{algebraic-fibre-Dimension-Estimate} to the regular map:
\[
\pi_2
\colon\ \ \
\pi_2^{-1}
\big(
{\kPN}
\big)
\cap
\mathcal{B}\,
\longrightarrow\,
\kPN,
\]
remembering:
\[
\dim\,
\kPN
=
N
-
k
\ \ \ \ \ \ \ \ \ \ \
{\scriptstyle{(k\,=\,0\,\cdots\,N)}},
\]
together with the identity~\thetag{\ref{fibre-dimension}},  
we receive the dimension estimate:
\begin{equation}
\label{estimate-dim-B}
\aligned
\dim\,
\pi_2^{-1}
\big(
\kPN
\big)
\cap
\mathcal{B}
&
\leqslant
\dim\,
\kPN
+
\dim\,
\mathbb{P}(\mathcal{M})
-
\max
\{
N-k-e+1,0
\}
-
e
\\
&
\leqslant
(N-k)
+
\dim\,
\mathbb{P}(\mathcal{M})
-
(N-k-e+1)
-
e
\\
&
=
\dim\,
\mathbb{P}(\mathcal{M})
-
1.
\endaligned
\end{equation}
Note that $\mathcal{B}$ can be written as the union of $(N+1)$ quasi-subvarieties:
\[
\aligned
\mathcal{B}
&
=
\pi_2^{-1}\,
\big(
\mathbb{P}_{\mathbb K}^{N}
\big)
\cap
\mathcal{B}
\\
&
=
\pi_2^{-1}\,
\big(
\cup_{k=0}^N\,
\kPN
\big)
\cap
\mathcal{B}
\\
&
=
\Big(
\cup_{k=0}^N\,
\pi_2^{-1}\,
\big(
\kPN
\big)
\Big)
\cap
\mathcal{B}
\\
&
=
\cup_{k=0}^N\,
\Big(
\pi_2^{-1}\,
\big(
\kPN
\big)
\cap
\mathcal{B}
\Big),
\endaligned
\]
each one being, thanks to~\thetag{\ref{estimate-dim-B}}, of dimension less than or equal to: 
\[
\dim\,\mathbb{P}(\mathcal{M})-1,
\]
and therefore we have the dimension estimate:
\[
\dim\,
\mathcal{B}
\leqslant
\dim\,\mathbb{P}(\mathcal{M})-1.
\]
Finally, \thetag{\ref{pi_1(B)!=whole-moduli-space}}
follows from the dimensional comparison: 
\[
\dim\,
\pi_1(\mathcal B)
\leqslant
\dim\,
\mathcal{B}
\leqslant
\dim\,\mathbb{P}(\mathcal{M})-1.
\eqno
\qed
\]

In the more general context of our moving coefficients method, we now want to have an everywhere full-rank property analogous to 
Lemma~\ref{full-rank-c*(N+1)} just obtained.

Observing that in~\thetag{\ref{F_i-moving-coefficient-method-full-strenghth}}, the number of terms in each polymonial $F_i$ is:
\[
(N+1)
+
\sum_{\ell=c+r+1}^{N}\,
\binom{N+1}{\ell+1}\,(\ell+1),
\]
and recalling that the $\mathbb{K}$-linear subspace $\mathcal{A}_{\epsilon_i}(\mathbb{K}^{N+1})\subset \mathbb{K}[z_0,\dots,z_N]$ spanned by all degree $\epsilon_i$ homogeneous polynomials is of dimension:
\[
\dim_{\mathbb{K}}\,
\mathcal{A}_{\epsilon_i}(\mathbb{K}^{N+1})
=
\binom{N+\epsilon_i}{N},
\] 
we may denote by
$\mathbb{P}_{\mathbb{K}}^{\text{\ding{169}}}$
the projectivized parameter space of such $c+r$ hypersurfaces, with the integer:
\begin{equation}
\label{euro = ?}
\text{\ding{169}}
:=
\bigg[
(N+1)
+
\sum_{\ell=c+r+1}^{N}\,
\binom{N+1}{\ell+1}\,(\ell+1)
\bigg]
\,
\sum_{i=1}^{c+r}\,
\binom{N+\epsilon_i}{N}.
\end{equation}

Now, by mimicking the construction of the matrix $\sf H$ 
in~\thetag{\ref{H-matrix}}, employing the notation in Subsection~\ref{The-global-moving-coefficients-method}, for every integer $\nu=0\cdots N$, let us denote by 
${\sf H}^{\nu}$ the $(c+r)\times (N+1)$ matrix whose $i$-th row copies the $(N+1)$ terms of $F_i$ in~\thetag{\ref{first-major-manipulation}}.
Also, for every integer $\tau=0\cdots N-1$, for every index 
$\rho=\tau+1\cdots N$,
 let us denote by ${\sf H}^{\tau,\rho}$ the $(c+r)\times (N+1)$ matrix whose $i$-th row copies the $(N+1)$ terms of $F_i$ in~\thetag{\ref{second-major-manipulation}}.

\begin{Lemma}\label{all-rank-H=c}
In $\mathbb{P}_{\mathbb{K}}^{\text{\ding{169}}}$,
there exists a proper algebraic subset:
\[
\Sigma
\subsetneqq
\mathbb{P}_{\mathbb{K}}^{\text{\ding{169}}}
\] 
such that,
for every choice of parameter outside $\Sigma$: 
\[
\big[
A_{\smallbullet}^{\smallbullet},
M_{\smallbullet}^{\smallbullet}
\big]\ \
\in\ \
\mathbb{P}_{\mathbb{K}}^{\text{\ding{169}}}
\setminus
\Sigma,
\]
on the corresponding intersection:
\[
X
=
H_1
\cap
\cdots
\cap
H_{c+r}
\,
\subset\,
\mathbb{P}_{\mathbb K}^N,
\]
all the matrices ${\sf H}^{\nu},{\sf H}^{\tau,\rho}$ have full rank $c$:
\[
\rank_{\mathbb K}\,{\sf H}^{\nu}(z)
=
c+r,
\quad
\rank_{\mathbb K}\,{\sf H}^{\tau,\rho}(z)
=
c+r
\ \ \ \ \ \ \ \ \ \ 
{\scriptstyle{(\forall\,[z]\in X)}}.
\]
\end{Lemma}

We can copy the proof of Lemma~\ref{full-rank-c*(N+1)} without much modification and thus everything works smoothly. Alternatively, we may present a short proof by applying Lemma~\ref{full-rank-c*(N+1)}.

\begin{proof}
{\em Observation 1.}
We need only prove this lemma separately
for each matrix ${\sf H}^{\nu}$ (resp. ${\sf H}^{\tau,\rho}$),
i.e. to show that there exists a proper algebraic subset:
\[
\Sigma^{\nu}\,(\text{resp. }\Sigma^{\tau,\rho})
\subsetneqq
\mathbb{P}_{\mathbb{K}}^{\text{\ding{169}}}
\] 
outside of which every choice of parameter succeeds. Then the union of all these proper algebraic subsets works:
\[
\Sigma
:=
\cup_{\nu=0}^{N}\,
\Sigma^{\nu}\ \
\cup
\cup_{\substack{\tau=0\cdots N-1\\ \rho=\tau+1\cdots N}}
\Sigma^{\tau,\rho}\ \ \ \ \ 
\subsetneqq\ \ \ \ \ 
\mathbb{P}_{\mathbb{K}}^{\text{\ding{169}}}.
\] 

{\em Observation 2.} For each matrix ${\sf H}^{\nu}$ (resp. ${\sf H}^{\tau,\rho}$), inspired by the beginning arguments in the proof of Lemma~\ref{full-rank-c*(N+1)}, especially \thetag{\ref{pi_1(B)!=whole-moduli-space}}, we only need to find one parameter:
\[
\big[
A_{\smallbullet}^{\smallbullet},
M_{\smallbullet}^{\smallbullet}
\big]\ \
\in\ \
\mathbb{P}_{\mathbb{K}}^{\text{\ding{169}}}
\setminus
\Sigma
\]
with the desired property.

{\em Observation 3.} 
Now, setting all the moving coefficients zero:
\[
M_{\smallbullet}^{\smallbullet}
:=
0,
\]
thanks to~\thetag{\ref{C_i^j*z_j^(d-delta_N)}}, \thetag{\ref{equation-T}},
the equations~\thetag{\ref{first-major-manipulation}}
become exactly the equations~\thetag{\ref{F_i-before-technical-lemma}},
and therefore all the matrices ${\sf H}^{\nu}$
become the same matrix ${\sf H}$ of Lemma~\ref{full-rank-c*(N+1)} (with $e=c+r$).
Similarly, so do all the matrices ${\sf H}^{\tau,\rho}$.

{\em Observation 4.} 
Now, a direct application of Lemma~\ref{full-rank-c*(N+1)} clearly yields more than one parameter, an infinity!
\end{proof}

Once again, by mimicking the construction of the matrix $\sf H$ in Lemma~\ref{full-rank-c*(N+1)}, employing the notation in subsection~\ref{The-moving-coefficients-method-for-intersections-with-coordinate-hyperplanes}, let us denote by ${}_{v_1,\dots,v_{\eta}}{\sf H}^{\nu}$ (resp. ${}_{v_1,\dots,v_{\eta}}{\sf H}^{\tau,\rho}$) the $c\times (N+1)$ matrix whose $i$-th row copies the $(N+1)$ terms of $F_i$ in \thetag{\ref{first-way-F_i}} (resp. \thetag{\ref{second-way-F_i}}). 

\begin{Lemma}
\label{all-rank-H=c-second}
In $\mathbb{P}_{\mathbb{K}}^{\text{\ding{169}}}$,
there exists a proper algebraic subset:
\[
{}_{v_1,\dots,v_{\eta}}\Sigma
\subsetneqq
\mathbb{P}_{\mathbb{K}}^{\text{\ding{169}}}
\] 
such that,
for every choice of parameter outside ${}_{v_1,\dots,v_{\eta}}\Sigma$: 
\[
\big[
A_{\smallbullet}^{\smallbullet},
M_{\smallbullet}^{\smallbullet}
\big]\ \
\in\ \
\mathbb{P}_{\mathbb{K}}^{\text{\ding{169}}}
\setminus
{}_{v_1,\dots,v_{\eta}}\Sigma
\]
on the corresponding intersection:
\[
X
=
H_1
\cap
\cdots
\cap
H_{c+r}
\,
\subset\,
\mathbb{P}_{\mathbb K}^N,
\]
all the matrices ${}_{v_1,\dots,v_{\eta}}{\sf H}^{\nu}$ and ${}_{v_1,\dots,v_{\eta}}{\sf H}^{\tau,\rho}$ have full rank $c+r$:
\[
\rank_{\mathbb K}\,{}_{v_1,\dots,v_{\eta}}{\sf H}^{\nu}(z)
=
c+r,
\quad
\rank_{\mathbb K}\,{}_{v_1,\dots,v_{\eta}}{\sf H}^{\tau,\rho}(z)
=
c+r
\ \ \ \ \ \ \ \ \ \ 
{\scriptstyle{(\forall\,[z]\,\in\,X)}}.
\eqno
\qed
\]
\end{Lemma}

The proof goes exactly the same way as in the preceding lemma.

\section{\bf Controlling the base locus}
\label{section: Controlling the base locus}

\subsection{Characterization of the base locus}
Now, we are in a position to characterize the base locus of all the obtained global twisted symmetric differential $n$-forms in~\thetag{\ref{first-class-symmetric-differential-n-forms}}, \thetag{\ref{second-class-symmetric-differential-n-forms}}:
\begin{equation}
\label{BS-definition}
\mathsf{BS}
:=
\text{Base Locus of }
\{
\phi_{j_1,\dots,j_n}^{\nu},
\psi_{j_1,\dots,j_n}^{\tau,\rho}
\}
^{
\nu,
\tau,
\rho
}_{
1
\leqslant 
j_1 
<
\cdots 
<
j_n 
\leqslant 
c
}
\ \
\subset
\ \
\mathbb{P}
\big(
\mathrm{T}_V\big\vert_X
\big),
\end{equation}
where $\mathbb{P}
\big(
\mathrm{T}_V\big\vert_X
\big)\subset \mathbb{P}(\mathrm{T}_{\mathbb{P}_{\mathbb{K}}^N})$ is given by:
\[
\mathbb{P}
\big(
\mathrm{T}_V\big\vert_X
\big)
:=
\Big\{
([z],[\xi])\
\colon\
F_i(z)
=
0,
dF_j\big{\vert}_z(\xi)
=
0,
\forall\,
i=1\cdots c+r,
\forall\,
j=1\cdots c
\Big\}
\]

To begin with, for every $\nu=0\cdots N$, let us study the specific base locus:
\[
\mathsf{BS}^{\nu}
:=
\text{Base Locus of }
\{
\phi_{j_1,\dots,j_n}^{\nu}
\}
_{1\leqslant j_1 < \cdots < j_n \leqslant c}
\ \
\subset
\ \
\mathbb{P}
\big(
\mathrm{T}_V\big\vert_X
\big)
\]
associated with only the twisted symmetric differential forms obtained in~\thetag{\ref{first-class-symmetric-differential-n-forms}}.

For each sequence of ascending indices:
\[
1\leqslant j_1 < \cdots < j_n \leqslant c,
\]
by mimicking the construction of the matrices ${\sf K}, \widehat{\mathsf{K}}_{j_1,\dots,j_n;\,j}$ at the end of Subsection \ref{subsection:Global-twisted-holomorphic-symmetric-differential-forms}, in accordance with the first kind of manipulation~\thetag{\ref{first-major-manipulation}}, we construct
the $(c+r+c)\times (N+1)$ matrix ${\sf K}^{\nu}$ in the obvious way, i.e. 
by copying terms, differentials, and then we define the analogous $\widehat{{\sf K}}_{j_1,\dots,j_n;\,j}^{\nu}$.

First, let us look at points 
$\big([z],[\xi]\big)\in \mathsf{BS}^{\nu}$ having all coordinates nonvanishing: 
\begin{equation}
\label{z_0...z_N=!0}
z_0
\cdots 
z_N
\neq 
0.
\end{equation}
For each symmetric horizontal differential $n$-form $\widehat\phi_{j_1,\dots,j_n}^{\nu}$ which corresponds to
$\phi_{j_1,\dots,j_n}^{\nu}$ in the sense of Propositions~\ref{general-holomorphic-symmetric-forms}, \ref{general-holomorphic-symmetric-horizontal-forms},
 for every
$j=0\cdots N$, we receive:
\[
\aligned
0
&
=
\widehat\phi_{j_1,\dots,j_n;\,j}^{\nu}\,(z,\xi)
\ \ \ \ \ \ \ \ \ \ \ \ \ \ \ \ \
\explain{since $([z],[\xi])\in \mathsf{BS}^{\nu}$}
\\
\explain{use \thetag{\ref{omega-det(K)-first}}}
\ \ \ \ \ \ \ \ \
&
=
\underbrace{
\frac
{(-1)^{j}}
{z_0^{\star}\cdots z_N^{\star}}
}_{
\neq\,0
}
\,
\det\big(\widehat{\sf K}_{j_1,\dots,j_n;\,j}^{\nu}\big)\,(z,\xi),
\endaligned
\] 
where all integers $\star$ are of no importance here.
Indeed, we can drop the nonzero factor $\frac{(-1)^{j}}{z_0^{\star}\cdots z_N^{\star}}$ and obtain:
\[
\det\big(\underbrace{\widehat{\sf K}_{j_1,\dots,j_n;\,j}^{\nu}}_{N\times N\text{ matrix}}\big)\,
(z,\xi)
=
0.
\]
In other words:
\[
\rank_{\mathbb{K}}\,
\widehat{\sf K}_{j_1,\dots,j_n;\,j}^{\nu}\,
(z,\xi)
\leqslant
N
-
1.
\]
Now, letting the index $j$ run from $0$ to $N$, we receive:
\begin{equation}
\label{rank-K^nu<=N-1}
\rank_{\mathbb{K}}\,
\underbrace{
{\sf K}_{j_1,\dots,j_n}^{\nu}
\,
(z,\xi)
}_{
N\times (N+1)
\text{ matrix}
}
\leqslant
N
-
1,
\end{equation}
where ${\sf K}_{j_1,\dots,j_n}^{\nu}$ is defined analogously to the matrix ${\sf C}_{j_1,\dots,j_n}^{\nu}$ before Proposition~\ref{general-holomorphic-symmetric-horizontal-n-forms} in the obvious way.

Note that the first $c+r$ rows of ${\sf K}_{j_1,\dots,j_n}^{\nu}$ constitute the matrix ${\sf H}^{\nu}$ in Lemma~\ref{all-rank-H=c}, which asserts that for a generic choice of parameter:
\[
\rank_{\mathbb{K}}\,
{\sf H}^{\nu}
(z)
=
c+r.
\]   
Now, in \thetag{\ref{rank-K^nu<=N-1}}, letting $1\leqslant j_1 < \cdots < j_n \leqslant c$ vary,  
and applying Lemma~\ref{rank(2c*(N-k+1)<N-k}, we immediately receive:
\[
\rank_{\mathbb{K}}\,
{\sf K}^{\nu}\,
(z,\xi)
\leqslant
N-1.
\]
Conversely, it is direct to see that any point $\big([z],[\xi]\big)\in {}_X\mathbb{P}({\mathrm{T}_{\mathrm{V}}})$ satisfying this rank inequality lies in the base locus $\mathsf{BS}^{\nu}$.

Note that a point $\big([z],[\xi]\big)\in \mathbb{P}(\mathrm{T}_{\mathbb{P}^N})$ lies in
$\mathbb{P}
\big(
\mathrm{T}_V\big\vert_X
\big)$ if and only if the sum of all columns of
${\sf K}^{\nu}\,
(z,\xi)$ vanishes.
Summarizing the above analysis, restricting to the coordinates nonvanishing part of 
${\mathbb{P}}(\mathrm{T}_{\mathbb{P}_{\mathbb{K}}^N})$:
\[
\oP
(
\mathrm{T}_{\mathbb{P}_{\mathbb{K}}^N}
)
\,
:=
\,
\mathbb{P}
(
\mathrm{T}_{\mathbb{P}_{\mathbb{K}}^N}
)
\cap
\{
z_0
\cdots 
z_N
\neq
0
\},
\]
we conclude the following generic characterization of:
\[
\mathsf{BS}^{\nu}\,
\cap\,
\oP
(
\mathrm{T}_{\mathbb{P}_{\mathbb{K}}^N}
),
\] 
where the exceptional locus $\Sigma$ just below is defined in Lemma~\ref{all-rank-H=c}.

\begin{Proposition}
\label{1}
For every choice of parameter 
outside $\Sigma$: 
\[
\big[
A_{\smallbullet}^{\smallbullet},
M_{\smallbullet}^{\smallbullet}
\big]\ \
\in\ \
\mathbb{P}_{\mathbb{K}}^{\text{\ding{169}}}
\setminus
\Sigma
\]
a point: 
\[
\big(
[z],
[\xi]
\big)\,
\in\,
\oP(\mathrm{T}_{\mathbb{P}_{\mathbb{K}}^N})
\]
lies in the base locus:
\[
\big(
[z],[\xi]
\big)\,
\in\, 
\mathsf{BS}^{\nu}
\]
if and only if:
\[
\rank_{\mathbb{K}}\,
{\sf K}^{\nu}\,
(z,\xi)
\leqslant
N-1,
\text{ and the sum of all columns vanishes}.
\eqno
\qed
\]
\end{Proposition}

Now, for every integer $\tau=0\cdots N-1$ and for every index $\rho=\tau+1\cdots N$, the base locus:
\[
\mathsf{BS}^{\tau,\rho}
:=
\text{Base Locus of }
\{
\psi_{j_1,\dots,j_n}^{\tau,\rho}
\}
_{1\leqslant j_1 < \cdots < j_n \leqslant c}\ \
\subset
\ \
\mathbb{P}
\big(
\mathrm{T}_V\big\vert_X
\big)
\]
associated with the twisted symmetric differential forms obtained in~\thetag{\ref{second-class-symmetric-differential-n-forms}} enjoys the following generic characterization on the coordinates nonvanishing set $\{z_0\cdots z_N \neq 0\}$.
Of course, the matrix ${\sf K}^{\tau,\rho}$ is defined analogously to the matrix ${\sf K}^{\nu}$ in the obvious way. A repetition of the preceding 
arguments yields:

\begin{Proposition}
\label{2}
For every choice of parameter outside $\Sigma$: 
\[
\big[
A_{\smallbullet}^{\smallbullet},
M_{\smallbullet}^{\smallbullet}
\big]\ \
\in\ \
\mathbb{P}_{\mathbb{K}}^{\text{\ding{169}}}
\setminus
\Sigma
\] 
a point: 
\[
\big(
[z],
[\xi]
\big)\,
\in\,
\oP(\mathrm{T}_{\mathbb{P}_{\mathbb{K}}^N})
\]
lies in the base locus:
\[
\big(
[z],[\xi]
\big)\,
\in\, 
\mathsf{BS}^{\tau,\rho}
\]
if and only if:
\[
\rank_{\mathbb{K}}\,
{\sf K}^{\tau,\rho}\,
(z,\xi)
\leqslant
N-1,
\text{ and the sum of all columns vanishes}.
\eqno
\qed
\]
\end{Proposition}

It is now time to clarify the (uniform) structures of the matrices ${\sf K}^{\nu}$, ${\sf K}^{\tau,\rho}$.
 
\smallskip
Thanks to the above two Propositions~\ref{1}, \ref{2}, we may now
receive a generic characterization of:
\[
\mathsf{BS}\,
\cap\,
\oP
(
\mathrm{T}_{\mathbb{P}^N}
).
\] 

Firstly, we construct the $(c+r+c)\times (2N+2)$ matrix $\mathsf{M}$ such that, for $i=1\cdots c+r, j=1\cdots c$, its $i$-row copies
the $(2N+2)$ terms of $F_i$ in~\thetag{\ref{first-rewrite-F_i}} in the exact order, and its $(c+r+j)$-th row is the differential of the $j$-th row. 
In order to distinguish the first $(N+1)$ `dominant' columns from the last $(N+1)$ columns of moving coefficient terms, we write $\sf M$ as:
\[
{\sf M}
=
\Big(
{\sf A}_0
\mid
\cdots
\mid
{\sf A}_N
\mid
{\sf B}_0
\mid
\cdots
\mid
{\sf B}_N
\Big).
\]

For every index $\nu=0\cdots N$, comparing~\thetag{\ref{first-major-manipulation}}, \thetag{\ref{equation-T}} with~\thetag{\ref{first-rewrite-F_i}}, the matrix ${\sf K}^{\nu}$ is nothing but:
\begin{equation}
\label{K^nu=(...)}
{\sf K}^{\nu}
=
\Big(
\mathsf{A}_0
\mid
\cdots
\mid
\widehat{\mathsf{A}_{\nu}}
\mid
\cdots
\mid
\mathsf{A}_N
\mid
\mathsf{A}_{\nu}+\sum_{j=0}^{N}\mathsf{B}_{j}
\Big).
\end{equation}
Similarly, for every integer $\tau=0\cdots N-1$ and for every index $\rho=\tau+1\cdots N$, comparing~\thetag{\ref{second-major-manipulation}}, \thetag{\ref{equation-E-R}} with~\thetag{\ref{first-rewrite-F_i}}, the matrix 
${\sf K}^{\tau,\rho}$ is nothing but: 
\begin{equation}
\label{K^(tau,rho)=(...)}
{\sf K}^{\tau,\rho}
=
\Big(
\mathsf{A}_0+\mathsf{B}_{0}
\mid
\cdots
\mid
\mathsf{A}_{\tau}+\mathsf{B}_{\tau}
\mid
\mathsf{A}_{\tau+1}
\mid
\cdots
\mid
\widehat{\mathsf{A}_{\rho}}
\mid
\cdots
\mid
\mathsf{A}_{N}
\mid
\mathsf{A}_{\rho}+\sum_{j=\tau+1}^N\mathsf{B}_{j}
\Big).
\end{equation}

Secondly, we
introduce the algebraic subvariety:
\begin{equation}
\label{M_(2c)^N}
\mathscr{M}_{2c+r}^N\ \
\subset\ \
{\sf Mat}_{(2c+r)\times 2(N+1)}(\mathbb K)
\end{equation}
consisting of all $(c+r+c)\times 2(N+1)$ matrices
$(\alpha_0\mid\alpha_1\mid\dots\mid\alpha_{N}\mid\beta_0\mid\beta_1\mid
\dots\mid\beta_{N})$ such that:
\begin{itemize}

\smallskip\item[{\bf (i)}]
the sum of these $(2N+2)$ colums is zero:
\begin{equation}
\label{sum-of-(2N+2)-columns=0}
\alpha_0
+
\alpha_1
+
\cdots
+
\alpha_{N}
+
\beta_0
+
\beta_1
+\cdots
+\beta_{N}
=
\mathbf{0};
\end{equation}

\smallskip\item[{\bf (ii)}]
for every index $\nu=0\cdots N$, replacing $\alpha_{\nu}$ with $\alpha_{\nu}+(\beta_0+\beta_1+\cdots+\beta_{N})$ in the collection of column vectors 
$\{\alpha_0,\alpha_1,\dots,\alpha_{N}\}$, there holds the rank inequality: 
\begin{equation}
\label{(ii) of M^n_2c}
\rank_{\mathbb{K}}\,
\big\{
\alpha_0,\dots,\widehat{\alpha_{\nu}},\dots,\alpha_{N},\alpha_{\nu}+(\beta_0+\beta_1+\cdots+\beta_{N})
\big\}
\leqslant 
N-1;
\end{equation}

\smallskip\item[{\bf (iii)}] 
for every integer $\tau=0\cdots N-1$, for every index $\rho=\tau+1 \cdots N$, replacing $\alpha_{\rho}$ with $\alpha_{\rho}+(\beta_{\tau+1}+\cdots+\beta_{N})$ 
in the collection of column vectors
$\{\alpha_0+\beta_0,\dots,\alpha_\tau+\beta_\tau,
\alpha_{\tau+1},\dots,\alpha_\rho,\dots,\alpha_{N}\}$, there holds the rank inequality:
\begin{equation}
\label{(iii) of M^n_2c}
\aligned
\ \ \ \ \ \ \ \ \ \
\rank_{\mathbb{K}}\,
\big\{
\alpha_0+\beta_0,\alpha_1+\beta_1,\dots,\alpha_\tau+\beta_\tau,
\alpha_{\tau+1},\dots,\widehat{\alpha_\rho},\dots,\alpha_{N},
\alpha_\rho+(\beta_{\tau+1}+\cdots+\beta_{N})
\big\}
\leqslant 
N-1.
\endaligned
\end{equation}
\end{itemize}
\smallskip

\begin{Proposition}
\label{characterization-BS-for-nonvanishing-coordinates}
For every choice of parameter outside $\Sigma$: 
\[
\big[
A_{\smallbullet}^{\smallbullet},
M_{\smallbullet}^{\smallbullet}
\big]\ \
\in\ \
\mathbb{P}_{\mathbb{K}}^{\text{\ding{169}}}
\setminus
\Sigma
\] 
a point: 
\[
\big(
[z],
[\xi]
\big)\,
\in\,
\oP(\mathrm{T}_{\mathbb{P}^N})
\]
lies in the base locus:
\[
\big(
[z],[\xi]
\big)\,
\in\, 
\mathsf{BS}
\]
if and only if:
\[
\mathsf{M}\,
(z,\xi)
\in
\mathscr{M}_{2c+r}^N. 
\eqno
\qed
\]
\end{Proposition}

Furthermore, for all integers $1\leqslant \eta \leqslant n-1$, for every sequence of ascending indices:
\[
0\leqslant v_1<\dots<v_{\eta}\leqslant N,
\] 
we also have to analyze the base locus of the twisted symmetric differential forms~\thetag{\ref{first-class-symmetric-differential-n-eta-forms}}, \thetag{\ref{second-class-symmetric-differential-n-forms-tau,rho}}:
\begin{equation}
\label{BS-second-definition}
{}_{v_1,\dots,v_\eta}\mathsf{BS}
:=
\text{Base Locus of }
\{
{}_{v_1,\dots,v_{\eta}}\phi_{j_1,\dots,j_{n-\eta}}^{\nu},\,
{}_{v_1,\dots,v_{\eta}}\psi_{j_1,\dots,j_{n-\eta}}^{\tau,\rho}
\}
_{
1
\leqslant 
j_1 
<
\cdots 
<
j_{n-\eta} 
\leqslant 
c
}
^{
\nu,
\tau,
\rho
}
\end{equation}
in the intersection of the $\eta$ hyperplanes: 
\[
{}_{v_1,\dots,v_\eta}
\mathbb{P}
(
\mathrm{T}_{\mathbb{P}^N}
)
:=
\mathbb{P}
(
\mathrm{T}_{\mathbb{P}^N}
)
\cap
\{
z_{v_1}
=
\cdots
=
z_{v_\eta}
=
0
\},
\]
and more specifically, we focus on the `interior part':
\[
{}_{v_1,\dots,v_\eta}
\oP
(
\mathrm{T}_{\mathbb{P}^N}
)
\,
:=
\,
{}_{v_1,\dots,v_\eta}
\mathbb{P}
(
\mathrm{T}_{\mathbb{P}^N}
)
\cap
\{
z_{r_0}
\cdots 
z_{r_{N-\eta}}
\neq
0
\}
\ \ \ \ \ \ \
\explain{see~\thetag{\ref{ascending order}}
for the indices $r_0,\dots,r_{N-\eta}$}.
\]

Firstly, we construct the $(c+r+c)\times (2N+2-2\eta)$ matrix ${}_{v_1,\dots,v_\eta}\mathsf{M}$, which will play the same role as the matrix $\mathsf{M}$,
whose $i$-row ($i=1\cdots c+r$) copies
the $(2N+2-2\eta)$ terms of \thetag{\ref{first-part-elt}} 
in the exact order, and whose $(c+r+j)$-th row ($j=1\cdots c$) is the differential of the $j$-th row. 

Secondly, in correspondence with $\mathscr{M}_{2c+r}^N$, by replacing plainly
$N$ with $N-\eta$, we introduce the algebraic variety:
\begin{equation}
\label{M_(2c)^(N-eta)}
\mathscr{M}_{2c+r}^{N-\eta}\,
\subset\,
{\sf Mat}_{(2c+r)\times 2(N-\eta+1)}(\mathbb K).
\end{equation}

Thirdly, let us recall the exceptional subvatiety:
\[
{}_{v_1,\dots,v_{\eta}}\Sigma
\subsetneqq
\mathbb{P}_{\mathbb{K}}^{\text{\ding{169}}}
\] 
defined in Proposition~\ref{all-rank-H=c-second}.

By performing the same reasoning as in the preceding proposition, we get:

\begin{Proposition}\label{characterization-BS-with-vanishing-coordinates}
For every choice of parameter outside ${}_{v_1,\dots,v_{\eta}}\Sigma$: 
\[
\big[
A_{\smallbullet}^{\smallbullet},
M_{\smallbullet}^{\smallbullet}
\big]\ \
\in\ \
\mathbb{P}_{\mathbb{K}}^{\text{\ding{169}}}
\setminus
{}_{v_1,\dots,v_{\eta}}\Sigma
\] 
a point: 
\[
\big(
[z],
[\xi]
\big) 
\in 
{}_{v_1,\dots,v_\eta}
\oP
(
\mathrm{T}_{\mathbb{P}^N}
)
\]
lies in the base locus~\thetag{\ref{BS-second-definition}}:
\[
\big(
[z],[\xi]
\big)\,
\in\, 
{}_{v_1,\dots,v_\eta}\mathsf{BS}
\]
if and only if:
\[
{}_{v_1,\dots,v_\eta}\mathsf{M}\,
(z,\xi)\,
\in\,
\mathscr{M}_{2c+r}^{N-\eta}. 
\eqno
\qed
\]
\end{Proposition}

\subsection{Emptiness of the base loci}
First, for the algebraic varieties~\thetag{\ref{M_(2c)^N}}, \thetag{\ref{M_(2c)^(N-eta)}}, we claim the following codimension estimates, which  
serve as the engine of the moving coefficients method. However, 
we will not present it here but in the next section. 

\begin{Lemma}[{\bf Core Lemma of MCM}]
\label{Core-lemma-of-MCM}
{\bf (i)}\,
For every positive integers $N\geqslant 1$, for every integers $c,r\geqslant 0$ with $2c+r\geqslant N$, there holds the codimension estimate:
\[
\cdim\,
\mathscr{M}_{2c+r}^N\,
\geqslant\,
\dim\,
\oP(\mathrm{T}_{\mathbb{P}^N})\,
=\,
2N
-
1.
\]

\noindent
{\bf (ii)}\,
For every positive integer
$\eta=1\cdots N-(c+r)-1$, for every sequence of ascending indices:
\[
0\leqslant v_1<\dots<v_{\eta}\leqslant N,
\] 
there holds the codimension estimate:
\[
\cdim\,
\mathscr{M}_{2c+r}^{N-\eta}\,
\geqslant\,
\dim\,
{}_{v_1,\dots,v_\eta}
\oP
(
\mathrm{T}_{\mathbb{P}^N}
)\,
=\,
2N
-
\eta
-
1.
\eqno
\qed
\]
\end{Lemma}

Now, let us show the power of this Core Lemma.

Bearing Proposition~\ref{characterization-BS-for-nonvanishing-coordinates} in mind, by mimicking the proof of Proposition~\ref{full-rank-c*(N+1)},
it is natural to introduce the
subvariety:
\[
M_{2c+r}^N
\hookrightarrow
\mathbb{P}_{\mathbb{K}}^{\text{\ding{169}}}
\times
\oP
(
\mathrm{T}_{\mathbb{P}^N}
),
\]
which is defined `in family' by:
\[
M_{2c+r}^N
\,:=\,
\Big\{
\big(
[A_{\smallbullet}^{\smallbullet},
M_{\smallbullet}^{\smallbullet}];\,
[z],[\xi]
\big)
\in
\mathbb{P}_{\mathbb{K}}^{\text{\ding{169}}}
\!\times\!
\oP({\rm T}_{\mathbb{P}^N})
\colon\,
{\sf M}(z,\xi)
\in
\mathcal{M}_{2c+r}^N
\Big\}.
\]

\begin{Proposition}\label{control-dimension}
There holds the dimension estimate:
\[
\dim\,
M_{2c+r}^N
\leqslant
\dim\,
\mathbb{P}_{\mathbb{K}}^{\text{\ding{169}}}.
\]
\end{Proposition}

\begin{proof}
Let $\pi_1,\pi_2$ be the two canonical projections:
\[
\xymatrix{
&
\mathbb{P}_{\mathbb{K}}^{\text{\ding{169}}}
\times
\oP
(
\mathrm{T}_{\mathbb{P}^N}
)
\ar[ld]_-{\pi_1}
\ar[rd]^-{\pi_2}
\\
\mathbb{P}_{\mathbb{K}}^{\text{\ding{169}}} 
& 
&  
\oP
(
\mathrm{T}_{\mathbb{P}^N}
).
}
\]
By mimicking Step~2 in Lemma~\ref{full-rank-c*(N+1)}, for every point 
$
(
[z],
[\xi]
)
\in 
\oP
(
\mathrm{T}_{\mathbb{P}^N}
)
$, 
we claim the fibre dimension estimate:
\begin{equation}
\label{fibre-dimenmsion-(x,xi)}
\aligned
\dim\,
\pi_2^{-1}
([z],[\xi])
\cap
M_{2c+r}^N
&
=
\dim\,
\mathbb{P}_{\mathbb{K}}^{\text{\ding{169}}} 
-
\cdim\,
\mathcal{M}_{2c+r}^N
\endaligned
\end{equation}

\begin{proof}
Noting that:
\[
\pi_2^{-1}([z],[\xi])
\cap
M_{2c+r}^N
=
\underbrace{
\pi_1
\Big(
\pi_2^{-1}([z],[\xi])
\cap
M_{2c+r}^N
\Big)
}_{\text{by Theorem~\ref{closed-image-theorem} is an algebraic set}}
\times
\underbrace{
\big\{
([z],[\xi])
\big\}
}_{\text{one point set}},
\] 
and considering the canonical projection:
\[
\widehat{\pi}
\colon
\ \ \
\mathbb{K}^{\text{\ding{169}}}
\setminus 
\{0\}\,
\longrightarrow\,
\mathbb{P}_{\mathbb{K}}^{\text{\ding{169}}},
\]
we receive:
\begin{equation}
\label{compute-dimension-3}
\aligned
\dim\,
\pi_2^{-1}([z],[\xi])
\cap
M_{2c+r}^N
&
=
\dim\,
\pi_1
\Big(
\pi_2^{-1}([z],[\xi])
\cap
M_{2c+r}^N
\Big)
\\
\explain{use Lemma ~\ref{algebraic-codim V=codim V-hat}}
\ \ \ \ \ \ \ \ \ \ \ \ \ \ \ \
&
=
\dim\,
\widehat\pi^{-1}
\Big(
\pi_1
\big(
\pi_2^{-1}([z],[\xi])
\cap
M_{2c+r}^N
\big)
\Big)
\cup
\{
0
\}
-
1.
\endaligned
\end{equation}

Now, notice that the set:
\[
\aligned
\mathbb{K}^{\text{\ding{169}}}\,
\supset\,
&
\widehat\pi^{-1}
\Big(
\pi_1
\big(
\pi_2^{-1}([z],[\xi])
\cap
M_{2c+r}^N
\big)
\Big)
\cup
\{0\}\,
\\
&
=
\Big\{
\big(
A_\smallbullet^\smallbullet,
M_\smallbullet^\smallbullet
\big)
\in
\mathbb{K}^{\text{\ding{169}}}\,
\colon\,
{\sf M}(z,\xi)
\in
\mathcal{M}_{2c+r}^N
\Big\}
\endaligned
\]
is nothing but the inverse image of:
\[
\mathscr{M}_{2c+r}^N\,
\subset\,
{\sf Mat}_{(2c+r)\times 2(N+1)}(\mathbb K)
\]
under the $\mathbb K$-linear map:
\[
\aligned
{\sf M}_{z,\,\xi}
\colon\ \ \ \ \ \ \ \ \ \
\mathbb{K}^{\text{\ding{169}}}
&
\longrightarrow
{\sf Mat}_{(2c+r)\times 2(N+1)}(\mathbb K)
\\
\big(
A_{\smallbullet}^{\smallbullet},
M_{\smallbullet}^{\smallbullet}
\big)
&
\longmapsto
 \ \ \
{\sf M}(z,\xi),
\endaligned
\]
which is surjective  
by the construction of {\sf M} --- see \thetag{\ref{first-rewrite-F_i}}, \thetag{\ref{C_i^j*z_j^(d-delta_N)}},
and by applying Lemma~\ref{linear-valuation-map-is-surjective}
--- since $z_0\neq 0,\dots,z_N\neq 0$ and $\xi\notin \mathbb{K}\cdot z$. 

Therefore, we have the codimension identity:
\begin{equation}
\label{compute-codimension-2}
\aligned
\cdim\,
\widehat\pi^{-1}
\Big(
\pi_1
\big(
\pi_2^{-1}([z],[\xi])
\cap
M_{2c+r}^N
\big)
\Big)
\cup
\{0\}\,
&
=
\cdim\,
\mathcal{M}_{2c+r}^N
\ \ \ \ \ \ \
\explain{${\sf M}_{z,\,\xi}$ is linear and surjective},
\endaligned
\end{equation}
and thereby we receive:
\[
\aligned
\dim\,
\pi_2^{-1}([z],[\xi])
\cap
M_{2c+r}^N
&
=
\dim\,
\widehat{\pi}^{-1}
\Big(
\pi_1
\big(
\pi_2^{-1}([z],[\xi])
\cap
M_{2c+r}^N
\big)
\Big)
\cup
\{
0
\}
-
1
\ \ \ \ \ \ \ \ \ \ \ \ \ \ \ \
\explain{use \thetag{\ref{compute-dimension-3}}}
\\
\explain{by definition of codimension}
\ \ \ \ \ \ \ \ \ \ \ 
&
=
\dim\,
\mathbb{K}^{\text{\ding{169}}}
-
\cdim\,
\widehat{\pi}^{-1}
\Big(
\pi_1
\big(
\pi_2^{-1}([z],[\xi])
\cap
M_{2c+r}^N
\big)
\Big)
\cup
\{0\}
-1
\\
\explain{why?}
\ \ \ \ \ \ \ \ \ \ \ 
&
=
\dim\,
\mathbb{P}_{\mathbb{K}}^{\text{\ding{169}}}
-
\cdim\,
\widehat{\pi}^{-1}
\Big(
\pi_1
\big(
\pi_2^{-1}([z],[\xi])
\cap
M_{2c+r}^N
\big)
\Big)
\cup
\{0\}
\\
\explain{use \thetag{\ref{compute-codimension-2}}}
\ \ \ \ \ \ \ \ \ \ \ 
&
=
\dim\,
\mathbb{P}_{\mathbb{K}}^{\text{\ding{169}}}
-
\cdim\,
\mathcal{M}_{2c+r}^N,
\endaligned
\]
which is exactly our claimed fibre dimension identity.
\end{proof}

Lastly, by applying the Fibre Dimension Estimate~\ref{algebraic-fibre-Dimension-Estimate}, we receive:
\[
\aligned
\dim\,
M_{2c+r}^N
&
\leqslant
\dim\,
\oP
(
\mathrm{T}_{\mathbb{P}^N}
)
+
\dim\,
\mathbb{P}_{\mathbb{K}}^{\text{\ding{169}}} 
-
\cdim\,
\mathcal{M}_{2c+r}^N
\qquad
\explain{use~\thetag{\ref{fibre-dimenmsion-(x,xi)}}}
\\
\explain{use Core Lemma~\ref{Core-lemma-of-MCM}}
\ \ \ \ \ \ 
&
\leqslant
\dim\,
\oP
(
\mathrm{T}_{\mathbb{P}^N}
)
+
\dim\,
\mathbb{P}_{\mathbb{K}}^{\text{\ding{169}}} 
-
\dim\,
\oP(\mathrm{T}_{\mathbb{P}^N})
\\
&
=
\dim\,
\mathbb{P}_{\mathbb{K}}^{\text{\ding{169}}},
\endaligned
\]
which is our claimed dimension estimate.
\end{proof}

Now, restricting the canonical projection $\pi_1$ to
$M_{2c+r}^N$:
\[
\pi_1
\colon\ \ \
M_{2c+r}^N\,
\longrightarrow\,
\mathbb{P}_{\mathbb{K}}^{\text{\ding{169}}},
\]
according to the dimension inequality just obtained, we gain:

\begin{Proposition}
\label{Observation-finite-fibre}
There exists a proper algebraic subset $\Sigma^\prime \subsetneqq\mathbb{P}_{\mathbb{K}}^{\text{\ding{169}}}$ such that,
for every choice of parameter outside $\Sigma^\prime$: 
\[
P
=
\big[
A_{\smallbullet}^{\smallbullet},
M_{\smallbullet}^{\smallbullet}
\big]\ \
\in\ \
\mathbb{P}_{\mathbb{K}}^{\text{\ding{169}}}
\setminus
\Sigma^\prime,
\]
the intersection of the fibre
$
\pi_1^{-1}
(P)
$
with
$M_{2c+r}^N$
is discrete or empty:
\[
\dim\,
\pi_1^{-1}
(P)\,
\cap\,
M_{2c+r}^N\,
\leqslant\,
0.
\eqno
\qed
\]
\end{Proposition}

Combining Propositions~\ref{characterization-BS-for-nonvanishing-coordinates} and
\ref{Observation-finite-fibre}, we receive:

\begin{Proposition}
Outside the proper algebraic subset: 
\[
\Sigma
\cup
\Sigma^\prime\,
\subsetneqq\,
\mathbb{P}_{\mathbb{K}}^{\text{\ding{169}}},
\]
for every choice of parameter: 
\[
\big[
A_{\smallbullet}^{\smallbullet},
M_{\smallbullet}^{\smallbullet}
\big]\ \
\in\ \
\mathbb{P}_{\mathbb{K}}^{\text{\ding{169}}}
\setminus
(
\Sigma
\cup
\Sigma^\prime
),
\] 
the base locus in the coordinates nonvanishing set:
\[
\mathsf{BS} 
\cap
\{
z_0
\cdots
z_N
\neq 
0
\}
\] 
is discrete or empty.
\qed
\end{Proposition}

Moreover, bearing in mind Proposition~\ref{characterization-BS-with-vanishing-coordinates}, by repeating the same reasoning as in the preceding proposition,
consider the 
subvariety:
\[
{}_{v_1,\dots,v_\eta}
M_{2c+r}^{N-\eta}
\hookrightarrow
\mathbb{P}_{\mathbb{K}}^{\text{\ding{169}}}
\times
{}_{v_1,\dots,v_\eta}
\oP
(
\mathrm{T}_{\mathbb{P}^N}
)
\]
which is defined `in family' by:
\[
{}_{v_1,\dots,v_\eta}M_{2c+r}^N
\,:=\,
\Big\{
\big(
[A_{\smallbullet}^{\smallbullet},
B_{\smallbullet}^{\smallbullet}];\,
[z],[\xi]
\big)
\in
\mathbb{P}_{\mathbb{K}}^{\text{\ding{169}}}
\times
{}_{v_1,\dots,v_\eta}\oP({\rm T}_{\mathbb{P}^N})
\colon\,
{}_{v_1,\dots,v_\eta}{\sf M}(z,\xi)
\in
\mathcal{M}_{2c+r}^{N-\eta}
\Big\},
\]
and hence receive a very analog of Proposition~\ref{control-dimension}.

\begin{Proposition}
There holds the dimension estimate:
\[
\dim\,
{}_{v_1,\dots,v_\eta}
M_{2c+r}^{N-\eta}
\leqslant
\dim\,
\mathbb{P}_{\mathbb{K}}^{\text{\ding{169}}}.
\eqno
\qed
\]
\end{Proposition}

Again, restricting the canonical projection $\pi_1$ to 
${}_{v_1,\dots,v_\eta}
M_{2c+r}^{N-\eta}$:
\[
\pi_1
\,
\colon
\,
{}_{v_1,\dots,v_\eta}
M_{2c+r}^{N-\eta}\,
\longrightarrow\,
\mathbb{P}_{\mathbb{K}}^{\text{\ding{169}}},
\]
according to the dimension inequality above, we receive:
\begin{Proposition}
\label{discrete base locus 2}
There exists a proper algebraic subset 
${}_{v_1,\dots,v_\eta}\Sigma^\prime \subsetneqq \mathbb{P}_{\mathbb{K}}^{\text{\ding{169}}}$ such that,
for every choice of parameter outside ${}_{v_1,\dots,v_\eta}\Sigma^\prime$: 
\[
P
=
\big[
A_{\smallbullet}^{\smallbullet},
M_{\smallbullet}^{\smallbullet}
\big]\ \
\in\ \
\mathbb{P}_{\mathbb{K}}^{\text{\ding{169}}}
\setminus
{}_{v_1,\dots,v_\eta}\Sigma^\prime,
\]
the intersection of the fibre
$\pi_1^{-1}(P)$
with
$M_{2c+r}^{N-\eta}$
is discrete or empty:
\[
\dim\,
\pi_1^{-1}
(P)\,
\cap\,
M_{2c+r}^{N-\eta}\,
\leqslant\,
0.
\eqno
\qed
\]
\end{Proposition}

Combining now Propositions~\ref{characterization-BS-with-vanishing-coordinates} 
and~\ref{discrete base locus 2}, we receive:

\begin{Proposition}
Outside the proper algebraic subset: 
\[
{}_{v_1,\dots,v_{\eta}}\Sigma\cup {}_{v_1,\dots,v_\eta}\Sigma^\prime\,
\subsetneqq\,
\mathbb{P}_{\mathbb{K}}^{\text{\ding{169}}}
\]
for every choice of parameter: 
\[
\big[
A_{\smallbullet}^{\smallbullet},
M_{\smallbullet}^{\smallbullet}
\big]\ \
\in\ \
\mathbb{P}_{\mathbb{K}}^{\text{\ding{169}}}
\setminus
(
{}_{v_1,\dots,v_{\eta}}\Sigma
\cup
{}_{v_1,\dots,v_\eta}\Sigma^\prime
),
\] 
the base locus in the corresponding `coordinates nonvanishing' set:
\[
{}_{v_1,\dots,v_\eta}\mathsf{BS}\,
\cap\,
\{
z_{r_0}
\cdots
z_{r_{N-\eta}}
\neq 
0
\}
\] 
is discrete or empty.
\qed
\end{Proposition}

\section{\bf The Engine of MCM}
\label{section: The engine of MCM}

\subsection{Core Codimension Formulas}
\label{Core-Codimension-Estimates}
Our motivation of this section is to prove the Core Lemma~\ref{Core-lemma-of-MCM}, which will succeed in Subsection~\ref{Proof of Core Lemma}.

As an essential step, by induction on positive integers $p\geqslant 2$ and $0\leqslant \ell\leqslant p$, we first estimate the codimension ${}_{\ell}C_{p}$ of the algebraic variety:
\begin{equation}
\label{r_X_p}
{}_{\ell}{\bf X}_{p}\,
\subset\,
{\sf Mat}_{p\times 2p}(\mathbb K)
\end{equation}
which consists of $p\times 2p$ matrices 
${\sf X}_p=(\alpha_1,\dots,\alpha_{p},\beta_1,\dots,\beta_{p})$ such that:
\begin{itemize}
\smallskip
\item[{\bf (i)}]
the first $p$ column vectors have rank: 
\begin{equation}
\label{rank(alpha_1,...,alpha_p) <= r}
\rank_{\mathbb{K}}\,
\big\{
\alpha_1,\dots,\alpha_{p}
\big\}
\leqslant 
\ell;
\end{equation}

\smallskip
\item[{\bf (ii)}] 
for every index $\nu=1\cdots p$, replacing $\alpha_{\nu}$ with $\alpha_{\nu}+(\beta_{1}+\cdots+\beta_{p})$ 
in the collection of column vectors
$\{\alpha_1,\dots,\alpha_p\}$, there holds the rank inequality:
\begin{equation}
\label{rank M^nu <= p-1}
\rank_{\mathbb{K}}\,
\big\{
\alpha_1,
\dots,
\widehat{\alpha_\nu},
\dots,
\alpha_p,
\alpha_{\nu}+(\beta_{1}+\cdots+\beta_{p})
\big\}
\leqslant 
p-1;
\end{equation}

\smallskip
\item[{\bf (iii)}] 
for every integer $\tau=1\cdots p-1$, for every index $\rho=\tau+1 \cdots p$, replacing $\alpha_{\rho}$ with $\alpha_{\rho}+(\beta_{\tau+1}+\cdots+\beta_{p})$ 
in the collection of column vectors
$\{\alpha_1+\beta_1,\dots,\alpha_\tau+\beta_\tau,
\alpha_{\tau+1},\dots,\alpha_\rho,\dots,\alpha_{p}\}$, there holds the rank inequality:
\begin{equation}
\label{rank M^(tau,rho)<=p-1}
\rank_{\mathbb{K}}\,
\big\{
\alpha_1+\beta_1,\dots,\alpha_\tau+\beta_\tau,
\alpha_{\tau+1},\dots,\widehat{\alpha_\rho},\dots,\alpha_{p},
\alpha_\rho+(\beta_{\tau+1}+\cdots+\beta_{p})
\big\}
\leqslant 
p-1.
\end{equation}
\end{itemize}
\smallskip

Let us start with the easy case $\ell=0$.

\begin{Proposition}
\label{Proposition: 0_C_p = p^2+1}
For every integer $p\geqslant 2$, the codimension value ${}_\ell C_{p}$ for $\ell=0$ is:
\begin{equation}
\label{0_C_p = p^2+1}
{}_0 C_{p}
=
p^2
+
1.
\end{equation}
\end{Proposition}

\begin{proof}
Now, {\bf (i)} is equivalent to:
\[
\underbrace{
\alpha_1
=
\cdots
=
\alpha_p
=
{\bf 0}
}_{{\sf codim}\,=\,p^2}\ \ .
\]
Thus {\bf (ii)} holds trivially, and the only nontrivial inequality in {\bf (iii)} is:
\[
\underbrace{
\rank_{\mathbb{K}}\,
\big\{
{\bf 0}+\beta_1,
\dots,
{\bf 0}+\beta_{p}
\big\}
\leqslant 
p-1
}_{{\sf codim}\,=\,1 \text{ by Lemma}~\ref{classical-codimension-formula}}\ \ ,
\]
which contributes one more codimension.
\end{proof}

For the general case $\ell=1\cdots p$, we will use Gaussian eliminations and do inductions on $p,\ell$. First, let us count the codimension of the exceptional locus of Gaussian eliminations.

\begin{Proposition}
\label{r_C_p^0=?}
For every integer $p\geqslant 2$, the codimensions ${}_\ell C_{p}^0$ of the algebraic varieties:
\[
\{
\alpha_1
+
\beta_1
=
{\bf 0}
\}\,
\cap\,
{}_{\ell}{\bf X}_{p}\,\,
\subset\,\,
{\sf Mat}_{p\times 2p}(\mathbb K)
\]
read according to the values of $\ell$ as:
\[
{}_\ell C_{p}^0
=
\begin{cases}
p+2
\qquad
&
{\scriptstyle(\ell\,=\,p-1,\,p)},
\medskip
\\
p+(p-\ell)^2
\qquad
&
{\scriptstyle(\ell\,=\,0\,\cdots\,p-2)}.
\end{cases}
\] 
\end{Proposition}

The following lemma is the key ingredient for the proof.

\begin{Lemma}
\label{lemma:rank(v_1,...,v_p,w) <= p-1}
In a field $\mathbb{K}$, let $W$ be a $\mathbb{K}$-vector space. Let
$p\geqslant 1$ be a positive integer. For any $(p+1)$ vectors: 
\[
\alpha_1,
\dots,
\alpha_p,
\beta\ \
\in\ \
W,
\] 
the rank restriction:
\begin{equation}
\label{rank(v_1,...,omit v_i,...,v_p,v_i+w) <= p-1}
\rank_{\mathbb K}\,
\{
\alpha_1,
\dots,
\widehat{\alpha_\nu},
\dots,
\alpha_p,
\alpha_\nu
+
\beta
\}\,
\leqslant\,
p-1
\qquad
{\scriptstyle(\nu\,=\,1\,\cdots\,p)},
\end{equation}
is equivalent to either:
\[
\rank_{\mathbb K}\,
\{
\alpha_1,
\dots,
\alpha_p,
\beta
\}\,
\leqslant\,
p-1,
\]
or:
\[
\rank_{\mathbb K}\,
\{
\alpha_1,
\dots,
\alpha_p
\}\,
=\,
p,
\qquad
(
\alpha_1
+
\cdots
+
\alpha_p
)
+
\beta
=
{\bf 0}.
\]
\end{Lemma}

\begin{proof}
Since `$\Longleftarrow$' is clear, 
we only prove the direction `$\Longrightarrow$'.

We divide the proof according to 
the rank of
$
\{
\alpha_1,
\dots,
\alpha_p
\}
$
into two cases.

\smallskip

{\em Case 1:}
$
\rank_{\mathbb K}\,
\{
\alpha_1,
\dots,
\alpha_p
\}
\leqslant
p-1
$. Assume on the contrary that:
\begin{equation}
\label{rank(v_1,...,v_p,beta)>=p}
\rank_{\mathbb K}\,
\{
\alpha_1,
\dots,
\alpha_p,
\beta
\}\,
\geqslant\,
p.
\end{equation}
Since we have the elementary estimate:
\begin{equation}
\label{rank-estimate-rank(a,b)<rank(a)+rank(b)}
\aligned
\rank_{\mathbb K}\,
\{
\alpha_1,
\dots,
\alpha_p,
\beta
\}\,
&
\leqslant\,
\rank_{\mathbb K}\,
\{
\alpha_1,
\dots,
\alpha_p
\}\,
+\,
\rank_{\mathbb K}\,
\{
\beta
\}\,
\\
&
\leqslant\,
(p-1)\,
+\,
1\,
\\
&
=
p,
\endaligned
\end{equation}
the inequalities `$\geqslant$' or `$\leqslant$' in~\thetag{\ref{rank(v_1,...,v_p,beta)>=p}}
and~\thetag{\ref{rank-estimate-rank(a,b)<rank(a)+rank(b)}}  are exactly equalities `$=$', and thus we have:
\begin{equation}
\label{beta-notin-(...)}
\beta\ \
\notin\ \
\mathsf{Span}_{\mathbb K}\,
\{
\alpha_1,
\dots,
\alpha_p
\},
\end{equation}
\[
\rank_{\mathbb K}\,
\{
\alpha_1,
\dots,
\alpha_p
\}\,
=\,
p-1.
\] 
Consequently, it is clear that we can find a certain index $\nu\in\{1,\dots,p\}$ such that:
\[
\rank_{\mathbb K}\,
\{
\alpha_1,
\dots,
\widehat{\alpha_\nu},
\dots,
\alpha_p
\}
=
p-1,
\] 
whence the above rank inequality~\thetag{\ref{rank(v_1,...,omit v_i,...,v_p,v_i+w) <= p-1}} implies:
\begin{equation}
\label{alpha_nu + beta in (....)}
\alpha_\nu
+
\beta\,
\in\,
\mathsf{Span}_{\mathbb{K}}\,
\{
\alpha_1,
\dots,
\widehat{\alpha_\nu},
\dots,
\alpha_p
\},
\end{equation}
which contradicts the 
formula~\thetag{\ref{beta-notin-(...)}}.

\smallskip

{\em Case 2:}
$
\rank_{\mathbb K}\,
\{
\alpha_1,
\dots,
\alpha_p
\}
=
p
$.
Here, inequalities~\thetag{\ref{rank(v_1,...,omit v_i,...,v_p,v_i+w) <= p-1}} also yield~\thetag{\ref{alpha_nu + beta in (....)}} for every $\nu$, whence:
\[
\beta
+
(
\alpha_1
+
\cdots
+
\alpha_p
)\,
\in\,
\mathsf{Span}_{\mathbb{K}}\,
\{
\alpha_1,
\dots,
\widehat{\alpha_\nu},
\dots,
\alpha_p
\}.
\]
Now, letting $\nu$ run from $1$ to $p$, and noting that:
\[
\cap_{\nu=1}^p\,
\mathsf{Span}_{\mathbb{K}}\,
\{
\alpha_1,
\dots,
\widehat{\alpha_\nu},
\dots,
\alpha_p
\}\,
=\,
\{\bf 0\},
\]
we immediately conclude the proof.
\end{proof}

\begin{proof}[Proof of Proposition~\ref{r_C_p^0=?}]
For every matrix
${\sf X}_p=(\alpha_1,\dots,\alpha_{p},\beta_1,\dots,\beta_{p})$ such that:
\begin{equation}
\label{alpha_1+beta+1=0,codimension=p}
\underbrace{
\alpha_1
+
\beta_1
=
{\bf 0}
}_{
{\sf codim}\,=\,p
}
,
\end{equation}
the conditions~\thetag{\ref{rank M^(tau,rho)<=p-1}} in
{\bf (iii)}
is trivial, and the restriction~\thetag{\ref{rank M^nu <= p-1}}, thanks to the lemma just obtained, is equivalent either to:
\begin{equation}
\label{rank (alpha_1,...,alpha_p,beta_1+...+beta_p) <= p-1}
\rank_{\mathbb K}\,
\{
\alpha_1,
\dots,
\alpha_p,
\beta_1
+
\cdots
+
\beta_p
\}\,
\leqslant\,
p-1,
\end{equation}
or to:
\begin{equation}
\label{rank (alpha_1,...,alpha_p) = p, beta_1+...=-(...)}
\rank_{\mathbb K}\,
\{
\alpha_1,
\dots,
\alpha_p
\}\,
=\,
p,
\qquad
\beta_1
+
\cdots
+
\beta_p\,
=\,
-\,
(
\alpha_1
+
\cdots
+
\alpha_p
).
\end{equation}

Now, since $\alpha_1+\beta_1=0$, adding the first column vector 
of~\thetag{\ref{rank (alpha_1,...,alpha_p,beta_1+...+beta_p) <= p-1}}
to the last one, we get:
\[
\underbrace{
\rank_{\mathbb K}\,
\{
\alpha_1,
\dots,
\alpha_p,
\beta_2
+
\cdots
+
\beta_p
\}\,
\leqslant\,
p-1
}_{{\sf codim}\,=\,2 \text{ by Lemma}~\ref{classical-codimension-formula}
}\ \ ,
\]
and similarly, \thetag{\ref{rank (alpha_1,...,alpha_p) = p, beta_1+...=-(...)}} is equivalent to:
\[
\rank_{\mathbb K}\,
\{
\alpha_1,
\dots,
\alpha_p
\}\,
=\,
p,
\qquad
\underbrace{
(
\alpha_2
+
\cdots
+
\alpha_p
)
+
(
\beta_2
+
\cdots
+
\beta_p
)\,
=\,
{\bf 0}
}_{
{\sf codim}\,=\,p
}\ .
\]
Therefore, when $\ell=p-1$ or $\ell=p$, we obtain the codimension formulas:
\[
\aligned
{}_{p-1} C_{p}^0
&
=
p+2,
\\
{}_p C_{p}^0
&
=
\min\,
\{
p+2,\,
p+p
\}
=
p+2.
\endaligned
\]

When $\ell=0\cdots p-2$, the restriction {\bf (ii)} is a consequence of {\bf (i)}:
\[
\aligned
&
\rank_{\mathbb{K}}\,
\big\{
\alpha_1,
\dots,
\widehat{\alpha_\nu},
\dots,
\alpha_p,
\alpha_{\nu}+(\beta_{1}+\cdots+\beta_{p})
\big\}
\\
\leqslant\,
& 
\rank_{\mathbb{K}}\,
\big\{
\alpha_1,
\dots,
\widehat{\alpha_\nu},
\dots,
\alpha_p
\big\}\,
+\,
\rank_{\mathbb{K}}\,
\big\{
\alpha_{\nu}+(\beta_{1}+\cdots+\beta_{p})
\big\}\,
\\
\leqslant\,
&
\rank_{\mathbb{K}}\,
\big\{
\alpha_1,
\dots,
\alpha_p
\big\}\,
+\,
1
\\
\leqslant\,
&
\ell+1
\\
\leqslant\,
&
p-1.
\endaligned
\]
Lastly, applying Lemma~~\ref{classical-codimension-formula}, restriction {\bf (i)}
contributes codimension $(p-\ell)^2$. Together with~\thetag{\ref{alpha_1+beta+1=0,codimension=p}}, this finishes the proof. 
\end{proof}

Now, we claim the following {\sl Codimension Induction Formulas}, the proof of which will appear in Subsection~\ref{subsection: proof of Codimension Induction Formulas}. In order to make sense of ${}_{\ell-2} C_{p-1}$ in~\thetag{\ref{eq:induction-formula-r_C_p}} when $\ell=1$, we henceforth make a {\em convention:} 
\[
{}_{-1}C_{p-1}
:=
\infty.
\]

\begin{Proposition}[{\bf Codimension Induction Formulas}]
\label{Codimension-induction-formulas}
{\bf (i)}\,
For every positive integer $p\geqslant 2$,
for $\ell=p$, the codimension value ${}_{p} C_{p}$ 
satisfies:
\begin{equation}
\label{p_C_p = min(.,.)}
{}_{p} C_{p}
=
\min\,
\big\{
p,\ \
{}_{p-1} C_{p}
\big\}.
\end{equation}

\smallskip
\noindent
{\bf (ii)}\,
For every positive integer $p\geqslant 3$,
for $\ell=p-1$, the codimension value ${}_{\ell} C_{p}$ 
satisfies:
\begin{equation}
\label{eq:induction-formula-(p-1)_C_p}
{}_{p-1} C_{p}
\geqslant
\min\,
\big\{
{}_{p-1} C_{p}^0,
\ \
{}_{p-1} C_{p-1}
+
2,
\ \
{}_{p-2} C_{p-1}
+
1,
\ \
{}_{p-3} C_{p-1}
\big\}.
\end{equation}

\smallskip
\noindent
{\bf (iii)}\,
For all integers $\ell=1\cdots p-2$, the codimension values ${}_{\ell} C_{p}$ 
satisfy:
\begin{equation}
\label{eq:induction-formula-r_C_p}
{}_{\ell} C_{p}
\geqslant
\min\,
\big\{
{}_{\ell} C_{p}^0, \ \
{}_{\ell} C_{p-1}
+
2(p-\ell)-1, \ \
{}_{\ell-1} C_{p-1}
+
(p-\ell),\ \
{}_{\ell-2} C_{p-1}
\big\}.
\end{equation}
\end{Proposition}

In fact, all the above inequalities `$\geqslant$' should be exactly equalities `$=$'. Nevertheless, `$\geqslant$' are already adequate
for our purpose.

\medskip

Now, let us establish the initial data for the induction process.
\begin{Proposition}
\label{Proposition: initial values p=2}
For the initial case $p=2$, 
there hold the codimension values:
\[
{}_0 C_2
=
5,
\qquad
{}_1 C_2
=
3,
\qquad
{}_2 C_2
=
2.
\] 
\end{Proposition}

\begin{proof}
Recalling formulas~\thetag{\ref{0_C_p = p^2+1}} and~\thetag{\ref{p_C_p = min(.,.)}}, we only need to prove 
${}_1 C_2 = 3$.

For every matrix:
\[
(\alpha_1,\alpha_2,\beta_1,\beta_2)\ \ 
\in\ \
{}_1{\bf X}_2\setminus {}_0{\bf X}_2,
\]
we have:
\begin{align}
\label{rank(alpha_1,alpha_2)=1}
\rank_{\mathbb K}\,
\big\{
\alpha_1,\,
\alpha_2
\big\}\,
&
=\,
1,
\\
\label{rank(alpha_1+(beta_1+beta_2),alpha_2)<=1)}
\rank_{\mathbb K}\,
\big\{
\alpha_1
+
(\beta_1+\beta_2),\,
\alpha_2
\big\}\,
&
\leqslant\,
1,
\\
\label{rank(alpha_1b,alpha_2+(beta_1+beta_2))<=1)}
\rank_{\mathbb K}\,
\big\{
\alpha_1,\,
\alpha_2
+
(\beta_1+\beta_2)
\big\}\,
&
\leqslant\,
1,
\\
\label{rank(alpha_1+beta_1,alpha_2+beta_2)<=1)}
\rank_{\mathbb K}\,
\big\{
\alpha_1
+
\beta_1,\,
\alpha_2
+
\beta_2
\big\}\,
&
\leqslant\,
1.
\end{align}
Either $\alpha_1$ or $\alpha_2$ is nonzero.
Firstly, assume $\alpha_1\neq {\bf 0}$. 
Then~\thetag{\ref{rank(alpha_1,alpha_2)=1}} yields:
\begin{equation}
\label{alpha_2 in K.alpha_1}
\alpha_2\ \
\in\ \
\mathbb{K}
\cdot
\alpha_1,
\end{equation}
and~\thetag{\ref{rank(alpha_1b,alpha_2+(beta_1+beta_2))<=1)}}
yields:
\[
\alpha_2
+
(\beta_1+\beta_2)
\ \
\in\ \
\mathbb{K}
\cdot
\alpha_1,
\]
whence by subtracting we receive:
\begin{equation}
\label{beta_1+beta_2 in K.alpha_1}
\beta_1+\beta_2\ \
\in\ \
\mathbb{K}
\cdot
\alpha_1.
\end{equation}
Next, adding the second column vector
of~\thetag{\ref{rank(alpha_1+beta_1,alpha_2+beta_2)<=1)}} to the first one, we see:
\begin{equation}
\label{rank(4 terms, 2 terms)<=1}
\rank_{\mathbb K}\,
\big\{
\alpha_1
+
\alpha_2
+
(
\beta_1
+
\beta_2
),
\,
\alpha_2
+
\beta_2
\big\}\,
\leqslant\,
1.
\end{equation}
By~\thetag{\ref{alpha_2 in K.alpha_1}} and~\thetag{\ref{beta_1+beta_2 in K.alpha_1}}:
\[
\alpha_1
+
\alpha_2
+
(
\beta_1
+
\beta_2
)\ \
\in\ \
\mathbb{K}
\cdot
\alpha_1,
\]
therefore~\thetag{\ref{rank(4 terms, 2 terms)<=1}} yields two possible situations, the first one is:
\begin{equation}
\label{alpha_1+alpha_2+beta_1+beta_2=0}
\alpha_1
+
\alpha_2
+
\beta_1
+
\beta_2
=
{\bf 0},
\end{equation}
and the second one is 
$
\alpha_1
+
\alpha_2
+
\beta_1
+
\beta_2
\neq
{\bf 0}
$ 
plus:
\[
\alpha_2
+
\beta_2\ \
\in\ \
\mathbb{K}
\cdot
\alpha_1.
\]
Recalling~\thetag{\ref{alpha_2 in K.alpha_1}}, 
the latter case immediately yields:
\[
\beta_2\ \
\in\ \
\mathbb{K}
\cdot
\alpha_1,
\]
and then~\thetag{\ref{beta_1+beta_2 in K.alpha_1}} implies:
\[
\beta_1\ \
\in\ \
\mathbb{K}
\cdot
\alpha_1,
\]
thus:
\begin{equation}
\label{rank(4 terms)=1}
\rank_{\mathbb K}\,
\big\{
\alpha_1,\,
\alpha_2,\,
\beta_1,\,
\beta_2
\big\}\,
=\,
1.
\end{equation}

Summarizing, the set:
\[
\big(
{}_1{\bf X}_2\setminus {}_0{\bf X}_2
\big)\,
\cap\,
\{\alpha_1\neq 0\}
\] 
is contained in the union of two algebraic varieties, the first one is defined by~\thetag{\ref{alpha_2 in K.alpha_1}}, \thetag{\ref{beta_1+beta_2 in K.alpha_1}}, \thetag{\ref{alpha_1+alpha_2+beta_1+beta_2=0}},
and the second one is defined by~\thetag{\ref{rank(4 terms)=1}}. Since both of the two varieties are of codimension 3, we get:
\[
\cdim\,
\big(
{}_1{\bf X}_2\setminus {}_0{\bf X}_2
\big)\,
\cap\,
\{\alpha_1\neq 0\}\,
\geqslant\,
3.
\]

Secondly, by symmetry, we also have:
\[
\cdim\,
\big(
{}_1{\bf X}_2\setminus {}_0{\bf X}_2
\big)\,
\cap\,
\{\alpha_2\neq 0\}\,
\geqslant\,
3.
\]
Hence the union of the above two sets satisfies:
\[
\cdim\,
{}_1{\bf X}_2\setminus {}_0{\bf X}_2\,
\geqslant\,
3.
\]
Now, recalling~\thetag{\ref{0_C_p = p^2+1}}:
\[
\cdim\,
{}_0{\bf X}_2\,
=\,
5\,
>
\,3,
\]
we immediately receive:
\[
\cdim\,
{}_1{\bf X}_2\,
\geqslant\,
3.
\]

Finally, noting that ${}_1{\bf X}_2$ contains the subvariety:
\[
\underbrace{
\Big\{
\rank\,
\{
\alpha_1,
\alpha_2,
\beta_1,
\beta_2
\}\,
\leqslant\,
1
\Big\}
}_{{\sf codim}\,=\,3 \text{ by Lemma}~\ref{classical-codimension-formula}}\ \
\subset\ \
{\sf Mat}_{2\times 4}(\mathbb{K}),
\]
it follows:
\[
\cdim\,
{}_1{\bf X}_2\,
\leqslant\,
3.
\]
In conclusion, the above two estimates squeeze out the desired codimension identity.
\end{proof}

Admitting temporally Proposition~\ref{Codimension-induction-formulas}, it is now time to deduce the crucial

\begin{Proposition}[{\bf Core Codimension Formulas}]
\label{Proposition: core codimension formulas}
For all integers $p\geqslant 2$, there hold the codimension estimates:
\begin{equation}
\label{r_C_p=r+(p-r)^2+1}
{}_{\ell}C_p\,
\geqslant
\,
\ell
+
(p-\ell)^2
+
1
\qquad
{\scriptstyle(\ell\,=\,0\,\cdots\,p-1)}\ ,
\end{equation}
and the codimension identity:
\[
\tag{\ref{r_C_p=r+(p-r)^2+1}$'$}
{}_{p}C_p
=
p.
\]
\end{Proposition}

\begin{proof}
The case $p=2$ is already done by the previous proposition.

Reasoning by induction, assume the formulas~\thetag{\ref{r_C_p=r+(p-r)^2+1}} and
\thetag{\ref{r_C_p=r+(p-r)^2+1}$'$} hold for some integer $p-1\geqslant 2$, and prove them for the integer $p$. 

Firstly,
formula~\thetag{\ref{0_C_p = p^2+1}} yields the case $\ell=0$. 

Secondly, for the case $\ell=p-1$,
thanks to 
Proposition~\ref{r_C_p^0=?} and to the induction hypothesis,  formula~\thetag{\ref{eq:induction-formula-(p-1)_C_p}} immediately yields:
\begin{equation}
\label{estimate (p-1)_C_p}
\aligned
{}_{p-1} C_{p}\,
&
\geqslant\,
\min\,
\big\{
{}_{p-1} C_{p}^0,
\ \
{}_{p-1} C_{p-1}
+
2,
\ \
{}_{p-2} C_{p-1}
+
1,
\ \
{}_{p-3} C_{p-1}
\big\}
\\
&
\geqslant\,
\min\,
\big\{
p+2,
\ \
(p-1)
+
2,
\ \
(p-2)
+
1^2
+
1
+
1,
\ \
(p-3)+2^2+1
\big\}
\\
&
=\,
p+1
\\
&
=
(p-1)
+
1^2
+
1.
\endaligned
\end{equation}
Similarly, for $\ell=1\cdots p-2$, recalling formula~\thetag{\ref{eq:induction-formula-r_C_p}}:
\[
{}_{\ell} C_{p}\,
\geqslant\,
\min\,
\big\{
{}_{\ell} C_{p}^0, \ \
{}_{\ell} C_{p-1}
+
2(p-\ell)-1, \ \
{}_{\ell-1} C_{p-1}
+
(p-\ell),\ \
{}_{\ell-2} C_{p-1}
\big\},
\]
and computing:
\[
\aligned
{}_{\ell} C_{p}^0
&
=
p
+
(p-\ell)^2
\\
&
=
\ell
+
(p-\ell)^2
+
\underbrace{
(p-\ell)
}_{\geqslant\,2}\ ,
\\
{}_{\ell} C_{p-1}
+
2(p-\ell)-1
&
\geqslant
\big[
\ell
+
\underline{
(p-1-\ell)^2
}
+1
\big]
+
2(p-\ell)
-
1
\\
&
=
\ell
+
\underline{
(p-\ell)^2
-
2(p-\ell)
+
1
}
+1
+2(p-\ell)
-
1
\\
&
=
\underbrace{
\ell
+
(p-\ell)^2
+
1
}_{
\text{the desired lower bound!}
}\ ,
\\
{}_{\ell-1} C_{p-1}
+
(p-\ell)
&
\geqslant
\big[
(\ell-1)
+
(p-\ell)^2
+1
\big]
+
(p-\ell)
\\
&
=
\ell
+
(p-\ell)^2
+
\underbrace{
(p-\ell)
}_{
\geqslant\,2
}\ ,
\\
{}_{\ell-2} C_{p-1}
&
\geqslant
(\ell-2)
+
(p-\ell+1)^2
+1
\\
&
=
(\ell-2)
+
\big[
(p-\ell)^2
+
2(p-\ell)
+
1
\big]
+
1
\\
&
=
\ell
+
(p-\ell)^2
+
\underbrace{
2(p-\ell)
}_{
\geqslant\,4
}\ ,
\endaligned
\]
we distinguish the desired lower bound without difficulty.

Lastly, the formula~\thetag{\ref{p_C_p = min(.,.)}} and \thetag{\ref{estimate (p-1)_C_p}} immediately yield \thetag{\ref{r_C_p=r+(p-r)^2+1}$'$}:
\[
{}_pC_p
=
p,
\]
which concludes the proof.
\end{proof}

\begin{Remark}
In fact, the above estimates ``$\geqslant$'' in~\thetag{\ref{r_C_p=r+(p-r)^2+1}} are exactly identities ``$=$''. By the same reasoning, in Section~\ref{The construction of hypersurfaces revisit and better lower bounds}, we will generalize the Core Codimension Formulas to cases of less number of moving coefficients terms, and thus receive better lower bounds on the 
hypersurfaces degrees. 
\end{Remark}

\subsection{Gaussian eliminations}
\label{subsection:Gauss-eliminations}
Following the notation in~\thetag{\ref{r_X_p}}, 
we denote by:
\[
{\sf X}_p
=
(\alpha_1,\dots,\alpha_{p},\beta_1,\dots,\beta_{p})
\] 
the coordinate columns of
${\sf Mat}_{p\times 2p}(\mathbb K)$, where each of the first $p$ columns explicitly writes as: 
\[
\alpha_i
=
(\alpha_{1,i},\dots,\alpha_{p,i})^{\mathrm{T}},
\] 
and where each of the last $p$ columns explicitly writes as:  
\[
\beta_i
=
(\beta_{1,i},\dots,\beta_{p,i})^{\mathrm{T}}.
\]

First, observing the structures of the matrices in~\thetag{\ref{rank M^nu <= p-1}}, \thetag{\ref{rank M^(tau,rho)<=p-1}}:
\[
\aligned
{\sf X}_p^{0,\nu}
&
:=
\big(
\alpha_1
\mid
\cdots
\mid
\underline{
\widehat{\alpha_\nu}}
\mid
\cdots
\mid
\alpha_p
\mid
\underline{
\alpha_{\nu}+(\beta_{1}+\cdots+\beta_{p})}
\big),
\\
{\sf X}_p^{\tau,\rho}
&
:=
\big(
\alpha_1+\beta_1
\mid
\cdots
\mid
\alpha_\tau+\beta_\tau
\mid
\alpha_{\tau+1}
\mid
\cdots
\mid
\underline{
\widehat{\alpha_\rho}}
\mid
\cdots
\mid
\alpha_{p}
\mid
\underline{
\alpha_\rho+(\beta_{\tau+1}+\cdots+\beta_{p})}
\big),
\endaligned
\]
where, slightly differently, the second underlined columns are understood to
appear in the first underlined removed places,
we realize that they
have the uniform shapes:
\begin{equation}
\label{X^(tau,rho)=X I^(tau,rho)}
\aligned
{\sf X}_p^{0,\nu}
&
=
{\sf X}_p\,
{\sf I}_{p}^{0,\nu},
\\
{\sf X}_p^{\tau,\rho}
&
=
{\sf X}_p\,
{\sf I}_{p}^{\tau,\rho},
\endaligned
\end{equation}
where the $2p\times p$ matrices ${\sf I}_{p}^{0,\nu}$ explicitly read as:

\medskip
\begin{center}
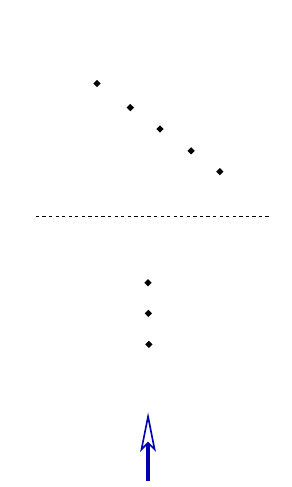
\end{center}

\noindent
the upper $p\times p$ submatrix being the identity, the lower $p\times
p$ submatrix being zero except its $\nu$-th column being a column of
$1$, and where lastly, the $2p\times p$ matrices ${\sf
I}_{p}^{\tau,\rho}$ explicitly read as:

\medskip
\begin{center}
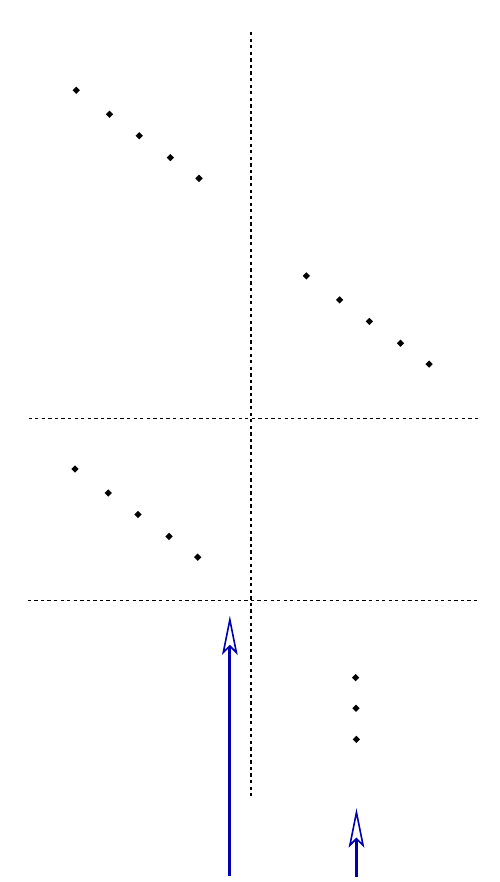
\end{center}

\noindent
the upper $p\times p$ submatrix being the identity, the lower $p\times
p$ submatrix being zero except $\tau$ copies of $1$ in the beginning
diagonal and $p-\tau$ copies of $1$ at the end of the $\rho$-th
column.
 
\begin{Observation}
\label{The key observation of the Gaussian eliminations}
For all $p\geqslant 3$, $\tau=1\cdots p-1$, $\rho=\tau+1\cdots p$, 
the matrices ${\sf I}_{p}^{\tau,\rho}$
transform to 
${\sf I}_{p-1}^{\tau-1,\rho-1}$
after deleting the first column and the rows $1, p+1$. 
\qed
\end{Observation}

Next, observe that all matrices 
${\sf X}^{\tau,\rho}$
have the same first column:
\[
\alpha_1
+
\beta_1
=
(
\
\alpha_{1,1}+\beta_{1,1}
\mid
\cdots
\mid
\alpha_{p,1}+\beta_{p,1}
)^{\mathrm{T}}.
\]
Therefore, when $\alpha_{1,1}+\beta_{1,1}\neq 0$, 
operating Gaussian eliminations 
by means of the matrix:
\begin{equation}
\label{Gauss-eliminations-matrix-G}
\mathsf{G}
:=
\begin{pmatrix}
1 
&
&
&
\\
-\frac{\alpha_{2,1}+\beta_{2,1}}{\alpha_{1,1}+\beta_{1,1}} 
& 
1
&
&
\\
\vdots 
& 
&
\ddots
&
\\
-\frac{\alpha_{p,1}+\beta_{p,1}}{\alpha_{1,1}+\beta_{1,1}} 
&
&
&
1
\end{pmatrix},
\end{equation}
these matrices ${\sf X}^{\tau,\rho}$ become simpler:
\begin{equation}
\label{G-multiplies-X_p}
{\sf G}\,
{\sf X}_p^{\tau,\rho}
=
\begin{pmatrix}
\alpha_{1,1}+\beta_{1,1} 
&
\bullet
&
\cdots
&
\bullet
\\
0 
& 
\star
&
\cdots
&
\star
\\
\vdots 
& 
\vdots
&
\ddots
&
\vdots
\\
0
&
\star
&
\cdots
&
\star
\end{pmatrix},
\end{equation}
where the lower-right $(p-1)\times (p-1)$ star submatrices enjoy amazing structural properties. At first, we need an:

\begin{Observation}
Let $p\geqslant 1$ be a positive integer, let $A$ be a $p\times 2p$
matrix, let $B$ be a $2p\times p$ matrix such that both its $1$-st, $(p+1)$-th rows are $(1,\underbrace{0,\dots,0}_{\text{zeros}})$. Then there holds:
\[
(A\,B)'
=
A''
\,
B''',
\]
where $(A\,B)'$ means the $(p-1)\times (p-1)$ matrix obtained by 
deleting the first row and column of $A\,B$, and where $A''$ means the $(p-1)\times 2(p-1)$ matrix obtained by deleting the first row and the columns $1,p+1$ of $A$, and where $B'''$ means the 
$2(p-1)\times (p-1)$ matrix obtained by deleting the first column and the rows $1,p+1$ of $B$.
\qed
\end{Observation}

Now, noting that:
\[
{\sf G}\,
{\sf X}_p^{\tau,\rho}
=
{\sf G}\,
\big(
{\sf X}_p\,
{\sf I}_{p}^{\tau,\rho}
\big)
=
\big(
{\sf G}\,
{\sf X}_p
\big)\,
{\sf I}_{p}^{\tau,\rho},
\]
thanks to the above two observations,
the $(p-1)\times (p-1)$ star submatrices enjoy the forms:
\begin{equation}
\label{star-matrix=...}
\begin{pmatrix}
\star
&
\cdots
&
\star
\\ 
\vdots
&
\ddots
&
\vdots
\\
\star
&
\cdots
&
\star
\end{pmatrix}
=
\mathsf{X}^{\sf G}_{p}
\,
\mathsf{I}_{p-1}^{\tau-1,\rho-1},
\end{equation}
where $\mathsf{X}^{\sf G}_p$ is the $(p-1)\times 2(p-1)$ matrix obtained by deleting the first row and the columns $1, p+1$ of ${\sf G}\,{\sf X}_p$.

Comparing \thetag{\ref{star-matrix=...}} and \thetag{\ref{X^(tau,rho)=X I^(tau,rho)}}, we immediately see that the star submatrices have the same structures as ${\sf X}_p^{0,\nu},{\sf X}_p^{\tau,\rho}$, which is the cornerstone of our induction approach.

\subsection{Study the morphism of left-multiplying by {\sf G}}
\label{Study the morphism of left-multiplying by G}
Let us denote by:
\[
D(\alpha_{1,1}+\beta_{1,1})\ \
\subset\ \
{\sf Mat}_{p\times 2p}(\mathbb K)
\]
the Zariski open set where
$\alpha_{1,1}+\beta_{1,1}\neq 0$.
Now, consider the regular map of left-multiplying by the function matrix
{\sf G}:
\[
\aligned
L_{\sf G}
\colon
\ \ \
D(\alpha_{1,1}+\beta_{1,1})\,
&
\longrightarrow\,
D(\alpha_{1,1}+\beta_{1,1})
\\
{\sf X}_p\,
&
\longmapsto\,
{\sf G}\,
{\sf X}_p.
\endaligned
\]
Of course, it is not surjective, as \thetag{\ref{G-multiplies-X_p}} shows that its image lies
in the variety:
\[
\cap_{i=2}^{p}\,
\{
\alpha_{i,1}
+
\beta_{i,1}
=
0
\}.
\]
In order to compensate this loss of surjectivity, combing with the regular map:
\[
\aligned
\text{\euro}
\colon
\ \ \
D(\alpha_{1,1}+\beta_{1,1})\,
&
\longrightarrow\,
{\sf Mat}_{(p-1)\times 1}(\mathbb K)
\\
{\sf X}_p\,
&
\longmapsto\,
(
\alpha_{2,1}+\beta_{2,1}
\mid
\cdots
\mid
\alpha_{p,1}+\beta_{p,1}
)^{\mathrm{T}},
\endaligned
\]
we construct a regular map:
\[
L_{\sf G}
\oplus
\text{\euro}
\colon
\ \ \
D(\alpha_{1,1}+\beta_{1,1})\,
\longrightarrow\,
\underbrace{
\bigg(
\cap_{\ell=2}^{p}\,
\{
\alpha_{i,1}
+
\beta_{i,1}
=
0
\}\,
\cap\,
D(\alpha_{1,1}+\beta_{1,1})
\bigg)\,\,
\oplus\,\,
{\sf Mat}_{(p-1)\times 1}(\mathbb K)
}_{
=:\,
\maltese\,
}\,,
\]
which turns out to be an isomorphism. 
In fact, it has the inverse morphism:
\[
\aligned
\underbrace{
\bigg(
\cap_{\ell=2}^{p}\,
\{
\alpha_{i,1}
+
\beta_{i,1}
=
0
\}\,
\cap\,
D(\alpha_{1,1}+\beta_{1,1})
\bigg)\,\,
\oplus\,\,
{\sf Mat}_{(p-1)\times 1}(\mathbb K)
}_{
=\,
\maltese\,
}\,
&
\longrightarrow\,
D(\alpha_{1,1}+\beta_{1,1})\,
\\
{\sf Y}\,
\oplus\,
(s_2,\dots,s_p)^{\sf T}\,
&
\longmapsto\,
{}^{-1}{\sf G}\,\cdot\,{\sf Y},
\endaligned
\]
where the matrix ${}^{-1}{\sf G}$ is the ``inverse'' of the regular function matrix ${\sf G}$ in~\thetag{\ref{Gauss-eliminations-matrix-G}}:
\begin{equation}
\label{Gauss-elimination-matrix-G-inverse}
\begin{pmatrix}
1 
&
&
&
\\
\frac{s_2}{\alpha_{1,1}+\beta_{1,1}} 
& 
1
&
&
\\
\vdots 
& 
&
\ddots
&
\\
\frac{s_p}{\alpha_{1,1}+\beta_{1,1}} 
&
&
&
1
\end{pmatrix}.
\end{equation}

Now, let us denote by:
\[
\pi_p
\colon
\ \ \
{\sf Mat}_{p\times 2p}(\mathbb K)\,
\longrightarrow\,
{\sf Mat}_{(p-1)\times 2(p-1)}(\mathbb K)
\]
the projection map obtained by deleting the first row and the columns $1,p+1$. Let us denote also:
\[
\mathcal{L}_{\sf G}
:=
\pi_p
\circ
L_{\sf G}.
\] 
It is worth to mention that there is a natural isomorphism:
\[
\aligned
{\sf R}
\colon
\qquad
\maltese\,
&
\xlongrightarrow{\sim}\,
D(\alpha_{1,1}+\beta_{1,1}),
\\
{\sf Y}\,
\oplus\,
(s_2,\dots,s_p)^{\sf T}\,
&
\longmapsto\,
\ \ ?
\endaligned
\]
where $?$ is ${\sf Y}$ but replacing $(b_{2,1},\dots,b_{p,1})^{\sf T}$ by
$(s_{2},\dots,s_{p})^{\sf T}$, and thus
we obtain a commutative diagram:
\begin{equation}
\label{key-diagram}
\xymatrix{
D(\alpha_{1,1}+\beta_{1,1}) 
\ar[r]^{\ \ \ \ \ \ \ \ L_{\sf G}\,\oplus\,\text{\euro}} 
\ar[rd]_{\!\!\!\!\!\!\mathcal{L}_{\sf G}} 
&
\maltese 
\ar[d]^{\pi_p\,\oplus\,{\bf 0}} 
\ar[r]^{\!\!\!\!\!\!\!\!\!\!\!\!\!\!\!\!\!\!\!\!\!\sf R}
&
D(\alpha_{1,1}+\beta_{1,1}) 
\ar[ld]^{\ \ \ \ \ \pi_p}
\\
&
{\sf Mat}_{(p-1)\times 2(p-1)}(\mathbb K),
}
\end{equation}
where the horizontal maps are isomorphisms, and where
the right vertical map is surjective with fibre:
\[
\underbrace{
\ker\,\pi_p
}_{
\mathbb{K}
\text{-linear space}
}
\,
\!\!\!
\cap\,\,\,
D(\alpha_{1,1}+\beta_{1,1}).
\]

Recalling the end of Subsection~\ref{subsection:Gauss-eliminations},
we in fact received the following key observation.

\begin{Observation}
\label{key-observation-of-Gauss-eliminations}
For every positive integer $p\geqslant 3$, for every integer
$\ell=0\cdots p-1$,
the image of the variety:
\[
{}_\ell{\bf X}_p\,
\cap\,
D(\alpha_{1,1}+\beta_{1,1})\,
\subset\,
D(\alpha_{1,1}+\beta_{1,1})
\]
under the map:
\[
\mathcal{L}_{\sf G}
\colon
\ \ \ 
D(\alpha_{1,1}+\beta_{1,1})\,
\longrightarrow\,
{\sf Mat}_{(p-1)\times 2(p-1)}(\mathbb K)
\]
is contained in the variety:
\[
{}_\ell{\bf X}_{p-1}\ \
\subset\ \
{\sf Mat}_{(p-1)\times 2(p-1)}(\mathbb K).
\eqno
\qed
\]
\end{Observation}

\subsection{A technical lemma}

Now, we carry out one preliminary lemma for the final proof
of Proposition~\ref{Codimension-induction-formulas}.

For all positive integers 
$p\geqslant 3$, for every integer 
$\ell=0 \cdots p-1$, for every fixed $(p-1)\times (p-1)$ matrix 
$J$ of rank $\ell$, 
denote the space 
which consists of all the $p\times p$ matrices of the form:
\[
\left(
\begin{array}{c:ccc}
z_{1,1} 
&
z_{1,2} 
&
\cdots  
&
z_{1,p}
\\
\hdashline
z_{2,1} 
& 
&  
& 
\\
\vdots
&\, 
& 
J
&  
\\
z_{p,1} 
&  
& 
&
\end{array}
\right)
\]
by ${}_JS_{p,\ell}\cong \mathbb{K}^{2p-1}$. For every integer $j=\ell,\ell+1$, denote by 
${}_JS_{p,\ell}^j\subset {}_JS_{p,\ell}$ the subvariety that consists of all the matrices having rank  
$\leqslant j$. 

\begin{Lemma}
\label{A-technical-lemma}
The codimensions of ${}_JS_{p,\ell}^j$ are:
\[
\cdim
\,
{}_JS_{p,\ell}^j
=
\begin{cases}
2(p-1-\ell)+1
\ \ \ \ \ &{\scriptstyle{(j\,=\,\ell)}},
\medskip
\\
p-1-\ell
\ \ \ \ \ &{\scriptstyle{(j\,=\,\ell+1)}}.
\end{cases}
\] 
\end{Lemma}

\begin{proof}
{\em Step 1.} 
We claim that the codimensions of ${}_JS_{p,\ell}^j$ are independent of the matrix $J$.

Indeed, choose two invertible 
$(p-1)\times (p-1)$ matrices $L$ and $R$, which
normalize the matrix $J$ by multiplications on both sides:
\[
L
\,
J
\,
R
=
\begin{pmatrix}
0
&
&
&
&
&
\\
&
\ddots
&
&
&
&
\\
&
&
0
&
&
&
\\
&
&
&
1
&
&
\\
&
&
&
&
\ddots
&
\\
&
&
&
&
&
1
\end{pmatrix}
=:
J_0,
\]
where all the entries of $J_0$ are zeros except the last $\ell$ copies of $1$ in the diagonal.
Therefore, we obtain an isomorphism:
\[
\aligned
LR
\,
\colon 
\,
{}_JS_{p,\ell}
\,
&
\xlongrightarrow{\sim}
\,
{}_{J_0}S_{p,\ell}
\\
S
&
\longmapsto
\begin{pmatrix}
1
&
\\
&
L
\end{pmatrix}
\,
S
\,
\begin{pmatrix}
1
&
\\
&
R
\end{pmatrix}
\endaligned
\]
whose inverse is:
\[
\aligned
L^{-1}R^{-1}
\colon\ 
{}_{J_0}S_{p,\ell}
\,
&
\xlongrightarrow{\sim}
\,
{}_JS_{p,\ell}
\\
S
&
\longmapsto
\begin{pmatrix}
1
&
\\
&
L^{-1}
\end{pmatrix}
\,
S
\,
\begin{pmatrix}
1
&
\\
&
R^{-1}
\end{pmatrix}.
\endaligned
\]
Since the map $LR$ preserves the rank of matrices, it induces an isomorphism
between ${}_JS_{p,\ell}^j$ and ${}_{J_0}S_{p,\ell}^j$, which concludes the claim. 

\smallskip
{\em Step 2.}
For $J_0$, doing elementary row and column operations, we get: 
\[
\aligned
&\ \
\rank_{\mathbb K}\,
\left(
\begin{array}{c:ccc}
z_{1,1} 
&
z_{1,2} 
&
\cdots  
&
z_{1,p}
\\
\hdashline
z_{2,1} 
& 
&  
& 
\\
\vdots
&\, 
& 
J_0 
&  
\\
z_{p,1} 
&  
& 
&
\end{array}
\right)
\\
=\,
&
\rank_{\mathbb K}\,
\left(
\begin{array}{c:ccc:ccc}
z_{1,1}
&
z_{1,2} 
&
\cdots  
&
z_{1,p-\ell}
&
z_{1,p-\ell+1}
&
\cdots
&
z_{1,p}
\\
\hdashline
z_{2,1} 
& 
\multicolumn{3}{c}{\multirow{3}{*}{\bf\Large 0}}
& 
\multicolumn{3}{c}{\multirow{3}{*}{\bf\Large 0}}
\\
\vdots
&\, 
&  
\\
z_{p-\ell,1} 
&  
& 
&
&
&
&
\\
\hdashline
z_{p-\ell+1,1}
&
\multicolumn{3}{c}{\multirow{3}{*}{\bf\Large 0}} 
&
1
&
&
\\
\vdots
&
&
&
&
&
\ddots
\\
z_{p,1}
&
&
&
&
&
&
1
\end{array}
\right)
\\
=\,
&
\rank_{\mathbb K}\,
\left(
\begin{array}{c:ccc:ccc}
z_{1,1}
-
\sum_{k=p-\ell+1}^p\,
z_{k,1}\,z_{1,k} 
&
z_{1,2} 
&
\cdots  
&
z_{1,p-\ell}
&
0
&
\cdots
&
0
\\
\hdashline
z_{2,1} 
& 
\multicolumn{3}{c}{\multirow{3}{*}{\bf\Large 0}}
& 
\multicolumn{3}{c}{\multirow{3}{*}{\bf\Large 0}}
\\
\vdots
&\, 
&  
\\
z_{p-\ell,1} 
&  
& 
&
&
&
&
\\
\hdashline
0
&
\multicolumn{3}{c}{\multirow{3}{*}{\bf\Large 0}} 
&
1
&
&
\\
\vdots
&
&
&
&
&
\ddots
\\
0
&
&
&
&
&
&
1
\end{array}
\right)
\\
=
\,
&
\rank_{\mathbb K}\,
\left(
\begin{array}{c:ccc}
z_{1,1}
-
\sum_{k=p-\ell+1}^p\,
z_{k,1}\,z_{1,k} 
&
z_{1,2} 
&
\cdots  
&
z_{1,p-\ell}
\\
\hdashline
z_{2,1} 
& 
\multicolumn{3}{c}{\multirow{3}{*}{\bf\Large 0}} 
\\
\vdots  
&
\\
z_{p-\ell,1} 
&  
\end{array}
\right)\,
+\,
\ell.
\endaligned
\]

{\em Step 3.}
In the $\mathbb K$-Euclidian space $\mathbb{K}^{2N-1}$ with coordinates
$(z_{1,1},z_{1,2},\dots,z_{1,N},z_{2,1},\dots,z_{N,1})$,
the algebraic subvariety defined by the rank inequality:
\[
\rank_{\mathbb K}\,
\left(
\begin{array}{c:ccc}
z_{1,1}
-
\sum_{k=p-\ell+1}^p\,
z_{k,1}\,z_{1,k} 
&
z_{1,2} 
&
\cdots  
&
z_{1,p-\ell}
\\
\hdashline
z_{2,1} 
& 
\multicolumn{3}{c}{\multirow{3}{*}{\bf\Large 0}} 
\\
\vdots  
&
\\
z_{p-\ell,1} 
&  
\end{array}
\right)
\leqslant
0
\qquad(\text{resp. }\leqslant 1)
\]
has codimension $2(p-1-\ell)+1$ (resp. $p-1-\ell$).
\end{proof}

\subsection{Proof of Proposition~\ref{Codimension-induction-formulas}}
\label{subsection: proof of Codimension Induction Formulas}
Recalling the definition~\thetag{\ref{r_X_p}}, and applying Lemma~\ref{lemma:rank(v_1,...,v_p,w) <= p-1}, we receive:
\begin{Corollary}
\label{sum of all 2p columns = 0}
For every integers $p\geqslant 1$, the difference of the varieties:
\[
{}_{p}{\bf X}_{p}\,
\setminus\,
{}_{p-1}{\bf X}_{p}\ \
\subset\ \
{\sf Mat}_{p\times 2p}(\mathbb K)
\]
is exactly the quasi-variety:
\[
\underbrace{
\Big\{
\alpha_1
+
\cdots
+
\alpha_p
+
\beta_1
+
\cdots
+
\beta_p
=
{\bf 0}
\Big\}
}_{
{\sf codim}\,
=\,
p
}\,\,
\cap\,\,
\Big\{
\underbrace{
\rank_{\mathbb{K}}\,
\{
\alpha_1,
\dots,
\alpha_p
\}\,
=\,
p
}_{
\Leftrightarrow\,
\det\,
(
\alpha_1
\mid
\cdots
\mid
\alpha_2
)\,
\neq\,
0
}
\Big\},
\]
whose codimension is $p$.
\end{Corollary}

\begin{proof}
For every $p\times 2p$ matrix:
\[
{}_{p}{\bf X}_{p}\,
\setminus\,
{}_{p-1}{\bf X}_{p}\ \ 
\ni\ \
{\sf X}_{p} 
=
(
\alpha_1,
\dots,
\alpha_{p},
\beta_1,
\dots,
\beta_{p}
),
\]
applying now Lemma~\ref{lemma:rank(v_1,...,v_p,w) <= p-1} to condition~\thetag{\ref{rank M^nu <= p-1}}:
\[
\rank_{\mathbb{K}}\,
\big\{
\alpha_1,
\dots,
\widehat{\alpha_\nu},
\dots,
\alpha_{p},
\alpha_{\nu}
+
\underbrace{
(\beta_{1}+\cdots+\beta_{p})
}_{
=:\,\beta
}
\big\}
\leqslant 
p-1
\qquad
{\scriptstyle(\nu\,=\,1\,\cdots\,p)},
\]
since:
\[
\rank_{\mathbb{K}}\,
\big\{
\alpha_1,
\dots,
\alpha_{p}
\big\}\,
=\,
p,
\]
we immediately receive: 
\[
\alpha_1
+
\cdots
+
\alpha_p
+
\beta_1
+
\cdots
+
\beta_p
=
{\bf 0}.
\]

On the other hand, for all matrices: 
\[
{\sf X}_{p} 
=
(
\alpha_1,
\dots,
\alpha_{p},
\beta_1,
\dots,
\beta_{p}
)
\]
satisfying the above identity,
{\bf (ii)} holds immediately.
Noting that
the $p\times p$ matrix in~\thetag{\ref{rank M^(tau,rho)<=p-1}} has a vanishing sum of all its columns, 
it has rank $\leqslant p-1$, i.e. {\bf (iii)} holds too. 
\end{proof}

\medskip

Now, we give a complete proof of the Codimension Induction Formulas.

\begin{proof}[Proof of~\thetag{\ref{p_C_p = min(.,.)}}]
This is a direct consequence of the above corollary.
\end{proof}

\begin{proof}[Proof of~\thetag{\ref{eq:induction-formula-(p-1)_C_p}}]
By Observation~\ref{key-observation-of-Gauss-eliminations},
under the map:
\[
\mathcal{L}_{\sf G}
\colon
\ \ \ 
D(\alpha_{1,1}+\beta_{1,1})\,
\longrightarrow\,
{\sf Mat}_{(p-1)\times 2(p-1)}(\mathbb K),
\]
the image of the variety:
\[
{}_{p-1}{\bf X}_p\,
\cap\,
D(\alpha_{1,1}+\beta_{1,1})
\]
is contained in the variety:
\[
{}_{p-1}{\bf X}_{p-1}\ \
\subset\ \
{\sf Mat}_{(p-1)\times 2(p-1)}(\mathbb K).
\]

Now, let us decompose the variety ${}_{p-1}{\bf X}_{p-1}$ 
into three pieces:
\begin{equation}
\label{decompose (p-2)_X_(p-1)}
{}_{p-1}{\bf X}_{p-1}\,
=\,
{}_{p-3}{\bf X}_{p-1}\,
\cup\,
\big(
{}_{p-2}{\bf X}_{p-1}\,
\setminus\,
{}_{p-3}{\bf X}_{p-1}
\big)\,
\cup\,
\big(
{}_{p-1}{\bf X}_{p-1}\,
\setminus\,
{}_{p-2}{\bf X}_{p-1}
\big),
\end{equation}
where each matrix $(\alpha_1,\dots,\alpha_{p-1},\beta_1,\dots,\beta_{p-1})$ in the $\underline{\text{first}}_{1}$ (resp. $\underline{\text{second}}_{2}$, $\underline{\text{third}}_{3}$) piece satisfies:
\begin{equation}
\label{rank(J)=p-2}
\rank_{\mathbb K}\,
(\alpha_1,\dots,\alpha_{p-1})\,
\underline{
\leqslant\,
p-3
}_{1}
\qquad
(
\text{resp. }
\underline{
=\,
p-2}_{2}\,, \
\underline{
=\,
p-1}_{3}
).
\end{equation}
Pulling back~\thetag{\ref{decompose (p-2)_X_(p-1)}} by the map $\mathcal{L}_{\sf G}$,
we see that:
\[
{}_{p-1}{\bf X}_p\,
\cap\,
D(\alpha_{1,1}+\beta_{1,1})
\] 
is contained in:
\begin{equation}
\label{(p-1)_X_p cap D(a+b) subset ...}
\mathcal{L}_{\sf G}^{-1}\,
(
{}_{p-3}{\bf X}_{p-1}
)\,
\cup\,
\mathcal{L}_{\sf G}^{-1}\,
\big(
{}_{p-2}{\bf X}_{p-1}\,
\setminus\,
{}_{p-3}{\bf X}_{p-1}
\big)\,
\cup\,
\mathcal{L}_{\sf G}^{-1}\,
\big(
{}_{p-1}{\bf X}_{p-1}\,
\setminus\,
{}_{p-2}{\bf X}_{p-1}
\big).
\end{equation}

Firstly, for every point in the first piece:
\[
{\sf Y}\ \
\in\ \
{}_{p-3}{\bf X}_{p-1},
\]
thanks to the commutative diagram~\thetag{\ref{key-diagram}}, we receive the fibre dimension:
\[
\aligned
\dim\,
\mathcal{L}_{\sf G}^{-1}
({\sf Y})
&
=
\dim\
\ker\,\pi_p
\,
\cap\,
D(\alpha_{1,1}+\beta_{1,1})
\\
&
=
\dim\
{\sf Mat}_{p\times 2p}(\mathbb K)
-
\dim\,
{\sf Mat}_{(p-1)\times 2(p-1)}(\mathbb K).
\endaligned
\]
Now, applying Corollary~\ref{transferred-codimension-estimate} to the regular map $\mathcal{L}$ restricted on:
\[
\mathcal{L}_{\sf G}^{-1}\,
(
{}_{p-3}{\bf X}_{p-1}
)\ \
\subset\ \
{\sf Mat}_{p\times 2p}(\mathbb K)
\] 
we receive the codimension estimate:
\begin{equation}
\label{codimension estimate of (p-3)_X_(p-1)}
\cdim\,
\mathcal{L}_{\sf G}^{-1}\,
(
{}_{p-3}{\bf X}_{p-1}
)\,
\geqslant\,
\underbrace{
\cdim\,
{}_{p-3}{\bf X}_{p-1}
}_{
{}_{p-3}C_{p-1}
}.
\end{equation}

Secondly, for every point in the second piece:
\[
{\sf Y}\ \
\in\ \
{}_{p-2}{\bf X}_{p-1}\,
\setminus\,
{}_{p-3}{\bf X}_{p-1},
\]
to look at the fibre of $\mathcal{L}_{{\sf G}}^{-1}({\sf Y})$,
thanks to the commutative diagram~\thetag{\ref{key-diagram}},
we can use:
\begin{equation}
\label{isormorphism L_G, pi_p}
\mathcal{L}_{\sf G}^{-1}
=
\big(
\underbrace{
{\sf R}
\circ
(L_{\sf G}\oplus\text{\euro})
}_{
\text{an isomorphism}
}
\big)^{-1}\,
\circ\,
\pi_p^{-1},
\end{equation}
and obtain:
\[
\aligned
&
\mathcal{L}_{\sf G}^{-1}
({\sf Y})\,\,
\cap\,\,
\big(
{}_{p-1}{\bf X}_p\,
\cap\,
D(\alpha_{1,1}+\beta_{1,1})
\big)
\\
\cong\,
&
\underline{
{\sf R}
\circ
(L_{\sf G}\oplus\text{\euro})
}\,
\mathcal{L}_{\sf G}^{-1}\,
({\sf Y})\
\cap\
\underline{
{\sf R}
\circ
(L_{\sf G}\oplus\text{\euro})
}\,
\big(
{}_{p-1}{\bf X}_p\,
\cap\,
D(\alpha_{1,1}+\beta_{1,1})
\big)
\\
\cong\,
&
\underbrace{
\pi_p^{-1}
({\sf Y})\,\,
\cap\,\,
{\sf R}
\circ
(L_{\sf G}\oplus\text{\euro})\,
\big(
{}_{p-1}{\bf X}_p\,
\cap\,
D(\alpha_{1,1}+\beta_{1,1})
\big)
}_{
=:\,\clubsuit
}
\qquad
\explain{use~\thetag{\ref{isormorphism L_G, pi_p}}}.
\endaligned
\]
Observe now that every matrix:
\[
(
\alpha_1
\mid
\cdots
\mid
\alpha_{p}
\mid
\beta_1
\mid
\cdots
\mid
\beta_{p}
)\,
\in\,
\clubsuit
\]
satisfies the rank estimate:
\[
\rank_{\mathbb K}\,
\big(
\alpha_1
\mid
\cdots
\mid
\alpha_{p}
\big)\,
\leqslant\,
p-1.
\]
Moreover, noting that the lower-right $(p-1)\times (p-1)$ submatrix
$J$ of 
$
\big(
\alpha_1
\mid
\cdots
\mid
\alpha_{p}
\big)
$
--- which is the left $(p-1)\times (p-1)$ submatrix of $\sf Y$ ---
has rank:
\[
\rank_{\mathbb K}\,
J
=
p-2
\qquad
\explain{see~\thetag{\ref{rank(J)=p-2}}},
\]
by applying Lemma~\ref{A-technical-lemma}, we get that:
\[
\clubsuit\ \
\subset\ \
\pi_p^{-1}
({\sf Y})\,
\]
has codimension greater or equal to:
\[
\cdim
\,
{}_JS_{p,p-2}^{p-1}
=
p-1-(p-2)
=
1.
\]
In other words:
\[
\mathcal{L}_{\sf G}^{-1}
({\sf Y})\,\,
\cap\,\,
\big(
{}_{p-1}{\bf X}_p\,
\cap\,
D(\alpha_{1,1}+\beta_{1,1})
\big)\ \
\subset\ \
\mathcal{L}_{\sf G}^{-1}({\sf Y})
\] 
has codimension $\geqslant 1$.
Thus, applying Corollary~\ref{transferred-codimension-estimate} to the map $\mathcal{L}_{\sf G}$ restricted on:
\[
\underbrace{
\mathcal{L}_{\sf G}^{-1}\,
\big(
{}_{p-2}{\bf X}_{p-1}\,
\setminus\,
{}_{p-3}{\bf X}_{p-1}
\big)\,
\cap\,
\big(
{}_{p-1}{\bf X}_p\,
\cap\,
D(\alpha_{1,1}+\beta_{1,1})
\big)
}_{
=:\,
{\bf II}
}\ \
\subset\ \
{\sf Mat}_{p\times 2p}(\mathbb K),
\] 
we receive the codimension estimate:
\begin{equation}
\label{codimension estimate of (p-2)_X_(p-1)}
\aligned
\cdim\,
{\bf II}\,
&
\geqslant\,
\cdim\,
\big(
{}_{p-2}{\bf X}_{p-1}\,
\setminus\,
{}_{p-3}{\bf X}_{p-1}
\big)\,
+\,
1
\\
&
\geqslant\,
\underbrace{
\cdim\,
{}_{p-2}{\bf X}_{p-1}\,
+\,
1
}_{
{}_{p-2}C_{p-1}\,
+\,
1
}.
\endaligned
\end{equation}

Thirdly, for every point in the third piece:
\[
{\sf Y}\ \
\in\ \
{}_{p-1}{\bf X}_{p-1}\,
\setminus\,
{}_{p-2}{\bf X}_{p-1},
\]
thanks to the diagram~\thetag{\ref{key-diagram}}:
\begin{equation}
\label{isormorphism L_G, (pi_p+0)}
\mathcal{L}_{\sf G}^{-1}
=
(
\underbrace{
L_{\sf G}\oplus\text{\euro}
}_{
\cong
})^{-1}\,
\circ\,
(\pi_p\oplus {\bf 0})^{-1},
\end{equation}
we receive:
\[
\aligned
&
\mathcal{L}_{\sf G}^{-1}
({\sf Y})\,
\cap\,
\big(
{}_{p-1}{\bf X}_p\,
\cap\,
D(\alpha_{1,1}+\beta_{1,1})
\big)
\\
\cong\,
&\underline{
(L_{\sf G}\oplus\text{\euro})}\,
\mathcal{L}_{\sf G}^{-1}
({\sf Y})\
\cap\
\underline{
(L_{\sf G}\oplus\text{\euro})}\,
\big(
{}_{p-1}{\bf X}_p\,
\cap\,
D(\alpha_{1,1}+\beta_{1,1})
\big)
\\
\cong\,
&
\underbrace{
(\pi_p\oplus {\bf 0})^{-1}
({\sf Y})\,
\cap\,
(L_{\sf G}\oplus\text{\euro})\,
\big(
{}_{p-1}{\bf X}_p\,
\cap\,
D(\alpha_{1,1}+\beta_{1,1})
\big)
}_{
=:\,{\sf \spadesuit}
}
\qquad
\explain{use~\thetag{\ref{isormorphism L_G, (pi_p+0)}}}.
\endaligned
\]
Recalling Corollary~\ref{sum of all 2p columns = 0}, the sum of all columns of ${\sf Y}$ ---
the bottom $(p-1)$ rows of  
$
(
\alpha_2
\mid
\cdots
\mid
\alpha_{p}
\mid
\beta_2
\mid
\cdots
\mid
\beta_{p}
)
$
---
is zero.
Thus, every element:
\[
(
\alpha_1
\mid
\cdots
\mid
\alpha_{p}
\mid
\beta_1
\mid
\cdots
\mid
\beta_{p}
)\,
\oplus\,
(s_2,\dots,s_p)^{\sf T}\ \
\in\ \
{\bf \spadesuit}
\]
not only satisfies:
\begin{equation}
\label{rank(alpha_1 | ... | alpha_p) <= p-1}
\rank_{\mathbb K}\,
(
\alpha_1
\mid
\cdots
\mid
\alpha_p
)\,
\leqslant\,
p-1,
\end{equation}
but also satisfies:
\[
\alpha_2
+
\cdots
+
\alpha_p
+
\beta_2
+
\cdots
+
\beta_p
=
(
\underbrace{
\alpha_{1,2}+\cdots+\alpha_{1,p}+\beta_{1,2}+\cdots+\beta_{1,p}
}_{
\text{only this first entry could be nonzero}
}\,,
\underbrace{
0\,,
\dots\,,
0
}_{(p-1) \text{ copies}}
)^{\sf T}.
\]
Remembering that:
\[
\alpha_1+\beta_1\,
=
(
\alpha_{1,1}+\beta_{1,1}
,
\underbrace{
0\,,
\dots\,,
0
}_{(p-1) \text{ copies}}
)^{\sf T},
\]
summing the above two identities immediately yields:
\begin{equation}
\label{sum of all alpha's and beta's = (*,0,...0)}
\alpha_1
+
\cdots
+
\alpha_p
+
\beta_1
+
\cdots
+
\beta_p
=
\alpha_{1,1}+\cdots+\alpha_{1,p}+\beta_{1,1}+\cdots+\beta_{1,p},
\underbrace{
0\,,
\dots,
0
}_{(p-1) \text{ copies}}
)^{\sf T}.
\end{equation}
Now, note that~\thetag{\ref{rank M^nu <= p-1}} (`matrices ranks') in condition {\bf (ii)} are preserved under the map $L_{\sf G}$ (`Gaussian eliminations'), in particular, for $\nu=1$, the image satisfies:
\[
\rank_{\mathbb{K}}\,
\big\{
\alpha_{1}+(\beta_{1}+\cdots+\beta_{p}),
\alpha_2,
\dots,
\alpha_p
\big\}
\leqslant 
p-1,
\]
which, by adding the column vectors $2\cdots p$ to the first one, is equivalent to:
\[
\rank_{\mathbb{K}}\,
\big\{
\alpha_1
+
\cdots
+
\alpha_p
+
\beta_1
+
\cdots
+
\beta_p,
\alpha_2
\dots,
\alpha_p
\big\}
\leqslant 
p-1.
\]
Remember~\thetag{\ref{sum of all alpha's and beta's = (*,0,...0)}}, and recalling Corollary~\ref{sum of all 2p columns = 0}: 
\begin{equation}
\label{lower (p-1)*(p-1) submatrix is of full rank}
\text{the bottom } 
(p-1)\times (p-1)
\text{ submatrix of } 
(\alpha_2\mid\cdots\mid\alpha_p)
\text{ is of full rank }
(p-1),
\end{equation} 
we immediately receive:
\[
\underbrace{
\alpha_{1,1}+\cdots+\alpha_{1,p}+\beta_{1,1}+\cdots+\beta_{1,p}
=
0
}_{{\sf codim}\,=\,1}.
\]
Therefore, by applying Lemma~\ref{A-technical-lemma}, the restrictions~\thetag{\ref{rank(alpha_1 | ... | alpha_p) <= p-1}} and~\thetag{\ref{lower (p-1)*(p-1) submatrix is of full rank}} contribute one extra codimension:
\[
\cdim
\,
{}_JS_{p,p-1}^{p-1}
=
1.
\]
Thus, we see that `the fibre in fibre':
\[
(\pi_p\oplus {\bf 0})^{-1}
({\sf Y})\,
\cap\,
\spadesuit\ \
\subset\ \
(\pi_p\oplus {\bf 0})^{-1}
({\sf Y})\,
\]
has codimension greater or equal to:
\[
1+1
=
2.
\]
Now, applying once again Corollary~\ref{transferred-codimension-estimate} to the map $\mathcal{L}_{\sf G}$ restricted on:
\[
\underbrace{
\mathcal{L}_{\sf G}^{-1}\,
\big(
{}_{p-1}{\bf X}_{p-1}\,
\setminus\,
{}_{p-2}{\bf X}_{p-1}
\big)\,
\cap\,
\big(
{}_{p-1}{\bf X}_p\,
\cap\,
D(\alpha_{1,1}+\beta_{1,1})
\big)
}_{
=:\,
{\bf III}
}\ \
\subset\ \
{\sf Mat}_{p\times 2p}(\mathbb K),
\] 
we receive the codimension estimate:
\begin{equation}
\label{codimension estimate of (p-1)_X_(p-1)}
\aligned
\cdim\,
{\bf III}\,
&
\geqslant\,
\cdim\,
\big(
{}_{p-1}{\bf X}_{p-1}\,
\setminus\,
{}_{p-2}{\bf X}_{p-1}
\big)\,
+\,
2
\\
&
\geqslant\,
\underbrace{
\cdim\,
{}_{p-1}{\bf X}_{p-1}\,
+\,
2
}_{
{}_{p-1}C_{p-1}\,
+\,
2
}.
\endaligned
\end{equation}

Summarizing \thetag{\ref{(p-1)_X_p cap D(a+b) subset ...}}, 
\thetag{\ref{codimension estimate of (p-3)_X_(p-1)}},
\thetag{\ref{codimension estimate of (p-2)_X_(p-1)}}, 
\thetag{\ref{codimension estimate of (p-1)_X_(p-1)}}, we receive the codimension estimate:
\[
\cdim\,
{}_{p-1}{\bf X}_p\,
\cap\,
D(\alpha_{1,1}+\beta_{1,1})\,
\geqslant\,
\min\,
\big\{ 
\cdim\,
{}_{p-3}{\bf X}_{p-1},\
\cdim\,
{}_{p-2}{\bf X}_{p-1}\,
+\,
1,\
\cdim\,
{}_{p-1}{\bf X}_{p-1}\,
+\,
2
\big\}.
\]
By permuting the indices, we know that all:
\[
{}_{p-1}{\bf X}_p\,
\cap\,
D(\alpha_{i,1}+\beta_{i,1})\ \
\subset\ \
{\sf Mat}_{p\times 2p}(\mathbb K)
\qquad
{\scriptstyle(i\,=\,1\,\cdots\,p)}
\]
have the same codimension, and so does their union:
\[
{}_{p-1}{\bf X}_p\,
\cap\,
D(\alpha_1+\beta_1)\,
=\,
\cup_{i=1}^{p}\,
\big(
{}_{p-1}{\bf X}_p\,
\cap\,
D(\alpha_{i,1}+\beta_{i,1})
\big)\ \
\subset\ \
{\sf Mat}_{p\times 2p}(\mathbb K).
\]
Finally, taking codimension on both sides of:
\[
{}_{p-1}{\bf X}_p\,
=\,
\big(
{}_{p-1}{\bf X}_{p}
\cap\,
V(\alpha_1+\beta_1)\,
\big)\ \
\cup\ \
\big(
{}_{p-1}{\bf X}_p\,
\cap\,
D(\alpha_1+\beta_1)
\big),
\]
Proposition~\ref{r_C_p^0=?} and the preceding estimate conclude the proof.
\end{proof}

\begin{proof}
[Proof of~\thetag{\ref{eq:induction-formula-r_C_p}}]
If $\ell\geqslant 2$, decompose the variety ${}_{\ell}{\bf X}_{p-1}$ 
into three pieces:
\[
{}_{\ell}{\bf X}_{p-1}\,
=\,
{}_{\ell-2}{\bf X}_{p-1}\,
\cup\,
\big(
{}_{\ell-1}{\bf X}_{p-1}\,
\setminus\,
{}_{\ell-2}{\bf X}_{p-1}
\big)\,
\cup\,
\big(
{}_{\ell}{\bf X}_{p-1}\,
\setminus\,
{}_{\ell-1}{\bf X}_{p-1}
\big);
\]
and if $\ell=1$, 
decompose the variety ${}_{\ell}{\bf X}_{p-1}$ 
into two pieces:
\[
{}_{\ell}{\bf X}_{p-1}\,
=\,
{}_{\ell-1}{\bf X}_{p-1}\,
\cup\,
\big(
{}_{\ell}{\bf X}_{p-1}\,
\setminus\,
{}_{\ell-1}{\bf X}_{p-1}
\big).
\]
Now, by mimicking the preceding proof, namely by applying Lemma~\ref{A-technical-lemma} and Corollary~\ref{transferred-codimension-estimate},
everything goes on smoothly with much less effort,
because there is no need to perform delicate codimension 
estimates such as \thetag{\ref{codimension estimate of (p-1)_X_(p-1)}}.
\end{proof}

\subsection{Proof of Core Lemma~\ref{Core-lemma-of-MCM}}
\label{Proof of Core Lemma}
If $N=1$, there is nothing to prove. Assume now $N\geqslant 2$.
Comparing~\thetag{\ref{M_(2c)^N}} and~\thetag{\ref{r_X_p}}, it is natural to introduce the projection:
\[
\aligned
\pi_{2c+r,N}
\colon\ \ \
{\sf Mat}_{(2c+r)\times 2(N+1)}(\mathbb K)\,
&
\longrightarrow\,
{\sf Mat}_{N\times 2N}(\mathbb K)
\\
\big(
\alpha_0,\dots,\alpha_{p},\beta_0,\dots,\beta_{p}
\big)\,
&
\longmapsto\,
\big(
\widehat{\alpha}_1,\dots,\widehat{\alpha}_p,\widehat{\beta}_1,\dots,\widehat{\beta}_p
\big),
\endaligned
\]
where each widehat vector is obtained by extracting the first $N$ rows (entries) out of the original $2c+r$ rows (entries). 

Now, for every point:
\[
(\alpha_0,\dots,\alpha_{p},\beta_0,\dots,\beta_{p})\ \
\in\ \
\mathscr{M}_{2c+r}^N\ 
\subset\ 
{\sf Mat}_{(2c+r)\times 2(N+1)}(\mathbb K),
\]
in restriction~\thetag{\ref{(ii) of M^n_2c}},
by setting $\nu=0$, we receive:
\[
\rank_{\mathbb{K}}\,
\big\{
\alpha_1,\dots,\alpha_{N},\alpha_{0}+(\beta_0+\beta_1+\cdots+\beta_{N})
\big\}
\leqslant 
N-1.
\]
Dropping the last column and extracting the first $N$
rows, we get:
\[
\rank_{\mathbb{K}}\,
\big\{
\widehat{\alpha}_1,
\dots,
\widehat\alpha_{N}
\big\}
\leqslant 
N-1.
\]
Similarly, in restriction~\thetag{\ref{(iii) of M^n_2c}}, 
by dropping the first column and extracting the first $N$ rows,
for all $\tau=0\cdots N-1$ and $\rho=\tau+1\cdots N$, we obtain:
\[
\rank_{\mathbb{K}}\,
\big\{
\widehat\alpha_1
+
\widehat\beta_1,
\dots,
\widehat\alpha_\tau
+
\widehat\beta_\tau,
\widehat\alpha_{\tau+1},
\dots
\fbox{$\widehat\alpha_\rho$}
\dots,
\widehat\alpha_{N},
\widehat\alpha_\rho
+
(\widehat\beta_{\tau+1}+\cdots+\widehat\beta_{N})
\big\}
\leqslant 
N-1,
\]
where we omit the column vector $\widehat\alpha_\rho$ in the box.
Summarizing the above two inequalities, $(\widehat{\alpha}_1,\dots,\widehat{\alpha}_p,\widehat{\beta}_1,\dots,\widehat{\beta}_p)$ satisfies the restriction~\thetag{\ref{rank(alpha_1,...,alpha_p) <= r}} -- \thetag{\ref{rank M^(tau,rho)<=p-1}}:
\[
\underbrace{
\pi_{2c+r,N}\,
(\alpha_0,\dots,\alpha_{p},\beta_0,\dots,\beta_{p})
}_{
=\,(\widehat{\alpha}_1,\dots,\widehat{\alpha}_p,\widehat{\beta}_1,\dots,\widehat{\beta}_p)
}
\in\ \
{}_{N-1}{\bf X}_{N}\ 
\subset\ 
{\sf Mat}_{N\times 2N}(\mathbb K).
\]
Therefore:
\[
\pi_{2c+r,N}\,
(\mathscr{M}_{2c+r}^N)\ \
\subset\ \
{}_{N-1}{\bf X}_{N}.
\]

Moreover, for every point ${\sf Y}\in {}_{N-1}{\bf X}_{N}$, 
the `fibre in fibre':
\[
\pi_{2c+r,N}^{-1}({\sf Y})\,
\cap\,
\mathscr{M}_{2c+r}^N\ \
\subset\ \
\pi_{2c+r,N}^{-1}({\sf Y}),
\]
thanks to~\thetag{\ref{sum-of-(2N+2)-columns=0}}, 
has codimension $\geqslant 2c+r$. Thus a direct application of Corollary~\ref{transferred-codimension-estimate} yields:
\[
\aligned
\cdim\,
\mathscr{M}_{2c+r}^N\,
&
\geqslant\,
\cdim\,
{}_{N-1}{\bf X}_{N}
+
2c
+r
\\
\explain{use~\thetag{\ref{r_C_p=r+(p-r)^2+1}}}
\qquad
&
\geqslant\,
N+1
+
2c+r.
\endaligned
\]
Repeating the same reasoning, we obtain:
\[
\aligned
\cdim\,
\mathscr{M}_{2c+r}^{N-\eta}\,
&
\geqslant\,
\cdim\,
{}_{N-\eta-1}{\bf X}_{N-\eta}
+
2c+r
\\
\explain{use~\thetag{\ref{r_C_p=r+(p-r)^2+1}}}
\qquad
&
\geqslant\,
N-\eta+1
+
2c+r.
\endaligned
\]
Remembering $2c+r\geqslant N$, we conclude the proof. 
\qed

\subsection{`Macaulay2', `Maple' et al. {\sl vs.} the Core Lemma}
\label{Macaulay2}
Believe it or not, concerning the Core Lemma or the Core Codimension Formulas, `Macaulay2' -- a professional software system devoted to supporting research in algebraic geometry and commutative algebra -- is not strong enough to compute the precise codimensions of the involved determinantal ideals, even in small dimensions $p\geqslant 4$. And unfortunately, so do other mathematical softwares, like `Maple'...

This might indicate some weaknesses of current computers. Since the Core Lemma or a variation of it should be a crucial step in the constructions of {\sl ample examples}, the dream of finding explicit examples with rational coefficients, firstly in small dimensional cases,
could be kind of a challenge 
for a moment.
 
\section{\bf A rough estimate of lower degree bound}
\label{Section: Rough estimates}
\subsection{Effective results}
Recalling 
Subsection~\ref{subsection: product coup},
we first provide an effective

\begin{Theorem}
For all $N\geqslant 3$, for any $\epsilon_{1},\dots,\epsilon_{c+r}\in \{1,2\}$, 
Theorem~\ref{Main Nefness Theorem 1/2}
holds 
for $\varheartsuit=1$.
and for all $d\geqslant N^{N^2/2}-1$.
\end{Theorem}

\begin{proof}
Setting $\delta_{c+r+1}=2$ in~\thetag{\ref{delta_(c+r+1)}},
and demanding all~\thetag{\ref{mu_0>?}} -- 
\thetag{\ref{delta_l=?}} to be equalities, 
we thus receive the desired estimate:
\[
\explain{see~\thetag{\ref{d>?}}}
\qquad
(N+1)\,\mu_{N,N}
\leqslant
 N^{N^2/2}-1.
\]
For the sake of completeness, we present all computational details in Subsection~\ref{subsection: Explicit computations} below.
\end{proof}

Hence, the {\sl product coup} 
in Subsection~\ref{subsection: product coup}
yields

\begin{Theorem}
\label{Main Nefness Theorem with d_0=?}
In Theorem~\ref{Main Nefness Theorem}, 
for $\varheartsuit=1$,
the lower bound
$\texttt{d}_0(-1)=N^{N^2}$ works.
\qed
\end{Theorem}

\subsection{Computational details}
\label{subsection: Explicit computations}
We specify~\thetag{\ref{mu_0>?}} -- \thetag{\ref{d>?}} as follows.
Recalling that $\delta_{c+r+1}=2$ and $\varheartsuit=1$,
for every integer $l=c+r+1 \cdots N$,
we choose:
\begin{equation}\label{mu_0>?'}
\mu_{l,0}\,
=\,
l\,
\delta_l
+
4\,l
+
1,
\end{equation}
and inductively we choose:
\begin{equation}\label{mu_k>?'}
\mu_{l,k}\,
=\,
\sum_{j=0}^{k-1}\,
l\,\mu_{l,j}
+
(l-k)\,\delta_l
+
4\,l
+
1
\qquad
{\scriptstyle{(k\,=\,1\,\cdots\,l)}}.
\end{equation}
Actually, we
take the above values in purpose, 
because they also work in the degree estimates  in our coming paper.

For every integer $l=c+r+1\cdots N$, for every integer $k=0\cdots l$,
let:
\begin{equation}
\label{S_(l,k) definion}
S_{l,k}\,
:=\,
\sum_{j=0}^{k}\,
\mu_{l,j}.
\end{equation}
For $k=1\cdots l$, we have:
\[
\explain{see~\thetag{\ref{mu_k>?'}}}
\qquad
\underbrace{
S_{l,k}
-
S_{l,k-1}}_{=\,\mu_{l,k}}\,
=\,
l\,S_{l,k-1}
+
(l-k)\,\delta_l
+
4\,l
+
1.
\]
Moving the term `$-S_{l,k-1}$' to the right hand side, we receive:
\[
S_{l,k}\,
=\,
(l+1)\,S_{l,k-1}
+
(l-k)\,\delta_l
+
4\,l
+
1.
\]
Dividing by $(l+1)^k$ on both sides, we receive:
\[
\underline{
\frac{S_{l,k}}{(l+1)^k}}\,
=\,
\underline{
\frac{S_{l,k-1}}{(l+1)^{k-1}}}\,
+\,
\Big(
l\,\delta_l
+
4\,l
+
1
\Big)\,
\frac{1}{(l+1)^k}\,
-\,
\delta_l\,
\frac{k}{(l+1)^{k}}.
\]
Noting that the two underlined terms have the same structure, doing induction backwards $k\cdots 1$, we receive:
\[
\frac{S_{l,k}}{(l+1)^k}\,
=\,
\frac{S_{l,0}}{(l+1)^{0}}\,
+\,
\Big(
l\,\delta_l
+
4\,l
+
1
\Big)\,
\sum_{j=1}^k\,
\frac{1}{(l+1)^j}\,
-\,
\delta_l\,
\sum_{j=1}^k\,
\frac{j}{(l+1)^{j}}.
\]

Now, applying the following two elementary identities:
\[
\aligned
\sum_{j=1}^k\,
\frac{1}{(l+1)^j}\,
&
=\,
\frac{1}{l}\,
\Big(
1
-
\frac{1}{(l+1)^k}
\Big),
\\
\sum_{j=1}^k\,
\frac{j}{(l+1)^{j}}\,
&
=\,
\frac{l+1}{l^2}\,
\Big(
1
+
\frac{k}{(l+1)^{k+1}}-\frac{1+k}{(l+1)^k}
\Big),
\endaligned
\]
and recalling~\thetag{\ref{mu_0>?'}}:
\[
S_{l,0}\,
=\,
\mu_{l,0}\,
=\,
l\,
\delta_l
+
4\,l
+
1,
\]
we obtain:
\[
\frac{S_{l,k}}{(l+1)^k}\,
=\,
l\,
\delta_l
+
4\,l
+
1\,
+\,
\Big(
l\,\delta_l
+
4\,l+1
\Big)\,
\frac{1}{l}\,
\Big(
1
-
\frac{1}{(l+1)^k}
\Big)\,
-\,
\delta_l\,
\frac{l+1}{l^2}\,
\Big(
1
+
\frac{k}{(l+1)^{k+1}}-\frac{1+k}{(l+1)^k}
\Big).
\]
Next,
multiplying by $(l+1)^k$ on both sides, we get:
\begin{equation}
\label{S_(l,k)=?}
\aligned
S_{l,k}\,
&
=\,
\Big(
l\,\delta_l
+
4\,l+1
\Big)\,
\Big(
\underline{
(l+1)^k
+
\frac{(l+1)^k}{l}
}
-
\frac{1}{l}
\Big)
-
\frac{\delta_l}{l^2}\,
\Big(
(l+1)^{k+1}
+
k
-
(1+k)\,(l+1)
\Big)
\\
&
=\,
\Big(
l\,\delta_l
+
4\,l+1
\Big)\,
\Big(
\underline{
\frac{(l+1)^{k+1}}{l}
}
-
\frac{1}{l}
\Big)
-
\frac{\delta_l}{l^2}\,
\Big(
(l+1)^{k+1}
+
k
-
(1+k)\,(l+1)
\Big).
\endaligned
\end{equation}
Recalling~\thetag{\ref{delta_l=?}}, we have:
\begin{align}
\delta_{l+1}\,
&
=\,
l\,\mu_{l,l}
\notag
\\
\explain{use~\thetag{\ref{mu_k>?'}} for $k=l$}
\qquad
&
=\,
\underline{
l\,
\Big(
\sum_{j=0}^{l-1}\,
l\,\mu_{l,j}
}
+
4\,l
+
1
\Big)
\notag
\\
\explain{use~\thetag{\ref{S_(l,k) definion}} for $k=l-1$}
\qquad
&
=\,
\underline{
l^2\,
S_{l,l-1}
}\,
+\,
l\,
(
4\,l
+
1
)
\notag
\\
\explain{use~\thetag{\ref{S_(l,k)=?}} for $k=l-1$}
\qquad
&
=\,
\Big(
l\,\delta_l
+
4\,l
+
1
\Big)\,
\Big(
l\,
(l+1)^{l}
-
l
\Big)
-
\delta_l\,
\Big(
(l+1)^{l}
+
l-1
-
l\,(l+1)
\Big)\,
+\,
l\,
(
4\,l
+
1
)
\notag
\\
\label{delta_(l+1)=?}
&
=\,
\delta_l\,
\underbrace{
\Big(
l^2\,
(l+1)^l
-
(l+1)^l
+
1
\Big)}_{\geqslant\,0}\,
+\,
(
4\,l
+
1
)\,
l\,
(l+1)^{l}.
\end{align}
Throwing away the first positive part, we receive the estimate:
\begin{equation}
\label{estimate of delta_(c+r+2) >= ?}
\delta_{l+1}\,
>\,
(
4\,l
+
1
)\,
l\,
(l+1)^{l}.
\end{equation}
Therefore, for all $l\geqslant c+r+2$, we have the estimate of~\thetag{\ref{delta_(l+1)=?}}:
\[
\aligned
\delta_{l+1}
&
=\,
l^2\,
(l+1)^l\,
\delta_l\,
-\,
\Big(
(l+1)^l
-
1
\Big)\,
\delta_l\,
+\,
(
4\,l
+
1
)\,
l\,
(l+1)^{l}
\\
\explain{use~\thetag{\ref{estimate of delta_(c+r+2) >= ?}}}
\qquad
&
<\,
l^2\,
(l+1)^l\,
\delta_l\,
-\,
\Big(
(l+1)^l
-
1
\Big)\,
\big(
4\,(l-1)
+
1
\big)\,
(l-1)\,
l^{l-1}\,
+\,
(
4\,l
+
1
)\,
l\,
(l+1)^{l}
\\
&
=\,
l^2\,
(l+1)^l\,
\delta_l\,
-\,
4\,
\underline{
\left[
\Big(
(l+1)^l
-
1
\Big)\,
(l-1)^2\,
l^{l-1}
-
l^2\,(l+1)^l
\right]}_{1}\,
\\
& \ \ \ \ \ \ \ \ \ \ \ \ \ \ \
-\,
\underline{
\left[
\Big(
(l+1)^l
-
1
\Big)\,
(l-1)\,
l^{l-1}\,
-\,
l\,(l+1)^l
\right]
}_{2}.
\endaligned
\]
Since $2c+r\geqslant N\geqslant 1$, $c,r$ cannot be both 
zero, hence $l\geqslant c+r+2\geqslant 3$ above, thus we may realize that the first underlined bracket is positive:
\[
\aligned
\left[
\Big(
(l+1)^l
-
1
\Big)\,
(l-1)^2\,
l^{l-1}
-
l^2\,(l+1)^l
\right]\,
&
=\,
l^2\,(l+1)^l\,
\left[
\Big(
1
-
\frac{1}{(l+1)^l}
\Big)\,
(l-1)^2\,
l^{l-3}
-
1
\right]
\\
&
\geqslant\,
l^2\,(l+1)^l\,
\left[
\Big(
1
-
\frac{1}{(3+1)^3}
\Big)\,
(3-1)^2\,
3^{3-3}
-
1
\right]
\\
&
>\,
0,
\endaligned
\]
and that the second underlined bracket is also positive:
\[
\aligned
\left[
\Big(
(l+1)^l
-
1
\Big)\,
(l-1)\,
l^{l-1}\,
-\,
l\,(l+1)^l
\right]\,
&
=\,
l\,(l+1)^l
\left[
\Big(
1
-
\frac{1}{(l+1)^l}
\Big)\,
(l-1)\,
l^{l-2}\,
-\,
1
\right]
\\
&
\geqslant\,
l\,(l+1)^l
\left[
\Big(
1
-
\frac{1}{(3+1)^3}
\Big)\,
(3-1)\,
3^{3-2}\,
-\,
1
\right]
\\
&
>\,
0.
\endaligned
\]
Consequently, we have the neat estimate suitable for the induction:
\begin{equation}
\label{estimate delta_(l+1) < ...}
\delta_{l+1}\,
\leqslant\,
l^2\,
(l+1)^l\,
\delta_l
\qquad
{\scriptstyle
(l\,=\,c+r+2\cdots\,N-1)
},
\end{equation}
which for convenience, we may assume to be satisfied for 
$l=N$ by just defining $\delta_{N+1}:=N\mu_{N,N}$.

In fact, using these estimates iteratively, we may proceed as follows:
\begin{align}
(N+1)\,
\mu_{N,N}\,
\notag
&
=\,
\frac{N+1}{N}\,
\underbrace{
N\,
\mu_{N,N}
}_{=\,\delta_{N+1}}
\notag
\\
\label{d=(N+1)/N delta...}
&
=\,
\frac{N+1}{N}\,
\delta_{N+1}
\\
\explain{use~\thetag{\ref{estimate delta_(l+1) < ...}}}
\qquad
&
<\,
\frac{N+1}{N}\,
\delta_{c+r+2}\,
\prod_{l=c+r+2}^{N}\,
l^2\,(l+1)^l.
\notag
\end{align}
Noting that \thetag{\ref{delta_(l+1)=?}} yields:
\[
\aligned
\delta_{c+r+2}\,
&
=\,
\left[
\delta_l\,
\Big(
l^2\,
(l+1)^l
-
(l+1)^l
+
1
\Big)\,
+\,
(4\,l+1)\,
l\,
(l+1)^{l}
\right]\,
\bigg\vert_{l=c+r+1}
\\
\explain{recall $\delta_{c+r+1}=2$}
\qquad
&
<\,
6\,l^2\,
(l+1)^l\,
\bigg\vert_{l=c+r+1},
\endaligned
\]
thus the above two estimates yield:
\begin{equation}
\label{d<... rough estimate}
\aligned
(N+1)\,
\mu_{N,N}\,
&
<\,
\frac{N+1}{N}\,
6\,
\prod_{l=c+r+1}^N\,
l^2\,(l+1)^l.
\endaligned
\end{equation}
For the convenience of later integration, 
we prefer the term $(l+1)^{l+1}$ to $(l+1)^l$, therefore we firstly transform: 
\[
\aligned
\prod_{l=c+r+1}^{N}\,
l^2\,(l+1)^l\,
&
=\,
\prod_{l=c+r+1}^{N}\,
\frac{l}{l+1}\,
l\,(l+1)^{l+1}\,
\\
&
=\,
\prod_{l=c+r+1}^{N}\,\frac{l}{l+1}\,
\prod_{l=c+r+1}^{N}\,
l\,
\prod_{l=c+r+1}^{N}\,(l+1)^{l+1}\,
\\
&
=\,
\frac{c+r+1}{N+1}\,
\prod_{l=c+r+1}^{N}\,
l\,
\prod_{l=c+r+1}^{N}\,(l+1)^{l+1},
\endaligned
\]
whence~\thetag{\ref{d<... rough estimate}} becomes:
\begin{equation}
\label{d<....+... rough estimate}
\aligned
(N+1)\,
\mu_{N,N}\,
&
<\,
6\,
\frac{c+r+1}{N}\,
\prod_{l=c+r+1}^{N}\,
l\,
\prod_{l=c+r+1}^{N}\,(l+1)^{l+1}
\\
\explain{recall $c+r\leqslant N-1$}
\qquad
&
\leqslant\,
6\,
\prod_{l=c+r+1}^{N}\,
l\,
\prod_{l=c+r+1}^{N}\,(l+1)^{l+1}.
\endaligned
\end{equation}
Now, we estimate the dominant term:
\[
\prod_{l=c+r+1}^{N}\,
l\,
\prod_{l=c+r+1}^{N}\,(l+1)^{l+1}.
\]
as follows.

Remembering $2c+r\geqslant N$, we receive:
\[
c
+
r\,
\geqslant\,
(2c+r)/2\,
\geqslant\, 
N/2,
\]
and hence for $N\geqslant 2$ we obtain:
\[
\aligned
\ln
\prod_{l=c+r+1}^{N}\,
l\,
&
=\,
\ln\,N\,
+\,
\sum_{l=c+r+1}^{N-1}\ln\,l\,
\\
&
<\,
\ln\,N\,
+\,
\int_{c+r+1}^{N}\,
\ln\,x\,\,
dx
\\
&
\leqslant\,
\ln\,N\,
+\,
\int_{N/2+1}^{N}\,
\ln\,x\,\,
dx
\\
&
=\,
\ln\,N\,
+\,
(x\,\ln x-x)
\bigg\vert
_{N/2+1}^{N}.
\endaligned
\]
Similarly, when $N\geqslant 4$, noting that $N\geqslant N/2+2$, we get the estimate:
\[
\aligned
\ln
\prod_{l=c+r+1}^{N}\,(l+1)^{l+1}\,
&
=\,
(N+1)\,\ln\,(N+1)\,
+\,
N\,\ln\,N\,
+\,
\sum_{l=c+r+2}^{N-1}\,
l\,
\ln\,
l
\\
&
\leqslant\,
(N+1)\,\ln\,(N+1)\,
+\,
N\,\ln\,N\,
+\,
\int_{N/2+2}^{N}\,
x\,
\ln\,
x\,\,
dx
\\
&
=\,
(N+1)\,\ln\,(N+1)\,
+\,
N\,\ln\,N\,
+\,
\Big(
\frac{1}{2}\,
x^2\,
\ln\,x\,
-\,
\frac{x^2}{4}
\Big)
\bigg\vert
_{N/2+2}^{N}.
\endaligned
\]
Summing the above two estimates,
for $N\geqslant 4$ we receive:
\begin{align}
&\ \ \ \ \
\ln
\prod_{l=c+r+1}^{N}\,
l\,
+\,
\ln
\prod_{l=c+r+1}^{N}\,(l+1)^{l+1}\,
\notag
\\
&
\leqslant\,
\ln\,N\,
+\,
(x\,\ln x-x)
\bigg\vert
_{N/2+1}^{N}\,
+\,
(N+1)\,\ln\,(N+1)\,
+\,
N\,\ln\,N\,
+\,
\Big(
\frac{1}{2}\,
x^2\,
\ln\,x\,
-\,
\frac{x^2}{4}
\Big)
\bigg\vert
_{N/2+2}^{N}
\notag
\\
&
=\,
\frac{1}{2}\,
N^2\,
\ln\,N\,
-\,
\frac{1}{2}\,
(N/2+2)^2\,
\ln\,(N/2+2)\,
-\,
\Big(
\frac{3}{16}\,N^2
-
2
\Big)\,
-\,
(N/2+1)\,\ln\,(N/2+1)\,
+\,
\notag
\\
&\qquad\qquad\qquad\!\!
+\,
(N+1)\,
\ln\,(N+1)\,
+\,
(2N+1)\,\ln\,N\,
\label{long trash terms}
\\
&
\label{3/8 N^2 estimate}
=\,
\frac{3}{8}\,
N^2\,
\ln\,N\,
-\,
O\,(N^2),
\qquad
\text{as}\,\,
N\rightarrow\infty.
\end{align}

In order to have a neat lower bound, we would like to have:
\begin{equation}
\label{desired bound}
(N+1)\,
\mu_{N,N}\,
\leqslant\,
N^{N^2/2}
-
1.
\end{equation}

In fact, using the estimates~\thetag{\ref{d<....+... rough estimate}}, \thetag{\ref{long trash terms}}, when $N\geqslant 48$, we can show by hand that~\thetag{\ref{desired bound}} holds true.
For $N=14\cdots 47$, we can use a mathematical software `Maple' to
 check the above estimate.
Finally, for $N=3\cdots 13$, we ask `Maple' to compute $\delta_{N+1}$ explicitly, and thereby, thanks to~\thetag{\ref{d=(N+1)/N delta...}}, we again prove the estimate~\thetag{\ref{desired bound}}. 

\section{\bf Some Improvements of MCM}
\label{The construction of hypersurfaces revisit and better lower bounds}
\subsection{General core codimension formulas}
In order to lower the degree bound $\texttt{d}_0$ of MCM, we will
modify the hypersurface constructions. Of course, we would like to reduce the number of moving coefficient terms, and this will be based on
the General Core Lemma~\ref{sharp-Core-lemma-of-MCM} below.

For every integers $p\geqslant q\geqslant 2$, for every integer $0\leqslant \ell\leqslant q$,  we first estimate the codimension ${}_{\ell}C_{p,q}$ of the algebraic variety:
\[
{}_{\ell}{\bf X}_{p,q}\,
\subset\,
{\sf Mat}_{p\times 2q}(\mathbb K)
\]
which consists of $p\times 2q$ matrices 
${\sf X}_{p,q}=(\alpha_1,\dots,\alpha_{q},\beta_1,\dots,\beta_{q})$ such that:
\begin{itemize}
\smallskip
\item[{\bf (i)}]
the first $q$ column vectors have rank: 
\[
\rank_{\mathbb{K}}\,
\big\{
\alpha_1,\dots,\alpha_{q}
\big\}
\leqslant 
\ell;
\]

\smallskip
\item[{\bf (ii)}] 
for every index $\nu=1\cdots q$, replacing $\alpha_{\nu}$ with $\alpha_{\nu}+(\beta_{1}+\cdots+\beta_{q})$ 
in the collection of column vectors
$\{\alpha_1,\dots,\alpha_q\}$, there holds the rank inequality:
\[
\rank_{\mathbb{K}}\,
\big\{
\alpha_1,
\dots,
\widehat{\alpha_\nu},
\dots,
\alpha_q,
\alpha_{\nu}+(\beta_{1}+\cdots+\beta_{q})
\big\}
\leqslant 
p-1;
\]

\smallskip
\item[{\bf (iii)}] 
for every integer $\tau=1\cdots q-1$, for every index $\rho=\tau+1 \cdots q$, replacing $\alpha_{\rho}$ with $\alpha_{\rho}+(\beta_{\tau+1}+\cdots+\beta_{q})$ 
in the collection of column vectors
$\{\alpha_1+\beta_1,\dots,\alpha_\tau+\beta_\tau,
\alpha_{\tau+1},\dots,\alpha_\rho,\dots,\alpha_{q}\}$, there holds the rank inequality:
\[
\rank_{\mathbb{K}}\,
\big\{
\alpha_1+\beta_1,\dots,\alpha_\tau+\beta_\tau,
\alpha_{\tau+1},\dots,\widehat{\alpha_\rho},\dots,\alpha_{q},
\alpha_\rho+(\beta_{\tau+1}+\cdots+\beta_{q})
\big\}
\leqslant 
q-1.
\]
\end{itemize}
\smallskip

Repeating the same reasoning as in Section~\ref{section: The engine of MCM},
we may proceed as follows.
Firstly, here is a very analogue of Proposition~\ref{Proposition: 0_C_p = p^2+1}:

\begin{Proposition}
For every integers $p\geqslant q\geqslant 2$, the codimension value ${}_\ell C_{p,q}$ for $\ell=0$ is:
\[
{}_0 C_{p,q}
=
p\,q
+
p
-
q
+
1.
\eqno
\qed
\]
\end{Proposition}

Next, we obtain an analogue of Proposition~\ref{r_C_p^0=?}:

\begin{Proposition}
For every integers $p\geqslant q\geqslant 2$, the codimensions ${}_\ell C_{p,q}^0$ of the algebraic varieties:
\[
\{
\alpha_1
+
\beta_1
=
{\bf 0}
\}\,
\cap\,
{}_{\ell}{\bf X}_{p,q}\,\,
\subset\,\,
{\sf Mat}_{p\times 2q}(\mathbb K)
\]
read according to the values of $\ell$ as:
\[
{}_\ell C_{p,q}^0
=
\begin{cases}
p
+
\min\,
\{
2\,(p-q+1),\,
p
\}
\qquad
&
{\scriptstyle(\ell\,=\,q)},
\medskip
\\
p
+
2(p-q+1)
\qquad
&
{\scriptstyle(\ell\,=\,q-1)},
\medskip
\\
p
+
(p-\ell)\,
(q-\ell)
\qquad
&
{\scriptstyle(\ell\,=\,0\,\cdots\,q-2)}.
\end{cases}
\] 
\end{Proposition}
The last two lines are easy to obtain, while the first line
is a consequence of Lemma~\ref{lemma:rank(v_1,...,v_p,w) <= p-1}.
\qed

\smallskip

Now, we deduce the analogue of Proposition~\ref{Codimension-induction-formulas}:

\begin{Proposition}[{\bf General Codimension Induction Formulas}]
{\bf (i)}\,
For every positive integers $p\geqslant q\geqslant 2$,
for $\ell=q$, the codimension value ${}_{q} C_{p,q}$ 
satisfies:
\[
{}_{q} C_{p,q}\,
=\,
\min\,
\big\{
p,\ \
{}_{q-1} C_{p,q}
\big\}.
\]

\smallskip
\noindent
{\bf (ii)}\,
For every positive integers $p\geqslant q\geqslant 3$,
for $\ell=q-1$, the codimension value ${}_{\ell} C_{p,q}$ 
satisfies:
\[
{}_{q-1} C_{p,q}\,
\geqslant\,
\min\,
\big\{
{}_{q-1} C_{p,q}^0,
\ \
{}_{q-1} C_{p-1,q-1}
+
p-q
+
2,
\ 
{}_{q-2} C_{p-1,q-1}
+
1,
\ \
{}_{q-3} C_{p-1,q-1}
\big\}.
\]

\smallskip
\noindent
{\bf (iii)}\,
For all positive integers $p\geqslant q\geqslant 3$, for all integers $\ell=1\cdots q-2$, the codimension values ${}_{\ell} C_{p,q}$ 
satisfy:
\[
{}_{\ell} C_{p,q}\,
\geqslant\,
\min\,
\big\{
{}_{\ell} C_{p,q}^0, \ \
{}_{\ell} C_{p-1,q-1}
+
(p-\ell)
+(q-\ell)
-1,\ \
{}_{\ell-1} C_{p-1,q-1}
+
(q-\ell),\ \
{}_{\ell-2} C_{p-1,q-1}
\big\}.
\eqno
\qed
\]
\end{Proposition}

Similar to Proposition~\ref{Proposition: initial values p=2}, we have:

\begin{Proposition}
For the initial cases $p\geqslant q=2$, 
there hold the codimension values:
\[
{}_0 C_{p,2}
=
3p
-
1,
\qquad
{}_1 C_{p,2}
=
2p-1,
\qquad
{}_2 C_{p,2}
=
p.
\eqno
\qed
\] 
\end{Proposition}

Finally, by the same induction proof as in Proposition~\ref{Proposition: core codimension formulas}, we get: 

\begin{Proposition}[{\bf General Core Codimension Formulas}]
\label{Proposition: General Core Codimension Formulas}
For all integers $p\geqslant q\geqslant 2$, there hold the codimension estimates:
\[
{}_{\ell}C_{p,q}\,
\geqslant\,
(p-\ell)\,
(q-\ell)
+
p
-
q
+
l
+
1
\qquad
{\scriptstyle(\ell\,=\,0\,\cdots\,q-1)}\ ,
\]
and the core codimension identity:
\[
{}_{q}C_{p,q}
=
p.
\eqno
\qed
\]
\end{Proposition}

Actually, we could prove that the above estimates are identities, yet it is not really necessary.

\subsection{General Core Lemma}
Similar to~\thetag{\ref{M_(2c)^N}}, for every integer $k=1\cdots N-1$, we introduce the algebraic subvariety:
\[
\mathscr{M}_{2c+r}^{N,\,k}\ \
\subset\ \
{\sf Mat}_{(2c+r)\times (N+1+k+1)}(\mathbb K)
\]
consisting of all $(c+r+c)\times (N+1+k+1)$ matrices
$(\alpha_0\mid\alpha_1\mid\dots\mid\alpha_{N}\mid\beta_0\mid\beta_1\mid
\dots\mid\beta_{k})$ such that:
\begin{itemize}
\smallskip\item[{\bf (i)}]
the sum of these $(N+1+k+1)$ colums is zero:
\begin{equation}
\label{sum of N+k+2 columns vanishes}
\alpha_0
+
\alpha_1
+
\cdots
+
\alpha_{N}
+
\beta_0
+
\beta_1
+\cdots
+\beta_{k}
=
\mathbf{0};
\end{equation}

\smallskip\item[{\bf (ii)}]
for every index $\nu=0\cdots k$, replacing $\alpha_{\nu}$ with $\alpha_{\nu}+(\beta_0+\beta_1+\cdots+\beta_{k})$ in the collection of column vectors 
$\{\alpha_0,\alpha_1,\dots,\alpha_{N}\}$, there holds the rank inequality: 
\[
\rank_{\mathbb{K}}\,
\big\{
\alpha_0,\dots,\widehat{\alpha_{\nu}},\dots,\alpha_{N},\alpha_{\nu}+(\beta_0+\beta_1+\cdots+\beta_{k})
\big\}
\leqslant 
N-1;
\]

\smallskip\item[{\bf (iii)}] 
for every integer $\tau=0\cdots k-1$, for every index $\rho=\tau+1 \cdots k$, replacing $\alpha_{\rho}$ with $\alpha_{\rho}+(\beta_{\tau+1}+\cdots+\beta_{k})$ 
in the collection of column vectors
$\{\alpha_0+\beta_0,\dots,\alpha_\tau+\beta_\tau,
\alpha_{\tau+1},\dots,\alpha_\rho,\dots,\alpha_{N}\}$, there holds the rank inequality:
\[
\aligned
\ \ \ \ \ \ \ \ \ \
\rank_{\mathbb{K}}\,
\big\{
\alpha_0+\beta_0,\alpha_1+\beta_1,\dots,\alpha_\tau+\beta_\tau,
\alpha_{\tau+1},\dots,\widehat{\alpha_\rho},\dots,\alpha_{N},
\alpha_\rho+(\beta_{\tau+1}+\cdots+\beta_{k})
\big\}
\leqslant 
N-1.
\endaligned
\]
\end{itemize}
\smallskip

\begin{Lemma}[{\bf Sharp Core Lemma of MCM}]
\label{sharp-Core-lemma-of-MCM}
For every positive integers $N\geqslant 3$, for every integers $c,r\geqslant 0$ with $2c+r\geqslant N$, for every integer $k=1\cdots N-1$, there holds the codimension estimate:
\[
\aligned
\cdim\,
\mathscr{M}_{2c+r}^{N,\,k}\,
&
\geqslant\,
{}_{k+1}C_{2c+r-N+k+1,k+1}
+
(2c+r)
\\
&
\geqslant\,
2\,
(2c+r)
-
N
+
k
+
1.
\endaligned
\]
\end{Lemma}

The term $(2c+r)$ comes from~\thetag{\ref{sum of N+k+2 columns vanishes}}. When $k=N-1$, there is nothing to prove. When
$k<N-1$, noting that all matrices in
{\bf(ii)} and {\bf (iii)} have the same last column
$\alpha_N$, we may do Gaussian eliminations with 
respect to this column, and then by much the same argument as before,
we receive the estimate.  
\qed
\smallskip

Actually, these two estimates are identities.

\subsection{Minimum necessary number of moving coefficient terms}
Firstly, letting:
\[
\cdim\,
\mathscr{M}_{2c+r}^{N,\,k}\,
\geqslant\,
\dim\,
\oP(\mathrm{T}_{\mathbb{P}^N})\,
=\,
2N
-
1,
\]
we receive the lower bound:
\[
k\,
\geqslant\,
3N
-
2\,
(2c+r)
-
2,
\]
which indicates that at the step $N$, the least number of moving coefficient terms, if necessary, 
should be:
\[
3N
-
2\,
(2c+r)
-
2
+
1.
\]
When 
$
3N
-
2\,
(2c+r)
-
2\,
\leqslant\,
0
$,
no moving coefficient terms are needed, thanks to the following:

\begin{Lemma}
{\bf(Elementary Core Lemma)}
Let:
\[
\mathscr{M}_{2c+r}^{N,\,-1}\ \
\subset\ \
{\sf Mat}_{(2c+r)\times (N+1)}(\mathbb K)
\]
consist of all $(c+r+c)\times (N+1)$ matrices
$(\alpha_0\mid\alpha_1\mid\dots\mid\alpha_{N})$ such that:
\begin{itemize}
\smallskip\item[{\bf (i)}]
the sum of these $(N+1)$ colums is zero:
\begin{equation}
\label{sum of N+k columns vanishes}
\alpha_0
+
\alpha_1
+
\cdots
+
\alpha_{N}
=
\mathbf{0};
\end{equation}

\smallskip\item[{\bf (ii)}]
there holds the rank inequality: 
\[
\rank_{\mathbb{K}}\,
\big\{
\alpha_0,\dots,\widehat{\alpha_{\nu}},\dots,\alpha_{N}
\big\}
\leqslant 
N-1
\qquad
{\scriptstyle(\nu\,=\,0\,\cdots\,N)}.
\]
\end{itemize}
Then one has the codimension identity:
\[
\cdim\,
\mathscr{M}_{2c+r}^{N,\,-1}\,
=\,
2(2c+r)
-
N
+
1.
\eqno
\qed
\]
\end{Lemma}

Next, for $\eta=1\cdots n-1$, in the step $N-\eta$, letting:
\[
\cdim\,
\mathscr{M}_{2c+r}^{N-\eta,\,k}\,
\geqslant\,
\dim\,
{}_{v_1,\dots,v_\eta}
\oP
(
\mathrm{T}_{\mathbb{P}^N}
)\,
=\,
2N
-
\eta
-
1,
\]
we receive:
\[
2\,
(2c+r)
-
(N-\eta)
+
k
+
1\,
\geqslant\,
2N
-
\eta
-
1,
\]
that is:
\[
k
\geqslant\,
3N
-
2\,
(2c+r)
-
2
-
2\eta,
\]
which indicates that, at the step $N-\eta$, the least number of moving 
coefficient terms, if necessary, 
should be:
\[
3N
-
2\,
(2c+r)
-
2
-
2\eta
+
1.
\]
When 
$
3N
-
2\,
(2c+r)
-
2
-
2\eta\,
\leqslant\,
0
$,
no moving coefficient terms are needed, thanks to the Elementary Core Lemma.

\subsection{Improved Algorithm of MCM}
When 
$3N
-
2\,
(2c+r)
-
2
>
0$, 
in order to lower the degrees,
we improve the hypersurface equations~\thetag{\ref{F_i-moving-coefficient-method-full-strenghth}} as follows.

Firstly, when 
$3N
-
2\,
(2c+r)
-
2
=
2p$ is even,
the following hypersurface equations are suitable for MCM:
\begin{equation}
\aligned
F_i
\,
&
=
\,
\sum_{j=0}^N\,
A_i^j\,
z_j^{d}
\,
+
\,
\sum_{\eta=0}^{p-1}\,
\sum_{0\leqslant j_0<\cdots<j_{N-\eta}\leqslant N}\,
\sum_{k=0}^{2p-2\eta}\,
\\
&\qquad
M_i^{j_0,\dots,j_{N-\eta};j_k}\,
z_{j_0}^{\mu_{N-\eta,k}}
\cdots
\widehat{z_{j_k}^{\mu_{N-\eta,k}}}
\cdots
z_{j_{2p-2\eta}}^{\mu_{N-\eta,k}}\,
z_{j_k}^{d-(2p-2\eta)\mu_{N-\eta,k}
-
2(N+\eta-2p)}\,
z_{j_{2p-2\eta+1}}^2
\cdots
z_{j_{N-\eta}}^2.
\endaligned
\end{equation}

Secondly, when 
$3N
-
2\,
(2c+r)
-
2
=
2p+1$ is odd,
the following hypersurface equations are suitable for MCM:
\begin{equation}
\aligned
F_i
\,
&
=
\,
\sum_{j=0}^N\,
A_i^j\,
z_j^{d}
\,
+
\,
\sum_{\eta=0}^{p}\,
\sum_{0\leqslant j_0<\cdots<j_{N-\eta}\leqslant N}\,
\sum_{k=0}^{2p+1-2\eta}\,
\\
&\qquad\qquad
M_i^{j_0,\dots,j_{N-\eta};j_k}\,
z_{j_0}^{\mu_{N-\eta,k}}
\cdots
\widehat{z_{j_k}^{\mu_{N-\eta,k}}}
\cdots
z_{j_{2p+1-2\eta}}^{\mu_{N-\eta,k}}\,
z_{j_k}^{d-(2p+1-2\eta)\mu_{N-\eta,k}
-
2(N+\eta-2p-1)}\,
z_{j_{2p-2\eta+2}}^2
\cdots
z_{j_{N-\eta}}^2.
\endaligned
\end{equation}

Of course, all integers $\mu_{\bullet,\bullet}$ and the degree $d$ are to be determined by some {\sl improved Algorithm}, so that
all the obtained symmetric forms are negatively
twisted. And then we may estimate the lower bound $\texttt{d}_0$ accordingly. We leave this standard process to the interested reader.

\subsection{Why is the lower degree bound $\texttt{d}_0$ so large in MCM}
\label{Final words: why large degrees?}
Because we could not enter the intrinsic difficulties, firstly of solving
some huge linear systems to obtain sufficiently many (negatively twisted, large  degree) symmetric differential forms (see \cite[Theorem 2.7]{Brotbek-2014-arxiv}), and secondly of 
proving that the obtained symmetric forms have discrete base locus.
What we have done is only focusing on the extrinsic negatively twisted symmetric forms with degrees $\leqslant n$, obtained by some minors
of the hypersurface equations\big/differentials matrix.

Our tool is coarse, based on some robust extrinsic geometric\big/algebraic structures, yet our goal is delicate, to certify the conjectured intrinsic ampleness. So a large lower degree bound $\texttt{d}_0\gg 1$ is a price we need to pay.

\section{\bf Uniform Very-Ampleness of 
$\mathsf{Sym}^{\kappa}\Omega_X$}
\label{Proof of the very ampleness theorem}

\subsection{A reminder}
In~\cite{Fujita-1987}, Fujita proposed the famous:

\begin{Conjecture}{\bf (Fujita)}
Let $M$ be an $n$-dimensional complex manifold with canonical line bundle $\mathcal{K}$. If
$\mathcal{L}$ is any positive holomorphic line bundle on $M$,
then:

\begin{itemize}
\item[\bf (i)]
for every integer $m\geqslant n+1$, the line bundle
$\mathcal{L}^{\otimes m}\otimes\mathcal{K}$
should be globally generated;

\smallskip\item[\bf (ii)]
for every integer
$m \geqslant n + 2$,
the line bundle
$\mathcal{L}^{\otimes m}\otimes\mathcal{K}$
should be very ample.
\end{itemize}
\end{Conjecture}

Recall that, given a complex manifold $X$ having ample cotangent bundle $\Omega_X$, the projectivized tangent bundle 
$\mathbb{P}(\mathrm{T}_X)$ is equipped with the ample Serre line bundle
$\mathcal{O}_{\mathbb{P}(\mathrm{T}_X)}(1)$.
Denoting $n:=\dim\,X$, one has:
\[
\dim\,\mathbb{P}(\mathrm{T}_X)\,
=\,
2n-1.
\]
Anticipating, we will show in Corollary~\ref{two formulas on canonical bundles} below that the canonical bundle
of $\mathbb{P}(\mathrm{T}_X)$ is:
\[
\mathcal{K}_{\mathbb{P}(\mathrm{T}_X)}\
\cong\
\mathcal{O}_{\mathbb{P}(\mathrm{T}_X)}(-\,n)\,
\otimes\,
\pi^*\,
\mathcal{K_X}^{\otimes\,2},
\]
where $\pi\colon \mathbb{P}(\mathrm{T}_X) \rightarrow X$ is the canonical projection.
Thus, for the complex manifold $\mathbb{P}(\mathrm{T}_X)$ and the ample Serre line bundle
$\mathcal{O}_{\mathbb{P}(\mathrm{T}_X)}(1)$,
the Fujita Conjecture implies:

\begin{itemize}
\item[\bf (i)]
for every integer
$m \geqslant 2n$, the line bundle
$\mathcal{O}_{\mathbb{P}(\mathrm{T}_X)}(m-n)\,
\otimes\,
\pi^*\,
\mathcal{K_X}^{\otimes\,2}$
is globally generated;

\item[\bf (ii)]
\smallskip

for every integer
$m \geqslant 2n+1$,
the line bundle
$\mathcal{O}_{\mathbb{P}(\mathrm{T}_X)}(m-n)\,
\otimes\,
\pi^*\,
\mathcal{K_X}^{\otimes\,2}$
is very ample.
\end{itemize}

\smallskip

In other words, we receive the following by-products of the Fujita Conjecture.

\smallskip
\noindent
{\bf A Consequence of the Fujita Conjecture.}
{\em
For any $n$-dimensional complex manifold X having ample cotangent bundle $\Omega_X$, there holds:

\begin{itemize}

\item[\bf (i)]
for every integer
$m \geqslant n$, the twisted $m$-symmetric cotangent bundle 
$\mathsf{Sym}^{m}\Omega_X\otimes\mathcal{K}_X^{\otimes\,2}$
is globally generated;

\smallskip\item[\bf (ii)]

for every integer
$m \geqslant n+1$,
the twisted $m$-symmetric cotangent bundle 
$\mathsf{Sym}^{m}\Omega_X\otimes\mathcal{K}_X^{\otimes\,2}$
is very ample.

\end{itemize}
}

\subsection{The canonical bundle of a projectivized vector bundle}
In this subsection, we recall some classical results in algebraic geometry.

Let $X$ be an $n$-dimensional complex manifold, and let $E$ be a
holomorphic vector bundle on $X$ having rank $e$.
Let $\mathbb{P}(E)$ be the projectivization of $E$. Now,
we compute its canonical bundle $\mathcal{K}_{\mathbb{P}(E)}$ as follows.

Let $\pi$ be the canonical projection:
\[
\pi
\colon\
\mathbb{P}(E)\,
\longrightarrow\,
X.
\]
First, recall the exact sequence which defines the relative tangent bundle $\mathrm{T}_{\pi}$:
\begin{equation}
\label{exact sequence of relative tangent bundle}
0\,
\longrightarrow\,
\mathrm{T}_{\pi}\,
\longrightarrow\,
\mathrm{T}_{\mathbb{P}(E)}\,
\longrightarrow\,
\pi^*\,
\mathrm{T}_{X}\,
\longrightarrow\,
0,
\end{equation}
and recall also the well known Euler exact sequence:
\begin{equation}
\label{Euler exact sequence}
0\,
\longrightarrow\,
\mathcal{O}_{\mathbb{P}(E)}\,
\longrightarrow\,
\mathcal{O}_{\mathbb{P}(E)}(1)\,
\otimes\,
\pi^*\,E\,
\longrightarrow\,
\mathrm{T}_{\pi}\,
\longrightarrow\,
0.
\end{equation}
Next, taking wedge products, the exact sequence~\thetag{\ref{exact sequence of relative tangent bundle}} yields:
\begin{equation}
\label{first wedge product identity}
\wedge^{n+e-1}\,
\mathrm{T}_{\mathbb{P}(E)}\,
\cong\,
\wedge^{e-1}\,
\mathrm{T}_{\pi}\,
\otimes\,
\pi^*
\wedge^{n}
\mathrm{T}_{X},
\end{equation}
and the Euler exact sequence~\thetag{\ref{Euler exact sequence}} yields:
\begin{equation}
\label{second wedge product identity}
\mathcal{O}_{\mathbb{P}(E)}(e)\,
\otimes\,
\pi^*
\wedge^{e}
E\,
\cong\,
\mathcal{O}_{\mathbb{P}(E)}\,
\otimes\,
\wedge^{e-1}\,
\mathrm{T}_{\pi}\,
\cong\,
\wedge^{e-1}\,
\mathrm{T}_{\pi}.
\end{equation}
Thus, we may compute the canonical line bundle as:
\[
\aligned
\mathcal{K}_{\mathbb{P}(E)}\,
&
=\,
\wedge^{n+e-1}\,
\Omega_{\mathbb{P}(E)}^1\,
\\
&
\cong\,
\big(
\wedge^{n+e-1}\,
\mathrm{T}_{\mathbb{P}(E)}
\big)^{\vee}
\\
\explain{use the dual of~\thetag{\ref{first wedge product identity}}}
\qquad
&
\cong\,
\big(
\wedge^{e-1}\,
\mathrm{T}_{\pi}
\big)^{\vee}\,
\otimes\,
\big(
\pi^*
\wedge^{n}
\mathrm{T}_{X}
\big)^{\vee}
\\
\explain{use the dual of~\thetag{\ref{second wedge product identity}}}
\qquad
&
\cong\,
\big(
\mathcal{O}_{\mathbb{P}(E)}(e)\,
\big)^{\vee}\,
\otimes\,
\big(
\pi^*
\wedge^{e}
E
\big)^{\vee}\,
\otimes\,
\big(
\pi^*
\wedge^{n}
\mathrm{T}_{X}
\big)^{\vee}
\\
&
\cong\,
\mathcal{O}_{\mathbb{P}(E)}(-\,e)\,
\otimes\,
\pi^*
\wedge^{e}
E^{\vee}\,
\otimes\,
\pi^*
\wedge^{n}
\Omega_{X}^1
\\
&
\cong\,
\mathcal{O}_{\mathbb{P}(E)}(-\,e)\,
\otimes\,
\pi^*
\wedge^{e}
E^{\vee}\,
\otimes\,
\pi^*\,
\mathcal{K}_X,
\endaligned
\]
where $\mathcal{K}_X$ is the canonical line bundle of $X$.

\begin{Proposition}
\label{canonical line bundle formula of P(E)}
The canonical line bundle $\mathcal{K}_{\mathbb{P}(E)}$
of $\mathbb{P}(E)$ satisfies the formula:
\[
\mathcal{K}_{\mathbb{P}(E)}\
\cong\
\mathcal{O}_{\mathbb{P}(E)}(-\,e)\,
\otimes\,
\pi^*
\wedge^{e}
E^{\vee}\,
\otimes\,
\pi^*\,
\mathcal{K}_X.
\eqno
\qed
\]
\end{Proposition}

In applications, first, we are interested in the case where $E$ is the tangent bundle $\mathrm{T}_X$ of $X$.

\begin{Corollary}
\label{two formulas on canonical bundles}
One has the formula:
\[
\mathcal{K}_{\mathbb{P}(\mathrm{T}_X)}\
\cong\
\mathcal{O}_{\mathbb{P}(\mathrm{T}_X)}(-\,n)\,
\otimes\,
\pi^*\,
\mathcal{K_X}^{\otimes\,2}.
\eqno
\qed
\]
\end{Corollary}

More generally, we are interested in the case where $X\subset V$ for some complex manifold $V$ of dimension $n+r$, and $E=\mathrm{T}_V\big{\vert}_X$.

\begin{Corollary}
\label{two more formulas on canonical bundles}
One has:
\[
\mathcal{K}_{\mathbb{P}(\mathrm{T}_V|_X)}\
\cong\
\mathcal{O}_{\mathbb{P}(\mathrm{T}_V|_X)}(-\,n-r)\,
\otimes\,
\pi^*\,
\mathcal{K_V}\big{\vert}_X\,
\otimes\,
\pi^*\,
\mathcal{K_X}.
\eqno
\qed
\]
\end{Corollary}

In our applications, $X, V$ are some smooth complete intersections
in $\mathbb{P}_{\mathbb{C}}^N$, so their canonical line bundles $\mathcal{K}_X, \mathcal{K}_V$ have neat expressions by the following classical theorem, whose 
 proof is based on the Adjunction Formula.

\begin{Theorem}
\label{degree of complete intersection}
For a smooth complete intersection:
\[
Y\,
:=\,
D_1
\cap
\cdots
\cap
D_k\
\subset\
\mathbb{P}_{\mathbb{C}}^N
\]
with divisor degrees:
\[
\deg\,
D_i\,
=\,
d_i
\qquad
{\scriptstyle(i\,=\,1\,\cdots\,k)},
\]
the canonical line bundle $\mathcal{K}_X$ of $X$ is:
\[
\mathcal{K}_X\,
\cong\,
\mathcal{O}_X\,
\Big(
-N-1
+
\sum_{i=1}^k\,d_i
\Big).
\eqno
\qed
\]
\end{Theorem}

\subsection{Proof of the Very-Ampleness Theorem~\ref{Very-Ampleness Theorem}}
Assume for the moment that the ambient field $\mathbb{K}=\mathbb{C}$.
Recall that in our Ampleness Theorem~\ref{Main Theorem}, $V=H_1\cap\dots\cap H_{c}$ and
$X=H_1\cap\dots\cap H_{c+r}$ with $\dim_{\mathbb{C}}\,X=n=N-(c+r)$. Then the above
Corollary~\ref{two more formulas on canonical bundles} and Theorem~\ref{degree of complete intersection} imply:
\[
\mathcal{K}_{\mathbb{P}(\mathrm{T}_V|_X)}\
\cong\
\mathcal{O}_{\mathbb{P}(\mathrm{T}_V|_X)}(-\,n-r)\,
\otimes\,
\pi_2^*\,
\mathcal{O}_{\mathbb{P}_{\mathbb{K}}^N}\,
\Big(
-2\,(N+1)
+
\sum_{i=1}^{c}\,d_i
+
\sum_{i=1}^{c+r}\,d_i
\Big).
\]

Also, recalling Theorem~\ref{Main Nefness Theorem} and
Proposition~\ref{nef + epsilon implies ample},
 for generic choices of $H_1,\dots,H_{c+r}$, for any positive integers $a> b\geqslant 1$, the negatively twisted line bundle below is ample:
\[
\mathcal{O}_{\mathbb{P}(\mathrm{T}_V|_X)}(a)\,
\otimes\,
\pi_2^*\,
\mathcal{O}_{\mathbb{P}_{\mathbb{K}}^N}\,
(-b).
\]

Recall the Fujita Conjecture that, by subsequent works of Demailly, Siu et al. (cf. the survey~\cite{Demailly-1998}), it is known that $\mathcal{L}^{\otimes\,m}\otimes\mathcal{K}^{\otimes\,2}$ is very ample  for all large $m\geqslant 2+ \binom{3n+1}{n}$.
Consequently, the line bundle:
\begin{equation}
\label{first very ample}
\mathcal{O}_{\mathbb{P}(\mathrm{T}_V|_X)}
\Big(
m\,a
-
2n
-
2r
\Big)\,
\otimes\,
\pi_2^*\,
\mathcal{O}_{\mathbb{P}_{\mathbb{K}}^N}\,
\Big(
-m\,b
-4\,(N+1)
+
2\sum_{i=1}^{c}\,d_i
+
2\sum_{i=1}^{c+r}\,d_i
\Big)
\end{equation}
is very ample.

Also note that, for similar reason as the ampleness of~\thetag{\ref{what is the minimum l}}, for all large integers
$\ell\geqslant \ell_0(N)$:
\[
\mathcal{O}_{\mathbb{P}(\mathrm{T}_{\mathbb{P}_{\mathbb{K}}^{N}})}
(1)\,
\otimes\,
\pi_0^*\,
\mathcal{O}_{\mathbb{P}_{\mathbb{K}}^{N}}(\ell)
\]
is very ample. Consequently, so is:
\begin{equation}
\label{Second very ample}
\mathcal{O}_{\mathbb{P}(\mathrm{T}_V|_X)}(1)\,
\otimes\,
\pi_2^*\,
\mathcal{O}_{\mathbb{P}_{\mathbb{K}}^N}\,
(\ell).
\end{equation}

Now, recall the following two facts: 

\begin{itemize}
\smallskip
\item[\bf (A)]
if $\mathcal{O}_{\mathbb{P}(\mathrm{T}_V|_X)}(\kappa)\,
\otimes\,
\pi_2^*\,
\mathcal{O}_{\mathbb{P}_{\mathbb{K}}^N}\,
(\star)$ is very ample, then for every $\star'\geqslant \star$, 
$\mathcal{O}_{\mathbb{P}(\mathrm{T}_V|_X)}(\kappa)\,
\otimes\,
\pi_2^*\,
\mathcal{O}_{\mathbb{P}_{\mathbb{K}}^N}\,
(\star')$ is also very ample;

\smallskip
\item[\bf (B)]
the tensor product of any two very ample
line bundles remains very ample.
\end{itemize}

\smallskip
\noindent
Therefore, thanks to the very-ampleness of~\thetag{\ref{first very ample}}, \thetag{\ref{Second very ample}},
we can already obtain the very-ampleness of 
$\mathcal{O}_{\mathbb{P}(\mathrm{T}_V|_X)}(\kappa)$ for all large integers $\kappa\geqslant \kappa_0$, for some non-effective $\kappa_0$.
In other words, the restricted cotangent bundle $\mathsf{Sym}^{\kappa}\,\Omega_V\big{\vert}_X$ is very ample on $X$ for every $\kappa\geqslant \kappa_0$. But to reach an explicit $\kappa_0$, one may ask the

\medskip
\noindent
{\bf Questions.}\,
{\bf (i)} Find one explicit $\ell_0(N)$. 

\smallskip

\ \ \ \ \ \ \ \ \ \ \ \ \ \ \ 
{\bf (ii)} 
Find one explicit $\kappa_0$. 

\medskip
\noindent
{\bf Answer of (i).} The value $\ell_0(N)=3$ works.
One can check by hand that the following global sections:
\begin{equation}
\label{l=3 is OK}
z_k\,
z_j^{\ell-1}\,
d\,
\Big(
\frac{z_i}{z_j}
\Big)
\qquad
{\scriptstyle(i,\,j,\,k\,=\,0\,\cdots\,N,\,i\,\neq\,j)}
\end{equation}
guarantee the very-ampleness of
$\mathcal{O}_{\mathbb{P}(\mathrm{T}_{\mathbb{P}_{\mathbb{K}}^{N}})}
(1)\,
\otimes\,
\pi_0^*\,
\mathcal{O}_{\mathbb{P}_{\mathbb{K}}^{N}}(\ell)$. 

\smallskip
\noindent
{\bf Answer of (ii).} 
The second fact {\bf (B)} above leads us to consider the semigroup $\mathcal{G}$ of the usual Abelian group $\mathbb{Z}\oplus\mathbb{Z}$ generated by elements $(\ell_1,\ell_2)$ such that
$\mathcal{O}_{\mathbb{P}(\mathrm{T}_V|_X)}(\ell_1)\,
\otimes\,
\pi_2^*\,
\mathcal{O}_{\mathbb{P}_{\mathbb{K}}^N}\,
(\ell_2)$ is very ample.
Then, the following elements
are contained in $\mathcal{G}$, 
for all $m\geqslant 2+\binom{3n+1}{n}$:
\[
\aligned
\explain{see~\thetag{\ref{first very ample}},
$\forall\, b\geqslant 1, a\geqslant b+1$}
\qquad
&
\Big(
m\,a
-
2n
-
2r,\,
-m\,b
-4\,(N+1)
+
2\sum_{i=1}^{c}\,d_i
+
2\sum_{i=1}^{c+r}\,d_i
\Big),
\\
\explain{see~\thetag{\ref{Second very ample}},
$\ell\geqslant \ell_0(N)=3$}
\qquad
&
(1,\ell).
\endaligned
\]
Also, the first fact {\bf (A)} above says that if
$(\ell_1,\ell_2)\in \mathcal{G}$, then 
$(\ell_1,\ell_3)\in \mathcal{G}$ for all $\ell_3\geqslant \ell_2$.
Thus, Question {\bf(ii)} becomes to
find one explicit $\kappa_0$ such that
$(\kappa,0)\in \mathcal{G}$ for all $\kappa\geqslant
\kappa_0$.

Paying no attention to optimality,
taking: 
\[
b
=
1,\
a
=
2,\
m
=
-4\,(N+1)
+
2\sum_{i=1}^{c}\,d_i
+
2\sum_{i=1}^{c+r}\,d_i
+
3,
\]
we receive that
$
(m\,a-2n-2r,-3)
\in
\mathcal{G}.
$
Adding $(1,3)\in \mathcal{G}$, we receive $(m\,a-2n-2r+1,0)\in\mathcal{G}$. Now, also using $(m\,a-2n-2r,0)\in \mathcal{G}$, 
recalling Observation~\ref{Bezout Theorem}, we may take:
\[
\kappa_0\,
=\,
(m\,a-2n-2r-1)\,
(m\,a-2n-2r)\,
\leqslant\,
a^2\,m^2,
\]
or the larger neater lower bound:
\[
\kappa_0\,
=\,
16\,
\Big(
\sum_{i=1}^{c}\,d_i
+
\sum_{i=1}^{c+r}\,d_i
\Big)^2.
\]

Thus, we have proved the Very-Ampleness Theorem~\ref{Very-Ampleness Theorem} for $\mathbb{K}=\mathbb{C}$. Remembering that 
{\em very-ampleness (or not) is preserved under any base change 
obtained by ambient field extension}, and noting the field extensions 
$\mathbb{Q}\hookrightarrow \mathbb{C}$ and
$\mathbb{Q}\hookrightarrow \mathbb{K}$
for any field $\mathbb{K}$ with characteristic zero,
by some standard arguments in algebraic geometry, we conclude the proof of the Very-Ampleness Theorem~\ref{Very-Ampleness Theorem}.

When $\mathbb{K}$ has positive characteristic,
by the same arguments, we could also receive the same very-ampleness theorem provided the similar results about the Fujita
Conjecture hold over the field $\mathbb{K}$.

\end{document}

%% file: X-hat-X.pdf_t
\begin{picture}(0,0)%
\includegraphics{X-hat-X.pdf}%
\end{picture}%
\setlength{\unitlength}{4144sp}%
\begingroup\makeatletter\ifx\SetFigFont\undefined%
\gdef\SetFigFont#1#2#3#4#5{%
  \reset@font\fontsize{#1}{#2pt}%
  \fontfamily{#3}\fontseries{#4}\fontshape{#5}%
  \selectfont}%
\fi\endgroup%
\begin{picture}(2989,3686)(1684,-3017)
\put(3270,-2976){\makebox(0,0)[lb]{\smash{{\SetFigFont{8}{9.6}{\familydefault}{\mddefault}{\updefault}{\color[rgb]{0,0,0}$0$}%
}}}}
\put(3186, 12){\makebox(0,0)[lb]{\smash{{\SetFigFont{8}{9.6}{\familydefault}{\mddefault}{\updefault}{\color[rgb]{0,0,0}$\lambda z$}%
}}}}
\put(3731,-908){\makebox(0,0)[lb]{\smash{{\SetFigFont{8}{9.6}{\familydefault}{\mddefault}{\updefault}{\color[rgb]{0,0,0}\green{$\xi$}}%
}}}}
\put(1804,-1544){\makebox(0,0)[lb]{\smash{{\SetFigFont{8}{9.6}{\familydefault}{\mddefault}{\updefault}{\color[rgb]{0,0,0}$\mathbb{C}^{N+1}$}%
}}}}
\put(3224,-1258){\makebox(0,0)[lb]{\smash{{\SetFigFont{8}{9.6}{\familydefault}{\mddefault}{\updefault}{\color[rgb]{0,0,0}$z$}%
}}}}
\put(2559,-686){\makebox(0,0)[lb]{\smash{{\SetFigFont{8}{9.6}{\familydefault}{\mddefault}{\updefault}{\color[rgb]{0,0,0}\blue{$\widehat{X}$}}%
}}}}
\put(3923,510){\makebox(0,0)[lb]{\smash{{\SetFigFont{8}{9.6}{\familydefault}{\mddefault}{\updefault}{\color[rgb]{0,0,0}\green{$\lambda\xi$}}%
}}}}
\end{picture}%

%% file: I-p-0-nu.pdf_t
\begin{picture}(0,0)%
\includegraphics{I-p-0-nu.pdf}%
\end{picture}%
\setlength{\unitlength}{4144sp}%
\begingroup\makeatletter\ifx\SetFigFont\undefined%
\gdef\SetFigFont#1#2#3#4#5{%
  \reset@font\fontsize{#1}{#2pt}%
  \fontfamily{#3}\fontseries{#4}\fontshape{#5}%
  \selectfont}%
\fi\endgroup%
\begin{picture}(1324,2207)(1336,-2575)
\put(1351,-1411){\makebox(0,0)[lb]{\smash{{\SetFigFont{12}{14.4}{\familydefault}{\mddefault}{\updefault}{\color[rgb]{0,0,0}\vpb}%
}}}}
\put(1591,-518){\makebox(0,0)[lb]{\smash{{\SetFigFont{12}{14.4}{\familydefault}{\mddefault}{\updefault}{\color[rgb]{0,0,0}\hba}%
}}}}
\put(2533,-1408){\makebox(0,0)[lb]{\smash{{\SetFigFont{12}{14.4}{\familydefault}{\mddefault}{\updefault}{\color[rgb]{0,0,0}\vpc}%
}}}}
\put(2645,-1829){\makebox(0,0)[lb]{\smash{{\SetFigFont{12}{14.4}{\familydefault}{\mddefault}{\updefault}{\color[rgb]{0,0,0}\vbe}%
}}}}
\put(1971,-2166){\makebox(0,0)[lb]{\smash{{\SetFigFont{12}{14.4}{\familydefault}{\mddefault}{\updefault}{\color[rgb]{0,0,0}$1$}%
}}}}
\put(1969,-1515){\makebox(0,0)[lb]{\smash{{\SetFigFont{12}{14.4}{\familydefault}{\mddefault}{\updefault}{\color[rgb]{0,0,0}$1$}%
}}}}
\put(1583,-683){\makebox(0,0)[lb]{\smash{{\SetFigFont{12}{14.4}{\familydefault}{\mddefault}{\updefault}{\color[rgb]{0,0,0}$1$}%
}}}}
\put(2638,-982){\makebox(0,0)[lb]{\smash{{\SetFigFont{12}{14.4}{\familydefault}{\mddefault}{\updefault}{\color[rgb]{0,0,0}\vbd}%
}}}}
\put(2438,-1291){\makebox(0,0)[lb]{\smash{{\SetFigFont{12}{14.4}{\familydefault}{\mddefault}{\updefault}{\color[rgb]{0,0,0}$1$}%
}}}}
\put(2070,-2442){\makebox(0,0)[lb]{\smash{{\SetFigFont{8}{9.6}{\familydefault}{\mddefault}{\updefault}{\color[rgb]{0,0,0}\blue{$\nu$-th column}}%
}}}}
\end{picture}%

%% file: I-p-tau-rho.pdf_t
\begin{picture}(0,0)%
\includegraphics{I-p-tau-rho.pdf}%
\end{picture}%
\setlength{\unitlength}{4144sp}%
\begingroup\makeatletter\ifx\SetFigFont\undefined%
\gdef\SetFigFont#1#2#3#4#5{%
  \reset@font\fontsize{#1}{#2pt}%
  \fontfamily{#3}\fontseries{#4}\fontshape{#5}%
  \selectfont}%
\fi\endgroup%
\begin{picture}(2271,4031)(1336,-3491)
\put(2920,-3078){\makebox(0,0)[lb]{\smash{{\SetFigFont{12}{14.4}{\familydefault}{\mddefault}{\updefault}{\color[rgb]{0,0,0}$1$}%
}}}}
\put(2918,-2427){\makebox(0,0)[lb]{\smash{{\SetFigFont{12}{14.4}{\familydefault}{\mddefault}{\updefault}{\color[rgb]{0,0,0}$1$}%
}}}}
\put(1488,180){\makebox(0,0)[lb]{\smash{{\SetFigFont{12}{14.4}{\familydefault}{\mddefault}{\updefault}{\color[rgb]{0,0,0}$1$}%
}}}}
\put(2343,-428){\makebox(0,0)[lb]{\smash{{\SetFigFont{12}{14.4}{\familydefault}{\mddefault}{\updefault}{\color[rgb]{0,0,0}$1$}%
}}}}
\put(2540,-669){\makebox(0,0)[lb]{\smash{{\SetFigFont{12}{14.4}{\familydefault}{\mddefault}{\updefault}{\color[rgb]{0,0,0}$1$}%
}}}}
\put(3395,-1277){\makebox(0,0)[lb]{\smash{{\SetFigFont{12}{14.4}{\familydefault}{\mddefault}{\updefault}{\color[rgb]{0,0,0}$1$}%
}}}}
\put(1482,-1552){\makebox(0,0)[lb]{\smash{{\SetFigFont{12}{14.4}{\familydefault}{\mddefault}{\updefault}{\color[rgb]{0,0,0}$1$}%
}}}}
\put(2337,-2160){\makebox(0,0)[lb]{\smash{{\SetFigFont{12}{14.4}{\familydefault}{\mddefault}{\updefault}{\color[rgb]{0,0,0}$1$}%
}}}}
\put(1351,-1411){\makebox(0,0)[lb]{\smash{{\SetFigFont{12}{14.4}{\familydefault}{\mddefault}{\updefault}{\color[rgb]{0,0,0}\vpg}%
}}}}
\put(1590,390){\makebox(0,0)[lb]{\smash{{\SetFigFont{12}{14.4}{\familydefault}{\mddefault}{\updefault}{\color[rgb]{0,0,0}\hbf}%
}}}}
\put(3444,-1401){\makebox(0,0)[lb]{\smash{{\SetFigFont{12}{14.4}{\familydefault}{\mddefault}{\updefault}{\color[rgb]{0,0,0}\vph}%
}}}}
\put(3590,-563){\makebox(0,0)[lb]{\smash{{\SetFigFont{12}{14.4}{\familydefault}{\mddefault}{\updefault}{\color[rgb]{0,0,0}\vbi}%
}}}}
\put(3587,-1847){\makebox(0,0)[lb]{\smash{{\SetFigFont{12}{14.4}{\familydefault}{\mddefault}{\updefault}{\color[rgb]{0,0,0}\vbk}%
}}}}
\put(3592,-2669){\makebox(0,0)[lb]{\smash{{\SetFigFont{12}{14.4}{\familydefault}{\mddefault}{\updefault}{\color[rgb]{0,0,0}\vbl}%
}}}}
\put(3022,-3362){\makebox(0,0)[lb]{\smash{{\SetFigFont{8}{9.6}{\familydefault}{\mddefault}{\updefault}{\color[rgb]{0,0,0}\blue{$\rho$-th column}}%
}}}}
\put(1725,-3359){\makebox(0,0)[lb]{\smash{{\SetFigFont{8}{9.6}{\familydefault}{\mddefault}{\updefault}{\color[rgb]{0,0,0}\blue{$\tau$-th column}}%
}}}}
\end{picture}%